\pdfoutput=1
\documentclass[a4paper, reqno, 12pt, toc]{amsart}
\usepackage[margin=0.8in]{geometry}
\usepackage[utf8]{inputenc}
\usepackage[T1]{fontenc}
\usepackage{lmodern, csquotes, graphicx}
\usepackage{amssymb, physics, bm}
\usepackage{hyperref}
\usepackage[subsetnonamb, bbsets]{jkmath}
\usepackage[toc,page]{appendix}
\usepackage{enumitem} 
\usepackage[makeroom]{cancel}
\usepackage[backend=biber,style=alphabetic,sorting=nyt]{biblatex}
\usepackage{dynkin-diagrams}
\usepackage[symbols, sort=def, counter=subsection, toc=false]{glossaries-extra}

\makeatletter
\renewcommand\tableofcontents{%
    \section*{\contentsname}%
    \@starttoc{toc}%
}
\makeatother

\newtheorem{theorem}{Theorem}[section]
\newtheorem{definition}[theorem]{Definition}
\newtheorem{lemma}[theorem]{Lemma}
\newtheorem{proposition}[theorem]{Proposition}
\newtheorem{corollary}[theorem]{Corollary}
\newtheorem{example}[theorem]{Example}
\newtheorem{notation}[theorem]{Notation}
\newtheorem{remark}[theorem]{Remark}
\newtheorem{conjecture}[theorem]{Conjecture}
\newtheorem{thmintro}{Theorem}

\newtheorem{conjintro}{Conjecture}

\usepackage{tikz}
\usepackage{tikz-cd}
\usepackage{mathrsfs}
\usetikzlibrary{cd}
\DeclareMathOperator{\id}{id}
\DeclareMathOperator{\im}{im}
\DeclareMathOperator{\ad}{ad}
\DeclareMathOperator{\Ad}{Ad}
\DeclareMathOperator{\End}{End}
\DeclareMathOperator{\Hom}{Hom}
\DeclareMathOperator{\lot}{l.o.t.}
\DeclareMathOperator{\ct}{ct}
\DeclareMathOperator{\cl}{cl}
\DeclareMathOperator{\ch}{ch}
\DeclareMathOperator{\diag}{diag}
\newcommand{\uq}{\mathbf {U}}
\newcommand{\comp}{\mathscr{U}}

\newcommand{\uqb}{\mathbf{B}}
\newcommand{\uqbc}{\mathbf{B}_{\boldsymbol{c}}}
\newcommand{\uqbd}{\mathbf{B}_{\boldsymbol{d}}}
\newcommand{\uqbs}{\mathbf{B}_{\boldsymbol{c},\boldsymbol{s}}}
\newcommand{\uqds}{\mathbf{B}_{\boldsymbol{d},\boldsymbol{t}}}

\newcommand{\Addresses}{{
		\bigskip
		\footnotesize
		\textsc{IMAPP, Radboud Universiteit, P.O. Box 9010, 6500 GL Nijmegen, The Netherlands}\par\nopagebreak
		\textit{E-mail address}: \texttt{stein.meereboer@ru.nl}
        \par\nopagebreak
		\textit{E-mail address}: \texttt{philip.schloesser@ru.nl}
}}

\title{Quantum Matrix-Spherical Functions as Orthogonal Polynomials}
\author{Stein Meereboer \& Philip Schlösser}
\date{October 2025}

\makenoidxglossaries
\glsxtrnewsymbol[description={The sets of positive integers, integers, rational numbers, and complex numbers, respectively.}]{numbers}{\ensuremath{\N,\Z,\Q,\C}}
\glsxtrnewsymbol[description={An indeterminate.}]{q}{\ensuremath{q}}
\glsxtrnewsymbol[description={$q$-Number.}]{qn}{\ensuremath{[a]_{q^b}}}
\glsxtrnewsymbol[description={$q$-Operator}]{qo}{\ensuremath{[K;a]_{q^b}}}
\glsxtrnewsymbol[description={The algebraic closure of $\C(q)$.}]{k}{\ensuremath{k}}
\glsxtrnewsymbol[description={Two lattices that are perfectly paired.}]{XY}{\ensuremath{X,Y}}
\glsxtrnewsymbol[description={A root system.}]{R}{\ensuremath{\mathcal{R}}}
\glsxtrnewsymbol[description={The root lattice and weight lattice of $\mathcal{R}$.}]{QP}{\ensuremath{Q\le P \le X}}
\glsxtrnewsymbol[description={The coroot and coweight lattices of $\mathcal{R}$.}]{QPcheck}{\ensuremath{Q^\vee\le P^\vee}}
\glsxtrnewsymbol[description={A set of simple roots of $\mathcal{R}$.}]{alphas}{\ensuremath{\alpha_1,\dots,\alpha_n}}
\glsxtrnewsymbol[description={The corresponding simple coroots.}]{hs}{\ensuremath{h_1,\dots,h_n}}
\glsxtrnewsymbol[description={Lengths of the simple roots, i.e. $\epsilon_i=\frac{\abs{\alpha_i}^2}{2}$.}]{epsilons}{\ensuremath{\epsilon_1,\dots,\epsilon_n}}
\glsxtrnewsymbol[description={Half-sum of positive roots, or an appropriate generalisation thereof}]{rho}{\ensuremath{\rho}}
\glsxtrnewsymbol[description={Embedding of $\lambda\in X$ into the $\Q$-span of $Y$.}]{hlambda}{\ensuremath{h_\lambda}}
\glsxtrnewsymbol[description={The Weyl group of $\mathcal{R}$.}]{W}{\ensuremath{W}}
\glsxtrnewsymbol[description={The longest element of $W$.}]{w0}{\ensuremath{w_0}}
\glsxtrnewsymbol[description={The diagram automorphism such that $w_0=-\tau_0$.}]{tau0}{\ensuremath{\tau_0}}
\glsxtrnewsymbol[description={A quantum group, usually a Drinfel'd--Jimbo quantum group.}]{U}{\ensuremath{\uq}}
\glsxtrnewsymbol[description={A Hopf algebra dual to $\uq$.}]{A}{\ensuremath{\mathcal{A}}}
\glsxtrnewsymbol[description={Matrix element of the $\uq$-module $M$ associated to $v\in M$ and $f\in M^*$}]{cfv}{\ensuremath{c^M_{f,v}}}
\glsxtrnewsymbol[description={The subalgebra of $\uq$ generated by $K_h$ for $h\in Y$.}]{U0}{\ensuremath{\uq^0}}
\glsxtrnewsymbol[description={A coideal subalgebra of $\uq$, usually arising from a quantum symmetric pair.}]{B}{\ensuremath{\uqb}}
\glsxtrnewsymbol[description={An admissible pair}]{Itau}{\ensuremath{(I_\bullet,\tau)}}
\glsxtrnewsymbol[description={The involutions of $X,Y$ associated with an admissible pair.}]{Theta}{\ensuremath{\Theta}}
\glsxtrnewsymbol[description={The parabolic subgroup of $W$ associated to the subset $I_\bullet$ of simple roots.}]{Wblack}{\ensuremath{W_\bullet}}
\glsxtrnewsymbol[description={The sub-Hopf algebra of $\uq$ generated by $E_i,F_i,K_i,K_i^{-1}$ for $i\in I_\bullet$}]{Ublack}{\ensuremath{\uq_\bullet}}
\glsxtrnewsymbol[description={The longest element of $W_\bullet$.}]{wblack}{\ensuremath{w_\bullet}}
\glsxtrnewsymbol[description={The elements of $W$ that commute with $\Theta$.}]{WTheta}{\ensuremath{W_\Theta}}
\glsxtrnewsymbol[description={The restriction of the root $\alpha$, equals $\frac{\alpha-\Theta(\alpha)}{2}$.}]{atilde}{\ensuremath{\tilde{\alpha}}}
\glsxtrnewsymbol[description={The restricted root system.}]{Sigma}{\ensuremath{\Sigma}}
\glsxtrnewsymbol[description={The Weyl group of $\Sigma$, equals $W_\Theta/W_\bullet$.}]{WSigma}{\ensuremath{W_\Sigma}}
\glsxtrnewsymbol[description={The sublattice of $Y$ of $\Theta$-invariants.}]{YTheta}{\ensuremath{Y^\Theta}}
\glsxtrnewsymbol[description={The sublattice of $X$ of elements that annihilate $Y^\Theta$ and have integer pairings with $\Sigma^\vee$}]{L}{\ensuremath{2L}}
\glsxtrnewsymbol[description={The subset of $Y$ of regular values.}]{Yreg}{\ensuremath{Y_{\mathrm{reg}}}}
\glsxtrnewsymbol[description={The Lusztig braid group operator associated to an element $w\in W$.}]{Tw}{\ensuremath{T_w}}
\glsxtrnewsymbol[description={The bar involution of $\uq$ or of highest-weight modules of $\uq$.}]{bar}{\ensuremath{\overline{\cdot}}}
\glsxtrnewsymbol[description={The $\imath$bar involution of $\uqb$.}]{psii}{\ensuremath{\psi^\iota}}
\glsxtrnewsymbol[description={The quasi $K$-matrix.}]{Upsilon}{\ensuremath{\Upsilon}}
\glsxtrnewsymbol[description={The universal $K$-matrix.}]{K}{\ensuremath{\mathcal{K}}}
\glsxtrnewsymbol[description={The simple $\uq$-module with highest weight $\lambda$}]{Llambda}{\ensuremath{L(\lambda)}}
\glsxtrnewsymbol[description={An equivalence class of finite-dimensional
  $\uqb$-modules.}]{gamma}{\ensuremath{\gamma}}
\glsxtrnewsymbol[description={A representative of the class $\gamma$.}]{Vgamma}{\ensuremath{V(\gamma)}}
\glsxtrnewsymbol[description={A distinguished element of the centre of $\uq$.}]{cmu}{\ensuremath{c_\mu}}
\glsxtrnewsymbol[description={The $\gamma$-well, i.e. dominant elements $\lambda\in X^+$ such that $L(\lambda)|_\uqb$ contains exactly one copy of $\gamma$.}]{Xgamma}{\ensuremath{X^+(\gamma)}}
\glsxtrnewsymbol[description={The bottom of the $\gamma$-well.}]{Bgamma}{\ensuremath{\mathfrak{B}^+(\gamma)}}
\glsxtrnewsymbol[description={The space of matrix-spherical functions.}]{EVW}{\ensuremath{E^{V,W}}}
\glsxtrnewsymbol[see=EVW, description={}]{Egamma}{\ensuremath{E^{\gamma,\gamma'}}}
\glsxtrnewsymbol[description={The space of elementary matrix-spherical functions.}]{EVWM}{\ensuremath{E^{V,W}(M)}}
\glsxtrnewsymbol[see=EVWM, description={}]{Egammalambda}{\ensuremath{E^{\gamma,\gamma'}(\lambda)}}
\glsxtrnewsymbol[description={The distinguished elementary matrix-spherical function mapping 1 to 1.}]{Philambdagamma}{\ensuremath{\Phi^\lambda_\gamma}}
\glsxtrnewsymbol[description={The bilinear pairing of matrix-spherical functions that produces zonal spherical functions.}]{Xi}{\ensuremath{\Xi_{V,W}}}
\glsxtrnewsymbol[description={The Haar state.}]{h}{\ensuremath{h}}
\glsxtrnewsymbol[description={The restriction map.}]{Res}{\ensuremath{\Res}}
\glsxtrnewsymbol[description={A $k$-valued distribution, usually symmetric Macdonald weight.}]{triangledown}{\ensuremath{\triangledown}}
\glsxtrnewsymbol[description={The $\lambda,\lambda'$-entry of the matrix weight.}]{Hlambdagamma}{\ensuremath{H^{\lambda,\lambda'}_\gamma}}
\glsxtrnewsymbol[description={A family of matrix-valued orthogonal polynomials.}]{Qlambda}{\ensuremath{Q_\lambda}}
\glsxtrnewsymbol[description={The Macdonald data.}]{SSRRLL}{\ensuremath{(S,S',R,R',L,L')}}
\begin{document}
\begin{abstract}
    A general theory of matrix-spherical functions for dual Hopf algebras and
    right coideal subalgebras is developed. We establish their existence and
    define their orthogonality relations. When specialized to Kolb and Letzter's
    quantum symmetric pair coideal subalgebras we associate, to each classical
    commutative triple a unique corresponding quantum commutative triple. This
    leads to families of vector-valued orthogonal polynomials, which diagonalize
    a commutative algebra of difference-reflection operators and are invariant
    under sending $q\to q^{-1}$. Various examples of these vector-valued
    orthogonal polynomials are given and identified with the intermediate
    Macdonald polynomials previously introduced by the second author and we propose a general conjecture regarding their correspondence.
\end{abstract}

\subjclass[2020]{Primary 33D80; Secondary 33D52,  17B37} 
\keywords{Matrix-spherical functions, Quantum groups, Quantum symmetric pairs, Intermediate Macdonald polynomials.} 

\maketitle
\addtocontents{toc}{\protect\setcounter{tocdepth}{-1}}
\tableofcontents
\addtocontents{toc}{\protect\setcounter{tocdepth}{1}}
\vfill

\newpage
\section{Introduction}
\subsection{Background}
	The starting point of this article is the realization of 
zonal spherical functions on symmetric spaces $G/K$ as Jacobi polynomials. 
In his 1987 preprint of \cite{Mac00}, Ian~Macdonald 
tentatively asked whether an analogous description exists 
in the recently emerging field of quantum groups, and proposed
a set of polynomials as candidates.
In the following decade, various authors such as Tom~Koornwinder,
Masatoshi~Noumi, Tetsuya~Sugitani, and Mathijs~Dijhuizen indeed
managed to find examples of quantum analoga of symmetric pairs
whose zonal spherical functions were indeed the Macdonald polynomials
described in \cite{Mac00} or Koornwinder--Macdonald polynomials
(\cite{Ko92}), a generalisation for root systems of type $\mathsf{BC}$.
In later publications, e.g. \cite{Mac03}, both types of polynomials
have been unified and are now in full generality known as Macdonald
polynomials.

During this process there existed a plethora of ideas of what kind of algebraic object the subgroup $K$ would correspond to (as it became clear very
quickly that $K$ cannot correspond to a quantum group),
ranging from Koornwinder's twisted primitive element (see \cite{ko})
to the various two-sided coideals and two-sided coideal subalgebras
that are obtained from reflection equations in the work of
Dijkhuizen, Noumi, and Sugitani (\cite{NS95},\cite{No96},\cite{NDS}).

In the late 1990s, Gail~Letzter showed (\cite{Le99}) that all of these approaches
are equivalent and should be formulated in the language of
coideal subalgebras.
She then went on to develop a general theory of quantum symmetric
pairs (including but not limited to \cite{Let00},\cite{Let02},\cite{Let03}, \cite{Let04}), culminating in the general
identification of the zonal spherical functions of most quantum
symmetric pairs (i.e. with reduced restricted root system and standard
parameters) with Macdonald's polynomials.
More details and a survey of the literature on the remaining cases
can be found in Section~\ref{sec:zonal}.

In addition to having been studied in such detail,
the coideal subalgebra approach (\enquote{twisted quantised enveloping algebras}) has another advantage over,
say Noumi's and Sugitani's \enquote{additive} approach (see \cite[\S2.4]{No96}):
it generalises naturally to matrix-spherical functions
for other representations of the coideal subalgebra that corresponds to $K$.
The central question we
answer with this paper is whether these quantum
matrix-spherical functions can be described as vector-valued
generalisations of Macdonald polynomials.
Previous forays into this direction have been made in
\cite{Ald} and \cite{Ald17}, as well as \cite{Et94}, \cite{Ko96}, \cite{Mee25} and \cite{OS05} for
1-dimensional spaces of intertwiners.

In \cite{steinbergVariation}, a vector-valued generalisation
of the Jacobi polynomials (so-called Intermediate Jacobi Polynomials) was defined and exhibited as 
matrix-spherical functions of compact symmetric pairs (of groups)
for a few examples.
This approach was then adapted for Macdonald polynomials by Philip~Schlösser (\cite{Sch23})
to define Intermediate Macdonald Polynomials.



\subsection{Main results}
Let $\mathfrak{g}$ be a reductive Lie algebra and let $\theta$
be an involutive automorphism of $\mathfrak{g}$.
Let $\mathfrak{k}:=\mathfrak{g}^\theta$ be the subalgebra fixed
by $\theta$.
The pair $(\mathfrak{g},\mathfrak{k})$ or of their universal
enveloping algebras is a \emph{symmetric pair}.

With $(\mathfrak{g},\mathfrak{h})$ the theory of quantum symmetric pairs, as developed by Letzter and Kolb \cite{Le99} \cite{Kol14}, associates a pair
$(\uq,\uqb)$ of associative unital algebras ($\uq$ is a Drinfel'd Jimbo quantum group 
and $\uqb$ is a right coideal subalgebra) over $k=\overline{\C(q)}$
whose $q=1$ specialisations are isomorphic to
$(U(\mathfrak{g}),U(\mathfrak{k}))$.
Such a pair $(\uq,\uqb)$ is called a \emph{quantum symmetric pair}, and
$\uqb=\uqbs$ depends on parameters $(\bm{c},\bm{s})$.

\subsubsection{Main results Part \ref{part:1}}
The main result of Part \ref{part:1} is the development of a general theory of matrix spherical functions.

Write $\mathcal{A}$ for the algebra of finite-dimensional
matrix coefficients of $\uq$.
Let $(\uq,\uqb)$ and $(\uq,\uqb')$ be quantum symmetric pairs
corresponding to potentially different choices of parameters.
Let $V$ be a finite-dimensional $\uqb$-module, and
$W$ a finite-dimensional $\uqb'$-module.
A function $f\in\mathcal{A}\otimes\Hom(W,V)$ is called a
\emph{matrix spherical function} if
\[
\forall b\in\uqb,x\in\uq,b'\in\uqb':\quad
f(bxb') = bf(x)b'.
\]
For a simple $\uq$-module $M$, we define the
subspace $E^{V,W}(M)$ of elementary matrix spherical functions,
i.e. elements of $f\in E^{V,W}$ contained in $\mathcal{A}(M)\otimes \Hom(W,V)$,
where $\mathcal{A}(M)$ consists of the matrix elements of $M$.
$E^{V,W}$ is the direct sum of $E^{V,W}(M)$ with $M$ ranging over
finite-dimensional simple modules.
In particular, if $V=W=k$ are the restrictions of the trivial
representation $\epsilon$ of $\uq$ (the counit), we call $f\in E^{\epsilon,\epsilon}(M)$ a
\emph{zonal spherical function}. In general, the set $E^{V,W}$ of such matrix spherical functions is
a vector space and a left module over the algebra $E^{\epsilon,\epsilon}$. 
In the setting of quantum symmetric pairs, this setup for zonal spherical functions is analogous to \cite{Let03}.

Moreover, we construct a $k$-bilinear pairing (see Definition 
\ref{def:XI})
\begin{equation}\label{eq:introfact}
	\Xi_{V,W}: E^{V,W}\otimes E^{W,V}\to E^{\epsilon,\epsilon},
\end{equation}
which can be combined with the Haar state $h$ on $\mathcal{A}$ to
give a non-degenerate bilinear pairing $h\circ\Xi_{V,W}: E^{V,W}\otimes E^{W,V}\to k$ with respect to which that the
$E^{V,W}(M)$, for varying non-isomorphic $M$, are orthogonal (see Proposition \ref{prop:nondegen}).

\begin{thmintro}[Proposition \ref{prop:nondegen}]
	Let $M$ be a simple $\uq$-module that is
	semisimple over $\uqb$ and $\uqb'$ with $\tr(K_M)\neq0$ and 
	let
	$\Phi\in E^{V,W}$ be a nonzero elementary matrix-spherical function
	for $M$. 
	Then there is $\Psi\in E^{W,V}$ elementary for $M$ such that $h\circ \Xi_{V,W}(\Phi,\Psi)=\frac{\tr(K_M|_V)\tr(K_M|_W)}{\tr(K_M)}$.
\end{thmintro}
This is a generalization the construction of \cite{Mee25} to higher dimensional representations and more general pairs of Hopf algebras and coideal subalgebras.

With respect to the $E^{\epsilon,\epsilon}$-module structure of
$E^{V,W}$, $\Xi_{V,W}$ interacts as follows:
\[
\Xi_{V,W}(fF, gG)
= f \Xi_{V,W}(F,G) g(S(\cdot)K_{-2\rho}),
\]
where $K_{-2\rho}$ corresponds to the element $-2\rho$ (sum of
negative roots) embedded into the Cartan subalgebra of
$\mathfrak{g}$.
Standard quantum group theory defines a commutative algebra $\uq^0\subset \uq$
corresponding to the Cartan subalgebra.
We show that $\uqb\uq^0\uqb'$ is Zariski dense in $\uq$ (see Theorem \ref{thm:cartandecompthm}), thus showing an analogue of the Cartan decomposition
for quantum symmetric pairs. 
\begin{thmintro}[Theorem \ref{thm:cartandecompthm}]\label{thm:A}
	Let $x\in\uq$.
	There exists a Zariski open subset $Y_{\mathrm{reg}}(x)\subset Y$
	such that $K_hx\in\uqb'\uq^0\uqb$.
\end{thmintro}
As a consequence
every matrix spherical function $f$ is defined by its restriction $\Res(f)$ to
$\uq^0$ (see Proposition \ref{prop:resa}), which is an element of $k[X]\otimes\Hom(W,V)$, where
$X$ is a lattice of weights of $\mathfrak{g}$.

In particular, there exists a matrix weight distribution
$\triangledown\in k[[X]]\otimes\End(\Hom(V,W))$ such
that
\[
h(\Xi_{V,W}(F,G)) = \tr_W\ct((-\rho)\triangleright\Res(F)
\triangledown \overline{(-\rho)\triangleright\Res(G)})
\]
(see Proposition \ref{prop-matrix-weight}), where $\ct:k[[X]]\to k$, the
constant term map,
maps $\sum_{\mu\in X}a_\mu e^\mu$ to $a_0$, and
where $(-\rho)\triangleright e^\mu = q^{-\langle\rho,\mu\rangle}e^\mu$ effects a shift by $-\rho$.
In essence, the symmetric bilinear pairing $h\circ\Xi_{V,W}$
thus factors through $(-\rho)\triangleright\Res(\cdot)$.
Using \eqref{eq:introfact}, the factorization of $h\circ\Xi_{V,W}$ through $(-\rho)\triangleright\Res(\cdot)$ of can be further refined by noting that
$E^{V,W}$ and $E^{W,V}$ are often free modules over $E^{\epsilon,\epsilon}$ (see Lemma \ref{lem:lin-indep}).
In this case the orthogonality relations are determined by the generators of the $E^{\epsilon,\epsilon}$-modules $E^{V,W}$, $E^{W,W}$ and the orthogonality relations for $E^{\epsilon,\epsilon}$. The orthogonality relations for $E^{\epsilon,\epsilon}$ can often be deduced from \cite{Let04} as they then coincide with the orthogonality relations of the corresponding Macdonald polynomials.

These pleasant module structures arise in the context of quantum commutative triples $(\uq,\uqb,V)$, which are defined by the condition that 
\[
\forall\lambda\in X^+:\quad \dim(\Hom_{\uqb}(V,L(\lambda))),
\dim(\Hom_{\uqb}(L(\lambda),V))\le 1.
\]
Using the specialization of $\uqb$-modules as developed in \cite{Let00} and the techniques from Watanabe's theory of integrable modules \cite{Wat2025}, we prove that the classification of these triples is independent of the quantization parameter $q$.

For any weight $\lambda\in X^+$ and $\uqb$-submodule $V\subset L(\lambda)$, we let $\widetilde{V}$ denote the $U(\mathfrak{k})$-submodule of $L(\lambda)$ as defined in \cite{Let00}. 
Let $\Gamma=\Gamma(\uqb)$ be the set of equivalence classes of 
finite-dimensional irreducible $\uqb$-modules that are contained
in the restriction of a finite dimensional $\uq$-module. 
For each class $\gamma \in\Gamma$, we fix a representative $V(\gamma)$ and define the set of   $\gamma$-spherical weights as $X^+(\gamma)=\{\lambda\in X^+\,:\, V(\gamma)\subset L(\lambda)\}$. To characterize the minimal elements in these sets, we define the \emph{bottom of the $\gamma,\gamma'$-well} as
\[
\mathfrak{B}^+(\gamma,\gamma')=\set{\lambda \in X^+(\gamma)\cap X^+(\gamma')\where
	\forall \mu\in 2L^+= X^+(\epsilon):\quad \lambda-\mu \notin X^+(\gamma)\cap X^+(\gamma')}.
\]
We write $ \mathfrak{B}^+(\gamma)= \mathfrak{B}^+(\gamma,\gamma)$.
In the special case where $\mathfrak{B}^+(\gamma)$ contains one element, we call
$\gamma$ an integrable small $\uqb$-type.
\begin{thmintro}[Theorem \ref{thm-bottom-pit}]\label{thmintro:C}
	Let $(\uq,\uqb,\gamma)$ be a quantum commutative triple
	then $X^+(\gamma)=\mathfrak B^+(\gamma)+2L^+$.
	Furthermore, if $(\uq,\uqb',\gamma')$ is a
	quantum commutative triple with different parameters and $\cl(\gamma)=\cl (\gamma')$, we have
	$X^+(\gamma)=X^+(\gamma')$.
	In particular, $[L(\lambda)_{\uqb}:V(\gamma)] = [\widetilde{L(\lambda)}|_{U(\mathfrak{k})}:V(\cl(\gamma))]$, i.e. the branch rules
	are preserved by specialisation in $q=1$.
\end{thmintro}
Consequently, we obtain a wide range of quantum commutative triples, as many are already known in the classical case, in particular, those classified in \cite{Pe23}.

As a consequence of Theorem \ref{thmintro:C} each elementary $\gamma$-spherical function $\Phi_\gamma^{b+\lambda}\in \mathcal{A}\otimes \End(V(\gamma))$ admits a unique expansion
$\Phi_\gamma^{b+\lambda}=\sum_{b\in\mathfrak{B}(\gamma)} \phi_b \Phi_\gamma^b$
where $\Phi_\gamma^b\in \mathcal{A}\otimes \End(V(\gamma))$ and for coefficients $\phi_b\in E^\epsilon$.
Using this expatiation, we define the vector $\underline{\Phi_\gamma^{b+\lambda}}=(\phi_b)_{b\in \mathfrak{B}(\gamma)} $ (see Proposition \ref{prop-matrix-weight}).
For any $\lambda\in 2L^+$, this allows us to introduce the matrix-valued polynomial $Q_\lambda\in \mathbb{M}_{\mathfrak{B}(\gamma)}\otimes k[2L]^{W_\Sigma}$ as
\[
Q_\lambda := \qty((-\rho)\triangleright\Res
\qty(\underline{\Phi^{b+\lambda}_\gamma}))_{b\in\mathfrak{B}(\gamma)}.
\]
Explicitly, $Q_\lambda$ is the matrix-valued polynomial whose $b$-th
column is $(-\rho)\triangleright
\Res\qty(\underline{\Phi^{b+\lambda}_\gamma})$. 

We denote $^\circ: k[2L]\to k[2L]$ the anti-linear
involution mapping $e^\lambda\mapsto e^\lambda$ for all
$\lambda\in X$ which we canonically extend to $  \mathbb{M}_{\mathfrak{B}(\gamma)}\otimes k[2L]^{W_\Sigma}$.
For quantum commutative triples, we can leverage techniques from the $\imath$quantum group program, initiated by Bao and Wang, to derive symmetries of the vector-valued orthogonal polynomials $Q_\lambda$.
For each the involution $\sigma=\tau\circ\tau_0$ (see Section \ref{sec:barinv}) acts on on the set of quantum commutative triples $\Gamma$. In fact, we conjecture ( see Conjecture \ref{conj:multfree} ) that $\sigma$ acts trivially on $\Gamma$.
The universal $K$-matrix $\mathcal{K}: L(\lambda) \to L(\lambda)$ \cite{Ko19}, under the condition that $\gamma=\sigma(\gamma)$, intertwines the $\uqb$-action on $L(\lambda)$. In the context of quantum commutative triples, it therefore acts by scalar multiplication on the $\uqb$-type. This invariant action, at the level of the associated vector-valued orthogonal polynomials, results in a $q \to q^{-1}$ symmetry (see Theorem \ref{thm:polyqtoq}).

\begin{thmintro}[Theorem \ref{thm:polyqtoq}]
	If $\gamma\cong \sigma(\gamma)$, then for each $\lambda \in 2L^+$ we have $Q_\lambda=Q_\lambda^\circ$.
\end{thmintro}

The quantum Cartan decomposition (see Theorem \ref{thm:A}) implies that every (elementary) matrix spherical
function is a vector-valued polynomial that diagonalises a
commutative algebra of matrix difference operators 
(the radial parts of the centre of $\uq$, see Theorem \ref{thm:qdfi}). 

\begin{thmintro}[Theorem \ref{thm:qdfi}]
	For each $\Omega\in \uq^{\uqb}$ the matrix-valued orthogonal polynomials $\{Q_\lambda:\lambda \in 2L^+\}$ satisfy the matrix-valued $q$-differential equation 
	\begin{equation}
		Q_\lambda \Lambda_{\Omega,\lambda}=\mathcal{D}_{\Omega }(Q_\lambda)
	\end{equation}
	where $\Lambda_{\lambda,\Omega}$ is a diagonal matrix and $\mathcal{D}_{\Omega }$ is a matrix-valued $q$-differential obtained from Theorem \ref{thm:A}.
\end{thmintro}

The procedure to calculate the radial parts is similar in spirit to the classical case \cite{Ca82}.

\subsubsection{Main results Part \ref{part:2}}
In Part \ref{part:2} we consider instructive examples. The main result is the computation of the
corresponding matrix weights, and the identification of the matrix
spherical functions with Intermediate Macdonald polynomials. 
These were first introduced in \cite{Sch23} by one of the authors,
and are a parameter-dependent family of Laurent polynomials in
$k[2L]$ that are invariant under a parabolic subgroup $W_J$ of the
Weyl group $W_\Sigma$ (symmetric and non-symmetric Macdonald
polynomials are special cases for $J=I$ and $J=\emptyset$).
Like the Macdonald polynomials, they diagonalise a commutative
algebra of $q$-difference operators,
and as is demonstrated in \cite[\S6]{Sch23}, there are straightforward ways
of interpreting them vector-valued polynomials.

As such, the results of Part \ref{part:2} can be seen as generalizations of Letzter's identification of zonal spherical functions with Macdonald polynomials \cite{Let04} by considering a more general notion of spherical functions, and then obtaining a corresponding generalization of the Macdonald polynomials. 
The contents of Part \ref{part:2} is very close to \cite{steinbergVariation} and \cite{vHor} in spirit, where in the classical case matrix-spherical functions are identified with Intermediate Jacobi polynomials.
In particular, we consider
\begin{enumerate}
	\item $(\uq,\uqb)$ of type $\mathsf{BII}_n$, corresponding to
	$Spin(2n+1)/Spin(2n)$ with a representation $\gamma$
	corresponding to $s\varpi_{n-1}$ and $s\varpi_n$, where
	$\varpi_{n-1},\varpi_n$ correspond to the spin representations of $Spin(2n)$
	\item $(\uq,\uqb)$ of type $\mathsf{CII}_{n,1}$, corresponding
	to $Sp(2n)/(Sp(2)\times Sp(2n-2))$ with a representation
	$\gamma$ corresponding to $s\varpi_1$, where $\varpi_1$
	is the standard representation of the $Sp(2)$ factor.
	\item $(\uq,\uqb)$ of type $\mathsf{DII}_n$, corresponding
	to $Spin(2n)/Spin(2n-1)$ with a representation $\gamma$
	corresponding to the spin representation of $Sp(2n-1)$.
	\item $(\uq,\uqb)$ of type $\mathsf{AI}_2$, corresponding to $SU(3)/SO(3)$, with
	a representation $\gamma$ corresponding to the standard representation of $SO(3)$.
	\item $(\uq,\uqb)$ of type $\mathsf{A}_2$-group,
	corresponding to $(SU(3)\times SU(3))/SU(3)$, with a representation
	$\gamma$ corresponding to the standard representation of $SL(3)$.
	\item $(\uq,\uqb)$ of type $\mathsf{AII}_5$, corresponding
	to $SU(6)/Sp(6)$, with a representation $\gamma$ corresponding
	to the standard representation of $Sp(6)$.
\end{enumerate}
For the first two cases, we obtain symmetric Askey--Wilson polynomials
(i.e. Intermediate Macdonald polynomials of rank 1, symmetric under the entire
Weyl group).
In the third case, we obtain non-symmetric Askey--Wilson polynomials
(i.e. Intermediate Macdonald polynomials of rank 1, symmetric under
the trivial group).
And in the last three cases, we obtain Intermediate Macdonald polynomials
for the root system $\mathsf{A}_2$ that are symmetric under
a non-trivial parabolic subgroup of the Weyl group. It is worth to mention that we are able to make the identification with Intermediate Macdonald polynomials by matching the orthogonality relations using the inner product methods, as developed in Part \ref{part:1}, and that (v) is a special case studied classically in \cite{Koe20}.

With these examples, we can formulate a general correspondence in rank 1
between matrix spherical functions and intermediate Macdonald polynomials (in this case:
symmetric and non-symmetric Askey--Wilson polynomials),
and in higher rank between matrix spherical functions of \enquote{small $\uqb$-types} (Definition~\ref{def:bottom}) and symmetric Macdonald polynomials.
Beyond that, we can make no general statement (except for the three examples (iv), (v), (vi)), 
and formulate a conjecture (Conjecture~\ref{conj:A}) instead. 
In Section \ref{sec-conc-examples} we describe an action of $W_\Sigma$ on $\mathfrak{B}(\gamma)$, which is used to formulate Conjecture \ref{conj:A}.

\begin{conjintro}[Conjecture \ref{conj-imp}]\label{conj:A}
	Assume that the action of $W_\Sigma$ on $\mathfrak{B}(\gamma)$
	is transitive and one
	of the stabilisers $W_J$ is a parabolic subgroup of $W_\Sigma$.
	Then the $-\rho$-shifted restrictions of MSF for $(\uq,\uqb,\gamma)$ are Intermediate Macdonald polynomials invariant under
	$W_J$.
\end{conjintro}

\subsection{Structure}
This paper is divided into two parts. Part \ref{part:1} consists of Section \ref{sec:msfha} - \ref{sec:opandex}, here we develop the general theory of quantum matrix-spherical functions. Part \ref{part:2} consists of Section \ref{sec-examples} and Appendix \ref{ap:A} - \ref{ap:small}, and is devoted to the explicit study of matrix spherical functions in a number of specific cases.

We begin in Section~\ref{sec:msfha} in the most general setting by reviewing the definitions of a Hopf algebra and a right coideal subalgebra. 
In this setting, we define the notion of a (matrix) spherical function 
and show some basic properties thereof, including orthogonality relations.

In Section~\ref{sec:Dinfeld} we then specialise by introducing the
concrete Hopf algebras and coideal subalgebras that will be considered 
in the rest of the paper: Drin'feld--Jimbo quantum groups and
quantum-symmetric pair coideal subalgebras.
In their representation theory, they correspond to compact symmetric pairs, which we also make concrete.
In that setting we develop some further techniques to describe the orthogonality relations between the matrix-spherical functions.

In Section~\ref{sec:cartandecomp} we utilise representation theory to prove an analogon of the
$KAK$ decomposition for symmetric pairs of groups.
This allows us to describe spherical functions by their restrictions to
the Cartan torus.
The representation theory part involves proving a statement for all
rank-1 diagram that is true for all finite-type rank-1 diagrams but fails already for some affine diagrams.
Therefore, we restrict to Cartan data of finite type from now on.

Section~\ref{sec:zonal} deals with zonal functions and the extensive
literature about them.
We collect and unify the most important results (for our purposes) from
said literature.
In particular, we are concerned with the cases for which the zonal spherical functions are known to
be Macdonald polynomials

In Section~\ref{sec:opandex} we associate of vector-valued polynomials to matrix-spherical functions. We combine the results of Sections \ref{sec:msfha} -- \ref{sec:zonal} to study the structural properties of these polynomials. In particular, with use of the quantum Cartan decomposition of Section \ref{sec:cartandecomp} the existence a of matrix valued $q$-difference equation for the associated vector-valued polynomials is shown, Moreover, we show the $q\mapsto q^{-1}$ invariance of the associated matrix valued orthogonal polynomials.

Section~\ref{sec-examples} is the first section of Part \ref{part:2} and is concerned with the examples.
The inner product from Section~\ref{sec:msfha} and some ad-hoc computations
are used to describe the matrix-spherical functions as Intermediate
Macdonald polynomials.
It is a conjecture of the authors that this can be done whenever the
reduced Weyl group acts transitively on the $M$-types (respectively,
the bottom elements of the well) and
the stabiliser subgroup is parabolic.
Using the examples we show that this is true for all analogues of small
$K$-types (that are integrable, i.e. can be described using our formalism),
for all rank-1 cases, and for all cases with restricted root system of type
$\mathsf{A}_2$ with three $M$-types.

\begin{center}	$\textsc{Acknowledgment}$
\end{center}
This work is funded by grant number \texttt{OCENW.M20.108} of the
Dutch Research Council.
The authors would like to thank Jasper Stokman for his help with
assembling the literature on zonal spherical functions.
They would like to thank Erik Koelink and Maarten van Pruijssen
for valuable feedback and useful discussions. Moreover, they would like to thank Tobias Simon for the suggestion to study small $K$-types.
\newpage
\part{Generalities on quantum matrix-spherical functions}\label{part:1}
\section{Matrix-Spherical Functions for Hopf Algebras and Coideal Subalgebras}\label{sec:msfha}
In this section we are going to define the notions of quantum groups,
coideal subalgebras, and matrix-spherical functions in great generality, without mentioning
root systems or Satake diagrams.
This generality is assumed in order to highlight the elementary nature of the constructions and proofs, and for potential generalizations.

\subsection{General Definitions}\label{sec:gen-def}
Let $k$ be an algebraicly closed field. We write \gls{numbers} for the sets of natural numbers (without 0), integers, rationals, and complex numbers. 

\begin{definition}
    A \emph{Hopf algebra} over $k$ is a tuple $(\gls{U},\Delta,\epsilon,S)$ of
    \begin{enumerate}
        \item an associative unital $k$-algebra $\uq$,
        \item an algebra homomorphism $\Delta: \uq\to\uq\otimes\uq$ (\enquote{comultiplication}), usually written in \emph{Sweedler notation} as
        \[
            \Delta(x) = \sum x_{(1)}\otimes x_{(2)}
            = x_{(1)}\otimes x_{(2)},
        \]
        \item an algebra homomorphism $\epsilon: \uq\to k$
        (\enquote{counit}), and
        \item an algebra antihomomorphism $S: \uq\to \uq$ (\enquote{antipode})
    \end{enumerate}
    such that
    \begin{align*}
        (\Delta\otimes\id)\circ\Delta &= (\id\otimes\Delta)\circ\Delta : \uq\to\uq\otimes\uq\otimes\uq\\
        (\id\otimes\epsilon)\circ\Delta &= (\epsilon\otimes\id)\circ\Delta = \id: \uq\to\uq\\
        m\circ(\id\otimes S)\circ\Delta &= m\circ (S\otimes\id)\circ\Delta = 1\cdot \epsilon:\uq\to\uq
    \end{align*}
    (where $m:\uq\otimes\uq\to\uq$ is the multiplication map).
\end{definition}

\begin{remark}
    The axioms can be written in Sweedler notation as
    \begin{align*}
        x_{(1)(1)}\otimes x_{(1)(2)}\otimes x_{(2)}
        &= x_{(1)}\otimes x_{(2)(1)}\otimes x_{(2)(2)}
        = x_{(1)}\otimes x_{(2)}\otimes x_{(3)}\\
        \epsilon(x_{(1)})x_{(2)} &= x_{(1)}\epsilon(x_{(2)}) = x\\
        x_{(1)}S(x_{(2)}) &= S(x_{(1)})x_{(2)} = \epsilon(x)
    \end{align*}
    for all $x\in\uq$.

    Note that it is exactly the extra structure of a Hopf algebra that
    allows us to equip the category of $\uq$-modules with the structure of
    a rigid monoidal category.
\end{remark}

\begin{definition}\label{def:moduleactions}
    Define the \emph{adjoint representation} $\ad:\uq\to\End(\uq)$ as
    follows:
    \[
        \ad(x)(v) := x_{(1)}vS(x_{(2)}).
    \]
    Let $M$ be a module over a Hopf algebra $\uq$, let $v\in M, f\in M^*$,
    then define
    \[
        c_{f,v}=\gls{cfv}: \uq\to k,\qquad
        x\mapsto f(xv).
    \]
    This is the \emph{matrix element} for the module $M$ obtained from
    $f,v$.
\end{definition}

For a pairing $\langle \cdot,\cdot\rangle:\uq\times \mathcal A\to k$, we equip the tensor products with the natural pairing $\langle \cdot,\cdot\rangle:\uq\otimes \uq \times \mathcal A\otimes \to k$ defined by
\[\langle f\otimes g,a\otimes b\rangle =\langle f,a\rangle\langle g,b\rangle,\qquad f,g\in \uq,\quad a,b\in \mathcal A.\]

\begin{definition}
    Two Hopf algebras $\uq$ and \gls{A} are said to be
    \emph{dual} to each other if there is a nondegenerate pairing
    $\langle\cdot,\cdot\rangle: \mathcal{A}\times\uq\to k$ (that we will tacitly also use to denote the pairing of $\mathcal{A}\otimes\mathcal{A}$ with $\uq\otimes\uq$ induced by it) such that
    \begin{enumerate}
        \item $\langle fg,x\rangle = \langle f\otimes g,\Delta(u)\rangle$ for $f,g\in\mathcal{A},x\in\uq$.
        \item $\langle f,xy\rangle = \langle\Delta(f),x\otimes y\rangle$ for $f\in\mathcal{A},x,y\in\uq$.
        \item $\langle 1,x\rangle = \epsilon(x)$ for $x\in\uq$.
        \item $\langle f,1\rangle = \epsilon(f)$ for $f\in\mathcal{A}$.
        \item $\langle f,Sx\rangle = \langle Sf,x\rangle$ for
        $f\in\mathcal{A},x\in\uq$.
    \end{enumerate}
    The Hopf algebra $\mathcal{A}$ is said to contain the matrix
    element $c_{f,v}$ of a $\uq$-module $M$ if there is a (necessarily
    unique) element $h\in\mathcal{A}$ such that
    \[
        \forall x\in\uq: \langle h,x\rangle = c_{f,v}(x) = f(xv).
    \]
    We shall also denote that element by $c_{f,v}$ and more generally
    think of $\mathcal{A}$ as a suitable subset of $\uq^*$ (and
    hence interpret the pairing $\langle\cdot,\cdot\rangle$ as
    the evaluation map).
\end{definition}

\begin{proposition}
    Let $\uq,\mathcal{A}$ be dual Hopf algebras. We can define left
    and right actions of $\uq$ on $\mathcal{A}$ as follows:
    \begin{align*}
        x\triangleright f &:= (\id\otimes x)\Delta(f)\\
        f\triangleleft x &:= (x\otimes\id)\Delta(f)
    \end{align*}
    where we identify $x\in\uq$ with the map $f\mapsto \langle f,x\rangle$, for $f\in\mathcal{A}$.
\end{proposition}
\begin{proof}
    We have
    \begin{align*}
        xy\triangleright f &= (\id\otimes xy)\Delta(f)
        = (\id\otimes x\otimes y)(\id\otimes\Delta)(\Delta(f))\\
        &= (\id\otimes x\otimes y)(\Delta\otimes\id)(\Delta(f))
        = (\id\otimes x)\Delta((\id\otimes y)\Delta(f))\\
        &= x\triangleright (y\triangleright f)\\
        1\triangleright f &= (\id\otimes 1)\Delta(f) = f
    \end{align*}
    and similarly for $\triangleleft$.
\end{proof}

\begin{example}\label{ex-maximal-dual}
    A common choice for a dual $\mathcal{A}$ of $\uq$ is the set
    \[
        \set{f\in\uq^*\where\exists I\le\uq: \dim(\uq/I)<\infty, f(I)=0}
    \]
    (where $I$ is a two-sided ideal), which can be shown
    (see \cite[Corollary~1.4.5]{Jos95}) to be generated by
    matrix elements of finite-dimensional representations.
    
    This is the maximal choice of $\mathcal{A}$: any other
    choice $\mathcal{A}'$ can be interpreted as a subset of $\uq^*$ satisfying $\Delta(\mathcal{A}')\subset\mathcal{A}'\otimes\mathcal{A}'$. By the proof of 
    \cite[Lemma~1.4.2]{Jos95} and by \cite[Lemma~1.4.1(i)]{Jos95} we then find that $\mathcal{A}'\subset\mathcal{A}$
    as defined above. In particular, any choice of
    $\mathcal{A}$ is generated by matrix elements of
    finite-dimensional modules.
\end{example}

\begin{remark}
    In case $\uq$ has infinite-dimensional simple representations,
    the matrix elements of finite-dimensional representations might
    not suffice to separate the elements of $\uq$. In other words:
    there might be no Hopf algebra $\mathcal{A}$ dual to $\uq$. In
    that case, we can take $\mathcal{A}\subset \uq^*$ to be the vector
    space of matrix elements of semisimple representations (or of
    another appropriate (full) subcategory of $\uq-\operatorname{Rep}$).

This choice $\mathcal{A}$ will not generally be a bialgebra (much less a
    Hopf algebra), 
    but it will still be an algebra with a left and right
    action of $\uq$ (via $uc_{f,v}=c_{f,uv}$ and $c_{f,v}u=c_{fu,v}$
    for $v\in M,f\in M^*, u\in\uq$), so the notions from
    Sections~\ref{sec:gen-def}\ref{sec:gen-msf} still make sense. 
    However, in that case we lack an important tool in the study of
    the matrix-spherical functions: the inner product.
    The Haar state whose existence we assume from Lemma~\ref{lem-quantum-schur-orthogonality} onwards only exists in case
    $\mathcal{A}$ is a cosemisimple Hopf algebra (\cite[Theorem~11.2.13]{ks}).
    
    Moreover, as \cite{Ko15} suggests, the proper treatment of 
    such cases, e.g. $\uq$ being the quantum analogue of an
    affine Kac--Moody algebra, these cases likely also require
    appropriate completions that we don't introduce here.
\end{remark}

\subsection{Matrix-Spherical Functions}\label{sec:gen-msf}
From now on, we shall fix a dual pair $\uq,\mathcal{A}$ of Hopf algebras.
\begin{definition}\label{def:coideal}
    A \emph{right coideal subalgebra}
    is a $k$-subalgebra $\gls{B}\le\uq$ that is a right coideal, i.e.
    \[
        \Delta(\uqb) \subset\uqb\otimes\uq.
    \]
    Write $\Gamma=\Gamma(\uqb)$ for the set of equivalence classes of 
    finite-dimensional irreducible $\uqb$-modules that are contained
    in the restriction of a finite dimensional $\uq$-module. For every $\gls{gamma}\in\Gamma$
    write $\gls{Vgamma}$ for a representative.
\end{definition}

\begin{definition}\label{def:sph}
    Let $\uqb,\uqb'$ be right coideal subalgebras. 
    Let $V$ be a finite dimensional $\uqb$-module, $W$ a finite dimensional $\uqb'$-module. 
    An element $f\in\mathcal{A}\otimes\Hom(W,V)$ is called a  
    \emph{$(\uqb,\uqb')$-matrix-spherical function} (MSF) for $V,W$ if
    \[
        \forall b\in\uqb,b'\in\uqb',x\in \uq:\quad
        f(bxb') = bf(x)b'
    \]
    (in other words, we require $b'\triangleright f = f\pi_W(b')$
    for $b'\in\uqb$ and $f\triangleleft b=\pi_V(b)f$ for $b\in\uqb'$).
    Write $E^{V,W}_{\uqb,\uqb'}=\gls{EVW}$ (or $\gls{Egamma}$ if $V\cong V(\gamma),W\cong V(\gamma')$ for $\gamma,\gamma'\in\Gamma$) for the vector space of MSF for $(\uq,\uqb,V,\uqb',W)$.

    For $V=W=k$ and the module structures are given by the counit
    $\epsilon$, we write $E^\epsilon:= E^{\epsilon,\epsilon}_{\uqb,\uqb'}$ regardless of whether $\uqb=\uqb'$ or not.
    
    In all other cases, we write $E^V:=E^{V,V}$ and
    $E^\gamma:=E^{\gamma,\gamma}$ which implies that $\uqb=\uqb'$.
\end{definition}

\begin{example}\label{ex-emsf}
    Let $M$ be a simple $\uq$-module whose matrix elements are 
    contained in $\mathcal{A}$ and $V,W$ simple 
    finite-dimensional modules over $\uqb,\uqb'$, respectively.
    Let $j\in\Hom_{\uqb'}(W,M), p\in\Hom_{\uqb}(M,V)$.
    Then $f: \uq\to\Hom(W,V)$, given by
    \[
        f(x)(w) := p(xj(w)),
    \]
    is a MSF. 
    Any such function is called an \emph{elementary} MSF.
    Write $E^{V,W}_{\uqb,\uqb'}(M)=\gls{EVWM}$ for
    the vector space of elementary MSF for
    $M$, and call
    the elements of $E^{\epsilon}_{\uqb,\uqb'}(M)$ \emph{($\uqb,\uqb'$)-zonal spherical functions} (ZSF).

\end{example}

\begin{lemma}\label{lem-elementary-msf-basis-kinda}
    \begin{enumerate}
        \item Let $M$ be a simple $\uq$-module and $V,W$ finite-dimensional modules over $\uqb,\uqb'$, respectively.
        Any function $f\in E^{V,W}$ that can be written in terms of matrix
        elements of $M$ is contained in
        $E^{V,W}(M)$.
        \item Assume that every finite-dimensional
        $\uq$-module splits as a direct sum of simple modules.
        Then
        \[
            E^{V,W} = \bigoplus_{M\in\widehat{\uq}}
            E^{V,W}(M),
        \]
        where $\widehat{\uq}$ denotes the set of equivalence classes of simple finite-dimensional $\uq$-modules.
    \end{enumerate}
\end{lemma}
\begin{proof}
    \begin{enumerate}
        \item Write
    \[
        f = \sum_{i=1}^r c_{f_i,v_i}\otimes m_i\otimes \lambda_i
    \]
    (viewing $\Hom(W,V)\cong V\otimes W^*$)
    for $v_i\in M,f_i\in M^*,m_i\in V,\lambda_i\in W^*$.
    Then the condition of being an MSF just means that
    \begin{align*}
        \sum_{i=1}^r c_{f_ib,v_i}\otimes m_i\otimes\lambda_i
        &= \sum_{i=1}^r c_{f_i,v_i}\otimes bm_i\otimes \lambda_i\\
        \sum_{i=1}^r c_{f_i,b'v_i}\otimes m_i\otimes \lambda_i &=
        \sum_{i=1}^r c_{f_i,v_i}\otimes m_i\otimes \lambda_ib'
    \end{align*}
    for $b\in\uqb, b'\in\uqb'$.
    Note that the map $M^*\otimes M\to\mathcal{A}$ given by
    $f\otimes v\mapsto c_{f,v}$ is injective. Applying its inverse to $f$ and then permuting tensor legs, we obtain
    \[
        F = \sum_{i=1}^r m_i\otimes f_i\otimes v_i\otimes \lambda_i
        \in \Hom_{\mathbf{B}}(M,V)\otimes\Hom_{\uqb'}(W,M).
    \]
    Writing $F = \sum_{i=1}^s p_i\otimes j_i$ and defining
    $f_i$ to be the elementary MSF defined by $(p_i,j_i)$, we have
    \[
        f = \sum_{i=1}^s f_i.
    \]
    \item Let $f\in E^{V,W}$. 
    We decompose $\mathcal{A}$ as a direct sum into 
    matrix elements of simple finite-dimensional $\uq$-modules. Write
    \[
        f = \sum_{i=1}^r c_{f_i,v_i}\otimes \phi_i
    \]
    where $v_i\in M_i,\lambda_i\in M_i^*$ 
    and $\phi_i\in\Hom(W,V)$ for
    simple $\uq$-modules $M_1,\dots,M_r$. Write
    \[
        f_i := \sum_{\substack{j=1\\M_j\cong M_i}}^r
        c_{f_j,v_j}\otimes \phi_j.
    \]
    Note that the decomposition of $\mathcal{A}$ into
    $\uq$-types is invariant under the left and right
    action of $\uq$. This shows that $f_i\in E^{V,W}$
    and hence with (i) that $f$ can be written as a linear
    combination of elementary MSF.
    \end{enumerate}
\end{proof}

\begin{proposition}
    When equipped with the product inherited from $\mathcal{A}$, the
    vector space $E^\epsilon$ is a $k$-algebra and $E^{V,W}$ is a left $E^\epsilon$-module.
\end{proposition}
\begin{proof}
    We only need to check that $E^\epsilon E^{V,W}\subset E^{V,W}$. Let
    $f\in E^\epsilon,\Phi\in E^{V,W}$ and $x\in\uq,b\in\uqb,b'\in\uqb'$ and $f\in E^\epsilon,\Phi\in E^{V,W}$, then
    \begin{align*}
        (f\Phi)(bx) &= f(b_{(1)}x_{(1)})\Phi(b_{(2)}x_{(2)})
        = \epsilon(b_{(1)}) f(x_{(1)}) \Phi(b_{(2)}x_{(2)})\\
        &= f(x_{(1)})\Phi(bx_{(2)})
        = b (f\Phi)(x)\\
        (f\Phi)(xb') &= f(x_{(1)}b'_{(1)}) \Phi(x_{(2)}b'_{(2)}) =
        \epsilon(b'_{(1)}) f(x_{(1)})\Phi(x_{(2)} u'_{(2)})\\
        &= f(x_{(1)}) \Phi(x_{(2)}u')
        = (f\Phi)(x)u'.\qedhere
    \end{align*}
\end{proof}

Lastly, we prove a statement for $V=W=k$ the trivial representation:
\begin{lemma}[{\cite[Proposition~4.25]{Mee24}}]\label{lem-antipode-spherical}
    Let $\varphi\in E^\epsilon_{\uqb,\uqb'}$ and assume that $S$ is invertible.
    Then $\varphi\circ S\in E^{\epsilon}_{S^{-2}(\uqb'),\uqb}$.
\end{lemma}
\begin{proof}
    Let $b\in\uqb,b'\in\uqb',x\in\uq$, then
    \begin{align*}
        \varphi(S(S^{-2}(b')xb)) &=
        \varphi(S(b)S(x)S^{-1}(b'))\\
        &= \varphi\qty(S\Big(\epsilon(b_{(1)})b_{(2)}\Big)
        S(x)
        S^{-1}\Big(\epsilon(b'_{(1)})b'_{(2)}\Big))\\
        &= \epsilon(b_{(1)})\:\varphi\qty(S(b_{(2)}) S(x) S^{-1}(b'_{(2)}))\:\epsilon(b'_{(1)}).
    \end{align*}
    Since the elements $b_{(1)}$ and $b'_{(1)}$ lie in $\uqb,\uqb'$,
    respectively, we obtain
    \begin{align*}
        &= \varphi\Big(b_{(1)}S(b_{(2)}) S(x)
        S^{-1}(b'_{(2)})b'_{(1)}\Big)\\
        &= \varphi\Big(b_{(1)}S(b_{(2)}) S(x)
        S^{-1}(S(b'_{(1)})b'_{(2)})\Big)\\
        &= \epsilon(b)\epsilon(b')\varphi(S(x))\\
        &= \epsilon(S^{-2}(b')) \varphi(S(x)) \epsilon(b).\qedhere
    \end{align*}
\end{proof}

\begin{corollary}\label{cor-antipode-spherical}
    Let $\varphi\in E^\epsilon_{\uqb,\uqb'}$ and let $S$ be invertible.
    Let $b\in\uqb,b'\in\uqb',x\in\uq$, then
    \[
        \varphi(S(b)xS^{-1}(b'))
        = \epsilon(S(b))\varphi(x)\epsilon(S^{-1}(b')).
    \]
\end{corollary}
\begin{proof}
    By Lemma~\ref{lem-antipode-spherical}, $\varphi\circ S$ is
    a ZSF for $S^{-2}(\uqb'),\uqb$.
    Consequently,
    \[
    \varphi(S(b)xS^{-1}(b'))
    = \varphi(S(S^{-2}(b')S^{-1}(x)b))
    = \epsilon(b)\epsilon(b')\varphi(x).\qedhere
    \]
\end{proof}

\subsection{Orthogonality}
We will now define inner product methods to study matrix-spherical functions. In later sections, these inner products will be crucial to identify the matrix-spherical functions with Intermediate Macdonald polynomials. 
The elementary constructions of this section generalize constructions in the classical setting of \cite[\S6]{vPr18}, or those of \cite[\S4]{Mee25} in the quantum case of characters. The latter is based on a similar construction appearing in \cite[\S6]{Ald}.

For this section we will assume that the square of the antipode can
be expressed as follows: $S^2 = \ad(K)$ for a $K\in\uq$ satisfying $\Delta(K)=K\otimes K$ (or more
generally for a natural transformation $K$ of the forgetful functor of a suitable monoidal subcategory of the $\uq$-modules
that contains the adjoint representation and the left representation on $\mathcal{A}$, such that
$K_{M\otimes N} = K_M\otimes K_N$).

\begin{remark}\label{rem-sovereign}
In \cite[\S3.2]{App25}, the element $K$ is referred to as a \emph{sovereign} element, and $\uq$ is then termed a \emph{sovereign} Hopf algebra (Appel's further assumption of quasi-triangularity is not necessary for this definition).
We will later see that
Drinfel'd--Jimbo quantum groups,
the main choice of $\uq$ used in this work, are sovereign.
\end{remark}

\begin{definition}\label{def:XI}
    For finite-dimensional modules $V,W$ over $\uqb,\uqb'$ we define
    \[
        \gls{Xi}: E^{V,W}\otimes E^{W,V}\to E^\epsilon,\qquad
        \Xi_{V,W}(\Phi,\Psi)(x) := \tr_V\qty(\Phi(x_{(1)})\Psi\big(S(x_{(2)})K\big)).
    \]
\end{definition}

\begin{proposition}\label{prop-xi-properties}
  \begin{enumerate}
  \item $\Xi_{V,W}$ is well-defined.
  \item If $\Phi\in E^{V,W},\Psi\in E^{W,V}$ and $f,g\in E^\epsilon$, then
    \[
      \Xi_{V,W}(f\Phi, g\Psi)
      = f \Xi_{V,W} (\Phi,\Psi)
      S(K\triangleright g).
    \]
  \item Let $\Phi\in E^{V,W},\Psi\in E^{W,V}$,
  then
  \[
    \Xi_{V,W}(\Phi,\Psi) = S(K\triangleright \Xi_{W,V}(\Psi,\Phi)).
  \]
  \end{enumerate}
\end{proposition}
\begin{proof}
  \begin{enumerate}
  \item $\Xi_{V,W}$ is evidently well-defined as map
  $(\mathcal{A}\otimes\Hom(V,W))\times (\mathcal{A}\otimes\Hom(W,V))\to \mathcal A$. It is $k$-bilinear and
  hence descends to the tensor product. It remains to show
  that it indeed restricts to $E^{V,W}\otimes E^{W,V}$ as claimed. 
  Note that $\Psi(S(\cdot)K) = \Psi(S(K^{-1}\cdot))=S(\Psi)\triangleleft K^{-1}$.
  Let $b\in\uqb, b'\in \uqb', x\in\uq$, then
    \begin{align*}
        \Xi_{V,W}(\Phi,\Psi)(xb') &= \tr(\Phi\big(x_{(1)} b'_{(1)}\big)\Psi\big(S(b'_{(2)})S(x_{(2)}) K\big))\\
            &= \tr(\Phi(x_{(1)})\Psi\big(b'_{(1)}S(b'_{(2)}) S(x_{(2)}) K\big))\\
            &= \Xi_{V,W} (\Phi,\Psi)(x)\epsilon(b') \\
        \Xi_{V,W}(\Phi,\Psi)(bx) &= \tr(\Phi\big(b_{(1)}x_{(1)}\big)
        \Psi\big(S(x_{(2)})S(b_{(2)}) K\big))\\
        &= \tr(b_{(1)} \Phi(x_{(1)})\Psi\big(S(x_{(2)})S(b_{(2)}) K\big))\\
        &= \tr(\Phi(x_{(1)}) \Psi\big(S(x_{(2)}) S(b_{(2)}) K b_{(1)}\big))\\
        &= \tr(\Phi(x_{(1)}) \Psi\big(S(x_{(2)}) S(b_{(2)}) S^2(b_{(1)}) K\big))\\
        &= \tr(\Phi(x_{(1)}) \Psi\big(S(x_{(2)}) S(S(b_{(1)})b_{(2)}) K\big))\\
        &= \epsilon(b)\Xi_{V,W}(\Phi,\Psi)(x)
    \end{align*}
    using the fact that $\Delta(b)\in \uqb\otimes\uq$ (and similar for primes), general
    properties of $\Delta$ and $S$, and the fact that $S^2=\ad(K)$ by
    assumption.
  \item Let $x\in \uq$, then
    \begin{align*}
      \Xi_{V,W}(f\Phi,g\Psi)(x) &=
                              \tr\qty((f\Phi)(x_{(1)})(g\Psi)\big(S(x_{(2)})K\big))\\
                            &= \tr(f(x_{(1)}) \Phi(x_{(2)})
                              g\big(S(x_{(4)})K\big) \Psi\Big(S(x_{(3)})K\big))\\
                            &= f(x_{(1)}) \tr\qty(\Psi(x_{(2)})
                              \Psi\big(S(x_{(3)})K\big))
                              g\big(S(x_{(4)})K\big)\\
                            &= f(x_{(1)})
                              \Xi_{V,W}(\Phi,\Psi)(x_{(2)})
                              S(K\triangleright g)(x_{(3)}),
    \end{align*}
    which equals $f\Xi_{V,W}(\Phi,\Psi) S(K\triangleright g)$, evaluated in $x$.
    \item
        Let $x\in \uq$, then
        \begin{align*}
            S(K\triangleright\Xi_{W,V}(\Psi,\Phi))(x) &= 
            \Xi_{W,V}(\Psi,\Phi)(S(x)K)\\
            &= \tr_W(\Psi(S(x)_{(1)}K)
            \Phi(S(S(x)_{(2)} K)K))\\
            &= \tr_W(\Psi(S(x_{(2)})K)
            \Phi(K^{-1} S^2(x_{(1)}) K))\\
            &= \tr_W(\Psi(S(x_{(2)})K)\Phi(x_{(1)})).
        \end{align*}
        Since $V,W$ are finite-dimensional,
        we can leverage the cyclicity of the trace to obtain
        \[
            = \tr_V(\Phi(x_{(1)})\Psi(S(x_{(2)})K))
            = \Xi_{V,W}(\Phi,\Psi)(x).\qedhere
        \]
  \end{enumerate}
\end{proof}

\begin{definition}
    A functional
    $\gls{h}:\mathcal{A}\to k$ is called \emph{left-invariant} if
    \[
        (\id\otimes h)\circ\Delta = h.
    \]
    A left-invariant functional $h$ satisfying $h(1)=1$ is called a
    \emph{Haar state}.
\end{definition}

\begin{lemma}\label{lem-quantum-schur-orthogonality}
    Let $\mathcal{A}$ have a unique Haar state $h$ and let
    $f,g\in\mathcal{A}$ be matrix coefficients of the simple modules
    $M,N$, respectively. Then $h(fS(g))\ne0$ implies that
    $M\cong N$. Moreover, restricted to the matrix coefficients
    of a simple module $M$, the map $(f,g)\mapsto h(fS(g))$ defines
    a non-degenerate pairing. In particular, we have
    \[
        h(c_{f,v}S(c_{g,w})) = 
        \frac{f(w) g(Kv)}{\tr(K_M)}
    \]
    for a simple $\uq$-module $M$ and $v,w\in M$ and $f,g\in M^*$,
    and for $K$ still chosen such that $S^2=\ad(K)$.
\end{lemma}
\begin{proof}
    The first claim follows from \cite[Theorem 11.2.13, Proposition 11.2.15]{ks}.

    In particular, pick a basis $v_1,\dots,v_n\in M$ and
    $\lambda_1,\dots,\lambda_n\in M^*$ the dual basis. Then
    assume that
    \[
        v = \sum_{i=1}^n a_i v_i,\qquad
        w= \sum_{i=1}^n b_i v_i,\qquad
        f = \sum_{i=1}^n c_i \lambda_i,\qquad
        g = \sum_{i=1}^n d_i \lambda_i.
    \]
    Note that $K$ is the intertwiner in question. Consequently, we have
    \begin{align*}
        h(c_{f,v}S(c_{g,w}))
        &= \sum_{i,j,k,l=1}^n a_i b_j c_k d_l
        h(c_{\lambda_k,v_i} S(c_{\lambda_l,v_j}))\\
        &= \sum_{i,j,k,l=1}^n a_i b_j c_k d_l
        \delta_{k,j}\frac{\lambda_l(Kv_i)}{\tr(K_M)}\\
        &= \frac{1}{\tr(K_M)}\sum_{k=1}^n c_k \lambda_k\qty(\sum_{j=1}^n b_jv_j)
        \sum_{l=1}^n d_k\lambda_l \qty(\sum_{i=1}^n a_iKv_i)\\
        &= \frac{f(w)g(Kv)}{\tr(K_M)}.\qedhere
    \end{align*}
\end{proof}

\begin{corollary}\label{cor-emsf-orthogonal}
    Assume there exists a unique Haar state. Then the map $h\circ\Xi_{V,W}$
    is a symmetric bilinear pairing of $E^{V,W}$ with $E^{W,V}$, 
    such that elementary MSFs for
    non-isomorphic simple $\uq$-modules are orthogonal.
\end{corollary}
\begin{proof}
    Let $\Phi,\Psi$ be elementary MSF for non-isomorphic simple $\uq$-modules, 
    say
    \[
        \Phi = \sum_i c_{f_i,v_i}\otimes\phi_i,\qquad
        \Psi = \sum_j c_{g_i,w_i}\otimes\psi_i.
    \]
    Then
    \[
        \Xi_{V,W}(\Phi,\Psi) = \sum_{i,j} c_{f_i,v_i}S(c_{g_j, Kw_j})\tr_V(\phi_i\psi_j)
    \]
    and hence
    \[
        h(\Xi_{V,W}(\Phi,\Psi)) = \sum_{i,j} h(c_{f_i,v_i}S(c_{g_j,Kw_j}))
        \tr_V(\phi_i\psi_j) = 0
    \]
    by Lemma~\ref{lem-quantum-schur-orthogonality}.
    
    To show symmetry, it suffices to consider the case $M=N$,
    in which we have
    \begin{align*}
       h(\Xi_{V,W}(\Phi,\Psi)) &= \sum_{i,j} h(c_{f_i,v_i}S(c_{g_j,Kw_j}))
        \tr_V(\phi_i\psi_j)\\
        &= \sum_{i,j} \frac{f_i(Kw_j)g_j(Kv_i)}{\tr_M(K)}
        \tr_V(\phi_i\psi_j)\\
        &= \sum_{i,j} h(c_{g_j,w_j}S(c_{f_i,Kv_i}))
        \tr_W(\psi_j\phi_i)\\
        &= h(\Xi_{W,V}(\Psi,\Phi))
    \end{align*}
    by Lemma~\ref{lem-quantum-schur-orthogonality}.
\end{proof}

\begin{proposition}\label{prop:nondegen}
    Let $M$ be a simple $\uq$-module that is
    semisimple over $\uqb$ and $\uqb'$ with $\tr(K_M)\neq0$ and 
    let
    $\Phi\in E^{V,W}$ be a nonzero elementary MSF
    for $M$. 
    Then there is $\Psi\in E^{W,V}$ elementary for $M$ such that $h\circ \Xi_{V,W}(\Phi,\Psi)=\frac{\tr(K_M|_V)\tr(K_M|_W)}{\tr(K_M)}$ (assuming that $V,W$ are realised as subsets of
    $M$).
\end{proposition}
\begin{proof}
    Assume that $V\subset M, W\subset M$ according to the
    embedding and projection that define $\Phi$.
    
    Let $v_1,\dots,v_n\in V\subset M$ and
    $w_1,\dots,w_m\in W\subset M$ be bases of $V,W$ and let
    $f_1,\dots,f_n\in V^*, g_1,\dots,g_m\in W^*$ be the dual
    bases extended to $M$ in such a way that the other $\uqb$-types are mapped to 0. We then have
    \[
        \Phi = \sum_{i=1}^n\sum_{j=1}^m
        c_{g_j,v_i}\otimes w_jf_i.
    \]
    Here $w_jf_i$ denotes the linear map $V\to W$ mapping
    $v\mapsto f_i(v)w_j$.
    Define
    \[
        \Psi := \sum_{i=1}^n\sum_{j=1}^m
        c_{f_i,w_j}\otimes v_ig_j \in E^{W,V}.
    \]
    Then we have
    \[
        \Xi_{V,W}(\Phi,\Psi) =
        \sum_{i,k=1}^n\sum_{j,l=1}^m
        c_{g_j,v_i}S(c_{f_k,Kw_l})
        \tr_V(w_jf_iv_k g_l)
        = \sum_{i=1}^n \sum_{j=1}^m
        c_{g_j,v_i} S(c_{f_i, Kw_j}).
    \]
    Applying $h$ to this, and applying Lemma~\ref{lem-quantum-schur-orthogonality}, we obtain
    \[
        h(\Xi_{V,W}(\Phi,\Psi)) =
        \sum_{i=1}^n\sum_{j=1}^m
        \frac{g_j(Kw_j)f_i(Kv_i)}{\tr(K_M)}
        = \frac{\tr(K_M|_V)\tr(K_M|_W)}{\tr(K_M)}.
        \qedhere
    \]
\end{proof}

\section{Drinfel'd--Jimbo Quantum Groups}\label{sec:Dinfeld}
We will now define the concrete Hopf algebras and coideal subalgebras we
will be working with for the rest of this paper. The notation is mostly
taken from \cite{Lus10} and \cite{Kol14}.

\subsection{Root System}
Let $I=\set{1,\dots,n}$ and $A=(a_{i,j})_{i,j\in I}$ be a symmetrisable Cartan
matrix and \gls{XY} be two lattices equipped with a perfect
pairing $\langle\cdot,\cdot\rangle: Y\times X\to\Z$, and with
linearly independent elements
\[
    \gls{alphas}\in X,\qquad
    \gls{hs}\in Y
\]
such that $\langle h_i,\alpha_j\rangle = a_{i,j}$ for all
$i,j\in I$. Note that in general we don't require $X\otimes\Q$ and
$Y\otimes\Q$ to be spanned by the $(\alpha_i)_{i\in I}$ or
$(h_i)_{i\in I}$. If e.g. $A$ is of affine type, this will not be
the case.

Let $\gls{epsilons}\in\N$ be such that $DA$ is
symmetric for $D=\operatorname{diag}(\epsilon_1,\dots,\epsilon_n)$. Equivalently, we can imagine $X$ to be equipped
with a symmetric bilinear form~$\cdot$ such that
\begin{align*}
     \forall i\in I: \quad &\alpha_i\cdot\alpha_i \in 2\N\\
            \forall i\ne j\in I: \quad &2 \frac{\alpha_i\cdot\alpha_j}{\alpha_i\cdot\alpha_i} \in -\N_0
\end{align*}
and such that
\[
    \forall i,j\in I: \quad
    \langle h_i,\alpha_j\rangle = 2\frac{\alpha_i\cdot\alpha_j}{\alpha_i\cdot\alpha_i}.
\]
Then $\epsilon_i = \frac{\alpha_i\cdot\alpha_i}{2}$. 

Define the embedding $X\to Y\otimes_\Z\Q, \lambda\mapsto \gls{hlambda}$ by
mapping $\lambda\in X$ to the unique element $h\in Y\otimes_\Z\Q$ satisfying
\[
  \forall \mu\in X: \mu\cdot\lambda = \langle h,\mu\rangle.
\]
In particular, $h_{\alpha_i}=\epsilon_ih_i\in Y$.

Let $H=(h_{i,j})_{i,j\in I}$ be the Coxeter matrix associated
to $A$, i.e. we have
\[
    h_{i,j} = \begin{cases}
        1 & i=j\\
        \frac{\pi}{\arccos(\frac{\sqrt{a_{i,j}a_{j,i}}}{2})} &
        a_{i,j}a_{j,i} < 4\\
        \infty & \text{otherwise}
    \end{cases}.
\]
The braid group $B$ and the Weyl group \gls{W} are then the
Artin braid group and the Coxeter group for the Coxeter matrix $H$. Both act on $X$ and $Y$ as follows:
\[
    s_i\mu := \mu - \langle h_i,\mu\rangle\alpha_i,\qquad
    s_ih := h - \langle h,\alpha_i\rangle h_i.
\]
Let $Q,Q^\vee$ be the lattices generated by $\alpha_1,\dots,\alpha_n$
and $h_1,\dots,h_n$, respectively,
and let $P^\vee, P$ be the corresponding dual lattices.
Note that in general, $P^\vee, P$ are quotients of $Y,X$, respectively.
Let
\[
    \omega_1,\dots,\omega_n\in P,\qquad
    \omega_1^\vee,\dots,\omega_n^\vee\in P^\vee
\]
be the dual bases of $h_1,\dots,h_n$ and $\alpha_1,\dots,\alpha_n$. 

These are the fundamental weights and coweights. Note
that in general, $P,P^\vee$ are quotients of $X,Y$, respectively.
The lattices $Q,Q^\vee,P,P^\vee$ are evidently also acted upon
by $W$.
Let $\gls{R}\subset Q$ be the union of all $W$-orbits of
$\alpha_1,\dots,\alpha_n$.
It is a reduced root system, \gls{QP} are its root and
weight lattices, and \gls{QPcheck} are its coroot and coweight lattices.
Moreover, $\omega_1,\dots,\omega_n,\omega_1^\vee,\dots,\omega_n^\vee$ are its fundamental (co)weights.

Write $\operatorname{ht}: Q\to\Z$ for the group homomorphism
mapping $\alpha_i\mapsto 1$ ($i\in I$).
In case $X=P,Y=Q^\vee$, the lattices $X,Y$ are said to
make up the \emph{simply connected root datum}. In case
$X=Q,Y=P^\vee$, the lattices are referred to as the \emph{adjoint root datum}.

We make $Q,Q^\vee$ into ordered abelian groups by declaring that the 
elements of $Q^+:=\sum_{i=1}^n\N_0\alpha_i$ and $Q^{\vee+}=\sum_{i=1}^n\N_0h_i$
are non-negative.
This extends to partial orders on $X,Y$.
Write $\mathcal{R}^+,\mathcal{R}^{\vee+}$ for the subsets of
$\mathcal{R},\mathcal{R}^\vee$ of positive elements.

$\omega\in P$ (resp. $\omega^\vee\in P^\vee$) is said to be dominant if
$\omega(Q^{\vee+})\subset\N_0$ (resp. $\omega^\vee(Q^+)\subset\N_0$).
Similarly $\mu\in X$ is said to be dominant if
$\langle Q^{\vee+},\mu\rangle\subset\N_0$.
Write $X^+$ for the subset of dominant elements.

Choose $2\gls{rho}\in Y\otimes_\Z\Q$ such that
\[
    \langle 2\rho,\alpha_i\rangle = 2\epsilon_i.
\]
In case $\mathcal{R}$ is finite (equivalently, $W$ is finite), we say that $I$ is of \emph{finite type}.
In that case, such a $2\rho$ is given for example by
\[
    2\rho := \sum_{\alpha\in\mathcal{R}^{+}} h_\alpha.
\]
Bijective maps $\tau: I\to I$ are called \emph{diagram automorphisms} if $a_{i,j}=a_{\tau(i),\tau(j)}$ for all $i,j\in I$.

\subsection{Quantum Groups}
Let \gls{q} be an indeterminate and let
$\gls{k}:=\overline{\C(q)}$, this is the base field for all our definitions,
although most notions can also be defined over $\Q(q)$.
We fix a group homomorphism $\Q\to k^\times$ mapping $1\mapsto q$.
This corresponds to a consistent choice of rational powers $q^x$ of $q$ ($x\in\Q$).
Later we will also adopt the notation of $q$-numbers. Let $K$ be an indeterminate, then we set:
\begin{equation}\label{def-q-numbers}
    \gls{qn} := \frac{q^{ab}-q^{-ab}}{q^b-q^{-b}},\qquad
    \gls{qo} := \frac{Kq^{ab}-K^{-1}q^{-ab}}{q^b-q^{-b}}
\end{equation}
(as in \cite[\S2.2]{Wat21}) as well as $[K]_{q^b}=[K;0]_{q^b}$ wherever this doesn't lead to confusion. Define
$q_i := q^{\epsilon_i}$.

\begin{definition}
    Define the \emph{Drinfel'd--Jimbo quantum group} \gls{U} to be the associative $k$-algebra generated by
    \[
        (E_i)_{i\in I},\qquad (F_i)_{i\in I},\qquad
        (K_h)_{h\in Y}
    \]
    subject to the relations
    \begin{align*}
        K_0 &= 1\\
        \forall h,h'\in Y:\quad K_h K_{h'} &= K_{h+h'}\\
        \forall i\in I,h\in Y:\quad K_h E_i &= q^{\langle h,\alpha_i\rangle} E_iK_h\\
        \forall i\in I,h\in Y\quad K_h F_i&= q^{-\langle h,\alpha_i\rangle}F_iK_h\\
        \forall i,j\in I:\quad \comm{E_i}{F_j} &=
        \delta_{ij}\frac{K_i - K_i^{-1}}{q_i-q_i^{-1}} = \delta_{ij}[K_i]_{q_i}\\
        \forall i,j\in I: \quad F_{ij}(E_i,E_j) &= F_{ij}(F_i,F_j)=0
    \end{align*}
    where $K_i:=K_{\epsilon_ih_i}$ and
    \[
        F_{ij}(x,y) = \sum_{p+p' = 1-a_{ij}}
        \frac{(-1)^{p'}}{[p]^!_{q_i}[p']^!_{q_i}} x^p yx^{p'}
    \]
    for $i,j\in I$.
    Here, we use
    \[
        [n]_q^! := [n]_q[n-1]_q\cdots [2]_q[1]_q
    \]
    for $n\in\N_0$.
    
    We equip $\uq$ with the following Hopf algebra structure:
    \begin{align*}
        \Delta(E_i)&:= E_i\otimes 1 + K_i\otimes E_i\\
        \Delta(F_i)&:= F_i\otimes K_i^{-1}+ 1\otimes F_i\\
        \Delta(K_h) &:= K_h\otimes K_h\\
        \epsilon(E_i)&:=\epsilon(F_i):= 0\\
        \epsilon(K_h)&:= 1\\
        S(E_i) &:= -K_i^{-1}E_i\\
        S(F_i) &:= -F_iK_i\\
        S(K_h) &:= K_{-h}
    \end{align*}
    for $i\in I$ and $h\in Y$.
\end{definition}
\begin{lemma}
    The thus-defined $(\uq,\Delta,\epsilon,S)$ is a Hopf algebra.
\end{lemma}{
\begin{proof}
This is concluded at the end of \cite[\mbox{\S3.3.4}]{Lus10}.
\end{proof}}
\begin{remark}
    Our setting corresponds to an $X$- and $Y$-regular root
    datum of type $(I,\cdot)$ from \cite{Lus10}, where the
    embeddings $I\to X$ and $I\to Y$ are given by
    $i\mapsto \alpha_i$ and $i\mapsto h_i$, respectively.
    In {\cite[\S2.1]{Kol14}}, an explicit construction is given
    for lattices $Y$ ($P^\vee$ in Kolb's notation) and $X$ (resp.
    $P$) where $Y$ is an extension of $Q^\vee$.
    In that setting of Kac--Moody algebras, the
    elements of $\mathcal{R}$ are referred to as the
    \emph{real} roots, in contrast to those weights of the
    Kac--Moody algebra that are not contained in $\mathcal{R}$, which are the \emph{imaginary} roots.
    This differentiation does not make much sense for quantum
    groups as the weights of $\uq$ form a lattice, $Q$,
    and not a root system.
\end{remark}

\begin{lemma}\label{lem-antipode-squared}
    If $2\rho$ as chosen earlier lies in $Y$ (as is the case for $\mathcal{R}$ finite), we have $S^2 = \ad(K_{-2\rho})$.
    This shows that $\uq$ is sovereign as in Remark~\ref{rem-sovereign}.
\end{lemma}
\begin{proof}
    Both $S^2$ and $\ad(K_{-2\rho})$ are endomorphisms of
    $k$-algebras, so it suffices to prove that they agree on the
    algebra generators. We have
    \begin{align*}
        S^2(K_h) &= K_h = \ad(K_{-2\rho})(K_h)\\
        S^2(E_i) &= -S(K_i^{-1}E_i) = -S(E_i)S(K_i^{-1}) = K_i^{-1}E_iK_i\\
        &= q^{-2\epsilon_i} E_i = K_{-2\rho} E_i K_{2\rho}
        = \ad(K_{-2\rho})(E_i)\\
        S^2(F_i) &= -S(F_iK_i) = K_i^{-1}F_i K_i\\
        &= q^{2\epsilon_i} F_i = K_{-2\rho}F_iK_{2\rho}
        = \ad(K_{-2\rho})(F_i).\qedhere
    \end{align*}
\end{proof}

\begin{definition}
    For $\lambda\in X^+$, write \gls{Llambda} for the simple
    $\uq$-module with highest weight $\lambda$, as defined e.g. in \cite[Proposition~3.5.6]{Lus10} (it is called $\Lambda_\lambda$ there; simplicity is proven in Corollary~6.2.3 under the condition of $X$- and
    $Y$-regularity).

    Similarly, we write
    $E^{V,W}_{\uqb,\uqb'}(\lambda)$ for
    $E^{V,W}_{\uqb,\uqb'}(L(\lambda))$.
\end{definition}

\begin{theorem}[{\cite[Corollary~6.2.3(c)]{Lus10}}]\label{thm-complete-reduc}
    Let $M$ be a $\uq$-module satisfying:
    \begin{enumerate}
        \item Weight module: For $\lambda\in X$ we define
        \[
            M_\lambda := \set{m\in M\where\forall h\in Y: K_hm = q^{\langle h,\lambda\rangle}m}.
        \]
        We require that $M = \bigoplus_{\lambda\in X}M_\lambda$;
        \item Integrability: For every $m\in M$ andH $i\in I$ there is
        $N\in N_0$ such that for all $n\ge N$:
        \[
            E_i^nm = F_i^nm = 0;
        \]
        \item Local boundedness above: For every $m\in M$ there is 
        $N\in\N_0$ such that for all $\lambda\in Q$ with $\operatorname{ht}(\lambda)\ge N$ and all $x\in\uq_\lambda$ (here $\uq_\lambda$ is the $\lambda$-weight space considered as a $\ad(\uq)$ module) we have
        $xm=0$.
    \end{enumerate}
    Then $M$ can be written as a sum of simple modules of the form
    $L(\lambda)$ for $\lambda\in X^+$.
\end{theorem}

\begin{corollary}[{\cite[Propositions~6.3.4,6.3.6]{Lus10}}]
    If $I$ is of finite type, the simple modules $L(\lambda)$ for
    $\lambda\in X^+$ are finite-dimensional. Furthermore, any
    finite-dimensional weight module or integrable $\uq$-module $M$ satisfies the
    condition of Theorem~\ref{thm-complete-reduc}.
\end{corollary}

\begin{remark}\label{rmk-antipode-squared}
    Note that even if $2\rho\not\in Y$, the following construction allows us to construct a natural transformation $K_{-2\rho}$
    (of the forgetful functor of the category of $\uq$-weight modules)
    such that
    \[
        S^2(x) = \ad(K_{-2\rho})(x)
    \]
    for all $x\in\uq$, and such that
    \[
        S^2(x)v=K_{-2\rho}xK_{-2\rho}^{-1}v
    \]
    whenever in addition $v$ is an element of a $\uq$-weight module.

    Let $h\in Y\otimes_\Z\Q$, let
    $n\in\N_0$ such that $nh\in Y$.
    We pick a $n$-th root $q_0$ of $q$, which exists since $k$ is algebraically closed.
    
    Given this choice, we define the natural transformation $K_h$ as follows:
    for any weight module $M$ we have
    \[
        \forall \lambda\in X,m\in M_\lambda:\quad
        K_h m:= q_0^{\langle nh,\lambda\rangle} m,
    \]
    which is well-defined since $\langle nh,\lambda\rangle\in\Z$. 
    This is a natural
    transformation since for every morphism $f: M\to N$ we have
    $f(M_\lambda)\subset N_\lambda$, and hence $f$ also intertwines the
    action of $K_h$.

    In particular, since $\uq$ is a weight module over $\uq$ (via $\ad$),
    the operator $\ad(K_{-2\rho})$ makes sense and with the same proof as for Lemma \ref{lem-antipode-squared} we obtain that $S^2=\ad(K_{-2\rho})$ for a suitably chosen
    $2\rho\in Y\otimes_\Z\Q$, which means that
    \[
        \langle 2\rho,\alpha_i\rangle = 2\epsilon_i.
    \]
\end{remark}

Recall that $k=\overline{\C(q)}$ and that $\C\subset k$ is a subfield. In this way, the quantum group $\uq$ is a $\C$-algebra.

\begin{definition}
    Define $\gls{bar}:\uq\to\uq$ to be the $\C$-algebra homomorphism mapping
    \[
        q^x\mapsto q^{-x},\qquad
        E_i\mapsto E_i,\qquad
        F_i\mapsto F_i,\qquad
        K_h\mapsto K_{-h}
    \]
    for $i\in I,h\in Y, x\in\Q$. This is well-defined by
    \cite[\S3.1.12]{Lus10} and is called the \emph{bar involution}. 
\end{definition}

\begin{definition}\label{def:barinv}
Let $\lambda\in X^+$ and $v_\lambda\in L(\lambda)_\lambda\setminus \{0\}$. Define $\overline{\,\cdot\,}:L(\lambda)\to L(\lambda)$ to be the unique $\C$-linear map with
$$\overline{xv_\lambda}=\overline{x}v_\lambda,\qquad x\in\uq$$
which extends to a $\C$-linear map $\overline{\,\cdot\,}: L(\lambda)^\ast\to L(\lambda)^\ast$ by setting
$$\overline{f}(v):=\overline{f(\overline{v})}.$$
This is well-defined by
    \cite[\S19.3.4]{Lus10}
\end{definition}
\begin{definition}
For each $w\in W$ write \gls{Tw} for the automorphism $T_{w,1}'':\uq\to \uq$ from \cite[\mbox{\S 37.1.3}, \S39.4.7]{Lus10}. 
\end{definition}
For all $h\in Y$ and $w\in W$, we recall that $T_{w}(K_h)=K_{w h}$.
\begin{definition}\label{def-scaling}
For a $\uq$-weight module $M$ and $\xi\in\Hom(X,k^\times) $ we define the endomorphism $\xi:M\to M$ as $\xi v_\mu=\xi(\mu)v_\mu$, for $v_\mu\in M_\mu.$
\end{definition}
We set $\uq^0=k\langle K_h:h\in Y\rangle.$ Recall that $\uq$ is a $\uq$-weight module via $\ad$, for each $\xi\in \Hom (X,k)$ the automorphism $\ad(\xi):\uq\to\uq$ (i.e. $\xi$ acting in the adjoint representation of $\uq$) acts as the identity on $\uq^0$ as it has weight $0$.
\begin{definition}\label{def:qfa}
    In case $I$ is of finite type, let $\mathcal{A}$ be the algebra
    generated by matrix elements for finite-dimensional modules, i.e.
    the maximal dual algebra from Example~\ref{ex-maximal-dual}. This is the \emph{quantum function algebra} and
    it is dual with $\uq$ via the evaluation map.

    If $\mathcal{R}$ is irreducible and of infinite type, by \cite[Lemma~7.1.15(ii)]{Jos95} every simple
    module is either 1-dimensional or infinite-dimensional.
    We conclude that the maximal dual algebra $\mathcal{A}$ from
    Example~\ref{ex-maximal-dual} consists of the
    multiplicative characters.
    Unless $\uq$ is commutative, $\mathcal{A}$ will not
    be large enough to be a dual Hopf algebra for
    $\uq$. This is why we will later assume that $I$ is of
    finite type.
\end{definition}

\begin{lemma}\label{lem-centre-u}
  If $I$ is of finite type, the centre $Z(U)$ is 
  spanned by elements \gls{cmu} for every $\mu\in X^+$
  satisfying $2h_\mu\in Y$.
  $c_\mu$ acts on the finite-dimensional simple module
  $L(\lambda)$ as the scalar
  \[
    \sum_{\nu\in X}q^{-2\lambda\cdot\nu - \langle 2\rho,\nu\rangle}\dim(L(\mu)_\nu)
  \]
\end{lemma}
\begin{proof}
  This follows from \cite[\S 7.1.17,\S7.1.19]{Jos95}.
  The fact that we work with arbitrary choices of lattices
  $X,Y$ instead of the simply-connected lattices, is
  reflected in the fact that the $c_\mu$ of Joseph is an
  element whose image under the Harish-Chandra map
  is made up of $K_{-2h_\nu}$ where $L(\mu)_\nu\ne0$,
  so in order for these elements to be contained in our
  quantum group, we also require that $2h_\mu\in Y$
  (which then also implies $2h_\nu\in Y$ for all other
  weights $\nu$ of $L(\mu)$).
\end{proof}

\subsection{Quantum Symmetric Pair Coideal Subalgebras}\label{se:QSP}
In this section we introduce quantum symmetric pair coideal subalgebras following \cite{Le99} and \cite{Kol14}.
\begin{notation}
    For each subset $I_\bullet\subset I$ of finite type, let \gls{wblack} denote the longest element in the parabolic subgroup $\gls{Wblack}=\langle s_i:i\in I_\bullet\rangle\subset W$.
\end{notation}

\begin{definition}
    An \emph{admissible pair} \gls{Itau} consists of a subset $I_\bullet\subset I$ of finite type and an involutive diagram automorphism $\tau:I\to I$ such that
    \begin{enumerate}
    	\item $\tau|_{I_\bullet}=-w_\bullet$;
    	\item If $i\in I\setminus I_\bullet$ and $\tau(i)=i$ then $\langle \rho^\vee_\bullet,\alpha_i\rangle\in \Z$.
    \end{enumerate}
    Here $\rho_\bullet^\vee$ denotes the half sum of positive coroots of the root system generated by $I_\bullet$. Set $I_\circ:=I\setminus I_\bullet$ and set
    \[
        I_{\mathrm{ns}}:=\set{i\in I_\circ\where \tau(i)=i,
        \forall j\in I_\bullet:\quad \langle h_j,\alpha_i\rangle =0}.
    \]
    An admissible pair is said to be \emph{compatible} with the datum of
    the lattices $X,Y$ if there is an involution $\tau$ on $X$ (that we extend
    to $Y$ by the pairing) satisfying $\tau(\alpha_i)=\alpha_{\tau(i)}$.
\end{definition}

\begin{remark}
    Note that in case $X$ and $Y$ are generated
    by roots and coweights or weights and coweights
    (respectively), the compatibility condition is
    automatically satisfied.
\end{remark}

\begin{example}\label{ex-group-case}
    Let $I=I'\sqcup I'$, where $I'$ are the
    simple roots of a root system $\Sigma$,
    and let $\tau$ be the diagram automorphism that
    exchanges the two copies of $I'$.
    Let furthermore $X', Y'$ be lattices used to
    define a quantum group with $I'$, and let
    $X:= X'\oplus X'$ (same for $Y$), and extend
    $\tau$ to these lattices by exchanging the
    two copies of $X'$ (resp. $Y'$).
    Then $(\emptyset, \tau)$ is an admissible pair
    for $I$ that is compatible with $X,Y$.
    This is called the \emph{group case} since
    this mirrors the construction $G=G'\times G'$
    and $K=\diag(G')$ of a symmetric pair $(G,K)$
    whose symmetric space $G/K$ is the group $G'$
    itself.
\end{example}

For each admissible pair $(I_\bullet,\tau)$ compatible with $X,Y$ we associate an involution $\gls{Theta}=-w_\bullet\circ\tau$ acting on $X,Y$ as
well as $Q,Q^\vee,P,P^\vee$.
For each root $\alpha\in \mathcal R$, set $\gls{atilde}= \frac{\alpha-\Theta(\alpha)}{2}$. The set 
\[
    \gls{Sigma} = \set{\widetilde{\alpha}\where \alpha\in \mathcal{R},\,\Theta(\alpha)\neq \alpha}
    \subset \frac{1}{2}X.
\]
is called the \emph{restricted root system}. 
Write $\Sigma^\vee\subset Y\otimes_\Z\Q$ for its dual.
Define
\[
    \gls{L} := \set{\lambda\in X\where \langle Y^\Theta,\lambda\rangle=0,
    \langle\Sigma^\vee,\lambda\rangle\subset 2\Z
    }
\]
($\gls{YTheta}$ denoting the $\Theta$-invariant elements of $Y$).
We set
\[
    \gls{Ublack}= k\langle E_j, F_j, K_j^\pm\,:\, j\in I_\bullet\rangle\le\uq
\]
and set
\[
    \uq_\Theta^0=k\langle K_h\,:\, h\in Y^\Theta\rangle\le\uq.
\]

\begin{remark}
    By \cite[Definition~2.3]{Kol14} and an argument similar to \cite[\S2.5]{Araki}, the
    finite-type admissible pairs are classified by the group
    cases (Example~\ref{ex-group-case}) and the Satake diagrams.
\end{remark}

\begin{definition}
    We now define the \emph{quantum symmetric pair coideal subalgebra} $\gls{B}=\uqbs$ associated to the admissible pair $(I_\bullet,\tau)$ compatible with $X,Y$. It is the subalgebra of $\uq$ generated by $\uq_\bullet$, $\uq_\Theta^0$, and the elements
    \[
        B_i= F_i+c_i T_{w_\bullet}(E_{\tau(i)})K_i^{-1}+s_iK_i^{-1}
        \qquad (i\in I_\circ).
    \]
    We refer to $(\uq,\uqbs)$ as a \emph{quantum symmetric pair} (QSP), and to the algebra $\uqbs$ as a QSP-coideal subalgebra of $\uq$.
    The QSP-coideal subalgebra $\uqbs$ depends on a family of scalars $(\boldsymbol{c},\boldsymbol{s})\in (k^\times)^{I_\circ}\times k^{I_\circ}$ that satisfy 
    \begin{align*}
    	\boldsymbol{c}\in \mathcal C&= \{\boldsymbol{d}\in(k^\times)^{I_\circ}: d_i=d_{\tau(i)}\quad \text{if}\quad \alpha_i\cdot\Theta(\alpha_i)=0 \}\\
    	\boldsymbol{s}\in \mathcal S&= \{\boldsymbol{t}\in k^{I_\circ}: t_j\neq 0\implies (j\in I_{\mathrm{ns}}\text{ and  }\langle h_i,\alpha_j\rangle\in -2\N_0\,\forall i\in I_{\mathrm{ns}}\setminus \{j\})\}.
    \end{align*}
\end{definition}

\begin{remark}
    Note that our choice of parameters $\bm{c}$ differs from \cite{Kol14}
    because we absorb the factors $s(I,\tau)(\alpha_i)$ into
    $c_i$.
    That the condition $c_i=c_{\tau(i)}$ for $\alpha_i\cdot\Theta(\alpha_i)=0$ remains unchanged follows from
    the fact that $\alpha_i\cdot\Theta(\alpha_i)=0$ for
    $i\ne\tau(i)$ and $i\in I_\circ$ implies that $\langle\rho_\bullet,\alpha_i\rangle=0$ as well.
\end{remark}

QSP-coideal subalgebras $\uqbs$ are right coideal subalgebras, i.e.
\[
\triangle (\uqbs)\subset \uqbs\otimes \uq.
\]
For $\boldsymbol{s}=0$, the corresponding QSP-coideal subalgebra is denoted by $\uqb$ and is referred to as  \emph{standard}.
As is explained in \cite[\S 9.1]{Kol14}, the group $\Hom(X,k)$ acts on the set of parameters by 
mapping 
\begin{equation}\label{eq:actionrel}
\xi: (\boldsymbol{c},\boldsymbol{s})\mapsto (\boldsymbol{d},\boldsymbol{t})\qquad \text{if}\qquad \Ad(\xi)(\uqbs)=\uqds. 
\end{equation}
This action induces an equivalence relation on the set of parameters.

\begin{remark}\label{rem:pareq}
    By \cite[\mbox{Rem 9.3}]{Kol14}, if $I$ is of finite type and $\Sigma$ is reduced, any two standard parameters are equivalent.
    Furthermore, if $I$ is irreducible and $\Sigma$ is
    non-reduced, there is $i\in I$ such
    that $(\boldsymbol{c},0)\sim(\boldsymbol{d},0)$ iff
    $c_i=d_i$.

    A concrete and complete criterion for equivalence
    under $\ad(\Hom(X,k))$ will be provided in \eqref{eq-h-related-parameters}.
\end{remark}

\begin{remark}
 From now on we assume that $I$ is of finite type, unless stated otherwise. It is worth to mention that most of the notions and some of the results can be extended to Kac--Moody type.
\end{remark}

\subsection{Specialization}\label{sec:specializable}
This section studies the specialization $q=1$ of quantum groups, QSP-coideal subalgebras and their modules. Our approach follows \cite{Let00}.

\begin{definition}
Let $\mathbf{A}$ be the smallest subring of $k$ containing $q$ such that every element in $\mathbf{A}$ that is not contained in the ideal generated by $(q - 1)$ has a square root and is invertible, cf. \cite[\mbox{Sec 1}]{Let00}.
\end{definition}

\begin{definition}
     
    Let $\widehat{\uq}$ denote the $\mathbf{A}$-subalgebra generated by $E_i,F_i,K_i^{-1}$ and $\frac{K_i-1}{q-1}$, with $i\in I$. Furthermore, for each subspace $W\subset \uq$ we write $\widehat{W}=W\cap \widehat{\uq}$.
\end{definition}

Recall that the quotient $\widehat{\uq}/(q-1)\widehat{\uq}$ is isomorphic to $U(\mathfrak{g})$, the universal enveloping algebra of $\mathfrak{g}$, the reductive
Lie algebra that can be defined from the Cartan subalgebra $Y\otimes_\Z\C$ and
the Dynkin diagram $I$.
For a dominant weight $\lambda\in X^+$ the $\mathbf{A}$-module $L_{\mathbf{A}}(\lambda):=\widehat{\uq}v_\lambda$ is specialized by setting $L(\lambda)^1:=L_{\mathbf{A}}(\lambda)/(q-1)L_{\mathbf{A}}(\lambda)$. 
The $U(\mathfrak{g})$-module $L(\lambda)^1$ is the highest weight module of weight $\lambda$.
We write $\cl$ for the quotient maps 
$\widehat{\uq}\to U(\mathfrak{g})$, 
$L^{\mathbf{A}}(\lambda)\to L(\lambda)^1$ and 
$\mathbf{A}\to \C$.

\begin{definition}\label{def:specia}
A $\uqbs$-module $V$ is called \emph{specializable} if there exists a basis $\mathcal B=\{v_1,\dots v_n\}\subset V$ such that $\widehat{\uqbs}\,\cdot\mathrm{span}_{\mathbf{A}}\mathcal B\subset \mathrm{span}_{\mathbf{A}}\mathcal B$.
\end{definition}

\begin{definition}\label{def:specializable1}
    A parameter $(\boldsymbol{c},\boldsymbol{s})$ is  \emph{specializable} if $c_i,s_i\in \mathbf{A}$ and 
    \[\cl\bigg(\frac{c_i}{c_{\tau(i)}}\bigg)=
    \begin{cases}
    1&\text{if }\tau(i)=i\\
    (-1)^{\langle 2\rho_\bullet^\vee,\alpha_i\rangle}&\text{if }\tau(i)\neq i
    \end{cases}\]
    for each $i\in I_\circ$
    and \emph{quasi-specializable} if it is contained in the orbit, regarding the action \eqref{eq:actionrel}, of a specializable parameter.
\end{definition}
In the case of a specializable parameter we have $U(\mathfrak{g})\supset \widehat{\uqbs}/(1-q)\widehat{\uqbs}=U(\mathfrak{k})$, cf. \cite[\mbox{Thm 4.9}]{Le99} or \cite[\mbox{Thm 10.8}]{Kol14}.  
We recall that specializable modules $V$ allow us to form the $U(\mathfrak{k})$-module $\widetilde{V}:= \sum_ {v\in\mathcal B}\mathbf{A} v\otimes _{\mathbf{A}}\C$. In particular, for each finite-dimensional simple $\uqbs$-module $V$ the existence of a specializable basis is assured by \cite[\mbox{Lemma 7.3}]{Let00}.

\begin{theorem}\cite[\mbox{Thm 4.15}]{Mee24} \label{thm-specialisation}
	Let $\lambda\in X^+$ be a dominant weight, and let $(\boldsymbol{c},\boldsymbol{s})$ be a quasi-specializable parameter. If the $\uq$-module $L(\lambda)$ decomposes as $\bigoplus_{i=1}^m L_i$ into simple $\uqbs$-modules, then the $U(\mathfrak{g})$ module $L(\lambda)^1$ decomposes as $\bigoplus_{i=1}^m \widetilde{L}_i$ into simple $U(\mathfrak{k})$ modules.
\end{theorem}
\begin{remark}
    From now on always assume that $(\boldsymbol{c},\boldsymbol{s})$ is quasi-specializable.
\end{remark}

\begin{lemma}
The bilinear pairing 
\[h\circ \Xi_{V,W}:E^{V,W}\times E^{W,V}\to k\]
is non-degenerate.
\end{lemma}

\begin{proof}
It is sufficient to show the claim for elementary spherical functions. Let $\Phi\in E^{V,W}$ be  nonzero elementary MSF and let $\Psi\in E^{W,V}$ be as in Proposition 
\ref{prop:nondegen}. Then it follows by Proposition 
\ref{prop:nondegen} that
\[h\circ \Xi_{V,W}(\Phi,\Psi)=\frac{\tr(K_M|_V)\tr(K_M|_W)}{\tr(K_M)}.\]
Since the eigenvalues of $K$ specialise
    to 1, the above expression specialises to
    \[
        \frac{\dim(V)\dim(W)}{\dim(M)},
    \]
    which is nonzero. Consequently, we have
    $\Xi_{V,W}(\Phi,\Psi)\ne0$.
\end{proof}

\subsection{Quantum commutative triples}
We continue the study of matrix-spherical functions from Section \ref{sec:gen-msf}. In this section we study $\uqb$-modules which are multiplicity-free in every integrable $\uq$-module. Following \cite{vPr18}, such modules give rise to a family of vector-valued orthogonal polynomials. We apply the techniques of Section \ref{sec:specializable} to show that the classification of such modules is independent of $q$.

\begin{definition}
    A $\uqbs$-module $V$ is \emph{integrable} if it occurs in an integrable $\uq$-module.
\end{definition}

This notion is equivalent to the integrablility of \cite[Def 3.3.4]{Wat2025}, as shown in \cite[Prop 4.3.1]{Wat2025}.
An analogous integrability definition exists for $\mathbf{U}(\mathfrak{k})$-modules. 
Recal from Definition \ref{def:coideal} that $\Gamma$ denotes the set of equivalence classes of 
    finite-dimensional irreducible $\uqb$-modules that are contained
    in the restriction of a $\uq$-module. 

\begin{definition}\label{def:qtrip}
    Let $\gamma\in\Gamma$. The triple $(\uq,\uqbs,\gamma)$ is a \emph{quantum commutative triple} if
    \[
        \forall\lambda\in X^+:\quad \dim(\Hom_{\uqbs}(V(\gamma),L(\lambda))),
        \dim(\Hom_{\uqbs}(L(\lambda),V(\gamma)))\le 1.
    \]
    If both inequalities are equalities and $L(\lambda)$ is finite-dimensional, $\lambda$ is said to be $\gamma$-\emph{spherical}. Write
    \gls{Xgamma} for the set of dominant $\gamma$-spherical weights.
    
    We write $\gls{Egammalambda}=E^{V(\gamma),V(\gamma')}(L(\lambda))$ (recall the notation \gls{EVWM} from
    Example~\ref{ex-emsf}).
\end{definition}

\begin{lemma}
    Write $\gamma=\epsilon$ for the trivial $\uqbs$-module
    (given by the counit $\epsilon$). If $I$ is of finite type and
    $\uqbs$ is either standard or real nonstandard, we have $X^+(\epsilon)=2L^+$.
\end{lemma}
\begin{proof}
    This follows from \cite[Theorem 3.2]{Let03}. Note that
    the proofs also work for the case where $X\otimes_\Z\Q$ is
    not spanned by the roots.
\end{proof}

\begin{lemma}
    Let $(\uq,\uqb,\gamma)$ and $(\uq,\uqb',\gamma')$ be quantum commutative triples and let $0\ne \Phi\in E^{\gamma,\gamma'}_{\uqb,\uqb'}$, then $X^+(\gamma)\cap X^+(\gamma')$ is nonempty.
    Moreover, if $\Phi\in E^{\gamma,\gamma'}_{\uqb,\uqb'}(\lambda)$, then $\lambda\in X^+(\gamma)\cap X^+(\gamma')$.
\end{lemma}
\begin{proof}
    By Lemma~\ref{lem-elementary-msf-basis-kinda}(ii) we can
    without loss of generality assume that $\Phi$ can be written
    in terms of matrix elements of a simple finite-dimensional module, 
    say of $L(\lambda)$.
    By Lemma~\ref{lem-elementary-msf-basis-kinda}(i), we can
    then write $\Phi$ in terms of elementary MSF. Therefore there
    exist nonzero intertwiners in $\Hom_{\uqb}(L(\lambda),V(\gamma))$ and $\Hom_{\uqb'}(V(\gamma'),L(\lambda))$.
    This shows that $\lambda\in X^+(\gamma)\cap X^+(\gamma')$,
    which is therefore nonempty.
\end{proof}

In case $\uqb=\uqb'$ and $\gamma=\gamma'$,
each one-dimensional space
$E^\gamma_\uqb(\lambda)$ (for $\lambda\in X^+(\gamma)$) has a distinguished element.

\begin{lemma}\label{lem-emsf-basis}
    Let $(\uq,\uqbs,\gamma)$ be a quantum commutative triple where $\uq$ is of finite type and let
    $\lambda\in X^+(\gamma)$. 
    There exists a unique
    elementary MSF $\gls{Philambdagamma}\in E^\gamma(\lambda)$
    satisfying $\Phi^\lambda_\gamma(1)=1$
    (the identity endomorphism of $V(\gamma)$). 
    Furthermore,
    the $(\Phi^\lambda_\gamma)_{\gamma\in X^+(\gamma)}$
    form a $k$-basis of $E^\gamma$.
\end{lemma}
\begin{proof}
    By Lemma~\ref{lem-elementary-msf-basis-kinda}(i),
    the space of elementary MSF for $L(\lambda)$ is
    one-dimensional, consequently the scaling requirement
    fixes a unique MSF.

    For linear independence observe that matrix elements for non-isomorphic
    simple modules are linearly independent, hence so are elementary MSF
    for non-isomorphic simple modules.

    To see that these elementary MSF form a generating
    system, we quote \cite[Proposition~6.3.4(b)]{Lus10} to
    see that every finite-dimensional $\uq$-module
    is also semisimple and then apply Lemma~\ref{lem-elementary-msf-basis-kinda}(ii).
\end{proof}

\begin{definition}\label{def:cl}
Define the map $\cl:\Gamma\to \Gamma_1$, mapping $\gamma$ to $ \cl(\gamma)$, the $U(\mathfrak{k})$-type of the specialized module $\widetilde{V(\gamma)}:=V(\cl(\gamma))$.
\end{definition}

This is well defined according to \cite[\mbox{Theorem 4.15}]{Mee24}, based on \cite{Let00}.

\begin{lemma}\label{lem:comptripspes}
    Let $(U(\mathfrak{g}),U(\mathfrak{k}),\cl(\gamma))$ be a commutative triple, then $(\uq,\uqbs,\gamma)$ is a commutative triple.
\end{lemma}
\begin{proof}
    In order to reach a contradiction, we assume the
    existence of $\lambda\in X^+$ such that 
    $[L(\lambda):V(\gamma)]>1$. Applying
    \cite[\mbox{Theorem 4.15}]{Mee24} and \cite[\mbox{Thm 7.9}]{Let00}
    yields $[L^1(\lambda):V(\cl(\gamma))]>1$,
    a contradiction. 
    Therefore, $[L(\lambda):V(\gamma)]\leq 1$ for all $\lambda\in X^+$.
\end{proof}

\begin{lemma}\label{lem:multfree}
    Let $(\mathbf{U}(\mathfrak{g}),\mathbf{U}(\mathfrak{k}),\cl(\gamma))$
    be a commutative triple (in the classical sense),
    and let $V(\gamma)$ and $V(\gamma')$ be integrable $\uqbs$-modules 
    with isomorphic specializations such that there
    exists a $\gamma,\gamma'$-spherical weight 
    $\lambda\in X^+$.
    Then, $V(\gamma)$ and $V(\gamma')$ are isomorphic $\uqbs$-modules.
\end{lemma}
\begin{proof}
    Let $m_\gamma$ and $m_{\gamma'}$ denote the
    multiplicities of $V(\gamma)$ and $V(\gamma')$ in
    $L(\lambda)|_{\uqbs}$ respectively. 
    Using \cite[\mbox{Theorem 4.15}]{Mee24}, 
    the decomposition
    \[
        L(\lambda)^1=\bigoplus_{i=1}^l m_{\gamma_i} V(\cl(\gamma_i))
    \]
    corresponds to a decomposition $L(\lambda)=\bigoplus_{i=1}^l m_{\gamma_i}V(\gamma_i)$. 
    Thus, due to multiplicity-freeness $V(\gamma)$ is isomorphic to $V(\gamma')$.
\end{proof}

To show that there is a one-to-one correspondence
between quantum commutative triples and 
classical commutative triples  we can then use
Lemmas~\ref{lem:multfree} and \ref{lem:comptripspes}.
The only missing step is then to show that if $(U(\mathfrak{g}),U(\mathfrak{k}),\cl(\gamma))$ 
is commutative and $\gamma,\gamma'$ specialize to 
$\cl(\gamma)$ that $\gamma,\gamma'$ are $\uqbs$-types in an integrable $\uq$-module.
We show this in Proposition~\ref{thm:branching}.

\begin{notation}
    Write $\Gamma^c\subset\Gamma$ for the equivalence 
    classes of modules for which $(\uq,\uqbs,\gamma)$ 
    is commutative. 
    Analogously define $\Gamma^c_1\subset\Gamma_1$.
\end{notation}

In Proposition~\ref{thm:branching} we use techniques of 
\cite{Wat2025}, as preparation we introduce a lattice and an abelian group.

Set $X^\imath$ equal to 
$X/\{\lambda-\Theta(\lambda): \lambda\in X\}$, and 
$\overline{\,\cdot\,}:X\to X^\imath$ equal to 
the corresponding quotient map. 
Set $Y^\Theta := \{h\in Y:\Theta(h) = h\}$ and $\langle\,,\,\rangle: X^\imath\times Y^\Theta$ to be the pairing induced from $X\times Y$.
These abelian groups play the role of the coroot, and weight lattices for $\uqbs$.

\begin{definition}[{Weight module, \cite[\S 3.2]{Wat2025}}]  
A $\uqbs$-module $M$ is a \emph{weight module} if it admits a decomposition
$$M=\bigoplus_{\zeta\in X^\imath}M_\zeta$$
such that
\begin{enumerate}
    \item $M_\zeta\subset \{m\in M | K_h m= q^{\langle h,\zeta\rangle}v\}$ for all $h\in Y^\Theta$;
    \item $E_j M_\zeta \subset M_{\zeta+\overline{\alpha_j}}$ , $F_j M_\zeta \subset M_{\zeta-\overline{\alpha_j}}$ for all $j \in I_\bullet$;
    \item  $B_kM_\zeta \subset  M_{\zeta-\overline{\alpha_k}}$ for all $k\in I_\circ$.
\end{enumerate}
\end{definition}
\begin{remark}
Note that each $\uq$-weight module $M$ is a $\uqbs$-weight module via
    \[
        M=\bigoplus_{\zeta\in X^\imath}V_\zeta \qquad V_\zeta=\bigoplus_{\lambda\in X,\, \overline{\lambda}=\zeta}V_\lambda.
    \]
    According to \cite[\S 5.1]{Wat2025}, each irreducible $\uqbs$-submodule $V(\gamma)\subset L(\lambda)$ is itself canonically a $\uqbs$-weight module.
\end{remark}



\begin{notation}\label{not:levi}
    Let $\mathbf{L}$ denote the Levi subalgebra 
    generated by $E_i,F_i$ with $i\in I_\bullet$ and 
    $\uq^0$. 
    We remark that $\mathbf{L}$ is itself a 
    quantum group, and we let
    $\mathbf{B}_{\mathbf{L}}(-\infty)$ denote 
    its canonical basis, cf. \cite[\mbox{\S 25}]{Lus10}.

    For $n\geq 0, k\in I_\circ$ and $\zeta\in X^\imath$ let $\mathfrak{B}^{(n)}_{k,\zeta}$ be the divided power as introduced in \cite[\S4.1]{Wat2025}.
\end{notation}

\begin{proposition}\label{thm:branching}
    Specialization provides a one-to-one correspondence between the quantum commutative triples and the classical commutative triples. That is, the map $\cl :\Gamma^c\to \Gamma^c_1$ is a bijection.
\end{proposition}
\begin{proof}
\begin{description}
\item [Well-Definedness \& Injectivity] 
    Let $\gamma\in \Gamma$ and $\eta\in \Gamma^c$ 
    such that $V(\cl(\gamma))\cong V(\cl(\eta))$ as 
    $U(\mathfrak{k})$-modules. 
    By \cite[Theorem~4.15]{Mee24}, this isomorphism induces a linear space isomorphism
    preserving the weight spaces 
    $V(\gamma)_{\xi}\to  V(\eta)_{\xi}$, for each
    $\xi\in X^\imath$. Let $v\in V(\gamma)_{\xi}$ 
    and $w\in V(\eta)_{\xi}$ be nonzero 
    $\uq_\bullet$-highest weight vectors with 
    identical specialization. 
	Let $(b_j)_{j \in I_\bullet} \in \mathbb{Z}_{\geq 0}^{I_\bullet}$ 
    and $(c_i)_{i \in I_\circ} \in \mathbb{Z}_{\geq 0}^{I_\circ}$ satisfy
	\[
        E_jw=F_j^{(b_j+1)}w=0\quad\text{and}\quad E_jv=F_j^{(b_j+1)}v=0\quad \text{for all}\quad j\in I_\bullet,
    \]
	and 
	\[
        \mathfrak{B}_{i,\xi}^{(c_k+1)}v=0,\qquad \mathfrak{B}_{i,\xi}^{(c_k+1)}w=0\quad \text{for all}\quad i\in I_\circ.
    \]
The existence of these integers is assured by \cite[Proposition 4.3.2.]{Wat2025}.
	Given $\mu \in X^+$ such that $\overline{\mu}=\xi$, we may, by replacing $\mu$ by $\mu+\nu$ where $\overline{\nu}=0$, assume that
	\begin{align*}
    	&\langle h_j,\mu\rangle \geq b_j,\\
    	&\langle h_i, \mu +\mathrm{wt}(b)\rangle \geq c_i
	\end{align*}
	for all $j\in I_\bullet$, $i\in I_\circ$ and 
    $b\in B_{V(\gamma)}=B_{V(\eta)} = \{b\in \mathbf{B}_{\mathbf{L}}(-\infty): bv\neq 0\}$. 
    Let $V^\imath(0, \mu)$ be the module from \cite[\S 3.3]{Wat2025}. By \cite[Theorem 4.2.6]{Wat2025}, we have
    surjective $\uqbs$-module homomorphisms
    \[
        V(\gamma )\overset{f}{\twoheadleftarrow} V^\imath(0,\mu)\overset{f'}{\twoheadrightarrow} V(\eta).
    \]
    By \cite[Proposition 3.3.3]{Wat2025} we have $V^\imath(0, \mu) \cong L^\imath(0, \mu) \cong \uq v_0 \otimes v_\mu \cong L(\mu)$.
    By Definition~\ref{def:qtrip}, there exist $\uqbs$-module
    module homomorphisms $V(\gamma)\overset{g}{\hookrightarrow}L(\mu)\overset{g'}{\hookleftarrow}V(\eta)$
    such that $g\circ f=\id_{V(\gamma)}$ and $g'\circ f'=\id_{V(\eta)}$.
    Then $f\circ g$ and $f'\circ g'$ are idempotent, so we get
    \[
        L(\mu)\cong V(\gamma)\oplus\ker(f\circ g)
        \cong V(\eta)\oplus \ker(f'\circ g').
    \]
    If $\gamma\not\cong\eta$, the multiplicity of $\cl(\gamma)$ in
    $\cl(\mu)$ is $>1$, which contradicts the assumption.
    Hence, $\gamma\cong\eta$ and $\cl: \Gamma^c\to \Gamma_1^c$ is 
    injective and well-defined.
\item[Surjectivity] Let $V^1$ be a simple integrable $U(\mathfrak{k})$-module. 
    Then there exists a dominant weight $\lambda\in X^+$, 
    such that $L(\lambda)^1=\oplus_{i=1}^l L_i$, 
    with $L_1=V^1$. 
    By \cite[Theorem 4.15]{Mee24}, this decomposition 
    corresponds to  $L(\lambda)=\bigoplus_{i=1}^l L_i$. 
    Thus there exists a $\uqbs$-module $V(\gamma)$ 
    specializing to $V^1$.\qedhere
\end{description}
\end{proof}

\begin{definition}\label{def:bottom}
    The \emph{bottom of the $\gamma,\gamma'$-well} is defined as
    \[
        \mathfrak{B}^+(\gamma,\gamma')=\set{\lambda \in X^+(\gamma)\cap X^+(\gamma')\where
        \forall \mu\in 2L^+:\quad \lambda-\mu \notin X^+(\gamma)\cap X^+(\gamma')}.
    \]
    We write $\gls{Bgamma}:= \mathfrak{B}^+(\gamma,\gamma)$.
    If $\mathfrak{B}^+(\gamma)$ contains one element, we say that
    $\gamma$ is an integrable small $\uqb$-type.
\end{definition}

\begin{theorem}\label{thm-bottom-pit}
    Let $(\uq,\uqbs,\gamma)$ be a quantum commutative triple
    then $X^+(\gamma)=\mathfrak B^+(\gamma)+2L^+$.
    Furthermore, if $(\uq,\uqb_{\bm{d},\bm{t}},\gamma')$ is a
    quantum commutative triple with different parameters and $\cl(\gamma)=\cl (\gamma')$, we have
    $X^+(\gamma)=X^+(\gamma')$.

    In particular, $[L(\lambda)_{\uqbs}:V(\gamma)] = [L(\lambda)^1|_{U(\mathfrak{k})}:V(\cl(\gamma))]$, i.e. the branch rules
    are preserved by specialisation in $q=1$.
\end{theorem}
\begin{proof}
Using Proposition~\ref{thm:branching} the branching rules are as in the classical commutative triple case, the classical result can be found in \cite[\mbox{Proposition 8.7}]{Pe23}.
\end{proof}

\subsection{Bar Involutions and Symmetries}\label{sec:barinv}
The goal of this section is to study spherical functions in light of bar-involutions. 
The goal is to derive elementary symmetries. These symmetries will yield symmetries of the \textit{matrix-valued weight} for the vector-valued polynomials that we construct later.

To ensure existence of the \emph{quasi and universal $K$-matrix}, we assume that the parameters are \emph{uniform}. 

\begin{definition}
    A parameter $(\bm{c},\bm{s})\in (k^\times)^{I_\circ}\times k^{I_\circ}$ is called \emph{uniform} if
    \begin{equation}\label{eq:uniformpar}
        c_{\tau(i)}=(-1)^{\langle 2\rho^\vee_\bullet, \alpha_i\rangle}q^{\alpha_i\cdot(\Theta(\alpha_i)-2\rho_\bullet)}\overline{c_i}\quad\text{and}\quad s_i=\overline{s_i}\quad \text{for each}\quad i\in I_\circ.
    \end{equation}
\end{definition}

\begin{notation}
    Let $\gls{w0}\in W$ denote the longest element.
\end{notation}

Let $\gls{tau0}:I\to I$ denote the diagram involution with $w_0(\alpha_i)=-\alpha_{\tau_0(i)}$ for each $i\in I_\circ$. 
Diagram automorphisms $\eta:I\to I$ act on $\uq$ by
\[
    E_i\mapsto E_{\eta(i)},\quad F_i\mapsto F_{\eta(i)},\quad K_i\mapsto K_{\eta(i)},\qquad \text{for each }i\in I.
\]
We set $\sigma:=\tau\circ\tau_0$ (where $\tau$ is used
to define $\uqbs$).

During the remainder of this section we furthermore
assume $c_{\sigma(i)}=c_i$ and $s_{\sigma(i)}=s_i$,
for each $i\in I_\circ$.
When these conditions are fulfilled, 
$\sigma$ restricts to an involutive automorphism of 
$\uqbs$ and thus acts on $\Gamma$.
Let $\comp$ denote an algebraic completion of $\uq$,
cf. \cite[\mbox{\S 3}]{Ko19}. 

\begin{definition}[Anti-linear]
Let $V$ be a $k$-vector space. An \textit{anti-linear} map $f:V\to V$ is a $\C$-linear endomorphism with $q^{x}\mapsto q^{-x}$ for each $x\in \Q$.
\end{definition}

\begin{theorem}[{\cite{Ko19}\cite{Bao18}}]\label{thm:Kmatrix}
    The following three statements hold
    \begin{enumerate}
        \item There exists a unique anti-linear
        involution $\gls{psii}: \uqbs\to \uqbs$
        such that $\psi^\iota|_{\uq_\bullet}=\gls{bar}|_{\uq_\bullet}$
        and
        \[
            B_i\mapsto B_i,\qquad K_h\mapsto K_{-h},\qquad\text{where}\qquad i\in I_\circ,h\in Y_\Theta.
        \]
        \item There exists an up to scalar multiple unique element
        $\gls{Upsilon}\in {\comp}^+$ such that
        \[
            \psi^\iota(x)\Upsilon = \Upsilon \overline{x},\qquad\text{where}\qquad x\in \uqbs.
        \]
        We refer to $\Upsilon$ as the \emph{quasi $K$-matrix.}
        \item For each $\lambda\in X^+$ there exists a
        $\xi\in \mathrm{Hom}(X,k)$ such that the operator 
        $\gls{K}=\Upsilon \xi T_{w_\bullet}^{-1}T_{w_0}^{-1} $ 
        satisfies
        \[
            \mathcal K xv= \sigma(x)\mathcal Kv,\qquad \text{where}\qquad x\in \uqbs, v\in L(\lambda), \lambda \in X^+.
        \]
        We refer to $\mathcal K$ as the \emph{universal $K$-matrix.}
    \end{enumerate}
\end{theorem}

\begin{lemma}\label{lem-K-twists-sigma}
  Let $\lambda\in X^+(\gamma)$ and let
  $p: L(\lambda)\to V(\gamma)$ and $i: V(\gamma)\to L(\lambda)$ be
  $\uqb$-linear maps satisfying $p\circ i = id_{V(\gamma)}$.
  
  Then $p\circ\mathcal{K}: L(\lambda)\to V(\sigma(\gamma))$ and
  $\mathcal{K}^{-1}\circ i: V(\sigma(\gamma))\to L(\lambda)$ are
  also $\uqb$-linear maps satisfying $p\circ \mathcal{K}\mathcal{K}^{-1}\circ i = \id_{V(\gamma)}$.
\end{lemma}
\begin{proof}
  Let $v\in L(\lambda)$ and $x\in\uqb$, then
  \[
    p(\mathcal{K}xv) = p(\sigma(x)\mathcal{K}v)
    = \sigma(x)p(\mathcal{K}v)
  \]
  by Theorem~\ref{thm:Kmatrix}(iii).
  The other claim follows analogously.
\end{proof}

Recall that the dual of a left $\uq$-module $M$ is a right $\uq$-module $M^\ast$ where the action is defined in Definition \ref{def:moduleactions}.
\begin{definition}[$\imath$bar invariant]
We say a vector $v\in L(\lambda)$ is \textit{$\imath$bar invariant} if $\Upsilon\overline{v}=v$ and we say a vector $f\in L(\lambda)^\ast$ is \textit{$\imath$bar invariant} if $\overline{f}\Upsilon^{-1}=f$.
\end{definition}

In Proposition~\ref{prop:blackmagic-ibarinv-bases}, we show the existence of a $\imath$bar invariant basis that is independent of embedding.
We first need a preliminary result. Recall the elements $\mathfrak{B}^{(\langle h_j, \lambda\rangle)}_{k,\overline{\lambda}}$ from Notation \ref{not:levi}.

\begin{lemma}\label{lem:vecunique}
Let $(\uq,\uqb,\gamma)$ be a quantum commutative triple. Then there exists a, up to scalar multiple, unique vector $v\in V(\gamma)$ with \begin{equation}\label{eq:watzero}
0\neq v\in L(\lambda)_{\overline{\lambda}},\qquad E_jv=0=F_j^{(\langle h_j, \lambda\rangle+1)}v,\quad \text{and}\quad \mathfrak{B}^{(\langle h_j, \lambda\rangle+1)}_{k,\overline{\lambda}} v=0,
\end{equation}
for all $j\in I_\bullet$ and $k\in I_\circ$.
\end{lemma}

\begin{proof}
Let $\lambda\in X$ be $\gamma$-spherical and let $\pi_\gamma: L(\lambda)\to V(\gamma)$ be the $\uqbs$-equivariant projection. As $v_\lambda$ is a cyclic vector for $L(\lambda)$ as a $\uqbs$-module, cf. \cite[Lem 6.2.1]{Wat24A}, existence of the vector $v$ follows from \cite[Theorem 4.2.6]{Wat2025}. Uniqueness, up to scalar, then follows from $\dim\big(\Hom_{\uqbs}(L(\lambda),V(\gamma)\big)=1$, since any $v\in V(\gamma)$, with (\ref{eq:watzero}) defines a $\uqbs$-linear map
$v_\lambda\mapsto v$, again by \cite[Theorem 4.2.6]{Wat2025}.
\end{proof}

For Proposition \ref{prop:blackmagic-ibarinv-bases}, recall the notion of a specializable basis as introduced in Definition \ref{def:specia}.

\begin{proposition}\label{prop:blackmagic-ibarinv-bases}
Let $(\uq,\uqb,\gamma)$ be a quantum commutative triple.  There exist bases $(v_j)_{j=1,\dots,n}\subset V(\gamma)$ and
  $(f_j)_{j=1,\dots,n}\subset V(\gamma)^*$ dual to each other such that
  for every $\lambda\in X^+(\gamma)$ there are 
  $\uqb$-linear maps $i: V(\gamma)\to L(\lambda)$
  and $p:L(\lambda)\to V(\gamma)$ such that $p\circ i = \id_{V(\gamma)}$
  and such that all $i(v_j)$ and $f_j\circ p$ ($j=1,\dots,n$) are
  $\imath$bar-invariant.
\end{proposition}

\begin{proof}
Fix a $\gamma$-spherial weight $\lambda '\in X^+(\gamma)$ and write $v=\pi_\gamma(v_{\lambda'})\in V(\gamma)$ for the vector from Lemma \ref{lem:vecunique}. Let $\{b_1 v,\dots,b_nv\}$ be a specializable basis of $V(\gamma)$ with $b_j\in\widehat{\uqb}$ for all $1\leq j \leq n$. We note that $\psi^\imath (b_j)\in \widehat{\uqb}$ for all $1\leq j \leq n$ so that by Definition \ref{def:specia} 
	\[v_j=(b_j+\psi^\imath(b_j)) v\in \mathrm{span}_{\mathbf{A}} \{b_1 v,\dots,b_nv\}.\]
Since $\cl\circ \psi^\iota=\cl$, it follows that $\{v_1,\dots,v_n\}$ specializes to $\widetilde{\{b_1v,\dots,b_nv\}}\subset \widetilde{V(\gamma)}.$ 
Next, consider a $\gamma$-spherical weight $\lambda \in X^+(\gamma)$ and the, up to scalar multiple unique, $\uqb$-linear map $i: V(\gamma)\to L(\lambda)$.
According to \cite[\mbox{Proposition 4.2.1 (a), Proposition 4.2.2}]{Ap25}, we have $\cl(\Upsilon)=1$. 
Moreover, we note that 
$\cl\circ \overline{\,\cdot\,}=\cl$. 
Together this yields that $\cl \circ \Upsilon\circ \overline{\,\cdot\,}: V(\gamma)\to \widetilde{V(\gamma)}$ so that
\[\widetilde{\Upsilon \overline {i(V(\gamma))}}=\widetilde{i(V(\gamma))}.\] 
Now
$\Upsilon \overline{V(\gamma)}=V(\gamma)$ follows by multiplicity-freeness. 
Thus 
\[\Upsilon\circ \overline{\,\cdot\,}: i(V(\gamma))\to i(V(\gamma)).\]
By $\uqb$-linearity $p(\Upsilon\overline{i(v)})$ defines another vector satisfying the properties of Lemma \ref{lem:vecunique}, thus by multiplying with a scalar, we can fix $i$ so that $\Upsilon\overline{i(v)}=i(v)$.
Let $\tilde{v}_j=i(v_j) = i(b_j+\psi^\imath(b_j))v$. 
Because $\Upsilon\overline{i(v)}=i(v)$ and $\psi^\imath(b_j+\psi^\imath(b_j))=b_j+\psi^\imath(b_j)$, it follows that the vectors $\{\tilde{v_1},\dots ,\tilde{v_n}\}$ are $\imath$bar invariant. 
  
Let $f_1,\dots,f_n\in V(\gamma)^*$ be the dual basis of $v_1,\dots,v_n$.
For $1\leq j \leq n$ let $\tilde{f}_j=f_j\circ p$, then it holds that
    \begin{align*}
        (\overline{\tilde{f}_k}\triangleleft\Upsilon^{-1})(\tilde{v}_j)
        =\overline{\tilde{f}_k}(\Upsilon^{-1}\tilde{v}_j)
        =\overline{\tilde{f}_k (\overline{ \Upsilon^{-1}\tilde{v}_j})}
        =\overline{\tilde{f}_k (\overline{ \overline{\tilde{v}_j}})}
        =\overline{f_k (v_j)}=\overline{\delta_{k,j}}
        =\delta_{k,j},
    \end{align*}
    which shows $\overline{\tilde{f}_k}\Upsilon^{-1}=\tilde{f}_k$. Write $L(\lambda)=V(\gamma)\bigoplus V(\gamma)^\perp$, with respect to the decomposition of $L(\lambda)$ as a $\uqb$-module, which is guaranteed by Definition \ref{def:qtrip}. As $\Upsilon\overline{V(\gamma)}=V(\gamma)$, it follows that $\Upsilon\overline{V(\gamma)^\perp}=V(\gamma)^\perp$. Let $w\in V(\gamma)^\perp$. Then
    \[\Upsilon^{-1}\triangleright\overline{\tilde{f}_k}(w)=f_k(p (\underbrace{\Upsilon\overline{w}}_{\in V(\gamma)^\perp})=f_k(0)=0.\]
    Thus, $\Upsilon^{-1}\triangleright\overline{\tilde{f}_k}=\tilde{f}_k$.
\end{proof}

Recall from Definition~\ref{def:qfa} the algebra $\mathcal A$ generated by finite-dimensional matrix elements
\begin{definition}
Denote by $\gls{Res}: \mathcal A\to k[X]$ the composition of $\Res_{\uq^0}:\mathcal A\to \Hom(\uq^0,k)$ with $\Hom(\uq^0,k)\cong k[X]$.

Let $M$ be a $\uq$-module. 
If $f\in M^*$, and $v\in M$ is written as $v=\sum_{\mu\in X} v_\mu$ for $v_\mu\in M_\mu$,
then
\[
  \Res(c_{f,v}^M) = \sum_{\mu\in X} f(v_\mu) e^\mu.
\]
\end{definition}

\begin{notation}\label{not:involution}
    Define the following action of $Y\otimes_\Z\Q$ on $k[X]$ (or even
    on $k[[X]]$):
    \[
        h\triangleright e^\lambda := q^{\langle h,\lambda\rangle}e^\lambda.
    \]
    This is an action by $k$-algebra automorphism.
    
 Moreover, denote $^\circ: k[X]\to k[X]$ the anti-linear
involution mapping $e^\lambda\mapsto e^\lambda$ for all
$\lambda\in X$.
\end{notation}

\begin{lemma}\label{lem-conjugation-matrix-elements}
  Let $\lambda\in X^+$, $v\in L(\lambda)$ and $f\in L(\lambda)^*$.
  Then
  \[
    c_{\overline{f},\overline{v}}^\lambda(K_h)
    = \overline{f(K_{-h}v)}.
  \]
\end{lemma}
\begin{proof}
  We have
  \[
    c_{\overline{f},\overline{v}}^\lambda(K_h)
    = \overline{f}(K_h\overline{v}).
  \]
  By definition of $\overline{f}$, this equals
  \[
    =\overline{f(\overline{K_h\overline{v}})}
    = \overline{f(K_{-h}v)}.\qedhere
  \]
\end{proof}

\begin{lemma}\label{lem-matrix-elements-bar-invariant}
  Let $\lambda\in X^+$, and let $v\in L(\lambda), f\in L(\lambda)^*$
  be $\imath$bar-invariant.
  Then
  \[
    c^\lambda_{\overline{f},\overline{v}}(K_h)
    = c^\lambda_{f\circ T_{w_\bullet}\mathcal{K},
    \mathcal{K}^{-1}T_{w_\bullet}^{-1}\circ v} (K_{w_0h}).
  \]
\end{lemma}
\begin{proof}
  Recall that $\mathcal{K} = \Upsilon \xi T_{w_\bullet}^{-1}T_{w_0}^{-1}$.
  This implies that
  \[
    \Upsilon = \mathcal{K}T_{w_0}T_{w_\bullet}\xi^{-1}
    = \mathcal{K}T_{w_0} (w_\bullet\xi)^{-1} T_{w_\bullet}.
  \]
  Moreover, since $\psi^{\iota}$ and $\overline{\cdot}$ coincide on
  $\uq_\bullet$, $\Upsilon$ commutes with $\uq_\bullet$ and hence also
  with its closure.
  This closure contains $T_{w_\bullet}$, so that
  \[
    \Upsilon = T_{w_\bullet}\Upsilon T_{w_\bullet}^{-1}
    = T_{w_\bullet} \mathcal{K} T_{w_0} (w_\bullet\xi)^{-1}.
  \]
  This allows us to conclude
  \begin{align*}
    c^\lambda_{\overline{f},\overline{v}}(K_h)
    &= c^\lambda_{f\Upsilon,\Upsilon^{-1}v}(K_h)
    = f(\Upsilon K_h \Upsilon^{-1}v)\\
    &= f(T_{w_\bullet}\mathcal{K}T_{w_0}(w_\bullet\xi)^{-1}
    K_h (w_\bullet\xi) T_{w_0}^{-1}\mathcal{K}^{-1}T_{w_\bullet}^{-1}v)\\
    &= f(T_{w_\bullet}\mathcal{K}T_{w_0}K_h T_{w_0}^{-1}\mathcal{K}^{-1}T_{w_\bullet}v)\\
    &= f(T_{w_\bullet}\mathcal{K} K_{w_0h} \mathcal{K}^{-1}T_{w_\bullet}^{-1}v)\\
    &= c^\lambda_{f\circ T_{w_\bullet}\mathcal{K}, \mathcal{K}^{-1}T_{w_\bullet}^{-1}v}(K_{w_0h}).\qedhere
  \end{align*}
\end{proof}

\begin{lemma}\label{lem-weyl-invariance}
    Let $\uqbs,\uqb_{\bm{d},\bm{t}}$ be coideal subalgebras derived
    from the same admissible pair,
    where $\bm{c},\bm{d}$ are related as
    follows:
    \[
        \begin{cases}
            c_i=d_i & i=\tau(i) \text{ or } \alpha_i\cdot\Theta(\alpha_i) = 0\\
            c_ic_{\tau(i)} = d_id_{\tau(i)} & i\ne\tau(i) \text{ and } \alpha_i\cdot\Theta(\alpha_i)\ne0
        \end{cases}.
    \]
    Let $f\in E^{\epsilon,\epsilon}_{\uqbs,\uqb_{\bm{d},\bm{t}}}$ and write $P:= (-\rho)\triangleright\Res(f)$.
    We have $P\in k[2L]^{\gls{WSigma}}$.
    
    In particular this implies $P=\tau_0(\overline{P})$ 
    and $w_\bullet P = P$.
\end{lemma}
\begin{proof}
    The first claim follows from
    \cite[Corollary~5.4]{Let03} and
    \cite[Theorem~5.15]{Mee24}).

    More precisely, in Letzter's language, the restriction of a
    ZSF for $(\uqb',\uqb_{\theta,\bm{s},\bm{d}})$
    is $W_\Sigma$-invariant if $\uqb' = \chi_{\bm{c}}(\uqb_{\theta,\bm{s}',\bm{c}^2\bm{d}})$ (here $\chi_{\bm{c}}$ can be written in terms of the Hopf-algebra automorphism $\ad(\xi)$ from Definition \ref{def-scaling}, see \cite[p. 277]{Let03})
    for $\bm{s}'$ an allowed non-standard parameter and
    $c_i\ne1$ implying that $\tau(i)\ne i$ and $\alpha_i\cdot\Theta(\alpha_i)\ne0$.
    In particular, we have
    \[
        \chi_{\bm{c}}(B_{i,\theta,s_i',c_i^2d_i})
        = c_iq^{\frac{1}{2}\langle\rho,\Theta(\alpha_i)+\alpha_i\rangle}\ad(K_\rho)(B_{\theta,i,s''_i,\frac{c_id_i}{c_{\tau(i)}}})
    \]
    for $s''_i = \frac{s'_i}{c_iq^{\frac{1}{2}\langle\rho,\Theta(\alpha_i)+\alpha_i\rangle}}$.
    Note that $s_i'\ne0$ is only allowed for cases when $c_i=1$ and
    $\alpha_i=-\Theta(\alpha_i)$, so that $s''_i=s'_i$.
    This shows in particular that we can rewrite the condition
    on $\uqb'$ as $\uqb' = \ad(K_\rho)(\uqb_{\theta,\bm{s}',\frac{\bm{c}\bm{d}}{\tau(\bm{d})}})$,
    or alternatively that $\uqb'=\ad(K_\rho)(\uqb_{\theta,\bm{s}',\bm{d}'})$ where
    \[
        \begin{cases}
            d_i=d_i' & i=\tau(i) \text{ or } \alpha_i\cdot\Theta(\alpha_i) = 0\\
            d_id_{\tau(i)} = d_i'd_{\tau(i)}' & i\ne\tau(i) \text{ and } \alpha_i\cdot\Theta(\alpha_i)\ne0.
        \end{cases}
    \]
    To translate this statement to our notations we assume that there
    is an involutive algebra automorphism and coalgebra antiautomorphism
    $\phi$ of $\uq$ mapping $\phi(\uqbs) = \uqb_{\theta,\bm{Y}\bm{s},\bm{X}\bm{c}}$
    for some structural constants $\bm{X},\bm{Y}$ that only depend on
    $I_\bullet,\tau$.
    Assume furthermore that $\phi(K_h)=K_{-h}$.

    Let now $f\in E^{\epsilon,\epsilon}_{\uqbs,\uqb_{\bm{d},\bm{t}}}$,
    then $P
    = \Res(f\triangleleft K_{-\rho})$ where
    $f\triangleleft K_{-\rho} \in E^{\epsilon,\epsilon}_{\ad(K_\rho)(\uqbs),\uqb_{\bm{d},\bm{t}}}$.
    Then $(f\triangleleft K_{-\rho})\circ\phi^{-1}$ is a ZSF for $\ad(K_\rho)(\uqb_{\theta,\bm{Y}\bm{s},\bm{X}\bm{c}})$ and $\uqb_{\theta,\bm{Y}\bm{t},\bm{X}\bm{d}}$
    in Letzter's parlance.
    The conditions that $\bm{c},\bm{d}$ satisfy can be rewritten
    as follows
    \[
        \begin{cases}
            c_iX_i=d_iX_i & i=\tau(i) \text{ or } \alpha_i\cdot\Theta(\alpha_i) = 0\\
            c_ic_{\tau(i)} X_iX_{\tau(i)} = d_id_{\tau(i)} X_i X_{\tau(i)} & i\ne\tau(i) \text{ and } \alpha_i\cdot\Theta(\alpha_i)\ne0,
        \end{cases}
    \]
    which by the previous argument shows that
    \[
        \Res((f\triangleleft K_{-\rho})\circ\phi^{-1})=\overline{P}
    \]
    is invariant under $W_\Sigma$, hence so is
    $P$.

    To prove the two conclusions, we recall that $W_\Sigma$
    and its actions are obtained from the group $\gls{WTheta}$, the
    subgroup of elements of $W$ that commute with $\Theta$,
    and its normal subgroup $W_\bullet$ (\cite[\S2.9]{Araki}).
    In particular, $P$ is invariant under
    $W_\Theta$.

    Denote by $R_\bullet$ the root system generated by the simple roots $\alpha_j$ with $j\in I_\bullet$.
    Since $\Theta$ leaves $R_\bullet$ invariant, $W_\bullet$ and
    in particular $w_\bullet$ commutes with $\Theta$.
    Moreover, we have $\Theta w_0\Theta = w_\bullet\tau w_0 w_\bullet \tau$.
    Since $\tau$ is a diagram morphism that permutes the elements of $I_\bullet$,
    it commutes with $w_\bullet$. Since it also permutes the elements of $I$, it
    commutes with $w_0$, so we have
    \[
        \Theta w_0\Theta = w_\bullet\tau w_0\tau w_\bullet
        = w_\bullet w_0 w_\bullet.
    \]
    Since $w_0$ can be written as the composition of $-1$ with the
    diagram automorphism $\tau_0$ (that is the nontrivial
    diagram automorphism if $\uq$ is of type $A_n, D_n$ ($n$ odd),
    or $E_6$), we have
    \[
        \Theta w_0\Theta = w_0 \tau_0(w_\bullet)w_\bullet.
    \]
Using \cite{Araki}, it can be checked that in every case $\tau_0$ permutes $I_\bullet$,
    and as such $\tau_0(w_\bullet)=w_\bullet$.
    Consequently, also $w_0$ commutes with $\Theta$ and hence also
    leaves $P$ invariant.
    \qedhere
\end{proof}

    
    

Recall the involutions $^\circ:k[2L]\to k[2L]$ and $\sigma=\tau\circ\tau_0:\uq\to\uq$ from Notation \ref{not:involution} and the beginning of Section \ref{sec:barinv}, respectively. 

\begin{definition}\label{def:matrixweight}
Let $\gamma\in \Gamma$ and $\lambda,\lambda'\in X^+(\gamma)$, then using Definition~\ref{def:XI} we define the \emph{matrix weight}
\[\gls{Hlambdagamma}:=(-\rho)\triangleright \Res\big(\Xi_\gamma(\Phi^\lambda_\gamma,\Phi^{\lambda'}_\gamma)\big).\]
\end{definition}

\begin{lemma}\label{lem-matrix-weight-conjugation}
  Let $\lambda,\lambda'\in X^+(\gamma)$, then
  $(H^{\lambda,\lambda'}_\gamma)^\circ = H^{\lambda,\lambda'}_{\sigma(\gamma)}$.
\end{lemma}
\begin{proof}
  Choose $\imath$bar-invariant dual bases $v_1,\dots,v_n\in V(\gamma)$
  and $f_1,\dots,f_n\in V(\gamma)^*$ and linear maps
  $p: L(\lambda)\to V(\gamma), p': L(\lambda')\to V(\gamma)$ and
  $i: V(\gamma)\to L(\lambda), i': V(\gamma)\to L(\lambda')$ as in
  Proposition~\ref{prop:blackmagic-ibarinv-bases} (for $\lambda$ and
  $\lambda'$).
  Then $p,i$ (resp. $p',i'$) can be used to define the elementary
  MSF $\Phi^\lambda_\gamma$ (resp. $\Phi^{\lambda'}_\gamma$).
  In particular, we have
  \[
    \Phi^\lambda_\gamma = \sum_{j,k=1}^n c^\lambda_{f_j\circ p, i(v_k)} v_j\otimes f_k.
  \]
  This shows that for $h\in Y$ we have
  \begin{equation}\label{eq:3591}
    H^{\lambda,\lambda'}_\gamma(h) =
    \sum_{j,k=1}^n c^\lambda_{f_j\circ p,i(v_k)}(K_{h-\rho})
    c^{\lambda'}_{f_k\circ p', i'(v_j)}(K_{-h-\rho}).
  \end{equation}
  We apply Lemma~\ref{lem-conjugation-matrix-elements} to obtain
  \begin{equation}\label{eq:3592}
    \eqref{eq:3591}= \qty(\sum_{j,k=1}^n c^\lambda_{\overline{f_j\circ p},
    \overline{i(v_k)}}(K_{-h+\rho})
    c^{\lambda'}_{\overline{f_k\circ p'},\overline{i'(v_j)}}(K_{h+\rho}))^{-}.
  \end{equation}
  Since by choice, the elements $f_j\circ p, i(v_j)$ ($j=1,\dots,n$,
  also the primed versions) are $\imath$bar-invariant, we can apply
  Lemma~\ref{lem-matrix-elements-bar-invariant} and obtain
  \begin{equation}\label{eq:3593}
   \eqref{eq:3592}= \qty(\sum_{j,k=1}^n c^\lambda_{f_j\circ p\circ T_{w_\bullet}\mathcal{K},
    \mathcal{K}^{-1}T_{w_\bullet}^{-1}i(v_k)}(K_{-w_0h-\rho})
    c^{\lambda'}_{f_k\circ p'\circ T_{w_\bullet}\mathcal{K},\mathcal{K}^{-1}T_{w_\bullet}^{-1}i'(v_j)}(K_{w_0h-\rho}))^{-}.
  \end{equation}
  Note that $p,i,p',i'$ are $\uq_\bullet$-linear, so they all commute with
  $T_{w_\bullet}$.
  Furthermore, $(T_{w_\bullet}^{-1}v_j)_{j=1,\dots,n}$ and
  $(f_jT_{w_\bullet})_{j=1,\dots,n}$ are dual bases of $V(\gamma)$ and its
  dual, and we can replace them with our original $(v_j)_{j=1,\dots,n}$
  and $(f_j)_{j=1,\dots,n}$, so that the above reads
  \begin{equation}\label{eq:3594}
    \eqref{eq:3593}= \qty(\sum_{j,k=1}^n c^\lambda_{f_j\circ p\circ\mathcal{K},
    \mathcal{K}^{-1}i(v_k)}(K_{-w_0h-\rho})
    c^{\lambda'}_{f_k\circ p'\circ\mathcal{K},\mathcal{K}^{-1}i'(v_j)}(K_{w_0h-\rho}))^{-}.
  \end{equation}
  We can now apply Lemma~\ref{lem-K-twists-sigma} to
  $(p,i)$ and $(p',i')$, which shows that
  \begin{equation}\label{eq:3595}
    \eqref{eq:3594}= \overline{H^{\lambda,\lambda'}_{\sigma(\gamma)}(-w_0h)}.
  \end{equation}
  By Lemma~\ref{lem-weyl-invariance} we have
  \begin{equation}\label{eq:3596}
    \eqref{eq:3595}= \overline{H^{\lambda,\lambda'}_{\sigma(\gamma)}(-h)}
    = \qty(H^{\lambda,\lambda'}_{\sigma(\gamma)})^\circ(h).\qedhere
  \end{equation}
\end{proof}
Recall the involution $\overline{\,\cdot\,}:k[2L]\to k[2L]$, which is the $k$-linear map with $e^\lambda\mapsto e^{-\lambda}$ for all $\lambda\in 2L$.
This allows us to prove the following properties of the matrix weight from Definition \ref{def:matrixweight}:
\begin{proposition}\label{prop-weight-matrix-props}
  Let $\lambda,\lambda'\in X^+(\gamma)$.
  \begin{enumerate}
    \item We have
    \[
      \overline{H^{\lambda,\lambda'}_\gamma} = H^{\lambda',\lambda}_\gamma.
    \]
    \item We have
    \[
      \overline{H^{\lambda,\lambda'}_\gamma} = \tau_0 H^{\lambda,\lambda'}_\gamma.
    \]
    \item We have
    \[
      H^{\lambda,\lambda'}_\gamma = \tau H^{\lambda,\lambda'}_\gamma.
    \]
    \item We have
    \[
      \qty(H^{\lambda,\lambda'}_\gamma)^\circ = H^{\lambda,\lambda'}_{\sigma(\gamma)}
    \]
    \item In case $\gamma\cong \sigma(\gamma)$, we conclude that
    $H^{\lambda,\lambda'}_\gamma$ is invariant under $\cdot^\circ$.
  \end{enumerate}
\end{proposition}
\begin{proof}
\begin{enumerate}
  \item We have
    \begin{align*}
      H^{\lambda,\lambda'}_\gamma(h) &=
      \tr(\Phi^\lambda_\gamma(K_{h-\rho})
      \Phi^{\lambda'}_\gamma(K_{-h-\rho}))\\
      &= \tr(\Phi^{\lambda'}_\gamma(K_{-h-\rho})
      \Phi^{\lambda}_\gamma(K_{h-\rho}))\\
      &= H^{\lambda',\lambda}_\gamma(-h).
    \end{align*}
  \item We have
  \[
    H^{\lambda,\lambda'}_\gamma(-h) =
    H^{\lambda,\lambda'}_\gamma(w_0\tau_0h)
    = H^{\lambda,\lambda'}_\gamma(\tau_0h)
  \]
  by Lemma~\ref{lem-weyl-invariance}.
  \item We have $\tau = -w_\bullet\Theta$.
    By Lemma~\ref{lem-weyl-invariance}, $H^{\lambda,\lambda'}_\gamma$ is
    supported in $2L$, which is $\Theta$-anti-invariant, hence
    \[
      \tau H^{\lambda,\lambda'}_\gamma =
      \overline{w_\bullet \Theta H^{\lambda,\lambda'}}
      = \overline{w_\bullet \overline{H^{\lambda,\lambda'}_\gamma}}
      = H^{\lambda,\lambda'}_\gamma
    \]
    by Lemma~\ref{lem-weyl-invariance}.
  \item
    This is the statement of Lemma~\ref{lem-matrix-weight-conjugation}.
  \item
    Clear.
  \end{enumerate}
\end{proof}

\begin{remark}
The invariance of Proposition~\ref{prop-weight-matrix-props}(iv) can be observed in the existing literature and in the examples of this paper, cf. Section~\ref{sec-singlevarsmalk}, \ref{sec-singlevar2}, \cite[Theorem 5.7. \& (5.7)]{Mee25} and \cite[\S 6.8]{Ald}. 
\end{remark}

\begin{conjecture}\label{conj:multfree}
For each multiplicity free triple $(\uq,\uqb,\gamma)$ the twisted module $V(\gamma)^\sigma$ is isomorphic to $V(\gamma).$
\end{conjecture}

Recall that a Satake diagram is said to be of \textit{split type} if $I_\bullet=\emptyset$ and $\tau=\mathrm{id}$.

\begin{remark}
    It should be noted that the case 
    $\sigma\neq \mathrm{id}$ is rather sparse. 
    This inequality occurs in the cases 
    $\mathsf{AI}$, $\mathsf{AII}$, $\mathsf{EI}$, 
    $\mathsf{EIV}$ and may occur for 
    $\mathsf{DI}_{n,p}$ with $p$ odd, cf. \cite[\mbox{p 32-33}]{Araki}. Here $\mathsf{AI}$, $\mathsf{EI}$ are split and $\mathsf{DI}$ might be split, and hence in these cases we can apply Proposition \ref{prop:splitcase}.
    However, in the remaining cases, it still seems that $V(\gamma)\cong V(\gamma)^\sigma$. 
    One can check this for the multiplicity free triples 
    for $\mathsf{AII}$ and $\mathsf{EIV}$ listed in \cite[\mbox{Table 2}]{Pe23}.
\end{remark}

\begin{remark}
It is generally not true that $V(\gamma)^\sigma\cong V(\gamma)$. For an explicit counterexample, take the rank one symmetric pair $(\mathfrak{g},\mathfrak{k})=(\mathfrak{sl}_{4},\mathfrak{sp}_4)$ of type $\mathsf{AII}$. Given that $\mathfrak{k}$ is semisimple  it follows from \cite[7.10]{Let00} that the isomorphism $V(\gamma)^\sigma\cong V(\gamma)$ holds if and only if $V(\gamma)^\sigma\cong V(\gamma)$ at $q=1$. Let $\mathfrak{h}$ be a Cartan subalgebra of $\mathfrak{g}$ with basis $\{h_1,h_2,h_3\}$  Note that $\{h_1,h_3\}\subset \mathfrak{k}$ and that $\mathrm{span}_\C\{h_1,h_3\}=\mathfrak{t}$ is a Cartan subalgebra for $\mathfrak{k}$. We note $\sigma(h_1)=h_3$. For a dominant integral weight $\mu$ of $\mathfrak{t}$, we have $L(\mu)\cong L(\mu)^\sigma$, precisely when $\mu$ is of the form $\mu= a(\omega_1^\vee+\omega_3^\vee)$ for some $a\geq 0$. A similar result is obtained in higher rank and for type $\mathsf{EIV}$.
\end{remark}

In Proposition \ref{prop:splitcase} we show that Conjecture \ref{conj:multfree} holds for split types, even when $\sigma\neq \mathrm{id}_I$.

\begin{proposition}\label{prop:splitcase}
Let $(I_\bullet,\tau)$ be of split type and $\gamma\in \Gamma$, then $V(\gamma)^\sigma\cong V(\gamma)$.
\end{proposition}

\begin{proof}
We prove the proposition using a specialization argument. Note that the classical limit of the universal $K$-matrix and the longest element of the Wang--Zhang braid group operator both specialize to $T_{w_0}$, cf. \cite[Definition 10.4]{Wa23} \cite[\mbox{Proposition 4.2.1 (a), Proposition 4.2.2}]{Ap25}. Thus, it suffices to prove that the Wang--Zhang braid group operator $\mathbf{T}_{w_0}$ preserves the modules $V(\gamma)\subset L(\lambda).$ Let $w_0=s_1\dots s_n$ be a reduced expression of $w_0\in W_\Sigma=W$, then by \cite[Thm 10.6]{Wa23} $\mathbf{T}_{w_0}=\mathbf{T}_{s_1}\dots \mathbf{T}_{s_n}$. By the explicit expressions of $\mathbf{T}_{s_i}$ given in \cite[\S3.1.]{Wang2025}, each operator $\mathbf{T}_{s_i}$ preserves the module $V(\gamma).$ 
Thus, $\mathbf{T}_{w_0}:V(\gamma)\to V(\gamma)$ and as a result $V(\gamma)\cong V(\gamma)^\sigma$.
\end{proof}


\begin{remark}
    With use of $\Ad(\eta)$, the symmetries of 
    this section can be extended to all parameters
    that are in the $\mathrm{Hom}(X,k)$ orbit of
    uniform parameters. In particular, this shows 
    that these symmetries are present for each
    standard parameter in the reduced case.
\end{remark}

\section{Cartan Decomposition}\label{sec:cartandecomp}
The goal of this section is to show the existence of 
a radial component decomposition for suitable 
elements of $\uq$ using a Cartan decomposition of 
$\uq$.

For reasons elaborated on later in Remark~\ref{rmk-admodule-not-minuscule-affine}
this means that we need to assume that $I$ is of finite type, or at least that all rank-1 subdiagrams of $I$ satisfy the
Lemma~\ref{lem:admodule}(iii). Which in particular means that the results of this section hold for quasi-split quantum symmetric pairs.

More concretely, in this section we show that $\uqbs\uq^0\uqb_{\bm{d},\bm{t}}$ lies
Zariski-dense in $\uq$.
This implies that MSF for $\uqbs,\uqb_{\bm{d},\bm{t}}$ are uniquely
determined by their restrictions to $\gls{U0}$, the subalgebra
generated by the $(K_h)_{h\in Y}$.
Since the MSF are sums
of matrix elements, they are already determined by their restrictions to $\set{K_h\where h\in Y}$.

In case $Y=Y^\Theta\oplus Y^{-\Theta}$, these results can be
strengthened by replacing $\uq^0$ and $Y$ by $\uq^0_{-\Theta}$ and
$Y^{-\Theta}$,
where $\uq^0_{-\Theta}\le \uq^0$ is the group algebra of $Y^{-\Theta}$ (the
$\Theta$-antiinvariants).





Recall from \cite[\mbox{Prop~6.3}]{Kol14} the quantum Iwasawa decomposition of $\uq$
\begin{equation}\label{eq:triangularbab1}
    \uq'  \cong   \mathbf{V}_\bullet ^+\otimes\uq^0_{I_\circ,\tau} \otimes \uqbs,
\end{equation}
where under certain conditions $\uq'=\uq$ and $\uq_\Theta'$ is the group
algebra of a complement of $Y^\Theta$.
Here $\mathbf{V}_\bullet ^+$ is the subalgebra of 
$\uq^+$ generated by $\mathrm{ad}(\uq_\bullet)(E_i)$
with $i\in I_\circ$ and $\uq^0_{I_\circ,\tau}=k[K_i^\pm : i\in I_{\circ,\tau}]$ 
where $I_{\circ,\tau}\subset I_\circ$ contains precisely one representative for each $\tau$-orbit of $I_\circ$.

If we let go of a need for unique representation, the same proofs show
that 
\begin{equation}\label{eq:kolbiwasawa}
  \uq \cong \mathbf{V}_\bullet^+ \otimes \uq^0 \uqb_{\bm{d},\bm{t}}. 
\end{equation}
In order to show the desired decomposition, it thus remains to show
that $\uqbs \uq^0 \uqb_{\bm{d},\bm{t}}$ intersects densely with
$\mathbf{V}_\bullet^+ \uq^0$.
%

To that end, we have a closer look at $\bm{V}_\bullet^+$ and
in particular study the $\uq_\bullet$-modules (via the adjoint action) $\ad(\uq_\bullet)(E_i)$ and $\ad(\uq_\bullet)(F_iK_i)$.

\subsection{The $\uq_\bullet$-modules \texorpdfstring{$\ad(\uq_\bullet)(E_i)$, $\ad(\uq_\bullet)(F_iK_i)$}{} and their product}\label{sec:ubulletmodules}

Recall from Definition \ref{def:moduleactions} the adjoint action of $\uq$ on itself.
The goal of this section is to determine the $q$-commutation relations between elements of 
$\ad(\uq_\bullet)(E_i)$ 
and $\ad(\uq_\bullet)(F_iK_i)$. 
Our approach is to study the $\ad(\uq_\bullet)$-modules $\ad(\uq_\bullet)(E_i)$, $\ad(\uq_\bullet)(F_iK_i)$, and their product as an $\ad(\uq_\bullet)$-module. 
As a preliminary step, we show that $\ad(\uq_\bullet)(E_i)$ and $\ad(\uq_\bullet)(F_iK_i)$ are simple modules corresponding to (quasi-)minuscule weights. Recall that a weight $\lambda$ is minuscule if the weights of $L(\lambda)$ are contained in a single Weyl group orbit, and quasi-minuscule possibly with the addition of the zero-weight space.

\begin{lemma}\label{lem:admodule}\cite[Lemma 3.5]{Kol14}
Let $i\in  I_\circ$.
\begin{enumerate}
    \item The subspace $\ad(\uq_\bullet )(E_i)$ of $\uq^+$ is a finite dimensional, irreducible $\ad(\uq_\bullet )$-
submodule of $\uq$ with highest weight vector $T_{w_\bullet} (E_i)$ and lowest weight vector $E_i$.
\item The subspace $\ad(\uq_\bullet )(F_iK_i)$ of $S(\uq^-)$ is a finite dimensional, irreducible
$\ad(\uq_\bullet )$-submodule of $\uq$ with highest weight vector $F_iK_i$ and lowest
weight vector $T^{-1}_{w_\bullet} (F_iK_i)$.
\item The weights corresponding to the simple $\ad(\uq_\bullet)$-modules of $(i)$ and $(ii)$ are quasi-minuscule if the rank one diagram corresponding to $i$ is of type $\mathsf{BII}$, else the weight is minuscule.
\end{enumerate}
\end{lemma}

\begin{proof}
Statement $(i)$ and $(ii)$ are the content of \cite[Lemma 3.5]{Kol14}.

Consider the irreducible $\ad(\uq_\bullet)$-module $\ad(\uq_\bullet)(E_i)$. The $\uq_\bullet$-weight of this module is determined entirely by the rank-one diagram associated with the index $i\in I_\circ$. The rank one cases where $I_\bullet$ is empty are trivial. The remaining rank-one cases require further examination.
\begin{description}
\item[In case $\mathsf{AII_3}$] For $j=1,3$ we have
    $$\langle h_j, w_\bullet \alpha_2\rangle=-\langle h_{\tau(j)},\alpha_2\rangle=1.$$
    The Dynkin type of $\uq_\bullet$ is $\mathsf{A_1\times A_1}.$
    As such $w_\bullet\alpha_1$ is identified with the $\mathsf{A_1}\times \mathsf{A_1}$-weight $\omega_1+\omega_3$, which is minuscule.
\item[In case $\mathsf{AIV}_n$] For $j\in \{2,\dots,n-1\}$ and $i\in\{1,n\}$ we have
    $$\langle h_j, w_\bullet \alpha_i\rangle=-\langle h_{n-j},\alpha_i\rangle=\delta_{|i-n+j|,1}$$
    The Dynkin type of $\uq_\bullet$ is $\mathsf{A}_{n-2}$. As such $w_\bullet\alpha_1$ is identified with a fundamental weight for $\mathsf{A}_{n-2}$, all of which are minuscule.
\item[In case $\mathsf{BII}_n$]
    Let $i=1$ and $j\in \{2,\dots ,n\}$. Then 
    $$\langle h_j, w_\bullet \alpha_i\rangle=-\langle h_{\tau(j)},\alpha_i\rangle=\delta_{2,j}.$$
    The Dynkin type of $\uq_\bullet$ is $\mathsf{B}_{n-1}$. As such $w_\bullet\alpha_1$ is identified with the first fundamental weight $\omega_2$ for $\mathsf{B}_{n-1}$, which is quasi-minuscule.
\item[In case $\mathsf{CII}_n$]
    Let $i=2$ and $j\in \{1,3,\dots ,n\}$. Then 
    $$\langle h_j, w_\bullet \alpha_i\rangle=-\langle h_{\tau(j)},\alpha_i\rangle=\delta_{1,j}+\delta_{3,j}.$$
    The Dynkin type of $\uq_\bullet$ is $\mathsf{A}_1\times \mathsf{C}_{n-2}$. As such $w_\bullet \alpha_i$ is identified with the miniscule $\mathsf{A}_1\times\mathsf{C}_{n-2}$-weight $\omega_1+\omega_3$.
\item[In case $\mathsf{DII}_n$]
    Let $i=1$ and $j\in \{2,\dots ,n\}$. Then 
    $$\langle h_j, w_\bullet \alpha_i\rangle=-\langle h_{\tau(j)},\alpha_i\rangle=\delta_{2,j}.$$
    The Dynkin type of $\uq_\bullet$ is $\mathsf{D_{n-1}}$. As such $w_\bullet \alpha_i$ is identified with the the first fundamental weight $\omega_2$ of $\mathsf{D}_{n-1}$, which is minuscule.
\item[In case $\mathsf{FII}$]
    Let $i=4$ and $j\in \{1,2,3\}$. Then 
    $$\langle h_j, w_\bullet \alpha_i\rangle=-\langle h_{j},\alpha_i\rangle=\delta_{3,j}.$$
    The Dynkin type of $\uq_\bullet$ is $\mathsf{B_{3}}$. As such $w_\bullet \alpha_i$ is identified the last fundamental weight for $\mathsf{B_{3}}$, which is minuscule.
\end{description}
A similar check can be carried out for the modules $\ad(\uq_\bullet)(F_iK_i)$.
\end{proof}

\begin{remark}\label{rmk-admodule-not-minuscule-affine}
  Note that Lemma~\ref{lem:admodule}(iii) is not true in general if
  $\uq$ is of affine type.
  Consider for example the diagram of type $\widetilde{\mathsf{BC}_1}$
  with one vertex coloured in: \dynkin[labels={0,1}, scale=1.8] A[2]2
  corresponding to an affine Lie algebra of type $\mathsf{A}_2^{(2)}$.
  We take $I=\set{0,1}$ and
  \[
    \langle h_0,\alpha_1\rangle = -1,\qquad
    \langle h_1,\alpha_0\rangle = -4.
  \]
  Taking $I_\bullet=\set{1}$, we have
  \[
    \rho_\bullet^\vee = \frac{h_1}{2},
  \]
  so that $\langle \rho_\bullet^\vee,\alpha_0\rangle = 2\in\Z$.
  This shows that $(I_\bullet,\id)$
  is an admissible Satake diagram.
  Furthermore, since $\langle h_1,-\alpha_0\rangle=4$, the projection
  of $-\alpha_0$ to the weight lattice of $\uq_\bullet$ is $4\omega_1$,
  which is neither a minuscule nor a pseudo-minuscule weight.
  
  We also note that it is this result that makes the remainder
  of this section wrong in general for admissible pairs of infinite type.
  If every rank-1 subdiagram has either a minuscule weight or is of
  type $\mathsf{BII}$, the proofs in the remainder of this section still
  work.
\end{remark}

\begin{notation}
Let $T^{\mathrm{ad}}$ be the Lusztig braid group operator associated with $\uq_\bullet$ acting on the $\ad(\uq_\bullet)$-modules $\ad(\uq_\bullet)(F_iK_i)$, $\ad(\uq_\bullet)(E_i)$ and their (tensor-)product.
\end{notation}

\begin{corollary}\label{cor:minuscule}
We have the decompositions $$\ad(\uq_\bullet)(E_i)=\mathrm{span}_k\{T^{\mathrm{ad}}(W_\bullet)(E_i)\}\oplus \ad(\uq_\bullet)(E_i)_0$$
and
$$\ad(\uq_\bullet)(F_iK_i)=\mathrm{span}_k\{T^{\mathrm{ad}}(W_\bullet)(F_iK_i)\}\oplus \ad(\uq_\bullet)(F_iK_i)_0.$$
Furthermore, $\dim_k\big(\ad(\uq_\bullet)(F_iK_i)_0\big),
\dim_k\big(\ad(\uq_\bullet)(E_i)_0\big)\leq 1$, with equality holding if and only if the rank-one diagram corresponding to $i\in I_\circ$ is of type $\mathsf{BII}$.
\end{corollary}

\begin{proof}
This follows from Lemma \ref{lem:admodule}.
\end{proof}

\begin{lemma}\label{lem-quotient}
Let $i,k\in I_\circ$. The $\ad(\uq_\bullet)$-module $V=\ad(\uq_\bullet)(F_{i}K_i)\ad(\uq_\bullet)(E_k)$ is a quotient of the $\ad(\uq_\bullet)$-module
$\ad(\uq_\bullet)(F_{i}K_i)\otimes\ad(\uq_\bullet)(E_k)$.
\end{lemma}

\begin{proof}
The action of $\ad(\uq_\bullet)$ on the tensor product is by the coproduct. As $\ad(x)(ab) = \ad(x_{(1)})(a)\ad(x_{(2)})(b)$, this implies that the multiplication map
$$m:\ad(\uq_\bullet)(F_{i}K_i)\otimes\ad(\uq_\bullet)(E_k)\to \ad(\uq_\bullet)(F_{i}K_i)\ad(\uq_\bullet)(E_k),\qquad\qquad x\otimes y\mapsto xy$$
is a surjective $\ad(\uq_\bullet)$-homomorphism. This establishes the claim.
\end{proof}

Let $M$ and $N$ be finite-dimensional integrable $\uq_\bullet$-modules. The Lusztig braid group action is related to the tensor product via the rank-one quasi $R$ matrix $L_j$ for $j\in I_\bullet$ cf. \cite[\S5.3.1]{Lus10}, then by \cite[Proposition~5.3.4]{Lus10}
\begin{equation}\label{eq:braidcoproduct}
T_j\otimes T_j(z)= T_j(L_i(z)),\qquad \qquad \forall z\in M\otimes N.
\end{equation}
On the right hand side the action of the braid group operator $T_j$ is regarded as the action on the tensor product $M\otimes N$. 

\begin{definition}
Let $J$ be a finite sequence of simple roots in $I_\bullet $ and $i,k\in I_\circ$. We set \[\kappa(J,i,k)=\langle h_i, \sum_{j\in J}\alpha_j+\alpha_k\rangle.\]
\end{definition}

Putting the previous results together, we have the tools to study the $q$-commutators
\[
\comm{x}{\ad(E_J)(E_k)}_{q_i^{\kappa(J,i,k)}} \qquad x\in \ad(\uq_\bullet)(F_{i}K_i).
\]
Firs, recall that for a finite sequence of simple roots $J=\{i_1,\dots ,i_n\}$ we write $E_J$ to denote the element $E_{i_1}\dots E_{i_n}$.

\begin{proposition}\label{prop:commutationlowerdegree}
Let $J$ be a finite sequence of simple roots, $i,k \in I_\circ $ and $x\in \ad(\uq_\bullet)(F_{i}K_i)$, then 
\begin{equation}\label{eq:prop:commutation}[x,\ad(E_J)(E_k)]_{q_i^{\kappa(J,i,k)}}\in \delta_{i,k}\bigg(\uq_\bullet \big(K_i^2-1\big)\bigg).
\end{equation}
\end{proposition}

\begin{remark}
Note that Proposition \ref{prop:commutationlowerdegree} is clearly true in case $I_\bullet=\emptyset$, the technical difficulty is dealing with the case $I_\bullet\neq\emptyset$.
\end{remark}

\begin{proof}
Assume $x,\ad(E_J)(E_k)$ are both nonzero. Consider the case where $y:=\ad(E_J)(E_k)$ is generic and $x= T^{\ad}_w (F_iK_i)$ with $w\in W_\bullet,$ cf. Corollary \ref{cor:minuscule}. Let $w=s_{i_1}\dots s_{i_n}$ be a reduced expression. The identity $xy=m(x\otimes y)$ implies that by repeated application of \eqref{eq:braidcoproduct}
\begin{align}\label{eq:reverse1}
xy&=m(T^{\ad}_w (F_iK_i)\otimes y)\\
&=m(T^{\ad}_w (F_iK_i)\otimes T^{\ad}_w\circ (T^{\ad}_w)^{-1}(y))\nonumber\\
&=m\big(T_{i_1}^{\ad}\ad(L_{i_1})\dots T_{i_n}^{\ad}\ad(L_{i_n})
(F_iK_i\otimes (T^{\ad}_w)^{-1}(y))\big).\nonumber
\end{align}
The action $T_{i_1}^{\ad},\dots, T_{i_n}^{\ad}$ on weight vectors of $\ad(\uq_\bullet)(F_{i}K_i)\otimes\ad(\uq_\bullet)(E_k)$ is given by elements $z_{i_1},\dots z_{i_n}\in\uq_\bullet$ acting on the $\ad(\uq_\bullet)$-module $\ad(\uq_\bullet)(F_{i}K_i)\otimes\ad(\uq_\bullet)(E_k)$, cf. \cite[\S5.2.1]{Lus10}. That is to say
\begin{align}\label{eq-adz}
 T_{i_1}^{\ad}\ad(L_{i_1})\dots T_{i_n}^{\ad}\ad(L_{i_n})
(F_iK_i\otimes (T^{\ad}_w)^{-1} (y))\\
=\ad(z_{i_1})\ad(L_{i_1})\dots\ad(z_{i_n}\ad(L_{i_n})(F_iK_i\otimes (T^{\ad}_w)^{-1}\nonumber (y))\\
=\ad(z)(F_iK_i\otimes (T^{\ad}_w)^{-1} (y))\nonumber
\end{align}
\cite[\S5.2.1]{Lus10} and \eqref{eq-adz} also imply that the element $z$ is determined by the homogeneous $\uq_\bullet$-weight of $x\otimes y$. From Lemma~\ref{lem:admodule}, it follows that 
\begin{align}\label{eq:reverse2}
&m\big(T_{i_1}^{\ad}\ad(L_{i_1})\dots T_{i_n}^{\ad}\ad(L_{i_n})
(F_iK_i\otimes (T^{\ad}_w)^{-1} (y))\big)\\\nonumber
&\stackrel{\eqref{eq-adz}}{=}m(\ad(z)(F_iK_i\otimes (T^{\ad}_w)^{-1} (y)))\\\nonumber
&\stackrel{Lem \ref{lem-quotient}}{}\ad(z)
(F_iK_i(T^{\ad}_w)^{-1}(y))\\\nonumber
&=q_i^{\kappa(J,i,k)}\ad(z)
((T^{\ad}_w)^{-1}(y)F_iK_i)\\
&\quad +\delta_{i,k}\bigg(\uq_\bullet\big(K_i^2-1\big)\bigg).\nonumber
\end{align}
As $z\in\uq_\bullet$ is determined by the $\uq_\bullet$-weight of $x\otimes y$, which is the same for $y\otimes x$, it follows from \cite[\S5.2.1]{Lus10} that $\ad(z)
((T^{\ad}_w)^{-1}(E_k))F_iK_i)=yx$, by reversing the steps in \eqref{eq:reverse1} and \eqref{eq:reverse2}. An analogous argument applies when $y= T_{w}^{\ad}(E_k)$ and $x$ is any weight vector. 

By Corollary \ref{cor:minuscule}, it remains to consider the case where $x\in \ad(\uq_\bullet)(F_iK_i)_0$ and $y\in \ad(\uq_\bullet)(E_k)_0 $. This implies both $i$ and $k$ are of type $\mathsf{BII}$. Assuming irreducibility of $I$, we must have $i=k$. The commutation relations in $\uq$ with $x=\ad(z^-)(F_iK_i)$ and $y=\ad(z^+)(E_i)$ for $z^-\in \uq_\bullet^-$ and $z^+\in \uq_\bullet^+$ directly yield \eqref{eq:prop:commutation}.
\end{proof}

\subsection{A Cartan decomposition of $V_\bullet^+ \uq^0$}


\begin{definition}\label{def:regular}[Regular pair]
A \emph{regular pair} $(K,J)$ consists of
\begin{enumerate}
    \item $K=(\alpha_{i_1},\dots ,\alpha_{i_n})$, a finite sequence of simple roots in $I_\circ$ ;
    \item $J=(J_1,\dots, J_n)$, a sequence of finite sequences of simple roots in $I_\bullet$;
\end{enumerate}
with the property that $\ad(E_{J_{i_j}})(E_{i_j})\neq 0$ for all $1\leq j \leq n.$ 
\end{definition}

\begin{remark}
In Definition \ref{def:regular} we allow empty sequences $()$ for $K$ and $J_j$, taking the convention $E_{()}=1$.
\end{remark}

\begin{definition}
Let $X$ be a finite set and let $\mathrm{Fin}(X)$ denote the set of finite sequences
of elements in $X$. Let $\varsigma : \mathrm{Fin}(X)\to\mathrm{Fin}(X)$ be the permutation map $\varsigma (x_1,\dots,x_n) =
(x_2,...,x_n,x_1).$
\end{definition}

\begin{notation}
For each regular pair $(K,J)$ and $h\in Y$, set
\begin{equation}\label{eq:CIL}
    C(J,K,h)= \frac{d_{\tau(i_1)}}{c_{\tau(i_1})}q-\mathrm{exp}\bigg(\sum_{j=2}^n\langle \epsilon_{i_{1}}h_{i_{1}},-\Theta\big(\alpha_{i_j}+\mathrm{wt}(J_{i_j})\big)\rangle -\langle h,\alpha_{i_1}-\Theta(\alpha_{i_1})\rangle \bigg).
\end{equation}
Moreover, set 
\begin{equation}
R(J,K,h)=\prod_{j=1}^n C(\varsigma^{j-1}(J),\varsigma^{j-1}(K),h).
\end{equation}
\end{notation}

\begin{notation}
Let $K=(\alpha_{i_1},\dots ,\alpha_{i_n})$, a finite sequence of simple roots in $I_\circ$, then we write $\mathrm{wt}(K)=\sum_{j=1}^n\alpha_{i_j}$.
\end{notation}

\begin{lemma}\label{rem:finmany}
    Let $(K,J)$ be a regular pair. Then, there are finitely many regular pairs $(K',J')$ with $\mathrm{wt}(K')\leq \mathrm{wt}(K)$.
\end{lemma}

\begin{proof}
Note that there are finitely many $K'$ with $\mathrm{wt}(K')\leq \mathrm{wt}(K)$. Fix such a finite sequence $K'=(\alpha_{i_1},\dots ,\alpha_{i_n})$. For each $1\leq j \leq n$ it follows by Lemma \ref{lem:admodule} that there are finitely many $J$ such that $\ad(E_J)(E_{i_j})\neq 0.$ Therefore there are finitely many regular pairs $(K',J')$ with $\mathrm{wt}(K')<\mathrm{wt}(K)$.
\end{proof}

\begin{definition}\label{def:refular}[Regular values]
Let $(K,J)$ be a regular pair, then the set of \emph{$(K,J)$-regular values} $\gls{Yreg}(K,J)\subset Y$ is defined as $Y$ if $K=\emptyset$ and otherwise it is inductively defined as the set of $h\in Y$ such that
\begin{enumerate}
    \item $h-\Theta(h_{i}),h, h+2h_i\in Y_{\mathrm{reg}}(K',J')$ for all $i\in I$, regular pairs $(K',J')$ with $\mathrm{wt}(K')<\mathrm{wt}(K)$;
    \item for all $1\leq i \leq n-1$ we have $R(\varsigma^i(K),\varsigma^i(J),h)\neq0$.
\end{enumerate}
\end{definition}

\begin{remark}\label{rem:zariski}
Let $(K,J)$ be a regular pair. Note that the set of $h\in Y$ that satisfy condition (ii) from Definition \ref{def:refular} for $(K,J)$ equals $Y$ minus a finite number of hyperplanes. The inductive definition of regular values together with Lemma \ref{rem:finmany} implies that $Y_{\mathrm{Reg}}(K,J)$ is contained in $Y$ minus a finite number of hyperplanes. Therefore, $Y_{\mathrm{Reg}}(K,J)$ is Zariski dense in $Y$.
\end{remark}

Next, we extend Definition \ref{def:refular} to $V_\bullet^+$ and $\uq$. For this let 
\[\{N(K_\varsigma,J_\varsigma)=\ad(E_{J_{\varsigma_1}})(E_{i_{\varsigma_1}})\dots \ad(E_{\varsigma_{\mathrm{max}}})(E_{{\varsigma_{\mathrm{max}}}}): \varsigma\in \delta\}\subset V_\bullet^+\]
be a basis of $V_\bullet^+$, for some index set $\delta$.
\begin{definition}[Regular values of $\uq$] \label{def-yreg}      
    Let $y=\sum_{(K,J)\in\delta } a_{\varsigma} N(K,J)\in V_\bullet^+$, then define 
    \[Y_{\mathrm{Reg}}(y)=\bigcap_{(K,J)\in\delta}Y_{\mathrm{Reg}}(K,J).\]
    Extending this to elements of $\uq$, making use of the quantum Iwasawa decomposition, let 
    \[x=\sum_{i=1}^n a_{i} N(\varsigma_i)K_{\mu_i} B_{L_i},\] then   
    \[
        Y_{\mathrm{Reg}}(x)=\bigcap_{i=1}^n\big(Y_{\mathrm{Reg}}\big(N(K_{\varsigma_i},J_{\varsigma_i})\big)-\mu_i\big).
    \]
\end{definition}

\begin{remark}
The definition of $Y_{\mathrm{Reg}}:V_\bullet^+\to \mathcal P(Y)$ depends on a choice of basis of $V_\bullet^+$, Section \ref{sec:ubulletmodules} gives us a canonical choice.
\end{remark}



\begin{notation}
Denote by $\sim$ the equivalence relation modulo 
$\uqbs \uq^0\uqds$. That is to say $x\sim y$ if and only if $x-y\in \uqbs \uq^0\uqds$.
\end{notation}

Next, we introduce the restricted grading of $\uq$ as in \cite[p. 97]{Let04}. We note that $V_\bullet ^+$ can be written as a direct sum of weight spaces. For $\beta\in X$, recall the notation $\tilde{\beta}=\frac{\beta-\Theta(\beta)}{2}$.

\begin{definition}
    
For a subset $S\subset \uq$ and weight $\beta\in X$, we write

\[S_{\beta,r}
=\sum_{\substack{\beta'\in X\\\widetilde{\beta}=\widetilde{\beta'}}}
S_{\beta'}.\]

\end{definition}

Recall that $\uq^+ \cong V_{\bullet}^+\otimes \uq^+_\bullet$ as linear spaces via the multiplication map, cf.~\cite[(3.11)]{Kol14}. As $\widetilde{\alpha_j}=0$ for all $j\in I_\bullet$, we note that 
\begin{equation}\label{eq:inductionstep}
\uq^+_{\beta,r} \cong \uq^+_\bullet\otimes V_{\bullet,\beta,r}^+.
\end{equation}

\begin{theorem}\label{thm:reduce}
Let $(K,J)$ be a regular pair and let $h \in Y_{\mathrm{Reg}}(K,J)$ be a regular value, then
	\begin{equation}\label{eq:sim0}
	    \ad(E_{J_1})(E_{i_1})\dots \ad (E_{J_k})(E_{i_k})K_h \sim  0.
	\end{equation}
\end{theorem}
\begin{proof}

We proceed by induction on the restricted weight $\widetilde{\beta}$ of $\ad(E_{J_1})(E_{i_1})\dots \ad (E_{J_k})(E_{i_k})$. 

For the induction base consider the case that $\widetilde{\beta}=0$. Then $K=()$ and hence
$$\ad(E_{J_1})(E_{i_1})\dots \ad (E_{J_k})(E_{i_k})K_h =  K_h\sim 0.$$
Now assume that \eqref{eq:sim0} holds for all $\ad(E_{N_1})(E_{m_1})\dots \ad (E_{N_p})(E_{m_p})$ with restricted weight $\tilde{\beta'}$, where $\tilde{\beta'}<\tilde{\beta}$.
    \begin{description}
        \item[Claim 1] Assume that 
        $\ad(E_{J_1})(E_{i_1})\dots \ad (E_{J_k})(E_{i_k})$ has restricted weight $\tilde{\beta}$. Then
        \begin{align*}
            \ad(E_{J_1})(E_{i_1})\dots \ad (J_k)(E_{i_k})K_h  
            \sim C(K,J,h)\ad(E_{J_2})(E_{i_2}) \dots \ad (E_{J_k})(E_{i_k})\ad(E_{J_1})(E_{i_1})K_h,
        \end{align*}
here $C(K,J,h)$ is defined in \eqref{eq:CIL}.
        \item[Proof Claim 1] 
        We can write
        \begin{equation}\label{eq:BiasEF}
        E_{i_1} = \frac{T_{w_\bullet}^{-1}(B_{\tau(i_1)}K_{\tau(i_1)}) - T_{w_\bullet}^{-1}(F_{\tau(i_1)}K_{\tau(i_1)}) - s_{\tau(i_1)}}{c_{\tau(i_1)}}. 
        \end{equation}
        Consequently,       
        \begin{align*}
          & \ad(E_{J_1})(E_{i_1})\dots \ad (E_{J_k})(E_{i_k})K_h  \\
           = &\frac{1}{c_{\tau(i_1)} }
            \ad(E_{J_1})(T_{w_\bullet}^{-1}(B_{\tau(i_1)} ))K_{w_\bullet\tau(i_1)}
            \ad(E_{J_2})(E_{i_2})\cdots \ad(E_{J_k})(E_{i_k}) K_{h}\\
            &-\ad(E_{J_1})\qty(\frac{s_{\tau(i_1)}}{c_{\tau(i_1)} })
            \ad(E_{J_2})(E_{i_2})\cdots \ad(E_{J_k})(E_{i_k})K_{h}\\
            &-\ad(E_{J_1})\qty(\frac{1}{c_{\tau(i_1)} }
            T_{w_\bullet}^{-1}(F_{\tau(i_1)} K_{\tau(i_1)}))
            \ad(E_{J_2})(E_{i_2})\cdots \ad(E_{J_k})(E_{i_k})K_{h}.
        \end{align*}
        Note $\ad: \uq_\bullet\to \End( \uqbs\uq^0\uqds)$ and that by definition $h + \epsilon_{i_1}w_\bullet(h_{\tau(i_1)}) = h - \Theta(h_{i_1})$ is regular for
        \[
        \big((\alpha_{i_2},\dots,\alpha_{i_k}), (J_2,\dots, J_k)\big).
        \]
Note that the restricted weight of $\ad(E_{J_2})(E_{i_2})\cdots \ad(E_{J_k})(E_{i_k})$ is less then the restricted weight of $\ad(E_{J_1})(E_{i_1})\cdots \ad(E_{J_k})(E_{i_k})$, moreover by definition we have
\[h\in Y_{\mathrm{reg}}((\alpha_{i_2},\dots ,\alpha_{i_k}),(J_2,\dots,J_k)).\]
As $s_{\tau_{i_1}}\in \uqb$ and $T_{w_\bullet}^{-1}(B_{\tau(i_1)} )\in\uqb$, we can apply the induction hypothesis to show that
        \begin{align}\label{eq:babsim0}
        &\ad(E_{J_1})(E_{i_1})\dots \ad (E_{J_k})(E_{i_k})K_h \sim \\\nonumber
        &-\ad(E_{J_1})\qty(\frac{1}{c_{\tau(i_1)} }
            T_{w_\bullet}^{-1}(F_{\tau(i_1)} K_{\tau(i_1)}))
            \ad(E_{J_2})(E_{i_2})\cdots \ad(E_{J_k})(E_{i_k})K_{h}.    
        \end{align}
        By \cite[Lemma 3.5]{Kol14} there exists a $J_1^-\in \uq_\bullet^-$ with 
        \begin{equation}\label{eq:jmin}
        \ad (E_{J_1})(T_{w_\bullet}^{-1}(F_{\tau(i_1)}K_{\tau(i_1)}))=\ad (E_{J_1^-})(F_{\tau(i_1)}K_{\tau(i_1)})\in S(\uq^-).
        \end{equation}
        By Proposition \ref{prop:commutationlowerdegree}, we have      
        \begin{align*}
        \ad (E_{J_1^-})(F_{\tau(i_1)}K_{\tau(i_1)})\ad(E_{J_2})(E_{i_2})\cdots \ad(E_{J_k})(E_{i_k})K_h\\
        \in \uq^+_\bullet V_{\tilde{\beta'}-2\tilde{\alpha}_{\tau(i_1)},r}^+K_{h + \epsilon_{i_1}h_{\tau({i_1})}\pm \epsilon_{i_1}h_{\tau({i_1})}}.\end{align*}
        
        Here, one applies the induction hypothesis, noting that regularity holds by definition, to deduce that
        \begin{equation}\label{eq:sim02}
         \uq^+_\bullet V_{\tilde{\beta'}-2\tilde{\alpha}_{\tau(i_1)},r}^+K_{h+\epsilon_{i_1}h_{\tau({i_1}})\pm \epsilon_{i_1}h_{\tau({i_1})}}{\sim} 0.
        \end{equation}     
Moreover, we note that 
\begin{align}\label{eq:bablasteq}
& \ad(E_{J_1^-})(F_{\tau(i_1)}K_{\tau(i_1)})K_{h }= \ad(E_{J_1})\big(T_{w_\bullet}^{-1}(F_{\tau(i_1)}K_{\tau(i_1)})\big)K_{h }\\\nonumber
&=q^{\langle h,\mathrm{wt}(J_1)-w_\bullet \alpha_{\tau(i_1)} \rangle} K_{h }\ad(E_{J_1})\big(T_{w_\bullet}^{-1}(F_{\tau(i_1)}K_{\tau(i_1)})\big)\\\nonumber
&\sim -d_{\tau(i_1)}q^{\langle h,\mathrm{wt}(J_1)-w_\bullet \alpha_{\tau(i_1)} \rangle} K_{h }\ad(E_{J_1})(E_{i_1})\\\nonumber
&= -d_{\tau(i_1)}q^{\langle h ,-w_\bullet \alpha_{\tau({i_1})}-\alpha_{i_1}\rangle}\ad(E_{J_1})(E_{i_1})K_{h }
\end{align}
        here, in \eqref{eq:bablasteq} the second to last equality follows by an argument analogous to \eqref{eq:babsim0}. Remark that $\ad(E_{J_2})(E_{i_2})\cdots \ad(E_{J_k})(E_{i_k})$ has weight
        $\sum_{j=2}^k\qty(\alpha_{i_j}+ \mathrm{wt}(J_j))$ and set 
        \[
        L(J,K)=q^{\langle \epsilon_{i_1}w_\bullet \tau(h_{i_1}), \sum_{j=2}^k\alpha_{i_j}+ \mathrm{wt}(J_j) \rangle},
        \]
        such that $\frac{c_{\tau(i_1)}}{d_{\tau(i_1})}q^{\langle h ,-w_\bullet \alpha_{\tau({i_1})}-\alpha_{i_1}\rangle}L(J,K)=C(J,K,h)$. As a final step we combine equations \eqref{eq:babsim0}-\eqref{eq:bablasteq} to deduce
{\allowdisplaybreaks
        \begin{align}\label{eq:babeq}
        &\ad(E_{J_1})(E_{i_1})\dots \ad (E_{J_k})(E_{i_k})K_h      \\
        &\stackrel{\eqref{eq:babsim0}}{\sim}\frac{-1}{c_{\tau(i)}} \ad(E_{J_1})(T_{w_\bullet}^{-1}(F_{\tau(i_1)}K_{\tau(i_1)})
            \ad(E_{J_2})(E_{i_2})\cdots \ad(E_{J_k})(E_{i_k})K_{h }\nonumber \\\nonumber         
         &\stackrel{\eqref{eq:jmin}}{=}\frac{-1}{c_{\tau(i)}}\ad(E_{J_1^-})(F_{\tau(i_1)}K_{\tau(i_1)})
            \ad(E_{J_2})(E_{i_2})\cdots \ad(E_{J_k})(E_{i_k})K_{h } \\\nonumber    
     &\stackrel{\eqref{eq:sim02}}{\sim}\frac{-1}{c_{\tau(i)}}L(J,K) \ad(E_{J_2})(E_{i_2})\cdots \ad(E_{J_k})(E_{i_k})\ad(E_{J_1^-})(F_{\tau(i_1)}K_{\tau(i_1)})K_{h }+\lot\\\nonumber
           &\stackrel{\eqref{eq:bablasteq}}{\sim} C(J,K,h)  \ad(E_{J_2})(E_{i_2})\cdots \ad(E_{J_k})(E_{i_k})\ad(E_{J_1})(E_{i_1})K_{h }.
        \end{align}         } 
 \item[Claim 2] Under the assumptions of Claim 1, $\ad(E_{J_1})(E_{i_1})\dots \ad (E_{J_k})(E_{i_k})K_h\sim0$.

  \item[Proof Claim 2]
        $k$-fold application of Claim~1 yields
        (note that $h$ is $(K,J)$-regular iff it is
        $(\varsigma(K),\varsigma(J))$-regular) that
        \begin{align*}
         &\ad(E_{J_1})(E_{i_1})\dots \ad (E_{J_k})(E_{i_k})K_h \sim\\&
            \qty(\prod_{j=1}^k C(\varsigma^{j-1}(J),\varsigma^{j-1}(K),h))
            \ad(E_{J_1})(E_{i_1})\dots \ad (E_{J_k})(E_{i_k})K_h .   
        \end{align*}
        Since $h\in Y_{\mathrm{Reg}}(K,J)$, we have
        \begin{align*}
            \prod_{j=1}^k C(\varsigma^{j-1}(J),\varsigma^{j-1}(K),h)=R(J,K,h)
            \ne 1.
        \end{align*}
        This implies $\ad(E_{J_1})(E_{i_1})\dots \ad (E_{J_k})(E_{i_k})K_h\sim0$.\qedhere
    \end{description}
\end{proof}

The following result then follows from assembling the results from
Proposition~\ref{thm:reduce}, Definition~\ref{def-yreg}, and the
quantum Iwasawa decomposition \eqref{eq:triangularbab1}.
\begin{theorem}\label{thm:cartandecompthm}
    Let $x\in\uq$.
    There exists a Zariski open subset $Y_{\mathrm{reg}}(x)\subset Y$
    such that $K_hx\in\uqb'\uq^0\uqb$.
\end{theorem}

\begin{proposition}\label{rem:order}
Let $(K,J)$ be a regular pair, then the radial part calculation in Theorem \ref{thm:reduce} has the property that for each $h \in Y_{\mathrm{Reg}}(K,J)$ the $\uqbs\uq^0\uqds$ decomposition of 
\[\ad(E_{J_1})(E_{i_1})\dots \ad (E_{J_k})(E_{i_k})K_h= b_{(1)}K_{h_{(2)}}b_{(3)}'\]
can be chosen such that $h\leq h_{(2)}\leq h+  \Theta(\mathrm{wt}(K))$.
\end{proposition}

\begin{proof}
We proceed by induction on the restricted weight of the $K$. The case $K=()$ needs no proof. Next, suppose the conclusion holds for all regular pairs $(K',J')$ with $\mathrm{wt}(K')\leq \mathrm{wt}(K)$. The dependence on the regular pairs of lower restricted weight in the inductive procedure in the proof of Theorem \ref{thm:reduce} is governed by the equations \eqref{eq:babsim0} and \eqref{eq:babeq}. From \eqref{eq:babsim0} and \eqref{eq:babeq} the induction hypothesis, we deduce that decomposition of 
\[\ad(E_{J_1})(E_{i_1})\dots \ad (E_{J_k})(E_{i_k})K_h= b_{(1)}K_{h_{(2)}}b_{(3)}'\]
satisfies $h\leq h_{(2)}\leq h+ \Theta(\mathrm{wt}(K))$.
\end{proof}

As in the classical case the restriction of spherical functions to a torus is injective. Recall the space of spherical functions $E^{\gamma,\gamma'}$ from Definition \ref{def:sph}.

\begin{proposition}\label{prop:resa}
The restriction $\mathrm{res}_{\uq^0}: E^{\gamma,\gamma'}\to \mathrm{Hom}\big(\uq^0,\mathrm{Hom}(V(\gamma),V(\gamma'))\big)$ is injective.
\end{proposition}

\begin{proof}
Let $\Phi\in E^{\gamma,\gamma'}$ and suppose that $\mathrm{res}_{\uq^0}(\Phi)=0$. Write $\Phi=\sum _{i,j=1}^n c_{f_j,v_i}\otimes E_{i,j}\in \mathcal A\otimes \mathrm{Hom}\big(V(\gamma),V(\gamma')\big)$.
If $c_{f_j,v_i}\neq0$, then there exists $x_{i,j}\in\uq$ such that $x_{i,j}\triangleright \phi_{i,j}(1)=1$. By Theorem \ref{thm:reduce} and $\mathrm{res}_{\uq^0}(\Phi)=0$, it follows that $x_{i,j}\triangleright c_{f_i,v_j}(K_h)=0$ for all $h \in Y_{\mathrm{Reg}}(x_{i,j})$. Since $Y_{\mathrm{Reg}}(x_{i,j})\subset Y$ is Zariski dense, it follows that $c_{f_i,v_j}=0$. A contradiction. Thus, we conclude $\Phi=0$.
\end{proof}
\begin{proposition}\label{prop:cartandecomp}
Let $\Phi\in E^{\gamma,\gamma'}$ and let $x\in \uq$. Then, using Theorem \ref{thm:reduce} there exists a function 
$$\Pi_{\gamma,\gamma'}(x):Y_{\mathrm{Reg}}(x)\to \mathrm{End}(V_\gamma)\otimes \uq^0\otimes  \mathrm{End}(V_{\gamma'}),\qquad K_h\mapsto L^{h,x}_{(1)}\otimes K_{x,{h_{(2)}}}\otimes R^{x,h}_{(3)}$$
such that
\begin{equation}
 (x\triangleright\Phi)(K_h)=L^{x,h}_{(1)}\Phi( K_{x,{h_{(2)}}}) R^{x,h}_{(3)}.
\end{equation}
In particular if 
$\Phi\in E^{\gamma,\gamma'}(\lambda)$
and $\Omega\in \uq^{\uqbs}=\{\Omega\in \uq: \Omega b=b\Omega\qquad \text{for all }b\in \uqbs\}$ then
\begin{equation}
a_{\Omega,\lambda}\Phi(K_h)=(\Omega\triangleright\Phi)(K_h)=L^{\Omega,h}_{(1)}\Phi(K_{\Omega,{h_{(2)}}}) R^{\Omega,h}_{(3)},\qquad\text{where}\qquad a_{\Omega,\lambda}\in k.
\end{equation}
\end{proposition}

\begin{definition}[Order]
For each element $x\in \uq$ \emph{order} $\mathrm{ord}(x)$ is defined as the minimal elements of
$$\{{h-h_{(2)}}: h\in Y_{\mathrm{Reg}}(x),\,\, K_hx= b_{(1)} K_{h_{(2)}} b_{(2)},\quad b_{(1)}\in \uqbs, b_{(2)}\in\uqds\}.$$
If $\mathrm{ord}(x)$ consists of one element $\{\mu\}$, we say that $x$ has order $\mu$.
\end{definition} 

Recall the distinguished elements $c_\mu\in Z(\uq)$ from
Lemma~\ref{lem-centre-u}.
By \cite[Theorem 7.7]{Let04} it follows that for $\gamma=\gamma'=\epsilon$,
the radial part of $c_\mu$ has order $-2h_\mu$.
We will show this in general.

First, we need a preliminary lemma. As preparation set 
\[\mathcal F^{n}(\uq^-)=\{F_J: J\in \mathrm{Fin}(I), \,|J|\leq n\}\]
and for a finite sequence of simple roots $J=(\alpha_{i_1},\dots,\alpha_{i_n})$ write $B_J=B_{i_1}\dots B_{i_n}$, taking the convention that $B_j=F_j$ if $j\in I_\bullet$.

\begin{lemma}\label{lem:kolbiwasawa}
Let $J$ be a finite sequence of simple roots and $F_J=\sum _{i=1}^n E_{I_i} K_{\mu_i}B_{J_i}$ be the corresponding quantum Iwasawa decomposition, then $-\mathrm{wt}(J)\leq K_{\mu_i}\leq 0$ and $0\leq \mathrm{wt}(J_i)\leq \mathrm{wt}(J)$.
\end{lemma}

\begin{proof}
We proceed by induction on the length of $J$. The case $J=\emptyset$ is clear. Next assume that the statement holds for all $J'\in \mathcal F^{n-1}(\uq^-)$ and let $F_J\in \mathcal F^n(\uq^-)\setminus\mathcal F^{n-1}(\uq^-) $. Then $B_J-F_J\in \uq^+ K_{j_1}^{-1} \mathcal F^{n-1}(\uq^-)$. Thus, using
the quantum Serre relations for $\uq^-$ it follows that $F_J- B_J$ is contained in the left $\uq^+$-submodule of generated by $K_{j_1}^{-1}\mathcal{F} ^{n-1}(\uq^-)$. By the induction hypothesis, the lemma follows.
\end{proof}

By construction, the central elements $c_\mu$ lie in the $\ad(\uq)$-module $\ad(\uq)(K_{-2h_\mu})$ and satisfy 
\begin{equation}\label{eq:central}
  c_\mu\in K_{-2h_\mu}+\ad(\uq_+)(K_{-2h_\mu}),
\end{equation}
where $\uq_+=\ker(\epsilon)$, cf. \cite[p. 109]{Let04}.

\begin{proposition}\label{prop:orderdo}
For $\mu\in X^+$ with $2h_\mu\in Y$, the central element $c_\mu$ has order $-2h_\mu$.
\end{proposition}

\begin{proof}
Let $\mu\in X^+$ and $h\in Y_{\mathrm{Reg}(c_\mu)}$. By \eqref{eq:central} and the triangular decomposition of $\uq$, we may write
\begin{equation*}
c_\mu=K_{-2\mu}+\sum_{\mathrm{wt}(A)=\mathrm{wt}(C)}c_{A,C}\ad(E_A F_C)(K_{-2h_\mu}).
\end{equation*}
Using the relations \cite[4.9 (5), Lemma 4.12]{Jan96} for the coproduct and antipode, we deduce that 
\begin{align*}
\sum_{\mathrm{wt}(A)=\mathrm{wt}(C)}c_{A,C}\ad(E_A F_C)(K_{-2h_\mu})&=\sum_{\mathrm{wt}(A+A')=\mathrm{wt}(C)}c'_{A,C}E_A K_{-2h_\mu}K_{\mathrm{wt}(C)}F_{C}E_{A'}.
\end{align*}
The commutation relation $E_iF_j-F_jE_i=\delta_{i,j}[K_i]_{q_i}$ implies that
\begin{align*}
\sum_{\mathrm{wt}(A+A')=\mathrm{wt}(C)}c'_{A,C}E_A K_{-2h_\mu}K_{\mathrm{wt}(C)}F_{C}E_{A'}=\sum_{\mathrm{wt}(A)=\mathrm{wt}(C)}c''_{A,C}E_A K_{-2h_\mu}K_{\mathrm{wt}(C)+h_A}F_{C}.
\end{align*}
with $0\leq h_A\leq 2\mathrm{wt}(A)$ and scalars $c''_{A,C}$. Now, apply Lemma \ref{lem:kolbiwasawa} to obtain
\begin{align*}
\sum_{\mathrm{wt}(A)=\mathrm{wt}(C)}c''_{A,C}E_A K_{-2h_\mu}K_{\mathrm{wt}(C)+h_A}F_{C}=\sum_{\mathrm{wt}(A)\geq \mathrm{wt}(C)}c'''_{A,C}E_A E_{C''} K_{-2h_\mu}K_{\mathrm{wt}(C)+h_A+h_C}B_{C'},
\end{align*}
with $-\mathrm{wt}(C)\leq h_C\leq 0$. According to Proposition \ref{rem:order} we have
$K_{h-2h_\mu}E_A E_{C''} = b_{(1)}h_{(2)}b_{(3)}$
with 
\begin{equation}\label{eq:inequality}
h-2h_\mu+\mathrm{wt}(C)+h_A+h_C \leq h_{(2)}\leq h-2h_\mu+\mathrm{wt}(C)+h_A+h_C+\Theta(\mathrm{wt}(K))-2h_\mu
\end{equation}
and $b_{(1)}\in \uqbs ,b_{(3)}\in \uqds$. Combining \eqref{eq:inequality} with the inequalities for $h_A$ and $h_C$, we conclude that $h-2h_\mu\leq h_{(1)}$. Thus, using \eqref{eq:central} the order of $c_\mu$ equals $-2h_\mu$.
\end{proof}

\section{Zonal Spherical Functions}\label{sec:zonal}
We now assume that $\mathcal{R}$ is finite.
In this section we summarise some results from the literature
about ZSFs of quantum symmetric pairs. 
Later, we deduce the matrix valued orthogonality weight for the matrix-spherical function from the zonal spherical case. Thus, it is of utmost importance to understand the zonal spherical functions.

From \cite[\S9.1]{Kol14} we can see that most coideal Hopf algebras
$\uqbs$ associated to a given admissible pair $(I_\bullet,\tau)$
are related via appropriate conjugation automorphisms:
we can interpret a group homomorphism $\xi: X\to k^\times$ as
a natural transformation of the forgetful functor of the category of
weight modules over $k$ by having $\xi$ act as $\xi(\mu)$ on the
weight space $M_\mu$ of a weight module $M$.
In particular, $\ad(\xi)$ acts on $\uq_\lambda$ as multiplication with
$\xi(\lambda)$.
Kolb calls the group of such automorphisms $\tilde{H}$.

In particular, $\uqbs,\uqb_{\bm{d},\bm{t}}$ are $\tilde{H}$-related
if and only if
\begin{equation}\label{eq-h-related-parameters}
    \begin{cases}
        \frac{c_i}{c_{\tau(i)}} = \frac{d_i}{d_{\tau(i)}} & i\in I_\circ,
        \tau(i)\ne i,\alpha_i\cdot\Theta(\alpha_i)\ne0\\
        c_it_i^2 = d_is_i^2 & i\in I_{ns}.
    \end{cases}
\end{equation}
In accordance to \cite{Let03} we write $\uqb_\Theta$ for the
set of right coideal subalgebras of $\uq$ associated to $(I_\bullet,\tau)$.
Evidently, $\tilde{H}$ acts on $\uqb_\Theta$.

\begin{lemma}\label{lem-can-choose-invariant-parameters}
    Let $\uqbs\in\uqb_\Theta$ and let $X\in\uqb_\Theta/\tilde{H}$ be
    an orbit.
    There exists at least one $\uqb_{\bm{d},\bm{t}}\in X$ such that
    $-\rho$-shifted restrictions of $E^{\epsilon,\epsilon}_{\uqbs,\uqb_{\bm{d},\bm{t}}}$ are $W_\Sigma$-invariant.
\end{lemma}
\begin{proof}
    Let $\uqb_{\bm{d}',\bm{t}'}\in X$. For $i\in I_\circ$ with
    $\tau(i)\ne i$ and $\alpha_i\cdot\Theta(\alpha_i)\ne0$ pick
    $d_i$ to be a square root of $\frac{c_i c_{\tau(i)}d_i'}{d_{\tau(i)}'}$ in such a way that $d_id_{\tau(i)}=c_ic_{\tau(i)}$.
    For any other $i\in I_\circ$ pick $d_i=c_i$.
    Lastly, for $i\in I_{ns}$ pick $t_i = \sqrt{\frac{d_i}{d'_i}}t'_i$.
    Then $\uqb_{\bm{d},\bm{t}}\in X$ by \eqref{eq-h-related-parameters}, and by Lemma~\ref{lem-weyl-invariance} the claim holds.
\end{proof}

\begin{lemma}[{\cite[Lemma~6.2]{Let03}}]\label{lem-relate-zsf}
    Let $\uqbs,\uqb_{\bm{d},\bm{t}}\in\uqb_\Theta$, let $\xi_1,\xi_2: X\to k^\times$ be group homomorphisms,
    and let $F\in E^{\epsilon,\epsilon}_{\uqbs,\uqb_{\bm{d},\bm{t}}}$.
    \begin{enumerate}
        \item $\xi_2\cdot F\cdot\xi_1^{-1} \in E^{\epsilon,\epsilon}_{\xi_1\cdot\uqbs,\xi_2\cdot\uqb_{\bm{d},\bm{t}}}$;
        \item For $\xi:X\to k^\times$, we define $\tilde{\xi}:2L\to k^\times$ by restriction,
        and having it act on $k[2L]$ by $\tilde{\xi}\triangleright e^\lambda
        = \xi(\lambda)e^{\lambda}$.
        Then $\Res(\xi_2\cdot F\cdot\xi_1^{-1}) = \tilde{\xi_2}\tilde{\xi_1}^{-1}\triangleright\Res(F)$,
        and $(-\rho)\triangleright\Res(\xi_2\cdot  F\cdot\xi_1^{-1}) = \tilde{\xi_2}\tilde{\xi_1}^{-1}\triangleright(-\rho)\triangleright\Res( F)$.
    \end{enumerate}
\end{lemma}
\begin{proof}
    Let $b\in \ad(\xi_1)(\uqbs),b'\in\ad(\xi_2)(\uqb_{\bm{d},\bm{t}})$,
    then $\ad(\xi_1^{-1})(b)=\xi_1^{-1}b\xi_1\in\uqbs$ and idem for
    $b'$.
    Let $x\in\uq$, then
    \begin{align*}
        (\xi_2\cdot F\cdot \xi_1^{-1})(bxb') &=
        F(\xi_1^{-1}bxb'\xi_2)
        = F(\xi_1^{-1}b\xi_1\xi_1^{-1}x\xi_2\xi_2^{-1}b'\xi_2)\\
        &= \epsilon(\ad(\xi_1^{-1})(b)) F(\xi_1^{-1}x\xi_2)
        \epsilon(\ad(\xi_2^{-1})(b'))\\
        &= \epsilon(b)\epsilon(b') (\xi_2\cdot F\cdot\xi_1^{-1})(x),
    \end{align*}
    which shows the first claim.

    For the second claim assume that $F = c^{M}_{f,v}$ with
    $v=\sum_{\lambda\in 2L}v_\lambda$ for $v_\lambda\in M_\lambda$.
    Then we have
    \[
        \Res(F) = \sum_{\lambda\in 2L} f(v_\lambda)e^\lambda.
    \]
    Moreover, we have $\xi_2\cdot F\cdot \xi_1^{-1}=c^M_{f\xi_1^{-1},\xi_2v}$ whose restriction is the same as that of
    $c^M_{f,\xi_1^{-1}\xi_2v}$, which is
    \begin{align*}
        \Res(\xi_2\cdot F\cdot\xi_1^{-1}) &= 
        \sum_{\lambda\in 2L} f((\xi_1^{-1}\xi_2v)_\lambda)e^\lambda
        = \sum_{\lambda\in 2L} \xi_1^{-1}(\lambda)\xi_2(\lambda)f(v_\lambda)e^\lambda\\
        &= \tilde{\xi_2}\tilde{\xi_1}^{-1}\triangleright
        \Res(F).\qedhere
    \end{align*}
\end{proof}

When attempting to classify the ZSFs for
arbitrary (allowed) choices of parameters for the 
coideal subalgebras $\uqbs,\uqb_{\bm{d},\bm{t}}$,
Lemmas~\ref{lem-can-choose-invariant-parameters} and
\ref{lem-relate-zsf} show that we can restrict ourselves to one
representative per $\tilde{H}\times\tilde{H}$ orbit of
$\uqb_\Theta\times\uqb_\Theta$.
It makes sense to have these representatives satisfy the
conditions of Lemma~\ref{lem-weyl-invariance}, so that the
resulting $-\rho$-shifted ZSFs are
$W_\Sigma$-invariant.

\begin{proposition}\label{prop-classify-orbits}
    Assume that $\Sigma$ is irreducible.
    The $\tilde{H}$-orbits of
    $\uqb_\Theta$ look as follows:
    \begin{enumerate}
        \item If $\mathcal{R}=\Sigma\sqcup\Sigma$ and $\tau$ exchanges
        the two copies of $\Sigma$, there is one orbit, the one of
        $\uqbc$ with $c_i=1$ ($i\in I_\circ$); otherwise,
        $\mathcal{R}$ is irreducible as well.
        \item If the Satake diagram for $(I_\bullet,\tau)$ is
        of the types
        $\mathsf{AIII}_{n,p}$ ($2p<n+1$), $\mathsf{AIV}$, $\mathsf{DIII}_{2p+1}$, or $\mathsf{EIII}$
        the orbits are characterised by one parameter up to sign:
        pick $i\in I$ such that $I_\bullet\cup\set{i,\tau(i)}$ is
        of type $\mathsf{AIV}_m$ and
        define $\bm{c}(a)$ by
        $c_i(a)=i^{m-2}a$, $c_{\tau(i)}=(-i)^{m-2}a^{-1}$ and $c_j=1$ for $j\in I_\circ\setminus\set{i,\tau(i)}$.
        Then every orbit contains two elements of the type
        $\uqb_{\bm{c}(\pm a),0}$ for
        $a\in k^\times$.
        In \cite{Let02}, this is called \enquote{Variation 1}.
        \item If the Satake diagram is of type $\mathsf{AI}_1$,
        $\mathsf{AIII}_{2p-1,p}$, $\mathsf{BI}_{n,2}$, $\mathsf{CI}_{n}$, $\mathsf{DI}_{n,2}$, $\mathsf{DIII}_{2p}$,
        or $\mathsf{EVII}$,
        the orbits are also characterised by one parameter up to sign:
        there is exactly one $i\in I_\circ$ such that 
        there are elements $\bm{s}\in\mathcal{S}$ with $s_i\ne0$.
        Then every orbit contains
        one or two elements
        $\uqb_{1,\pm\bm{s}}$ with
        $s_j=0$ for $j\ne i$.
        In \cite{Let02}, this is called \enquote{Variation 2}.
        \item In all other cases, there is one orbit.
    \end{enumerate}
    All representatives listed here yield $W_\Sigma$-invariant
    $-\rho$-shifted restrictions of ZSFs.
\end{proposition}
\begin{proof}
    By \cite{Araki}, $\Sigma$ is irreducible if either $\mathcal{R}=\Sigma\sqcup\Sigma$ and $\tau$ maps one copy of $\Sigma$ to the
    other (as described in the first case), or if $\mathcal{R}$ is
    irreducible as well, in which case the classification from
    \cite[pp.~32--33]{Araki} holds.

    In the first case, we have $\tau(i)\ne i$ for all $i$ as well as
    $\alpha_i\cdot\Theta(\alpha_i)=0$.
    By \eqref{eq-h-related-parameters}, all elements of $\uqb_\Theta$
    are $\tilde{H}$-conjugate, which proves (i).

    It now suffices to check the admissible Satake diagrams for
    occurrences of $i\in I_\circ$ such that $\tau(i)\ne i$ and
    $\alpha_i\cdot\Theta(\alpha_i)\ne0$ and for allowed non-zero
    non-standard parameters.

    Note that by \cite[Remark~9.3]{Kol14}, Variation~1 can only
    occur in case $\Sigma$ is of type $\mathsf{BC}$.
    Examining these cases (in e.g. \cite{Araki}) yields that
    the cases listed under (ii) are in fact the ones for which
    Variation~1 can be made.
    In every case, $\mathcal{S}=\set{0}$, and there is exactly one $\tau$-orbit $\set{i,\tau(i)}$ with two elements
    satisfying
    $\alpha_i\cdot\Theta(\alpha_i)\ne0$.
    In every case, adding this $\tau$-orbit to $I_\bullet$ yields
    a diagram of type $\mathsf{AIV}_m$.

    Equation~\ref{eq-h-related-parameters} shows that
    up to conjugacy, all $c_j$ ($j\ne i,\tau(i)$) can be set to 1, and that $c_i,c_{\tau(i)}$ can be scaled by the
    same number.
    Consequently, without loss of generality we can assume
    that $c_j=1$ ($j\ne i,\tau(i)$) and that
    $c_ic_{\tau(i)}=1$.
    In other words, $\bm{c}=\bm{c}(a)$ for some $a\in k^\times$.
    Furthermore, $\uqbs$'s orbit comprises one more element:
    $\uqb_{\bm{c}(-a),0}$.

    For Variation 2, inspection yields the diagram types listed under (iii),
    with one non-zero non-standard parameter permitted in each case: $s_i$.
    Given a coideal subalgebra $\uqbs$, $\bm{c}$ can be scaled to be 1 in every component,
    and two coideal subalgebras $\uqb_{1,\bm{s}} $ and
    $\uqb_{1,\bm{t}}$ are then
    conjugate iff $s_i^2=t_i^2$, which explains the ambiguity in
    sign.

    In all other cases we have $\mathcal{S}=\set{0}$ and no $i\in I_\circ$
    satisfying $\tau(i)\ne i$ and $\alpha_i\cdot\Theta(\alpha_i)\ne0$.
    Consequently, all coideal subalgebras are conjugate and there is
    only one orbit, which proves (iv).

    We now pick two representatives listed here,
    e.g. two times $\uqb_{1,0}$ for (i) or (iv),
    $\uqb_{\bm{c}(a),0},\uqb_{\bm{c}(b),0}$ for (ii),
    or $\uqb_{1,a\bm{e}_i},\uqb_{1,b\bm{e_i}}$ for (iii) (where $\bm{e}_i$ has 1 as its $i$-th entry, and 0 everywhere else, and is chosen such that $\bm{e}_i\in\mathcal{S}$)
    All of these pairs satisfy the conditions of
    Lemma~\ref{lem-weyl-invariance}.
\end{proof}

\begin{proposition}\label{prop-zsf}
    Let $(\uqbs,\uqb_{\bm{d},\bm{t}})$ be one of the representatives
    from Proposition~\ref{prop-classify-orbits}.
    \begin{enumerate}
        \item If $\Sigma$ is reduced and $\bm{s}=\bm{t}=0$, define
        \gls{SSRRLL} as follows:
        \[
                S=S'=S((2\Sigma)^\vee)^\vee),\qquad
                R=R'=2\Sigma,\qquad
                L=L'=2P(\Sigma)
        \]
        (where we give the span of $2\Sigma$ the inner product it
        inherits from $X$),
        let $\psi\in\Sigma$ be the highest short root and define
        $\epsilon:=2\abs{\psi}^2$.
        \item If $(I_\bullet,\tau)$ is of any of the following
        types: $\mathsf{AI}_1,\mathsf{AIII}_{n,p},\mathsf{AIV}_n,\mathsf{CII}_{n,p},\mathsf{DIII}_{2p+1}$, write
        $\Sigma^0$ for the immultipliable subsystem of $\Sigma$ and
        define
        \begin{align*}
            S = S' &= S(2\Sigma^0)\cup 2S(2\Sigma^0)^\vee,\\
            R=R'&=2\Sigma^0,\\
            L=L'&=P(2\Sigma^0)=P(2\Sigma)
        \end{align*}
        and let $\epsilon=2$.
    \end{enumerate}
    Furthermore, pick the parameters as follows: if $\Sigma$ is
    reduced and $\bm{s}=\bm{t}=0$, we pick $l:S\to k$ as follows:
    \[
        l(2\alpha+c) = \frac{\abs{\alpha}^2\dim(\mathfrak{g}_\alpha)}{2\abs{\psi}^2}.
    \]
    Otherwise, pick
    \begin{align*}
        \mathsf{AI}_1:\qquad (l_1,\dots,l_4) &= \qty(\frac{\sigma-\tau+1}{2},\frac{\sigma+\tau+1}{2},\frac{-\sigma+\tau}{2},\frac{-\sigma-\tau}{2})\\
        \mathsf{AIII}_{n,p}:\quad (l_1,\dots,l_5) &=
        \qty(\frac{\sigma-\tau+1}{2},\frac{\sigma+\tau+1}{2},\frac{-\sigma+\tau}{2}+n+1-2p,\frac{-\sigma-\tau}{2},1)\\
        \mathsf{AIV}_n:\quad(l_1,\dots,l_4) &=
        \qty(\frac{\sigma-\tau+1}{2},\frac{\sigma+\tau+1}{2},\frac{-\sigma+\tau}{2}+n-1,\frac{-\sigma-\tau}{2})\\
        \mathsf{CII}_{n,p}:\quad (l_1,\dots,l_5)&= \qty(\frac{3}{2}+n-2p,\frac{3}{2},n-2p,0,2)\\
        \mathsf{DIII}_{2p+1}:\quad (l_1,\dots,l_5) &=
        \qty(\frac{3}{2},\frac{1}{2},1,0,2),
    \end{align*}
    where $\sigma,\tau$ are dependent on the parameters $\bm{c},\bm{s},\bm{d},\bm{t}$:
    \begin{itemize}
        \item For $\mathsf{AI}_1$, we assume that $t_1=[\sigma]_q$ and
    $s_1 = [\tau]_q$
        and $c_1=d_1=1$.
        \item For $\mathsf{AIII}_{n,p}$ for $2p < n+1$ and $\mathsf{AIV}_n$, we assume that
        $\bm{s}=\bm{t}=0$,
        $\bm{c}_i=\bm{d}_i$ for $i<p$ and $i>n+1-p$ and
        \begin{alignat*}{2}
            \bm{c}_p &= -q^{n+1-2p+\tau}\quad &\bm{c}_{n+1-p} &=(-1)^{n+1}q^{-\tau}\\
            \bm{d}_p &= -q^\sigma&\bm{d}_{n+1-p}&= (-1)^{n+1}q^{n+1-2p-\sigma}.
        \end{alignat*}
        \item For $\mathsf{AIII}_{2p-1,p}$ we assume that
        $\bm{c}=\bm{d}=1$, that $t_p = [\sigma]_q$, and
        $s_p = [\tau]_q$.
        \item For $\mathsf{DIII}_{2p+1}$, we assume that $\bm{s}=\bm{t}=0$ and
        that $\bm{c}=\bm{d}=\bm{c}(1)$.
    \end{itemize}
    Let now $\phi\in E^{\epsilon,\epsilon}_{\uqbs,\uqb_{\bm{d},\bm{t}}}(\lambda)$,
    then there is $C\in k$ such that
    \[
        (-\rho)\triangleright\Res(\phi) = CP_\lambda
    \]
    where $P_\lambda$ is the symmetric Macdonald polynomial
    \cite[\S5.3.1]{Mac03} for the Macdonald data
    $(S,S',R,R',L,L')$, the labelling $l$, and the base parameter $q^\epsilon$.
\end{proposition}
\begin{proof}
    For $\Sigma$ reduced and $\bm{s}=\bm{t}=0$, this follows from
    \cite[Theorem~8.2]{Let04}.
    For type $\mathsf{AI}_1$, this follows from \cite[Theorem~5.2]{ko}.
    For type $\mathsf{AIII}_{n,p}$, it follows from
    \cite[Theorem~3.4]{NDS}, with added explanation of the parameter
    provided in \cite[Lemma~A.2]{Mee25} for $2p<n+1$ and in
    \cite{OS05} for $2p=n+1$.
    For type $\mathsf{CII}_{n,p}$, it follows from \cite[Theorem~3.1]{Sug99}. Note that for $p=n$, $\Sigma$ is reduced and the weight 
    (and the orthogonal polynomials) are the same as would be
    obtained from Letzter.
    For type $\mathsf{DIII}_{2p+1}$, it follows from \cite[Theorem~3(type~7)]{NS95}.
\end{proof}

\begin{remark}
    The attentive reader might note, that Proposition
    \ref{prop-zsf} fails to cover the following non-standard cases:
    \[
        \mathsf{BI}_{n,2},\mathsf{CI}_n,\mathsf{DI}_{n,2},\mathsf{DIII}_{2p},\mathsf{EVII}
    \]
    where we expect two-parameter families of orthogonal polynomials,
    and the following standard cases:
    \[
        \mathsf{DIII}_{2p+1},\mathsf{EIII},\mathsf{FII},
    \]
    of which we expect two-parameter families only for $\mathsf{DIII}_{2p+1},\mathsf{EIII}$, and a family
    for $\mathsf{FII}$ that has no parameters.

    Note also that in the cases $\mathsf{AI}_1,\mathsf{AIII}_{n,p},\mathsf{AIV}_n$
    where the labelling $l$ explicitly depends on the parameters $\bm{c},\bm{s},\bm{d},\bm{t}$, we can remove
    the restriction on the values that these parameters might take (i.e. that they be powers of $q$ or $q$-numbers).
    The modern theory of Macdonald polynomials (as given in \cite{Mac03}) is phrased in
    terms of parameters $\tau_i,\tilde{\tau}_i$ ($i=1,\dots,p$; notation borrowed from \cite[Definition~2.28]{Sch23}).
    Since the coefficients of the ZSF can be shown to be rational functions in the parameters $\bm{c},\bm{d},\bm{s},\bm{t}$,
    and the coefficients of the Macdonald polynomials can be shown to be rational functions in the $\tau_i,\tilde{\tau}_i$,
    a Zariski argument shows that the conclusion of Proposition~\ref{prop-zsf} holds for all parameter values.

    In that case, the relations between $\bm{c},\bm{d},\bm{s},\bm{t}$ and $\tau_i,\tilde{\tau}_i$ ($i=1,\dots,p$) are given by
    \begin{align*}
        \mathsf{AI}_1&:\quad \tau_0\tau_1 = q, \tilde{\tau}_0\tilde{\tau}_1=1,\quad s_1=-[\tilde{\tau}_1;0]_q, t_1=[\tau_1;-1]_q\\
        \mathsf{AIII}_{n,p} (2p<n+1)&:\quad
        \set{a,b,c,d} = \set{(-1)^n\frac{\bm{d}_p}{\bm{c}_{n+1-p}}q,
        (-1)^n\frac{\bm{d}_{n+1-p}}{\bm{c}_p}q, \frac{\bm{d}_p}{\bm{c}_p}q^{n+2-2p}, \frac{\bm{d}_{n+1-p}}{\bm{c}_{n+1-p}}q^{n+2-2p}}\\
        \mathsf{AIII}_{2p-1,p}&:\quad
        \tau_0\tau_1=q,\tilde{\tau}_0\tilde{\tau}_1=1,\tau_1=q\\
        &\qquad
        s_p=-[\tilde{\tau}_p;0]_q,
        t_p=[\tau_p;-1]_q\\
        \mathsf{AIV}_n&:\quad
        \set{a,b,c,d} = \set{(-1)^n\frac{\bm{d}_1}{\bm{c}_n}q,
        (-1)^n\frac{\bm{d}_n}{\bm{c}_1}q, \frac{\bm{d}_1}{\bm{c}_1}q^{n}, \frac{\bm{d}_n}{\bm{c}_n}q^{n}}
    \end{align*}
\end{remark}
where we use \eqref{def-q-numbers} for the notation $[\tau_1;-1]_q$.

In the following we provide more details on how exactly to match
the results from \cite{Let04,ko,NDS,Sug99,NS95} to our conventions.

\begin{remark}[Letzter, $\Sigma$ reduced and $\bm{s}=\bm{t}=0$]\label{rmk-letzter-macdonald}
    In \cite[Theorem~8.2]{Let04} and the definitions before,
    Letzter refers to the Macdonald polynomials as defined in
    \cite{Mac00}.
    In contrast to \cite{Mac03}, these Macdonald polynomials depend
    not on an affine root system $S$, but on two root systems $(R,S)$,
    which Letzter chooses to be $(2\Sigma,(2\Sigma)^\vee)$.
    In particular, the Macdonald polynomials that Letzter works with
    are orthogonal with respect to the weight distribution
    $\triangledown=\Delta^+\overline{\Delta^+}$ with
    \begin{equation}\label{eq-letzter-weight}
        \Delta^+ = \prod_{\alpha\in \Sigma^+}
        \frac{(e^{2\alpha}; q^{2\abs{\alpha}^2})_\infty}{(q^{\abs{\alpha}^2\dim(\mathfrak{g}_\alpha)}e^{2\alpha};q^{2\abs{\alpha}^2})_\infty}.
    \end{equation}
    To translate this into the language of \cite{Mac03}, note that
    the infinite $q$-Pochhammer symbols have potentially varying
    powers of $q$ as base (if $\Sigma$ is not simply laced).
    This can only occur for $(S,S',R,R',L,L')$ of type \cite[\S1.4.2]{Mac03}.
    Matching \eqref{eq-letzter-weight} with \cite[\S5.1.28(ii)]{Mac03}, we obtain that $R^\vee=2\Sigma$, i.e. that $S=S'=S((2\Sigma)^\vee)^\vee$ and $L=L'=P^\vee((2\Sigma)^\vee)=P(2\Sigma)$.

    More concretely, we have
    \[
        S((2\Sigma)^\vee)^\vee = \set{2\alpha + 2\abs{\alpha}^2r
        \where \alpha\in\Sigma,r\in\Z}.
    \]
    Let $\psi\in \Sigma$ be the highest short root, then
    $(2\psi)^\vee$ is the highest root of $(2\Sigma)^\vee$.
    Consequently, the zeroth simple affine root is given by
    $-2\psi + 2\abs{\psi}^2$, so that the fundamental constant function
    is $c=2\abs{\psi}^2$, which is what we defined as $\epsilon$.

    The symmetric weight function from \cite[\S5.1.28(ii)]{Mac03}
    (fixing the typo in that formula) is then
    \[
            \triangledown = \prod_{\alpha\in\Sigma}
            \frac{\qty(e^{2\alpha};q_{(2\alpha)^\vee})_\infty}{\qty(q^{l(2\alpha)} e^{2\alpha}; q_{(2\alpha)^\vee})_\infty}.
    \]
    By \cite[\S5.1.13]{Mac03}, we have
    \[
        q_{(2\alpha)^\vee}
        = q^{\frac{\abs{(2\psi)^\vee}^2}{\abs{(2\alpha)^\vee}^2}}
        = q^{\frac{\abs{\alpha}^2}{\abs{\psi}^2}}.
    \]
    We therefore see that to match this with \eqref{eq-letzter-weight},
    we need to replace $q$ by $q^\epsilon$.
    Furthermore, this shows that $q^{l(2\alpha)}$ from Macdonald
    corresponds to $q^{2\abs{\psi}^2l(2\alpha)}$, which we need to
    equal $q^{\abs{\alpha}^2\dim(\mathfrak{g}_\alpha)}$.
    Consequently, we need to have
    \[
        l(2\alpha) = \frac{\abs{\alpha}^2\dim(\mathfrak{g}_\alpha)}{2\abs{\psi}^2},
    \]
    which implies that 
    \[
        l^\vee((2\alpha)^\vee) = \frac{\dim(\mathfrak{g}_\alpha)}{2}.
    \]
\end{remark}

\begin{remark}[Koornwinder, $\mathsf{AI}_1$]
    We write $K=K_{h_1/2},  E:=E_1,F:=F_1$, then
    the generators $A,B,C,D$ from \cite[\S 3]{ko} correspond to
    \[
        A = K,\qquad D = K^{-1},\qquad B = -iEK^{-1},\qquad
        C = iKF.
    \]
    Then the twisted primitive $X_\sigma$ from \cite[\S4]{ko} equals
    \[
        X_\sigma = q^{-1/2}K\qty(F + EK^{-2} +  [\sigma]_q(K^{-2}-1))
        = q^{-1/2} K(B_1 - \epsilon(B_1)),
    \]
    where $B_1$ is the generator of $\uqb=\uqb_{1,[\sigma]_q}$. Consequently, the left ideal
    generated by $X_\sigma$ is equivalently generated by $x-\epsilon(x)$ ($x\in \uqb$).
    Furthermore,
    \[
        X_\sigma = q^{-1/2} \ad(K)(B_1-\epsilon(B_1)) K,
    \]
    where $K=K_\rho$. Consequently, the right ideal generated by $X_\sigma$ is equivalently
    generated by $x-\epsilon(x)$ ($x\in\ad(K)(\uqb)$).

    This shows that every $f\in\mathcal{A}$ satisfying $X_\sigma\cdot f = f\cdot X_\tau=0$
    (what is called \enquote{$(\sigma,\tau)$-spherical} in \cite{ko}), is a ZSF for $\ad(K_\rho)(\uqb_{1,[\tau]_q})$ and $\uqb_{1,[\sigma]_q}$. 
    Consequently, $f\cdot K_\rho$ is a ZSF
    for $\uqb_{1,[\tau]_q},\uqb_{1,[\sigma]_q}$ and $\Res(f) = (-\rho)\triangleright \Res(f\cdot K_\rho)$.
    In other words, $f\mapsto f\cdot K_\rho$ maps the $(\sigma,\tau)$-spherical functions from
    Koornwinder to our ZSFs, and the restrictions that Koornwinder determines in
    \cite[Theorem~5.2]{ko} are precisely the polynomials we are looking for.
\end{remark}

\begin{remark}\label{rem:noumi}[Noumi, Dijkhuizen, Sugitani, $\mathsf{AIII}$]
	We consider $L^\pm_{ij}$ for $\uq=\uq_q(\mathfrak{gl}(n))$.
	By the conventions of \cite{OS05}, they can be mapped to our conventions
	by having
	\begin{equation}\label{eq:Loperator}
	L^\pm_{ii} = K_{\pm\epsilon_i},\qquad
	L^+_{i,i+1} = (q-q^{-1})K_{\epsilon_i}F_i,\qquad
	L^-_{i+1,i} = -(q-q^{-1})E_i K_{-\epsilon_i}.
	\end{equation}

    As is explained in \cite[Section~3]{OS05}, the
    authors of \cite{NDS} use a solution $J^\sigma$ to the
    reflection equation to define a two-sided coideal (not an algebra!)
    $\mathfrak{k}^\sigma$ (depending on a parameter $\sigma$)
    as the span of $L^+J^\sigma-J^\sigma L^-$.
    In \cite{NDS}, an element $f\in\mathcal{A}$ is a ZSF if $\mathfrak{k}^\sigma\cdot f=f\cdot\mathfrak{k}^\tau=0$.

    In order to relate the ZSF from \cite{NDS} to ours, we need to show
    that $\uq\mathfrak{k}^\sigma\supset\uq\uqb_{\bm{d}',+}$
    and that $\mathfrak{k}^\tau\uq\supset\uqb_{\bm{c},+}\uq$
    for appropriately chosen $\bm{c}',\bm{d}'$ ($\uqb_+$ is the augmentation ideal of $\uqb$, i.e. the kernel of $\epsilon$).
    For $n=2l$, we refer to \cite[Section~3]{OS05}, all other cases we discuss here.

    For that we consider $L^+J^\sigma+J^\sigma L^-$.
    
    For $l<i,j<n+1-l$, its $i,j$-th entry is $L^+_{ij}-L^-_{ij}$.
    Picking $i=j$ and $i=j\pm1$, we obtain that $\uq_{\bullet,+}\subset\uq\mathfrak{k}^\sigma\cap\mathfrak{k}^\tau\uq$.

    For $i=1,\dots,l$, the $i,n+1-i$-th entry is proportional to
    $L^+_{ii}-L^-_{n+1-i, n+1-i}$, which shows that
    $K_{\epsilon_i+\epsilon_{n+1-i}}^{\pm1}-1\in \uq\mathfrak{k}^\sigma\cap\mathfrak{k}^\tau\uq$, and hence that for all
    $h\in Y^\Theta$ we have $K_\mu-1\in\uq\mathfrak{k}^\sigma\cap\mathfrak{k}^\tau\uq$.

    To establish the claimed inclusions, it suffices to show that
    \begin{align*}
        F_i + \bm{d}_i' E_{n-i}K_i^{-1}&\in\uq\mathfrak{k}^\sigma\qquad i<l, i>n-l\\
        F_l + \bm{d}_l' E_{l+1}\cdots E_{n-l}K_l^{-1}&\in\uq\mathfrak{k}^\sigma\\
        F_{n-l}+\bm{d}_{n-l}' E_{n-l-1}\cdots E_lK_{n-l}^{-1}&\in\uq\mathfrak{k}^\sigma\\
        F_i+\bm{c}_i'E_{n-i}K_i^{-1}&\in\mathfrak{k}^\tau\uq\qquad
        i<l,i>n-l\\
        F_l + (-q^{-1})^{n-2l-1}\bm{c}_l' E_{n-l}\cdots E_{l+1}K_l^{-1}&\in \mathfrak{k}^{\tau}\uq\\
        F_{n-l}+(-q^{-1})^{n-2l-1}\bm{c}_{n-l}' E_l\cdots E_{n-l-1}K_{n-l}^{-1}&\in\mathfrak{k}^\tau\uq
    \end{align*}
    as these expression equal $B_i$ ($i\in I_\circ$) up to terms
    contained in $\uq\uq_{\bullet,+}$ or $\uq_{\bullet,+}\uq$.
    For $i<l$ we have
    \begin{align*}
        (L^+J^\sigma+J^\sigma L^-)_{i,n-i}
        =& -q^\sigma (L_{i,i+1}^+-L^-_{n+1-i,n-i})\\
        \in &-q^\sigma (q-q^{-1})K_{\epsilon_i}(F_i + E_{n-i}K_i^{-1})
        + \uq\uq^0_{\Theta,+}\\
        \cap&-q^{\sigma-1}(q-q^{-1})(F_i + q^2E_{n-i}K_i^{-1})K_{\epsilon_i}+\uq^0_{\Theta,+}\uq
    \end{align*}
    in the conventions of \cite{OS05}.
    And similarly for $n-i,i$.
    This establishes the inclusion $B_i\in\uq\mathfrak{k}^\sigma$
    (for $\bm{d}_i'=1$) and $B_i\in\mathfrak{k}^\tau\uq$
    (for $\bm{c}_i'=q^2$) for $i<l$ and $n-l<i$.

    It remains to find $B_l,B_{n-l}$.
    For that we consider
    \begin{align*}
        &(L^+J^\sigma-J^\sigma L^-)_{l,l+1}
        = L^+_{l,l+1}+q^\sigma L_{n+1-l,l+1}^-\\
        & \in (q-q^{-1})K_{\epsilon_l}\qty( F_l
        - q^{\sigma} E_{l+1} \cdots E_{n-l} (K_l^{-1} +
         K_{-\epsilon_l-\epsilon_{l+1}}(1-K_{2\epsilon_{l+1}}))) + \uq\uq_{\bullet,+}\\
        &\subset (q-q^{-1})K_{\epsilon_l}\qty(F_l-q^\sigma E_{l+1}\cdots E_{n-l}K_l^{-1}) + \uq\uq_{\bullet,+} + \uq\uq^0_{\Theta,+}\\
        &(L^+J^\sigma-J^\sigma L^-)_{n-l,l} =
        -q^\sigma L_{n-l,n+1-l}^+-L_{n-l,l}^-\\
        &\in -q^\sigma (q-q^{-1})K_{\epsilon_{n-l}}\qty(F_{n-l}
        + (-1)^n q^{n-2l-\sigma}E_{n-l-1}\cdots E_l (K_{n-l}^{-1} + K_{-\epsilon_{n-l}-\epsilon_l}(1-K_{\epsilon_l+\epsilon_{n+1-l}})))
        \\
        &+ \uq\uq_{\bullet,+}\subset -q^\sigma (q-q^{-1})K_{\epsilon_{n-l}}\qty(F_{n-l}
        + (-1)^n q^{n-2l-\sigma}E_{n-l-1}\cdots E_l K_{n-l}^{-1})
        + \uq\uq_{\bullet,+}+\uq\uq^0_{\Theta,+},
    \end{align*}
    thus we conclude that $B_l,B_{n-l}\in\uq\mathfrak{k}^\sigma$ for the
    parameters $\bm{d}_l' = -q^\sigma$ and $\bm{d}_{n-l}'= (-1)^nq^{n-2l-\sigma}$.
    Similarly, we obtain
    \begin{align*}
        &(L^+J^\tau-J^\tau L^-)_{l,l+1}\in 
        \\
        &q^{-1}(q-q^{-1})\qty(F_l + (-1)^nq^{n-2l+\tau} E_{n-l}\cdots E_{l+1}(q^2K_l^{-1}+K_{-\epsilon_l-\epsilon_{l+1}}(1-q^2K_{2\epsilon_{l+1}})))K_{\epsilon_l} + \uq_{\bullet,+}\uq\\
        &\subset q^{-1}(q-q^{-1})\qty(F_l + (-1)^nq^{n+2-2l+\tau} E_{n-l}\cdots E_{l+1}K_l^{-1})K_{\epsilon_l} + \uq_{\bullet,+}\uq
        + \uq^0_{\Theta,+}\uq\\
        &(L^+J^\tau-J^\tau L^-)_{n-l,l}    \in\\
        &
        -q^{\tau-1}(q-q^{-1})\qty(F_{n-l} -q^{1-\tau} E_l\cdots E_{n-l-1} (qK_{n-l}^{-1} + K_{-\epsilon_l-\epsilon_{n-l}}(1-qK_{\epsilon_l+\epsilon_{n+1-l}})))K_{\epsilon_{n-l}} + \uq_{\bullet,+}\uq\\
        &\subset-q^{\tau-1}(q-q^{-1})\qty(F_{n-l}-q^{2-\tau}E_l\cdots E_{n-l-1}K_{n-l}^{-1})K_{\epsilon_{n-l}}
        + \uq_{\bullet,+}\uq + \uq^0_{\Theta,+}\uq,
    \end{align*}
    which shows that $B_l,B_{n-l}\in\mathfrak{k}^\tau\uq$ for the
    parameters $\bm{c}_l'=-q^{2n-4l+1+\tau}$ and $\bm{c}_l'=(-1)^n q^{n+1-2l-\tau}$.

    We see that $\uqb_{\bm{c}'}=\ad(K_{\rho})(\uqbc)$ with
    \[
        \bm{c}_i = \begin{cases}
            1 & i<l, n-l<i\\
            -q^{n-2l+\tau} & i=l\\
            (-1)^nq^{-\tau} & i=n-l
        \end{cases}
    \]
    and conclude that if $f$ is a ZSF from \cite{NDS}, then
    $f\in E^\epsilon_{\uqb_{\bm{c}'}, \uqb_{\bm{d}'}}$ and
    $f\cdot K_{\rho}\in E^\epsilon_{\uqbc,\uqbd}$ (where $\bm{d}=\bm{d}'$),
    where $\bm{c},\bm{d}$ satisfy the conditions of Lemma~\ref{lem-weyl-invariance}.
    Moreover, $\Res(f)=(-\rho)\triangleright \Res(f\cdot K_\rho)$.
\end{remark}

\begin{remark}[Sugitani, $\mathsf{CII}$]
\label{rem:CII}
    The author of \cite{Sug99} uses a similar approach as in Remark \ref{rem:noumi},
    except that the coideal $\mathfrak{k}_q$ is
    defined as the span of $L^+J-JS(L^-)^T$, and $f\in\mathcal{A}$ is a ZSF if
    \begin{align}\label{eq:trans3}
    \mathfrak{k}_q\cdot f=f\cdot \ad(K_\rho)(\mathfrak{k}_q^*)=0.
    \end{align}
    Suppose $\varphi\in E^\epsilon$, we would like to show
    that $\varphi\triangleleft K_{-\rho}$ is a ZSF in the
    sense of Sugitani.
    
    It is shown in \cite[Lemma~6.1, p. 762]{Le99} that
    the left ideal $\uq\mathfrak{k}_q$ is generated by $a-\epsilon(a)$ for $a\in \uqbs$. 
    Consequently, we have
    \[
        (a-\epsilon(a))\triangleright\varphi\triangleleft K_{-\rho} = 0.
    \]
    Moreover, we have
    \[
        \varphi\triangleleft K_{-\rho}\triangleleft
        \ad(K_\rho)(\mathfrak{k}_q^*)\uq =0 \Leftrightarrow
        \varphi \triangleleft \mathfrak{k}_q^*\uq =0.
    \]
    Moreover, by the end of \cite[\S1.1]{Sug99}, the ideal
    $\uq\mathfrak{k}_q$ is invariant under $*\circ S$.
    Consequently, $(\uq\mathfrak{k}_q)^*=\mathfrak{k}_q^*\uq=S(\mathfrak{k})\uq$.
    The right ideal $S(\mathfrak{k})\uq$ is generated by
    $S(a)-\epsilon(a)$, and by Corollary~\ref{cor-antipode-spherical},
    we indeed have
    \[
        \varphi\cdot (S(a)-\epsilon(a))=0.
    \]
    By uniqueness of the ZSF, we thus conclude that our ZSF
    equals the one from \cite{Sug99} up to a scalar.
\end{remark}

\begin{remark}[Noumi, Sugitani, $\mathsf{DIII}$]
    The authors of \cite{NS95} use the matrix $J$ to
    define a coideal subalgebra $U^{\mathrm{tw}}_q(\mathfrak{k})$, which by \cite[\S6]{Le99}
    (more specifically, Lemma~6.3) is a right coideal
    subalgebra.
    Therefore, there is an automorphism of $\uq$
    mapping $U^{\mathrm{tw}}_q(\mathfrak{k})$ to
    $\uqb_{\bm{c}(a)}$ for a suitable choice of $a$.
    The ZSFs that are then
    considered are elements $f\in\mathcal{A}$ such that
    \begin{align*}
        \forall a\in U^{\mathrm{tw}}_q(\mathfrak{k}):\qquad
        a\cdot f &= \epsilon(a)f,\\
        \forall b\in \ad(K_\rho)(U^{\mathrm{tw}}_q(\mathfrak{k})^*):\qquad
        f\cdot b &= \epsilon(b)f.
    \end{align*}
    An argument analogous to Remark \ref{rem:CII} shows that
    \[
        (-\rho)\triangleright\Res(\Phi^\lambda_\epsilon)
        = CP_\lambda.
    \]
    Note that from the reflection matrix $J$ from \cite{NS95} it can be shown that
    the twisted Hopf algebra $U^{\mathrm{tw}}_q(\mathfrak{k})$, which is just our
    coideal subalgebra $\uqbs$,
    is generated by the elements of the matrix
    $S(L^+)JS(L^-)^T$.
    In the correct entry, this matrix contains (up to scalars) the elements
    $B_nK_{n-1}K_n$ and $B_{n-1}K_{n-1}^{-1}K_n^{-1}$ (and also $K_{n-1}K_n$)
    for the parameters $c_{n-1}=c_n=-\frac{a_{n-1}}{a_n}$
    (to show this, the conventions of \cite[\S8.5.2]{ks} were used, with the caveat that
    the Chevalley map $\omega$ from \cite[\S3.2]{Kol14} has to be applied afterwards
    to match conventions).
    Consequently, $\uqbs$ is conjugate to $\uqb_{\bm{c}(1),0}$.
\end{remark}

In Lemma \ref{lem-other-weight-functions} we use the identification of zonal spherical functions with orthogonal polynomials to obtain an explicit description of the bilinear form $h\circ \Xi_\epsilon: E^\epsilon\otimes E^\epsilon\to k$ from Proposition \ref{prop:nondegen}.

\begin{lemma}\label{lem-other-weight-functions}
Let $\varphi_\lambda\in E^{\epsilon,\epsilon}_{\uqb,\uqb'}$ for all $\lambda\in 2L^+$ be nonzero zonal
spherical functions.
Assume that $((-\rho)\rhd\Res(\varphi_\lambda))_{\lambda\in 2L^+}$ is a family of orthogonal polynomials with respect to $\gls{triangledown}$, a $k$-valued distribution over $L$, i.e.
\[
    \forall \mu,\lambda\in 2L^+:\quad
    \ct( -\rho \triangleright \Res(\varphi_\lambda) \overline{-\rho \triangleright \Res(\varphi_\mu)}\triangledown)=\delta_{\lambda,\nu}h_\mu
\]
with $h_\lambda\ne0$ ($\lambda\in L$).
Then for all $\psi,\phi\in E^{\epsilon,\epsilon}_{\uqb,\uqb'}$ we have 
\[
    h\circ \Xi_\epsilon(\phi,\psi)=\ct (-\rho \triangleright \Res(\psi) \overline{-\rho \triangleright \Res(\phi)}\triangledown)/\ct(\triangledown).
\]
\end{lemma}
\begin{proof}
By \cite[Theorem~4.2]{Let03}, $h\circ\Xi_\epsilon$
factors through $-\rho\triangleright\Res(\cdot)$.
Consequently, there is a distribution $\triangledown'$ such that
\[
    h(\Xi(\phi,\psi))=\ct(-\rho\triangleright\Res(\phi)\overline{-\rho\triangleright\Res(\psi)}\triangledown').
\]
The maps $l':f\mapsto \ct(f\triangledown')$ and
$l:f\mapsto\ct(f\triangledown)$ are both elements of
$k[2L]^*$.
Let $\lambda\in (2L^+)\setminus\set{0}$, then
\[
    0 = h(\Xi(\varphi_\lambda,\varphi_0)) = 
    \overline{\varphi_0(1)}\ct(\triangledown'(-\rho)\triangleright\Res(\varphi_\lambda))
    = \overline{\varphi_0(1)}l'(-\rho\triangleright\Res(\varphi_\lambda))
\]
as $\phi_0$ is a constant.
Similarly, $l(-\rho\triangleright\Res(\varphi_\lambda))=0$.
We thus see that $l,l'$ have the same kernel of
codimension 1: the span of $\varphi_\lambda$ for $\lambda\ne0$.
Consequently, $l,l'$ span the same 1-dimensional
vector space, so they are linearly dependent.
Since both are non-zero, they are proportional.
Since $l,l'$ depend injectively on $\triangledown,\triangledown'$, we obtain that $\triangledown,\triangledown'$
differ by a non-zero scalar.
Evaluation of $l,l'$ at 1 gives that scalar.
\end{proof}

\section{Orthogonal Polynomials}\label{sec:opandex}
In this section, we associate to each quantum commutative triple $(\uq,\uqb,\gamma)$, a family of vector-valued orthogonal polynomials. We are interested in the general properties of these associated polynomials, which include:
\begin{enumerate}
    \item highest order terms (Proposition \ref{prop-lt-zsf});
    \item orthogonality relations (Proposition \ref{prop-matrix-weight});
    \item matrix valued interpretations (Definition \ref{def:matrixtovect});
    \item $q$-difference equations (Theorem \ref{thm:qdfi});
    \item $q\mapsto q^{-1}$ invariance (Theorem \ref{thm:polyqtoq});
    \item recurrence relations (Lemma \ref{lem:rec}).  
\end{enumerate}
Properties (i) and (ii) will be crucial for identifying the associated vector-valued polynomials with known families of vector-valued orthogonal polynomials.
\subsection{Leading Terms}\label{sec:leadorth}
From now on we will assume that $\uqb=\uqb'$ and
that $V=W$, i.e. that our MSF consist of square
matrices that are endomorphisms of the same
representation space.
This allows us to define 
$\Phi^\lambda_\gamma\in E^\gamma(\lambda)$ as the unique
MSF satisfying $\Phi^\lambda_\gamma(1)=1$ as in Lemma~\ref{lem-emsf-basis}.

Let $\varpi_1,\dots\varpi_r$ be a basis of $2L$ made up of dominant
elements. 
For $\lambda = \sum_{i=1}^r n_i\varpi_i\in X^+(\epsilon)$ we
define $\varphi_\lambda\in E^\epsilon(\lambda)$ to be the ZSF such
that $-\rho\triangleright\Res(\varphi_\lambda)=e^\lambda + \lot$ (this is possible by \cite[Lemma~4.1]{Let03},
whose proof can be extended to cases where
$L\otimes_\Z\Q$ is not spanned by $\Sigma$).
We also define
\[
    \psi_\lambda := \prod_{i=1}^r \varphi_{\varpi_i}^{n_i}.
\]

\begin{proposition}\label{prop-lt-zsf}
    Let $\lambda\in X^+(\epsilon)$. We have
    \[
        \psi_\lambda \in \varphi_\lambda + \sum_{\substack{\mu\in X^+(\epsilon)\\\mu<\lambda}} k\varphi_\mu,\qquad
        \varphi_\lambda \in \psi_\lambda+ \sum_{\substack{\mu\in X^+(\epsilon)\\\mu<\lambda}} k\psi_\mu.
    \]
\end{proposition}
\begin{proof}
    Note that the product of two matrix elements is a matrix element for
    the tensor product of the representations.
    As such, $\psi_\lambda$ is a matrix element of
    \[
        L(\varpi_1)^{\otimes n_1}\otimes\cdots\otimes L(\varpi_n)^{\otimes n_r}
        = L(\lambda) \oplus\bigoplus_{\mu<\lambda} L(\mu)^{m_\mu}
    \]
    for some multiplicities $m_\mu \in \N_0$. 
    Therefore, by the proof of Lemma~\ref{lem-emsf-basis}, $\phi_\lambda$ can be
    expressed as claimed with an arbitrary constant
    $C$.
    But since the leading terms of $-\rho\rhd\Res(\varphi_\lambda)$
    and $-\rho\triangleright\Res(\psi_\lambda)$ 
    both equal $e^\lambda$, this constant $C$ is one.
    
    The reverse equation follows from the triangularity and the fact that
    the dominance order on $X^+$ has finite down-sets.
\end{proof}

\begin{corollary}
    We have $E^\epsilon = k[\varphi_{\varpi_1},\dots\varphi_{\varpi_r}]$.
\end{corollary}

Recall the elementary spherical function $\Phi^b_\lambda$ from Lemma \ref{lem-emsf-basis} and the set $\mathfrak{B}(\gamma)$ from Definition \ref{def:bottom}. By \cite[Proposition 4.2]{Pe23} it follows that each $\lambda\in X^+(\gamma)$ has a unique decomposition $\lambda=b+\mu$ with $\mu\in X^+(\epsilon)$.

\begin{lemma}\label{lem-leading-term-expansion}
    Assume that $(\Phi^b_\gamma)_{b\in\mathfrak{B}(\gamma)}$ is linearly independent over $E^\epsilon$. 
    For
    \[
        \lambda = b + \mu \in X^+(\gamma)
    \]
    (with $b\in\mathfrak{B}(\gamma)$), we define
    \[
        \Psi^\lambda_\gamma := \psi_\mu \Phi^b_\gamma.
    \]
    Then for every $\lambda\in X^+(\gamma)$, there is $C\ne0$ such that
    \[
        \Psi^\lambda_\gamma \in C\Phi^\gamma_\lambda + \sum_{\substack{\nu<\lambda\\\nu\in X^+(\gamma)}} k\Phi^\nu_\gamma,\qquad
        \Phi^\gamma_\lambda \in C^{-1}\Psi^\lambda_\gamma + \sum_{\substack{\nu<\lambda\\\nu\in X^+(\gamma)}} k\Psi^\nu_\gamma.
    \]
\end{lemma}
\begin{proof}
    By a proof very similar to Proposition~\ref{prop-lt-zsf}, we see that
    for every $\lambda\in X^+(\gamma)$ there is $C\in k$ such that
    \begin{equation}\label{eq-phi-psi-triangular}
        \Psi^\lambda_\gamma = C\Phi^\gamma_\lambda + \sum_{\nu<\lambda}
        k\Phi^\nu_\gamma.
    \end{equation}
    It remains to show that $C\ne0$. 
    \eqref{eq-phi-psi-triangular} shows that $\Psi^\nu_\gamma$ ($\nu\le\lambda$) are contained in the span of $\Phi^\nu_\gamma$ ($\nu\le\lambda$). Since they are also linearly independent (since we assumed
    $(\Phi^b_\gamma)_{b\in\mathfrak{B}(\gamma)}$ to be $E^\epsilon$-linearly independent) and are contained in a finite-dimensional vector space, they also form a basis. The change-of-basis matrix between these
    two bases is triangular by \eqref{eq-phi-psi-triangular}. Consequently, the diagonal entries, i.e. the $C$'s, are nonzero, and its inverse is also triangular
    with the corresponding $C^{-1}$'s on the diagonal.
\end{proof}

Proposition~\ref{prop:cartanproj} is a quantum analog of \cite[Theorem
3.1]{vPr18} and will show that the assumptions of Lemma
\ref{lem-leading-term-expansion} are satisfied.

\begin{proposition}\label{prop:cartanproj}
Let $\lambda \in X^+(\epsilon)$, $\mu\in X^+(\gamma)$, let $\mathrm{span}_k\{v_{\mathrm{sph}}\}=V(\epsilon)\subset L(\lambda)$.
Then, the $\uqbs$-equivariant embedding $\eta: V(\gamma)\to L(\lambda)\otimes L(\mu)$ defined by
$$v\mapsto v_{\mathrm{sph}}\otimes v\qquad\text{for}\qquad v\in V(\gamma)$$
composed with the Cartan projection is nonzero. 

\end{proposition}
\begin{proof}
Write $\pi_{\lambda+\mu}: L(\lambda)\otimes L(\mu)\to L(\lambda+\mu)$ for the Cartan projection.
Indeed, $V(\gamma)$ contains an element $v$ with a nonzero component in the highest weight vector of $L(\mu)$, i.e. $v=v_\mu+\mathrm{l.o.t.}$. Similarly $v_{\mathrm{sph}}=v_\lambda++\mathrm{l.o.t.}$. Hence $\pi_{\lambda+\mu}(v_{\mathrm{sph}\otimes v})=\pi_{\lambda+\mu}(v_\lambda\otimes v_\mu+\mathrm{l.o.t.})\neq 0.$

\end{proof}
As a straightforward consequence of Proposition~\ref{prop:cartanproj}, we show in Lemma~\ref{lem:lin-indep}, under its assumptions, that the assumptions of Lemma~\ref{lem-leading-term-expansion} are satisfied. 
\begin{lemma}\label{lem:lin-indep}
$E^\gamma$ is a free left module over $E^\epsilon$ with basis $(\Phi^b_\gamma)_{b\in\mathfrak{B}(\gamma)}$. 
\end{lemma}
\begin{proof}
Let $\phi_b\in E^\epsilon$ for each $b\in \mathfrak{B}(\gamma)$ and suppose that $\sum _{b\in \mathfrak{B}(\gamma)}\phi_b \Phi_b=0$. Recall that, up to a scalar multiple, $\Res(\phi_b)=v_{\lambda_b}+\lot$. Let $b_0\in \mathfrak{B}(\gamma)$ be such that $v_{\lambda_{b_0}}$ is maximal and $\phi_{b_0}\neq 0$. Using Proposition \ref{prop:cartanproj}, up to a nonzero scalar multiple, $\phi_{b_0}\Phi_{b_0}= \Phi_{b_0+\lambda_{b_0}}+\lot$. By maximality of $b_0$ this shows that $\sum _{b\in \mathfrak{B}(\gamma)}\phi_b \Phi_b\neq 0$, a contradiction.
\end{proof}

\subsection{Intermediate Macdonald Polynomials and Their Vector Versions}
In this subsection we use the identification of the zonal spherical functions with orthogonal polynomials to deduce the matrix valued orthogonality weight for the matrix valued spherical functions. For this we recall the MSF satisfying $\Phi^\lambda_\gamma(1)=1$ from Lemma~\ref{lem-emsf-basis}.

\begin{proposition}\label{prop-matrix-weight}
    Let $\triangledown$ be a $k$-valued distribution over $L$ such that
    \[
    h(\Xi_\epsilon(\Phi,\Psi)) = \ct((-\rho\triangleright\Res(\Phi))\overline{-\rho\triangleright\Res(\Psi)}\triangledown)
    \]
    for $\Phi,\Psi\in E^\gamma$.
    Expand $\Phi\in E^\gamma$ as
    \[
        \sum_{b\in\mathfrak{B}(\gamma)} \phi_b\Phi^b_\gamma
    \]
    for $\phi_b\in E^\epsilon$, we define $\underline{\Phi}:= (\phi_b)_{b\in\mathfrak{B}(\gamma)}$. Let
    \[
        M_{b,b'} := (-\rho)\triangleright \Res(\Xi_\gamma(\Phi^b_\gamma,\Phi^{b'}_\gamma)),
    \]
    then
    \[
        h(\Xi_\gamma(\Phi,\Psi)) = \ct((-\rho\triangleright\Res(\underline{\Phi}))^T
        M \overline{-\rho\triangleright\Res(\underline{\Psi})} \triangledown)
    \]
    for $\Phi,\Psi\in E^\gamma$. In particular $h(\Xi_\gamma(\Phi^\lambda_\gamma,\Phi^\lambda_\gamma))=0$ for $\gamma$-spherical weights $\lambda\neq \mu$.
\end{proposition}
\begin{proof}
    Expand $\Phi,\Psi\in E^\gamma$ as
    \[
        \Phi = \sum_{b\in\mathfrak{B}(\gamma)} \phi_b
        \Phi^b_\gamma,\qquad
        \Psi = \sum_{b\in\mathfrak{B}(\gamma)} \psi_b\Phi^b_\gamma.
    \]
    Using Proposition \ref{prop-xi-properties}, it follows that
    \begin{align*}
        h(\Xi_\gamma(\Phi,\Psi)) &=
        \sum_{b,b'\in\mathfrak{B}(\gamma)} 
        h(\Xi_\gamma(\phi_b\Phi^b_\gamma,\psi_{b'}\Phi^{b'}_\gamma))\\
        &= \sum_{b,b'\in\mathfrak{B}(\gamma)} 
        h\Big(\phi_b \Xi_\gamma(\Phi^b_\gamma,\Phi^{b'}_\gamma)
        S(K_{-2\rho}\triangleright \psi_{b'})\Big)\\
        &= \sum_{b,b'\in\mathfrak{B}(\gamma)} 
        h\Big(\Xi_\epsilon\big(\phi_b \Xi_\gamma(\Phi^b_\gamma,\Phi^{b'}_\gamma),
        \psi_{b'}\big)\Big)\\
        &= \sum_{b,b'\in\mathfrak{B}(\gamma)}
        \ct((-\rho\triangleright\Res(\phi_b))M_{b,b'}\overline{-\rho\triangleright\Res(\psi_{b'})}\triangledown)\\
        &= \ct((-\rho\triangleright\Res(\underline{\Phi}))^TM\overline{-\rho\triangleright\Res(\underline{\Psi})}\triangledown)
    \end{align*}
    as claimed. The last statement now follows from Corollary \ref{cor-emsf-orthogonal}.
\end{proof}

In the following, we recall some notions from \cite{Sch23} needed to understand Corollary \ref{cor-Intermediate-macdonald-polys}. Let $J\subset I$, then we denote 
\[
2L_{+,J}=\{\lambda\in 2L\,:\, \forall j\in J:\quad \langle \lambda,\alpha_j\rangle\geq0\}
\]
the set of $J$-dominant elements. The $J$-dominant elements label
$W_J$-invariant monomials
\[
    m_{J,\lambda}=\sum_{\mu\in W_J\lambda} e^\mu\qquad
    \lambda\in L_{+,J}
\]
and distinguished $W_J$-invariant polynomials 
\[P_{J,\lambda}=m_{J,\lambda}+\lot\qquad\lambda \in 2L_{+,J}\]
known as the \emph{Intermediate Macdonald polynomials}, cf. \cite[Definition 4.7]{Sch23}. 
They generalize the symmetric and non-symmetric Macdonald polynomials, alternatively Intermediate Macdonald polynomials can be interpreted as vector-valued Macdonald polynomials, cf. \cite[\S 6]{Sch23}. We refer the reader to \cite{Sch23} for a detailed exposition.

\begin{corollary}\label{cor-Intermediate-macdonald-polys}
    Let $J$ be a subset of the simple roots of $\Sigma$ and let
    $W_J\le W_\Sigma$ be the corresponding parabolic subgroup of the restricted Weyl group. Define $e_b:=(-\rho)\triangleright\Res(\Phi^b_\gamma)$ for
    $b\in\mathfrak{B}(\gamma)$ and let
    \[
        \Gamma: k[2L]^{W_J}\to k[2L]^{W_\Sigma}\otimes\operatorname{span}_k\set{e_b\where b\in\mathfrak{B}(\gamma)}
    \]
    be a $k[2L]^{W_\Sigma}$-linear isomorphism. Equip the target space
    of $\Gamma$ with the following monomial basis: $m_\lambda\otimes e_b$ for $\lambda\in 2L_+,b\in\mathfrak{B}(\gamma)$ ($m_\lambda$ as in \cite[\S5.3]{Mac03}),
    whose order is inherited from the dominance order of $X^+(\gamma)$ by the isomorphism $(b,\lambda)\mapsto b+\lambda$.
    
    Let $t: 2L_{+,J}\to X^+(\gamma)$ be a bijection
    with $t^{-1}$ monotonic, such that $\Gamma^{-1}$ is triangular (with
    nonzero \enquote{diagonal} terms) with respect to $t^{-1}$. 
    Assume that there is a triangular $k[2L]^{W_\Sigma}$-linear automorphism $C$ of $k[2L]^{W_J}$ such that
    \[
        \forall b,b'\in\mathfrak{B}(\gamma):\quad
        \frac{1}{\#W_{\Sigma}}
        \sum_{w\in W_\Sigma}
        w \frac{\Gamma^{-1}(e_b)\overline{C(\Gamma^{-1}(e_{b'}))}}{\Delta^0}
        = M_{b,b'}
    \]
    (left-hand side similar to \cite[Lemma~6.1]{Sch23}, right-hand side as in Proposition~\ref{prop-matrix-weight}), where $\Delta^0$ as in \cite[\S5.1.25]{Mac03} is chosen in conventions such that $\triangledown$
    from Proposition~\ref{prop-matrix-weight} coincides with $\triangledown$
    from \cite[\S5.1.26]{Mac03}. 
    
    Then for every
    $\lambda\in 2L_{+,J}$, there is a nonzero constant $C_\lambda$ such
    that $\Gamma(P_{J,\lambda}) = C_\lambda(-\rho)\triangleright \Res(\Phi^{t(\lambda)}_\gamma)$.
\end{corollary}
\begin{proof}
    As a consequence of Lemma~\ref{lem:lin-indep} and of $k[2L]^{W_\Sigma}$-linearity and injectivity we
    thus have $(-\rho)\triangleright\Res(E^\gamma)\subset\im(\Gamma)$. Let
    now $\lambda,\mu\in X^+(\gamma)$ and write
    $\Phi:=\Phi^\gamma_\lambda$ and $\Psi:=\Phi^\mu_\gamma$. We now
    have
    \begin{align*}
        &h(\Xi_\gamma(\Phi,\Psi))\\
        &= \ct((-\rho\triangleright\Res(\underline{\Phi}))^TM\overline{-\rho\triangleright\Res(\underline{\Psi})}\triangledown)\\
        &= \frac{1}{\#W_\Sigma}\sum_{w\in W_\Sigma}
        \sum_{b,b'\in\mathfrak{B}(\gamma)}\ct((-\rho\triangleright\Res(\phi_b))
        \overline{-\rho\triangleright\Res(\phi_{b'})}
        w\Gamma^{-1}(e_b)\overline{C(\Gamma^{-1}(e_{b'}))}\Delta)\\
        &= \sum_{b,b'\in\mathfrak{B}(\gamma)}\ct(\Gamma^{-1}((-\rho\triangleright\Res(\phi_b))e_b)
        \overline{C(\Gamma^{-1}((-\rho\triangleright\Res(\phi_{b'}))e_{b'}))}
        \Delta)\\
        &= \ct(\Gamma^{-1}(-\rho\triangleright\Res(\Phi))\overline{C(\Gamma^{-1}(-\rho \triangleright \Res(\Psi)))}\Delta)\\
        &=
        \qty(\Gamma^{-1}(-\rho\triangleright\Res(\Phi)),
        C(\Gamma^{-1}(-\rho\triangleright\Res(\Phi)))^\circ),
    \end{align*}
    where $(\cdot,\cdot)$ is the non-symmetric Macdonald inner product
    with weight $\Delta$ (see \cite[\S\S5.1.7,5.1.17]{Mac03}) and
    $\circ:k[2L]\to k[2L]$ is the ring automorphism mapping
    $q\mapsto q^{-1}$ and $e^\lambda\mapsto e^\lambda$.
    
    As a consequence, $(\Gamma^{-1}((-\rho)\triangleright\Res(\Phi^\gamma_\lambda)))_{\lambda\in X^+(\gamma)}$
    is a sequence of orthogonal polynomials. By Lemma~\ref{lem-leading-term-expansion}, we can expand
    \[
        (-\rho)\triangleright\Res(\Phi^{b+\lambda}_\gamma) = a m_\lambda\otimes e_b
        + \lot
    \]
    for $a\ne0$. Since $\Gamma^{-1}$ is triangular with respect
    to $t^{-1}$, so is
    $C\circ \Gamma^{-1}$.
    Since $t^{-1}$ is also monotonic, we then have
    \[
        C\Gamma^{-1}(-\rho\triangleright\Res(\Phi^{b+\lambda}_\gamma))
        = a' m_{J,t^{-1}(\lambda+b)} + \lot
    \]
    where $m_{J,\nu}$ is defined as in \cite[\S4.2 p.23]{Sch23}. The two facts
    of orthogonality and leading terms identify our polynomials as
    Intermediate Macdonald polynomials.
\end{proof}

As a consequence, to establish a link between elementary MSFs for
$(\uq,\uqbs,\gamma)$ it is sufficient to match $2L_{+,J}$ with the
spherical weights, and to match a $k[2L]^{W_\Sigma}$ basis of $k[2L]^{W_J}$ with the $E^\epsilon$-basis
$\set{\Phi^b_\gamma}_{b\in\mathfrak{B}(\gamma)}$ in such a way
that the matrix weights from Proposition~\ref{prop-matrix-weight} and
\cite[Lemma~6.1]{Sch23} (using $\overline{\cdot}$ instead of $*$) match.

\subsubsection{Matrix-Valued Interpretations}\label{sec:matrixvalued}
In this section we assosiate to each family of matrix-spherical functions a family of matrix valued orthogonal polynomials. This construction generalizes that of \cite{Ald} in the quantum case and that of \cite{vPr18} in the classical case. In this section this viewpoint is taken and we study the properties of the associated matrix valued orthogonal polynomials. In particular, we show that the polynomials are invariant under the involution $q\mapsto q^{-1}$ and are eigenfunctions of matrix valued $q$-difference operators.

\begin{notation}
Let $\mathbb{M}_{\mathfrak{B}(\gamma)}=k^{\mathfrak{B}(\gamma)\times\mathfrak{B}(\gamma)}$, 
the space of square matrices over $k$ whose rows and
columns are indexed by the finite set $\mathfrak{B}(\gamma)$.
\end{notation}

\begin{remark}\label{rem:normalization}
Recall that we assumed that for each $b\in \mathfrak{B}(\gamma)$ the elementary spherical functions $\Phi_\gamma^b$ from Lemma \ref{lem-emsf-basis} are normalized with $\Phi_\gamma^b(1)=1$. Using Lemma~\ref{lem-leading-term-expansion}, we now assume that $\Phi_\gamma^{b+\lambda}$ is normalized so that $\Phi_\gamma^{b+\lambda}=\varphi_\lambda \Phi_\gamma^{b}+\lot$,
and $\varphi_\lambda$ such that $-\rho \,\triangleright\,\Res(\varphi_\lambda)=m_\lambda+\lot$ .
\end{remark}

For each spherical function $\Phi\in E^{\gamma}$, recall the notation $\underline{\Phi}\in (k[2L]^{W_\Sigma})^{\mathfrak{B}(\gamma)}$, as introduced in Proposition \ref{prop-matrix-weight}.

\begin{definition}\label{def:matrixtovect}
    For $\lambda\in 2L^+$ define $\gls{Qlambda}\in \mathbb{M}_{\mathfrak{B}(\gamma)}\otimes k[2L]^{W_\Sigma}$ to be
    \[
        Q_\lambda := \qty((-\rho)\triangleright\Res
        \qty(\underline{\Phi^{b+\lambda}_\gamma}))_{b\in\mathfrak{B}(\gamma)},
    \]
    i.e. as the matrix-valued polynomial whose $b$-th
    column is $(-\rho)\triangleright
        \Res\qty(\underline{\Phi^{b+\lambda}_\gamma})$.
\end{definition}

\begin{remark}
    Lemma~\ref{lem-leading-term-expansion} and the
    choice of normalisation show that the diagonal entries of $Q_\lambda$ have leading term $m_\lambda$ and that the off-diagonal $(b,b')$-th entry has
    leading exponent $< \lambda+b'-b$.
\end{remark}

\begin{definition}
Define the symmetric matrix-valued bilinear form
$\langle\cdot,\cdot\rangle$ on $\mathbb{M}_{\mathfrak{B}(\gamma)}\otimes k[2L]^{W_\Sigma}$
by
\begin{equation}\label{eq:pairing}
\langle A,B\rangle _{b,b'} :=  \mathrm{ct}\big(( A^T M \overline{B})_{b,b'}\triangledown\big),\quad \text{where}\quad A,B\in\mathbb{M}_{\mathfrak{B}(\gamma)}\otimes k[2L]^{W_\Sigma}.
\end{equation}

Here, the involution $\overline{\,\cdot\,}$ is defined component-wise. 
\end{definition}

The matrix valued polynomials $Q_\lambda$ are the objects of interest of this subsection.

\begin{proposition}
    $\langle\cdot,\cdot\rangle$ is a 
    matrix-valued symmetric bilinear form.
\end{proposition}
\begin{proof}
    Let $F,G\in \mathbb{M}_{\mathfrak{B}(\gamma)}\otimes k[2L]^{W_\Sigma}$ and
    $A,B\in\mathbb{M}_{\mathfrak{B}(\gamma)}$, then
    \begin{align*}
        \langle FA, GB\rangle &= 
        \ct((FA)^T M\overline{GB}\triangledown)
        = \ct(A^TF^TM\overline{G}\overline{B}\triangledown)\\
        &= A^T\ct(F^TM\overline{G}\triangledown)B
        = A^T\langle F,G\rangle B
    \end{align*}
    (recall that $\overline{\cdot}$ is $k$-linear and
    in particular does nothing to constant matrices).
    Combined with the fact that $\langle\cdot,\cdot\rangle$ is $k$-bilinear, it therefore is
    $\mathbb{M}_{\mathfrak{B}(\gamma)}$-bilinear,
    
    Furthermore, for $F,G\in \mathbb{M}_{\mathfrak{B}(\gamma)}\otimes k[2L]^{W_\Sigma}$ we have
    \[
        \langle F, G\rangle = 
        \ct(F^TM\overline{G}\triangledown)
        = \ct(\overline{G}^T M^T F\triangledown)^T
        = \ct(G^T\overline{M}^T\overline{F} \triangledown)^T.
    \]
    From Proposition \ref{prop-weight-matrix-props} we find that
    $\overline{M}^T=M$, so that the above equals
    \[
        =\langle G,F\rangle^T.\qedhere
    \]
\end{proof}

Using Proposition~\ref{prop-matrix-weight} and Proposition~\ref{prop:nondegen}, we conclude the following theorem.
\begin{theorem}\label{thm:matrixdefining}
The matrix-valued polynomials $\{Q_\lambda\,:\,\lambda\in 2L^+\}$ span
$\mathbb{M}_{\mathfrak{B}(\gamma)}\otimes k[2L]^{W_\Sigma}$
$\mathbb{M}_{\mathfrak{B}(\gamma)}$-linearly and satisfy
$$\langle Q_\lambda,Q_{\lambda'}\rangle=\delta_{\lambda,\lambda'}D_\lambda,\qquad \text{where}\qquad \lambda,\lambda'\in 2L^+$$
and where $D_\lambda\in \mathbb{M}_{\mathfrak{B}(\gamma)}$ is diagonal and invertible.
\end{theorem}

Using the Cartan decomposition, we continue by describing a family of matrix valued $q$-difference operators for which $\{Q_\lambda\,:\,\lambda\in 2L^+\}$ are simultaneous eigenfunctions. 
Recall that

\[
  \uq^{\uqbs}=\{\Omega\in \uq: \Omega b=b\Omega\qquad \text{for all }b\in \uqbs\}.
\]

By Schur's Lemma and multiplicity freeness, for each $\Omega\in \uq^{\uqbs}$ and $\lambda\in X^+(\gamma)$, there exists a scalar $C^\gamma_\Omega$ such that $\Omega\triangleright \Phi^\gamma_\lambda= C^\gamma_\Omega \Phi^\gamma_\lambda$.

Let $\mathfrak{B}^+(\gamma)=\{b_1,\dots , b_d\}$ and let $V_\gamma: \big(E^{\epsilon}\big)^{\oplus d\times d}\to\big(E^{\gamma}\big)^{\oplus d}$ denote the unique linear isomorphism with:

\begin{equation}\label{eq:defvgamma}
    \qty(\phi_{i,j})_{1\le i,j\le d}
    \mapsto
    \qty(\sum_{j} \phi_{j,i}\Phi_\gamma^{b_j})_{1\le i\le d},
\end{equation}
such that
\[
Q_\lambda = -\rho\triangleright\Res\qty(V_\gamma^{-1}\qty((\Phi^{b_i+\lambda})_{1\le i\le d}))
\]
\begin{definition}
For each $\Omega\in \uq^{\uqbs}$ we let $\mathcal D_\Omega$ denote the unique operator making the following diagram commute:
\begin{equation}\label{eq:qdifferentialop}
\begin{tikzcd}
	\big(E^{\gamma}\big)^{\oplus d} \quad\arrow[rr, "{\qquad(\Phi_{\gamma}^{b_1},\dots ,\Phi_{\gamma}^{b_d})\mapsto (\Omega\triangleright \Phi_{\gamma}^{b_1},\dots ,\Omega\triangleright \Phi_{\gamma}^{b_d})}"] &  & \quad\qquad\big(E^{\gamma}\big)^{\oplus d} \arrow[d, "V_\gamma^{-1}"]    \\
	\big(E^{k}\big)^{\oplus d\times d} \arrow[u, "V_\gamma"] \arrow[d, "-\rho\triangleright \Res(-)","\thicksim" ']                                              &  & \big(E^k\big)^{\oplus d\times d} \arrow[d, "-\rho\triangleright \Res(-)","\thicksim" '] \\
	{\mathbb{M}_{\mathfrak{B}(\gamma)}\otimes k[2L]^{W_\Sigma}} \arrow[rr, "\mathcal D_\Omega"]                                                                         &  & {\mathbb{M}_{\mathfrak{B}(\gamma)}\otimes k[2L]^{W_\Sigma}}                           
\end{tikzcd}
\end{equation}

Write $\mathcal{D}_\mu:= \mathcal{D}_{c_\mu}$ for the operator obtained from \eqref{eq:qdifferentialop}
and the distinguished central element $c_\mu\in Z(\uq)$ from
Lemma~\ref{lem-centre-u}.
\end{definition}

By the construction of the matrix-valued orthogonal polynomials and \cite[\S7.1.19]{Jos95}, we conclude that:

\begin{corollary}\label{cor:MVDO}
	The matrix-valued orthogonal polynomials $\{Q_\lambda:\lambda\in 2L^+\}$ satisfy the relation
	\[
        \mathcal D_\mu (Q_\lambda)=Q_\lambda \sum_{\nu\in X}
        \diag\qty(q^{(\lambda+b_1+\rho)\cdot\nu},
        \dots,
        q^{(\lambda+b_d+\rho)\cdot\nu})
        \dim(L(\mu)_\nu)
        \quad \lambda \in 2L^+,\mu \in X^+, 2h_\mu\in Y.
    \]
\end{corollary}

Using Proposition~\ref{prop:cartandecomp}, Equation~\eqref{eq:defvgamma},
it follows that conjugation with the elementary MSF from $\mathfrak{B}^+(\gamma)$ yield a $q$-difference equation for the matrix-valued orthogonal polynomials $\{Q_\lambda:\lambda\in 2L^+\}$. 

\begin{theorem}\label{thm:qdfi}
For each $\Omega\in \uq^{\uqbs}$ the matrix-valued orthogonal polynomials $\{Q_\lambda:\lambda \in 2L^+\}$ satisfy the matrix-valued $q$-differential equation 
\begin{equation}
Q_\lambda \Lambda_{\Omega,\lambda}=\mathcal{D}_{\Omega }(Q_\lambda)
\end{equation}
where $\Lambda_{\lambda,\Omega}$ is a diagonal matrix.
\end{theorem}

Thus, the matrix-valued orthogonal polynomials diagonalize a commutative algebra 
\[
  k\langle \mathcal D_\mu:\mu\in X^+, 2h_\mu\in Y\rangle
\]
of difference-reflection operators.
From Proposition~\ref{prop:orderdo}, we know that
$c_\mu$ has order $-2h_\mu$,
which shows that $\mathcal{D}_\mu$ is a difference
operator of order $2h_\mu$.
This order needs to be understood modulo the annihilator
of $2L$, as we know that all entries of $Q_\lambda$ are
contained in $k[2L]$.

Recall the anti-linear involution $^0:k[2L]^{W_\Sigma}\to k[2L]^{W_\Sigma}$ introduced in Section \ref{sec:barinv} and the involution $\sigma=\tau\circ\tau_0$ also as introduced in Section \ref{sec:barinv}. We extend $^0:k[2L]^{W_\Sigma}\to k[2L]^{W_\Sigma}$ to $^0:\mathbb{M}_{\mathfrak{B}(\gamma)}\otimes k[2L]^{W_\Sigma}\to \mathbb{M}_{\mathfrak{B}(\gamma)}\otimes k[2L]^{W_\Sigma}$ component-wise. In the case $V(\gamma)\cong V(\gamma)^\sigma$, we show that $Q_\lambda=Q_\lambda^0$.

\begin{lemma}\label{lem-msf-0}
  Let $\mu\in X^+(\gamma)$, let $p,i$ be a $\uqb$-linear 
  projection and injection with $p\circ i=\id_{V(\gamma)}$ for $V(\gamma)$ and $L(\mu)$,
  and let $v\in V(\gamma),f\in V(\gamma)^*$ such that
  $i(v)$ and $f\circ p$ are $\imath$bar-invariant.
  Then
  \[
    (-\rho)\triangleright\Res(f\Phi^\mu_\gamma v)
    = w_0\qty((-\rho)\triangleright\Res(fT_{w_\bullet})\Phi^\mu_{\sigma(\gamma)}T_{w_\bullet}^{-1}v))^\circ,
  \]
  where we take $\Phi^\mu_\gamma$ and $\Phi^\mu_{\sigma(\gamma)}$ to be defined by $(p,i)$
  and $(p\circ\mathcal{K},\mathcal{K}^{-1}\circ i)$.
\end{lemma}
\begin{proof}
  Let $h\in Y$, then
  \begin{align*}
    (-\rho)\triangleright\Res(f\Phi^\mu_\gamma v)(h)
    &= c^\mu_{f\circ p, i(v)}(K_{h-\rho})\\
    &= \overline{c^\mu_{\overline{f\circ p}, \overline{i(v)}}(K_{-h+\rho})}\\
    &= \overline{c^\mu_{fT_{w_\bullet}\circ p\circ\mathcal{K},
    \mathcal{K}^{-1}i(T_{w_\bullet}^{-1}v)}(K_{-w_0h-\rho})}\\
    &= \overline{fT_{w_\bullet}\Phi^\mu_{\sigma(\gamma)}(K_{-w_0h-\rho})
    T_{w_\bullet}v}\\
    &= \qty((-\rho)\triangleright\Res(fT_{w_\bullet}
    \Phi^\mu_{\sigma(\gamma)}T_{w_\bullet}^{-1}v))^\circ(w_0h)
  \end{align*}
  by Lemmas~\ref{lem-conjugation-matrix-elements} and \ref{lem-matrix-elements-bar-invariant}.
\end{proof}

\begin{theorem}\label{thm:polyqtoq}
If $\gamma\cong \sigma(\gamma)$, then for each $\lambda \in 2L^+$ we have $Q_\lambda=Q_\lambda^0$.
\end{theorem}

\begin{proof}
    As in Proposition~\ref{prop:blackmagic-ibarinv-bases} let $v_1,\dots,v_n\in V(\gamma)$ and $f_1,\dots,f_n\in V(\gamma)^*$ be
    dual bases that are $\imath$bar-invariant under every
    injection and projection used to define the functions used here.
    
    For $j,k=1,\dots,n$ we have
    \begin{align*}
      &\sum_{b'\in\mathfrak{B}^+(\gamma)}
      (Q_\lambda)_{b',b} w_0\qty((-\rho)\triangleright\Res(f_jT_{w_\bullet}\Phi^{b'}_\gamma T_{w_\bullet}^{-1}v_k))^\circ\\
      \stackrel{Lem \ref{lem-msf-0} \,\&\, \gamma\cong \sigma(\gamma)}{=}& \sum_{b'\in\mathfrak{B}^+(\gamma)}
      (Q_\lambda)_{b',b}
      (-\rho)\triangleright\Res(f_j\Phi^{b'}_\gamma v_k)\\      
      \stackrel{Def \ref{def:matrixtovect}}{=}& (-\rho)\triangleright\Res\qty(\sum_{b'\in\mathfrak{B}^+(\gamma)} \underline{\Phi^{b+\lambda}_\gamma}_{b'}
      f_j\Phi^{b'}_\gamma v_k)\\      
      =& (-\rho)\triangleright\Res(f_j\Phi^{b+\lambda}_\gamma v_k)\\      
      \stackrel{Lem \ref{lem-msf-0} \,\&\, \gamma\cong \sigma(\gamma)}{=}& w_0\qty((-\rho)\triangleright\Res(f_jT_{w_\bullet})\Phi^{b+\lambda}_{\sigma(\gamma)}T_{w_\bullet}^{-1}v_k))^\circ\\      
      =& \sum_{b'\in\mathfrak{B}^+(\gamma)}
      w_0(Q_\lambda)^\circ_{b',b}
      w_0\qty((-\rho)\triangleright\Res(f_jT_{w_\bullet}\Phi^{b'}_\gamma T_{w_\bullet}^{-1}v_k))^\circ. 
    \end{align*}
    Since the $\Phi^{b'}_\gamma$ are linearly independent, and
    since $\Res$ is injective, we conclude
    \[
       w_0 Q_\lambda^\circ = Q_\lambda,
    \]
    which implies the claim as $w_0Q_\lambda=Q_\lambda$.

\end{proof}
Theorem \ref{thm:polyqtoq} is the marix-valued analogue of the $q\to q^{-1}$ invariance of Macdonald polynomials, which is shown in \cite[\S5.3.2]{Mac03}. Moreover, in the cases we know the $\{Q_\lambda:\lambda\in 2L^+\}$ to be Macdonald polynomials, cf. \cite{Let04} \cite{Mee24} and Section \ref{sec-singlevarsmalk} (in these cases $V(\gamma)\cong V(\gamma)^\sigma$), this gives a new proof and interpretation of the $q\to q^{-1}$ invariance.

Lastly, we show the existence of recurrence relations. Recall that $\{\varpi_1,\dots\varpi_r\}$ denotes a generating set of the monoid $2L^+$.

\begin{notation}
For $\lambda \in X^+$, let $\Lambda(\lambda)$ denote the weights occurring in the module $L(\lambda)$.
\end{notation}

\begin{lemma}\label{lem:rec}
For each $1\leq i\leq r$ and $\lambda \in 2L^+$ we have
$$-\rho\triangleright\Res(\phi_{\varpi_i}) Q_\lambda=\sum _{\mu\in \Lambda(\varpi_i)} a_{i,\mu,\lambda}Q_{\mu+\lambda},\qquad \qquad a_{i,\mu,\lambda}\in k.$$
\end{lemma}

\begin{proof}
Lemma~\ref{lem-leading-term-expansion} shows that $-\rho\triangleright\Res(\varphi_{\varpi_i}) Q_\lambda=\sum _{\mu\in X^+(\gamma)} a_{i,\lambda,\mu}Q_{\mu}$. The scalar $a_{i,\lambda,\mu}$ is zero if $\mu$ is not a highest weight contained in the tensor product decomposition $L(\varpi_i)\otimes L(\lambda+b)$, for some $b\in \mathfrak{B}(\gamma)$. As the highest weights in the tensor product decomposition $L(\varpi_i)\otimes L(b+\lambda)$ are contained in $\lambda+b +\Lambda (\varpi_i)$, cf. \cite[\S24.4]{Hum72}, the result follows.
\end{proof}

\begin{remark}
From Proposition~\ref{prop:cartanproj}, we furthermore observe that $a_{i,\lambda+\varpi_i,\lambda}\neq 0$.
\end{remark}

Since $$k[2L]^{W_\Sigma}=k[-\rho\triangleright\Res(\varphi_{\varpi_1}),\dots -\rho\triangleright\Res(\varphi_{\varpi_r})],$$
the $E^\epsilon$-module structure of $E^\gamma$ can be understood through the coefficients $a_{i,\lambda,\mu}$.
\part{Examples}\label{part:2}
In Part \ref{part:1}, a general theory of quantum matrix-spherical functions was developed. In particular, Section \ref{sec:opandex} associated to each quantum commutative triple a family of vector-valued orthogonal polynomials. We now explicit study of these polynomials. Specifically, Section \ref{sec-singlevarsmalk} treats the cases corresponding to Satake diagrams of type $\mathsf{BII}n$ and $\mathsf{CII}{n,1}$; Section \ref{sec-singlevar2} treats type $\mathsf{DII}_n$; and Section \ref{sec:multivar} treats types $\mathsf{AI}_2$, the $\mathsf{A}_2$-group case, and $\mathsf{AII}_5$.
As summarised in Section~\ref{sec-conc-examples}, this establishes
a general correspondence between MSF of quantum commutative triples and intermediate Macdonald polynomials for cases of rank 1, and for
small integrable $\uqb$-types (cf. Definition~\ref{def:bottom}), and
for three more examples of rank 2.
This provides the motivation to formulate a more general conjecture
(Conjecture~\ref{conj-imp}).

\section{Examples}\label{sec-examples}
We now provide some further examples of MSF that are not ZSFs that we
can nevertheless equate to polynomials within the Macdonald framework, polynomials with a slightly relaxed symmetry requirement:
the Intermediate Macdonald polynomials from \cite{Sch23}.
This identification will be achieved using Corollary~\ref{cor-emsf-orthogonal} and the fact that the inner product $h\circ\Xi_\gamma$ can be expressed as an inner product with a matrix
weight that also occurs when describing Intermediate Macdonald polynomials.

\subsection{Integrable Small $\uqb$-Type: $\mathsf{BII}$ and $\mathsf{CII}_{n,1}$}\label{sec-singlevarsmalk}
We start with two examples of what for groups is called a small $K$-type:
two integrable small $\uqb$-types. 
(We show later that these two examples are the only two that are not known already.)
For this case, as in
\cite{Mee25}, we expect to obtain symmetric Macdonald polynomials.

\subsubsection{Notation}
Let $(I,I_\bullet,\tau)$ correspond to the following
Satake diagrams: $\mathsf{BII}_n$ or $\mathsf{CII}_{n,1}$
\[\dynkin[labels={1,2,n-1,n}, scale=1.8] B{II},\qquad \dynkin[labels={1,2,3,n-1,n}, scale=1.8]C{*o*...**}.\]
Let $X$ be the weight lattice and let $Y$ be the coroot lattice. Let
$\alpha_1,\dots,\alpha_n$ be the roots, $h_1,\dots,h_n$ the coroots,
and $\omega_1,\dots,\omega_n$ the fundamental weights. Moreover,
we pick $\epsilon_1=\dots=\epsilon_{n-1}=2$ and $\epsilon_n=1$, or vice-versa.

In case of $\mathsf{CII}_{n,1}$ we also assume that
$n>2$; otherwise, it is identical to $\mathsf{BII}_2$.

The subalgebra $\uq_\bullet$ is then generated by
$E_i,F_i,K_i$ for $i\ne 1$ ($\mathsf{BII}_n$) or $i\ne2$
($\mathsf{CII}_{n,1}$).

The coideal subalgebra is then generated by $\uq_\bullet$ and
\begin{align*}
    \mathsf{BII}_n:\quad B_1 &= F_1 + c\ad(E_2\cdots E_n^2\dots E_2)(E_1)K_1^{-1}\\
    \mathsf{CII}_{n,1}:\quad B_2 &= F_2 + c\ad(E_1E_3\cdots E_n\cdots E_3)(E_2)K_2^{-1}.
\end{align*}
Note that these elements lie in the weight spaces
\[
    \uq_{-\alpha_1}\oplus\uq_{2\omega_1-\alpha_1},\qquad\uq_{-\alpha_2}\oplus\uq_{\omega_2-\alpha_2},
\]
respectively.
The restricted root system is
\begin{align*}
    \mathsf{BII}_n:\quad \Sigma &= \set{\pm\omega_1}\\
    \mathsf{CII}_{n,1}:\quad \Sigma&=\set{\pm\frac{\omega_2}{2},\pm\omega_2}.
\end{align*}
Consequently, $2L$ is spanned by $\omega_1$ and
$\omega_2$, respectively.

\subsubsection{Classical Branching Rules}\label{sec-higher-spin-classical}
For the classical branching rules, we consult
\cite[Table~2]{Pe23}.
In lines B.4 and B.7.6, we find that
for $\gamma$ corresponding to the simple module of highest
weight (in the notation of loc.cit.)
$s\varpi_n$ or $s\varpi_{n-1}$, or
$s\varpi_1$, respectively, we have
\[
    X^+(\gamma) = \begin{cases}
        s\omega_n + \N_0\omega_1&\mathsf{BII}_n\\
        s\omega_1 + \N_0\omega_2&\mathsf{CII}_{n,1}.
    \end{cases}
\]
Moreover, we have
\[
    \begin{cases}
    L(s\omega_n)|_{\mathfrak{k}}=\bigoplus_{r=0}^s
    L_{\mathfrak{k}}(r\varpi_{n-1}+(s-r)\varpi_n)
    & \mathsf{BII}_n\\
    L(s\omega_1)|_{\mathfrak{k}}=\bigoplus_{r=0}^s
    L_{\mathfrak{k}}(r\varpi_1 + (s-r)\varpi_1') & \mathsf{CII}_{n,1}.
    \end{cases}
\]
To further examine how these $\mathfrak{k}$-representations restrict
to $\mathfrak{m}$, we note that for $\mathsf{CII}_{n,1}$ we have
$\mathfrak{m}\cong \mathfrak{su}(2)\oplus\mathfrak{sp}(2n-4)$, where
the first summand is the diagonal subspace of two copies of
$\mathfrak{su}(2)$.
Therefore, the branch rule $\mathfrak{k}\to\mathfrak{m}$ may
be obtained in that case, by branching the second summand
$\mathfrak{sp}(2n-2)$ of $\mathfrak{k}$ into
$\mathfrak{su}(2)\oplus\mathfrak{sp}(2n-4)$, and then
taking the diagonal of the two $\mathfrak{su}(2)$ summands (i.e.
considering the tensor product decomposition).

As such, from Lines~B.5 and B.7.6, we obtain
\[
    \begin{cases}
        L_{\mathfrak{k}}(r\varpi_{n-1} + (s-r)\varpi_n)|_{\mathfrak{m}} = \bigoplus_{t=0}^{r\wedge(s-r)}
        L_{\mathfrak{m}}(t\omega_{n-2} + (s-t)\omega_{n-1}) & \mathsf{BII}_n\\
        L_{\mathfrak{k}}(r\varpi_1 + (s-r)\varpi_1')|_{\mathfrak{m}}
        = \bigoplus_{t=0}^{s-r}
        \bigoplus_{u=\abs{r-t}}^{r+t} L_{\mathfrak{m}}((u-s+r+t)\omega_1 + (s-r-t)\omega_3) & \mathsf{CII}_{n,1}.
    \end{cases}
\]
We obtain precisely one summand if $r=0$ or $r=s$ in the first case,
and $r=s$ in the second case.
Moreover we see that in the first case, there is exactly one
summand isomorphic to $L_{\mathfrak{m}}(s\omega_{n-1})$ (namely $t=0$); and in the
second case, exactly one summand is isomorphic to $L_{\mathfrak{m}}(s\omega_1)$ (namely $t=s-r$ and $u=s$).

This shows that $L(s\omega_n)|_{\mathfrak{m}}$ contains
$s+1$ copies of $L_{\mathfrak{m}}(s\omega_n)$ ($=L_{\mathfrak{m}}(s\omega_{n-1})$) in the first case,
and $L(s\omega_1)|_{\mathfrak{m}}$ contains $s+1$ copies of
$L_{\mathfrak{m}}(s\omega_1)$ in the second case.

\begin{remark}\label{rmk-m-branching}
    There is an error in \cite[Table~2, Line~B.5]{Pe23}.
    One of the generators is given as $(\omega_n,-\varpi_{n-1})$, which corresponds to a branch rule for
    the $\mathfrak{so}(2n+2)$ and $\mathfrak{so}(2n+1)$-modules
    with highest weights
    \[
        (a_1,\dots,a_{n+1})=\qty(\frac{1}{2},\dots,\frac{1}{2},-\frac{1}{2}),\qquad (b_1,\dots,b_n)=(1,\dots,1,0).
    \]
    This contradicts the well-known branching rules for this
    case, which require all differences $a_i-b_j$ to be
    integers.

    It seems that this generator should be replaced by
    $(\omega_n+\omega_{n+1},-\varpi_{n-1})$, which corresponds
    to
    \[
        (1,\dots,1,0),\qquad (1,\dots,1,0),
    \]
    which conforms to the well-known branching rules.
\end{remark}

\subsubsection{Describing $V(\gamma)$}
As is suggested by the previous section, we will now describe $V(\gamma)$ as it lies in $L(b):=L(s\omega_n)$ or $L(s\omega_1)$, respectively. 
Moreover, since classically for $K=\operatorname{Spin}(2n)$
or $K=\operatorname{Sp}(2)\times\operatorname{Sp}(2n-2)$
the subgroup $M$ (=normaliser in $K$ of a maximal commutative
subalgebra in $\mathfrak{g}^{-\Theta}$) is analytic,
we obtain that the specialisation of $V(\gamma)$ is 
simple over $\mathfrak{m}$, and hence that
$V(\gamma)$ is simple over $\uq_\bullet$.

It is therefore sufficient to find a $\uq_\bullet$-highest-weight
vector $v$ for $V(\gamma)$.

\begin{notation}
    For a tuple $J=(i_1,\dots,i_r)\in I^r$ write
    $\overline{J}:=(i_r,\dots,i_1)$ and
    \[
        F_J := F_{i_1}\cdots F_{i_r}
    \]
    similarly for $E_J,B_J$ (for the latter we interpret
    $B_i=F_i$ if $i\in I_\bullet$).
    Let moreover
    \[
            K_J := K_{i_1}\cdots K_{i_r}
            = K_{\sum_{j=1}^r\epsilon_{i_j}h_{i_j}}.
    \]
\end{notation}

\begin{lemma}\label{lem-higher-spin-hwv}
    Let
    \[
        V':= \bigoplus_{r=0}^sV'_r:=\begin{cases}
            \bigoplus_{r=0}^s L(s\omega_n)_{s\omega_n-r\omega_1} & \mathsf{BII}_n\\
            \bigoplus_{r=0}^s L(s\omega_1)_{s\omega_1-r\omega_2} & \mathsf{CII}_{n,1},
        \end{cases}\qquad
        K := \begin{cases}
            (n,\dots,1) & \mathsf{BII}_n\\
            (2,\dots,n,\dots,1) & \mathsf{CII}_{n,1}.
        \end{cases}
    \]
    Then $V'$ is an $s+1$-dimensional vector space
    and $E_J,F_{\overline{J}},B_{\overline{J}}$ restrict
    to endomorphisms of it.
\end{lemma}
\begin{proof}
    Note that $E_J$ and $F_{\overline{J}}$ have weight
    $\pm\omega_1$ for $\mathsf{BII}_n$ and
    $\pm\omega_2$ for $\mathsf{CII}_{n,1}$.
    Moreover,
    \[
        B_{\overline{J}}\in\begin{cases}
            \uq_{\omega_1}\oplus\uq_{-\omega_1} & \mathsf{BII}_n\\
            \uq_{\omega_2}\oplus\uq_0\oplus\uq_{-\omega_2} & \mathsf{CII}_{n,1}.
        \end{cases}
    \]
    As a consequence, these three operators map $V'$ to itself.
    Note that every vector in $v\in V'$ is a $\uq_\bullet$-highest weight vector of weight
    $b:=s\omega_n$
    or $s\omega_1$ (which we understand to be meant modulo
    weights annihilated by all $h_i$ ($i\in I_\bullet$)).
    
    This follows from the fact that $s\omega_n-r\omega_1+\alpha_i$ ($i=2,\dots,n$) lies outside the support of
    $L(s\omega_n)$ and the same for $s\omega_1-r\omega_2+\alpha_i$ ($i=1,3,\dots,n$), so hence $E_iv$ for
    $i\in I_\bullet,v\in V'$ is zero.

    Consequently, $L(b)|_{\uq_\bullet}$ contains
    $\dim(V')$-many simple $\uq_\bullet$-submodules
    of highest weight $\tilde{b}$.
    From Remark~\ref{rmk-m-branching} we know that there
    should be exactly $s+1$-many.
\end{proof}
\begin{lemma}\label{lem-sl2-triple}
    $E_J,F_{\overline{J}},K_J$ acts as a $U_q(\mathfrak{sl}(2))$-triple on $V'$, i.e. as endomorphisms of $V'$ we have
    \begin{align*}
      K_J E_J K_J^{-1} &= q^2 E_J\\
      K_J F_{\overline{J}} K_J^{-1} &= q^{-2} F_{\overline{J}}\\
      \comm{E_J}{F_{\overline{J}}} &= \frac{K_J-K_J^{-1}}{q-q^{-1}},
    \end{align*}
    see \cite[(4.4.5)]{Jan96}.
\end{lemma}

\begin{proof}
See Appendix \ref{ap:lem-sl2-triple}.
\end{proof}

\begin{corollary}
    If $\eta\in L(b)$ is the highest-weight vector, then
    \[
        E_JF_{\overline{J}}^r\eta = [r]_q[s+1-r]_qF_{\overline{J}}^{r-1}\eta.
    \]
    In particular $F_{\overline{J}}^r\eta$ for $r=0,\dots,s$ form a basis of
    $V'$.
\end{corollary}
\begin{proof}
    The equation is a standard identity from the representation
    theory of $U_q(\mathfrak{sl}(2))$ and follows by induction.
    It shows in particular that the elements
    $F_{\overline{J}}^r\eta$ ($r=0,\dots,s$) are all nonzero.
    Since they have different weights, they are linearly independent, and since there are $s+1$ of them, they therefore
    form a basis.
\end{proof}

\begin{lemma}\label{lem-B-action-V'}
    $B_{\overline{J}}$ acts as follows on $V'$:
    \[
        B_{\overline{J}} F_{\overline{J}}^r\eta
        = F_{\overline{J}}^{r+1}\eta +
        b_r F_{\overline{J}}^r\eta +
        c_r F_{\overline{J}}^{r-1}\eta
    \]
    where
    \begin{align*}
        b_r &= \begin{cases}
            0 & \mathsf{BII}_n\\
            -c\qty(q^{r-1-s}[r]_q + q^{r+3-2n}[s-r]_q) & \mathsf{CII}_{n,1}
        \end{cases}\\
        c_r &= \begin{cases}
            -c\frac{(-1)^n}{(q-q^{-1})^2}
        q^{3-2n} [2]_q \qty(1-q^{2r})\qty(1-q^{2(r-s-1)})
        & \mathsf{BII}_n\\
        -\frac{c^2}{(q-q^{-1})^2}q^{2-2n}\qty(1-q^{2r})\qty(1-q^{2(r-s-1)}) & \mathsf{CII}_{n-1}
        \end{cases}
    \end{align*}
\end{lemma}

\begin{proof}
See Appendix \ref{ap:lem-sl2-triple}
\end{proof}

\begin{lemma}\label{lem-BII-CII-hwv}
    Let $v\in V(\gamma)\subset L(b)$ be a $\uq_\bullet$-highest-weight vector.
    Then $v\in V'$ and there is $i\in\set{0,\dots,s}$ such that
    \[
        v = \sum_{r=0}^{s}
        a_{s-r} F_{\overline{J}}^r\eta
    \]
    where
    \[
        a_r = a^{-r} (q^{2s};q^{-2})_r
        K_i(q^{2r};-Cq^{2s};q^{-2})
    \]
    ($q$-Kravchuk polynomial, see \cite[\S 14.15]{kls}). 
    All vectors are given in a similar way but for a different value
    of $i$ belong to other simple $\uqb$-modules.
    The constants $a,C$ are chosen as follows:
    \begin{align*}
        a &= \begin{cases}
            \pm\frac{q-q^{-1}}{\sqrt{(-1)^{n}[2]_qc}}q^{s+n-\frac{3}{2}}
            &\mathsf{BII}_n\\
            \mp \frac{q-q^{-1}}{c}q^{n+s-1\pm(2-n)}
            &\mathsf{CII}_{n,1}
        \end{cases}\\
        C &= \begin{cases}
            -1 & \mathsf{BII}_n\\
            -q^{\pm(4-2n)}& \mathsf{CII}_{n,1}
        \end{cases}
    \end{align*}
\end{lemma}
\begin{proof}
    Note that $c_{s+1-r}=\frac{Cq^{2s}}{a^2}\qty(1-q^{-2r})\qty(1-q^{-2r+2s+2})$ in
    every case.

    $V(\gamma)$ is a simple $\uq_\bullet$-module, and as such
    its intersection with $V'$ is one-dimensional (as it has
    only one ray of highest-weight vectors).
    Consequently, any such highest-weight vector $v$ has to
    be an eigenvector of $B_{\overline{J}}$.
    
    Assume that $B_{\overline{J}}v=\lambda v$.
    Expanded in terms of the $(F_{\overline{J}}^r\eta)_{0\le r\le s}$, this
    yields the following recursion relation:
    \[
        \lambda a_r = a_{r+1} + b_{s-r}a_r + c_{s+1-r}a_{r-1}
    \]
    (where $a_{-1}=a_{s+1}=0$).
    This shows that $a_r=a_r(\lambda)$ is a polynomial in $\lambda$ of degree
    $r$.
    In particular, all eigenspaces are one-dimensional and
    $\lambda$ is a root of the polynomial
    $(\lambda-b_0) a_s(\lambda) - c_1a_{s-1}(\lambda)$.

    The relation between $c_{s+1-r}$ and $C,a$ suggests that
    the polynomials are normalised dual $q$-Kravchuk polynomials (see \cite[\S 14.17]{kls}).
    In particular, supposing
    \[
        a_r(\lambda) = a^{-r}
        (q^{2s};q^{-2})_r K_r(a\lambda+b;C,s|q^{-2}),
    \]
    we can match our recursion relation provided that
    $a,C$ are as given in the claim and
    \[
        b = \begin{cases}
            0 & \mathsf{BII}_n\\
            q^s\qty(q^{\mp s}-q^{\pm(s+4-2n)}) & \mathsf{CII}_{n,1}.
        \end{cases}
    \]
    This shows that $a\lambda + b = q^{2i}+Cq^{2s-2i}$ for an
    $i=0,\dots,s$,
    and that hence
    \[
    a_r = a^{-r}(q^{2s};q^{-2})_r
    K_r(q^{2i}+Cq^{2s-2i};C,s|q^{-2})
    = a^{-r}(q^{2s};q^{-2})_r
    K_i(q^{-2r};-Cq^{2s},s;q^{-2})
    \]
    ($q$-Krav\v{c}uk polynomials, see \cite[\S 14.15]{kls}).

    Since we have shown that on $V'$, the operator
    $B_{\overline{J}}$ is diagonalisable, any subspace of
    $V'$ that is invariant under $B_{\overline{J}}$ is a direct
    sum of eigenspaces.
    In particular, if we can (as $\uqb$-modules) decompose
    $L(b)=V(\gamma)\oplus V(\gamma)^\perp$, then
    $V(\gamma)\cap V'=kv$ and $V(\gamma)^\perp\cap V'$
    are both invariant under $B_{\overline{J}}$.
    We conclude that $V(\gamma)^\perp\cap V'$ consists of
    those eigenspaces for other values of $i$.
\end{proof}

\begin{conjecture}\label{conj-eigenvalue}
    The number $i$ as in Lemma~\ref{lem-BII-CII-hwv} equals
    $0$ or $n$, in such a way that
    $V(\gamma)$ is generated by a vector $v\in V'$ satisfying
    \[
        B_{\overline{J}}v=\lambda v
    \]
    for
    \[
        \lambda = \begin{cases}
            \pm\sqrt{(-1)^n[2]_q}q^{\frac{3}{2}-n}[s]_q &\mathsf{BII}_n\\
           2cq^{1-n}[s+2-n]_q & \mathsf{CII}_{n,1}.
        \end{cases}
    \]
\end{conjecture}

\subsubsection{MSF For the Bottom Element}
Even without Conjecture~\ref{conj-eigenvalue}, we can still proceed a bit further in full generality.

\begin{lemma}
    Write $v_i$ for the $i$-th eigenvector from
    Lemma~\ref{lem-BII-CII-hwv}.
    
    Define $p:L(b)\to V(\gamma)$ as follows:
    we decompose $L(b)$ into simple $\uq_\bullet$-modules, map all modules whose highest weight
    is not $s\omega_n$ (or $s\omega_1$, respectively) to 0,
    and which operates as follows on all other modules, i.e.
    $v\in \uq_\bullet V'$: $v$ can be written as
    \[
        v = \sum_{r=0}^s x_r F_{\overline{J}}^r\eta
    \]
    for unique $x_r\in U_\bullet^-$.
    We then define
    \[
        p(v) := \sum_{r=0}^s x_r v_i
        \frac{(a^{-1}Cq^{2s})^{r-s}}{(q^{-2};q^{-2})_{s-r}}
        K_i(q^{2s-2r};-Cq^{2s},s;q^{-2}).
    \]
    Then $p$ is a $\uqb$-intertwiner.
\end{lemma}
\begin{proof}
    Note that
    \[
        p(v_j) = v_i \sum_{r=0}^s (Cq^{2s})^{r-s}
        \frac{(q^{2s};q^{-2})_{s-r}}{(q^{-2};q^{-2})_{s-r}} K_iK_j(q^{2s-2r};-Cq^{2s},s;q^{-2})
        ,
    \]
    which is a non-zero multiple of $v_i$ if
    $i=j$, and 0 otherwise, see \cite[(14.15.2)]{kls}.
    This shows that it is a $\uq_\bullet$ intertwiner.
    Since the splitting of $L(b)$ into simple
    $\uq_\bullet$-modules is subordinate to the one into
    simple $\uqb$-modules, we obtain it is also a
    $\uqb$-intertwiner.
\end{proof}

\begin{notation}
    Let $v=xv_i\in V(\gamma)$ for $x\in \uq_{\bullet,-\nu}$, then we say that $b-\nu$ is its \emph{weight}.
    These are representatives for the equivalence classes
    modulo elements annihilated by $Y^\Theta=\{y\in Y: \Theta(h)=h\}$.

    This gives us a definition of weight spaces and of
    formal characters:
    \[
        \ch_\gamma := \sum_{\nu\in\Z\mathcal{R}_\bullet}
        \dim(V(\gamma)_{b-\nu}) e^{b-\nu}.
    \]
\end{notation}

\begin{corollary}
    Let $v\in V(\gamma)_\mu$, then
    \[
        \Phi^b_\gamma(K_h)v = \frac{p(K_hv)}{p(v)} = \frac{\sum_{r=0}^s q^{\langle h,\mu-r\omega_k\rangle}(Cq^{2s})^{r-s}
        \frac{(q^{2s};q^{-2})_{s-r}}{(q^{-2};q^{-2})_{s-r}} K_i(q^{2s-2r};-Cq^{2s},s;q^{-2})^2}{
        \sum_{r=0}^s
        (Cq^{2s})^{r-s} \frac{(q^{2s};q^{-2})_{s-r}}{(q^{-2};q^{-2})_{s-r}} K_i(q^{2s-2r};-Cq^{2s},s;q^{-2})^2
        }
        v
    \]
    where $k=1,2$ depending on whether we are considering
    the $\mathsf{BII}_n$ or $\mathsf{CII}_{n,1}$ case.
    The denominator is a scalar given by the orthogonality relation of the $q$-Kravchuk polynomials.
    Assuming Conjecture~\ref{conj-eigenvalue}, this reduces to
    \[
        \frac{q^{\langle h,\mu-s\omega_k\rangle}
        \qty(C^{-1}q^{\langle h,\omega_k\rangle};q^{-2})_s}{\qty(C^{-1};q^{-2})_s}v
        = \frac{q^{\langle h,\mu\rangle}\qty(Cq^{-\langle h,\omega_k\rangle};q^2)_s}{(C;q^2)_s}v
    \]
    by \cite[(1.11.1),(1.8.7)]{kls}.
\end{corollary}

\begin{theorem}\label{thm-example-BII-CII}
    Let $s\in\N_0$
    \[
        (l_1,l_2,l_3,l_4) = \begin{cases}
            \qty(\frac{2n-1}{2}, \frac{2n-1}{2}+s, 0, 0) & 
            \mathsf{BII}_n\\
            \qty(\frac{2n-1}{2},\frac{3}{2}+s,n-2,0) &\mathsf{CII}_{n,1},
        \end{cases}
    \]
    i.e. the labelling from Proposition~\ref{prop-zsf} with $l_2$ increased by $s$.
    For every $m\in\N_0$, there is a non-zero constant $D$
    such that
    \[
        DP_m = 
        \begin{cases}
            -\rho\triangleright \Res\qty(\frac{\Phi^{s\omega_n + m\omega_1}_\gamma}{\Phi^{s\omega_n}_\gamma}) &
            \mathsf{BII}_n\\
            -\rho\triangleright
            \Res\qty(\frac{\Phi^{s\omega_1+m\omega_2}_\gamma}{\Phi^{s\omega_1}_\gamma}) & \mathsf{CII}_{n,1}
        \end{cases}
    \]
    where $P_m$ is the $m$-th symmetric
    Macdonald polynomial for the Macdonald data
    \begin{align*}
        S=S'&=S(2\Sigma^0)\cup 2S(2\Sigma^0)^\vee\\
        R=R'&=2\Sigma^0\\
        L=L'&=2P(\Sigma)=\Z\omega_k
    \end{align*}
    (with $k=1,2$ for the $\mathsf{BII}_n,\mathsf{CII}_{n,1}$ case) of type $(C_1^\vee, C_1)$,
    the labelling $l$, and base $q^2$.
\end{theorem}
\begin{proof}
    Since $\Phi^b_\gamma$ is a basis of the left $E^\epsilon$-module $E^\gamma$, the fraction
    $\frac{\Phi^{\mu}_\gamma}{\Phi^b_\gamma}$ (for $\mu\in X^+(\gamma)$) can be understood to be the unique element $\phi\in E^\epsilon$ such that $\phi\Phi^b_\gamma=\Phi^\mu_\gamma$, which can
    then be restricted to $\uq^0$ and $-\rho$-shifted.

    Note that $\omega_1=\alpha_1+\cdots+\alpha_n$ ($\mathsf{BII}_n$) and that
    $\omega_2=\alpha_1+2\alpha_2+\cdots+2\alpha_{n-1}+\alpha_n$ ($\mathsf{CII}_{n,1}$), so that
    \[
        \Res(\Phi^b_\gamma)v = \frac{e^\mu \qty(Ce^{-\omega_k};q^2)_s}{(C;q^2)_s}v
        = \qty(\frac{e^\mu}{(C;q^2)_s} + \lot)v
    \]
    for $v\in V(\gamma)_\mu$, so that the leading exponent of $\Res(\Phi^b_\gamma)$ is
    $e^b$.
    
    We conclude that for $\phi\in E^\epsilon$, the leading exponent of $\Res(\phi)$ is $\lambda$ iff
    the leading exponent of $\Res(\phi\Phi^b_\gamma)=e^{b+\lambda}$.
    This shows that there is $D'\in k^\times$ such that
    \[
        \Res\qty(\frac{\Phi^{b+\lambda}_\gamma}{\Phi^b_\gamma}) = D' e^\lambda + \lot,
    \]
    where $\lot$ refers to lower terms in $X$ with respect to the cone $Q^+$.
    Since the right-hand side is a restriction of a zonal spherical function, we know that
    all terms of its restriction are contained in $2L$ (whose generator $\omega_k$ lies in
    $Q^+$, so the order it inherits from $X$ is the usual order on $\Z$).
    Moreover, there is $D=q^{-\langle\rho,\lambda\rangle}D'$ such that
    \[
        -\rho\triangleright\Res\qty(\frac{\Phi^{b+\lambda}_\gamma}{\Phi^b_\gamma})
        = Dm_\lambda + \lot
    \]
    (where $\lot$ refers to lower terms in $2L_+$).

    Moreover, if the $-\rho$-shifts of restrictions of zonal spherical functions are 
    orthogonal with respect to the weight $\triangledown$, then the family
    \[
        \mqty(-\rho\triangleright\Res\qty(\frac{\Phi^{b+\lambda}_\gamma}{\Phi^b_\gamma}))_{\lambda\in\N_0\omega_k}
    \]
    is orthogonal with respect to the weight $\triangledown'=\triangledown\qty(-\rho\triangleright\Res(\Xi_\gamma(\Phi^b_\gamma,\Phi^b_\gamma)))$.
    This follows from Lemma~\ref{lem-other-weight-functions}.

    We have
    \begin{align*}
        \qty(-\rho\triangleright\Res(\Xi_\gamma(\Phi^b_\gamma,\Phi^b_\gamma)))(h)
        &= \tr_{V(\gamma)}(\Phi^b_\gamma(K_{h-\rho})\Phi^b_\gamma(K_{-h-\rho}))\\
        &= \sum_{\mu} \frac{\dim(V(\gamma)_\mu)}{(C;q^2)_s^2}
        q^{\langle -2\rho,\mu\rangle}\qty(Cq^{-\langle h-\rho,\omega_k\rangle},Cq^{-\langle -h-\rho,\omega_k\rangle};q^2)_s\\
        &= \frac{\operatorname{ch}_\gamma(-2\rho)}{(C;q^2)_s^2}
        \qty(Cq^{-\langle h-\rho,\omega_k\rangle},Cq^{-\langle -h-\rho,\omega_k\rangle};q^2)_s.
    \end{align*}
    For $\mathsf{BII}_n$ we have $\omega_1=\alpha_1+\cdots+\alpha_n$, so that
    $\langle\rho,\omega_1\rangle = \epsilon_1+\cdots+\epsilon_n=2n-1$.
    For $\mathsf{CII}_{n,1}$ we have $\omega_2 = \alpha_1+2\alpha_2+\cdots+2\alpha_{n-1}+\alpha_n$,
    so that
    \[
        \langle\rho,\omega_2\rangle = \epsilon_1 + 2\epsilon_2+\cdots+2\epsilon_{n-1}+\epsilon_n
        = 2n-1,
    \]
    hence
    \[
        -\rho\triangleright\Res(\Xi_\gamma(\Phi^b_\gamma,\Phi^b_\gamma))
        = \frac{\operatorname{ch}_\gamma(-2\rho)}{(C;q^2)_s^2}
        \qty(Cq^{2n-1}e^{-\omega_k}, Cq^{2n-1}e^{\omega_k};q^2)_s.
    \]
    Recall that
    \[
        C = \begin{cases}
            -1 & \mathsf{BII}_n\\
            -q^{4-2n} & \mathsf{CII}_{n,1},
        \end{cases}
    \]
    so that our weight function becomes $\triangledown'=\frac{\operatorname{ch}_\gamma(-2\rho)}{(C;q^2)_s^2} \triangle^{\prime+}\overline{\triangle^{\prime+}}$ with
    \[
        \triangle^{\prime+} = \begin{cases}
            \frac{(e^{2\omega_1};q^2)_\infty}{\qty(q^{2n-1}e^{\omega_1},-q^{2s+2n-1}e^{\omega_1},qe^{\omega_1},-qe^{\omega_1};q^2)_\infty} & \mathsf{BII}_n\\
            \frac{(e^{2\omega_2};q^2)_\infty}{\qty(q^{2n-1}e^{\omega_2},-q^{2s+3}e^{\omega_2},q^{2n-3}e^{\omega_2},-qe^{\omega_2};q^2)_\infty} &\mathsf{CII}_{n,1}
        \end{cases}
    \]
    (where we used Proposition~\ref{prop-zsf} and \cite[\S5.1.28]{Mac03} to identify an appropriate
    weight for the zonal spherical functions).
    We conclude that $\triangledown'$ is proportional to the weight function for
    the Macdonald polynomials of type $(C_1^\vee, C_1)$ with the data and parameters claimed.

    From
    \[
    -\rho\triangleright\Res\qty(\frac{\Phi^{b+\lambda}_\gamma}{\Phi^b_\gamma})
        = Dm_\lambda + \lot
    \]
    we can thus conclude that
    \[
        -\rho\triangleright\Res\qty(\frac{\Phi^{b+\lambda}_\gamma}{\Phi^b_\gamma})
        = DP_\lambda.\qedhere
    \]
\end{proof}

\subsection{Single-Variable 2-Vector: $\mathsf{DII}$ with Spin Representation}\label{sec-singlevar2}
We now consider another example of rank one. But now, $\#\mathfrak{B}(\gamma)=2$. In the group case, this corresponds to $(Spin(2n),Spin(2n-1))$ for
$n\ge1$. This case is studied by van Horssen and van Pruijssen \cite{vHor}, they identified matrix spherical functions on the
associated with the fundamental spin-representation of
$Spin(2n -1)$ with non-symmetric Jacobi polynomials of type $\mathsf{BC}_1$.

Consequently, we would expect to find not symmetric but non-symmetric
Macdonald polynomials, which will indeed be the case. 

Note that for $n=1$, this corresponds to an Abelian group with
trivial subgroup, for $n=2$, this corresponds to the group case of type $\mathsf{A}_1$ (i.e. a diagram of type $\mathsf{A}_1\times\mathsf{A}_1$ with white nodes and
$\tau$ interchanging the two nodes), for $n=3$, this corresponds to
a Satake diagram of type $\mathsf{AII}_3$, and for $n\ge4$, this corresponds to
a Satake diagram of type $\mathsf{DII}_n$.

For the case $n=1$, we take $X,Y$ to be lattices of rank 1 that are perfectly
paired and $\uq$ to be the group algebra of $Y$. The algebra $\uqbs$ is
then the trivial algebra 1, which has only one simple module $\gamma$. Similarly, for every element $\mu\in X$, the
corresponding simple representation of $\uq$ is 1-dimensional and we have
\[
    K_h\cdot 1 = q^{\langle h,\mu\rangle}.
\]
Since there are no roots, we have $\rho=0$, so that
\[
    (-\rho)\triangleright\Res(\Phi^\mu_\gamma) = e^\mu,
\]
which is the $\mu$-th non-symmetric Askey--Wilson polynomial with
parameters $(1,q^{1/2},-1,-q^{1/2})$ (alternatively: non-symmetric
$q^2$-ultraspherical polynomial with $k=0$). This covers the case of $n=1$. The
cases $n=3$ and $n\ge4$ can be covered together, which will be done in the
following. The case $n=2$ will be covered later.

\subsubsection{Notation}
Let $\mathcal{R}$ be a root system of type $\mathsf{D}_n$, let $X$ be the weight
lattice and $Y$ the coroot lattice. Let $\alpha_1,\dots,\alpha_n$ be the
simple roots, $h_1,\dots,h_n$ the coroots, and $\omega_1,\dots,\omega_n$ the
fundamental weights. We pick $\epsilon_1=\cdots=\epsilon_n=1$, which implies
that $K_i = K_{h_i}$.

Let $I_\bullet=\set{2,\dots,n}$ and $\tau=\id_I$. The coideal subalgebra
$\uqbs$ is then generated by
\[
    E_2,\dots,E_n,F_2,\dots,F_n,K_2,\dots,K_n
\]
and
\[
    B_1 := F_1 + c\ad(E_2\cdots E_nE_{n-2}\cdots E_2)(E_1)K_1^{-1}.
\]
The restricted root system is again $\Sigma=\set{\pm\omega_1}$ so that
$2L=\Z\omega_1$. The corresponding Satake diagram is
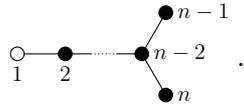
\[\dynkin[scale =1.8, labels={1,2,n-2,n-1,n}
, label directions={,,right,,}
] D{II}.\]

\subsubsection{Classical Branching Rules}
We describe the branching rules of $\mathfrak{g}=\mathfrak{so}(2n)$,
$\mathfrak{k}=\mathfrak{so}(2n-1)$, and $\mathfrak{m}=\mathfrak{so}(2n-2)$.
As explained in Section~\ref{sec-higher-spin-classical}, irreducible
representations of $\mathfrak{g},\mathfrak{k},\mathfrak{m}$ are
described by half-integers $a_1,\dots,a_n,b_1,\dots,b_{n-1},c_1,\dots,c_{n-1}$ such that
\[
    a_1\ge\cdots\ge a_{n-1}\ge \abs{a_n},\qquad
    b_1\ge\cdots\ge b_{n-1}\ge0,\qquad
    c_1\ge\cdots\ge c_{n-2}\ge\abs{c_{n-1}}
\]
and such that all the $a_i,b_i,c_i$ each only differ by integers. One
irreducible representation is contained in the other precisely once when
the numbers interlace and differ by integers. 

We are
interested in the representation described by
$b_1=\cdots=b_{n-1}=\frac{1}{2}$, i.e. the spin representation. It contains the two simple $\mathfrak{m}$-modules described by $c_1=\cdots=c_{n_2}=\frac{1}{2}$ and $c_{n-1}=\pm\frac{1}{2}$, and it is contained in the simple $\mathfrak{g}$-modules
described by $a_1\in\frac{1}{2}+\Z$, $a_2=\cdots=a_{n-1}=\frac{1}{2}$, and $a_n=\pm\frac{1}{2}$. In other words, if $\gamma$ is the corresponding
$\uqbs$-type, then $X^+(\gamma)=\set{\omega_{n-1},\omega_n}+\N_0\omega_1$.
This implies that $\mathfrak{B}(\gamma)=\set{\omega_{n-1},\omega_n}$.

Lastly, we note that both bottom elements restrict to the simple
$\mathfrak{k}$-module we are considering. Consequently, we pick
$V(\gamma):=L(\omega_n)|_{\uqbs}$ and equip it (and analogously $L(\omega_{n-1})$) with the basis $v_\mu$
($\mu\in W\omega_n$) described in \cite[\S5A.1]{Jan96}.

We also remark that $\Theta$ acts on $X$ as the reflection that negates
$a_1$ (in the notation given here) and leaves all other numbers intact.

\subsubsection{MSF for Bottom Elements}\label{sec-DII-msf}
Using the notation of Lemma \ref{lem-emsf-basis} we write $\Psi_1:=\Phi^{\omega_{n-1}}_\gamma$ and $\Psi_2:=\Phi^{\omega_n}_\gamma$.

\begin{lemma}
    We have
    \begin{align*}
        \Psi_1(K_h)v_\mu &= q^{\langle h,\Theta(\mu)\rangle}v_\mu\\
        \Psi_2(K_h)v_\mu &= q^{\langle h,\mu\rangle}v_\mu
    \end{align*}
    for $\mu\in W\omega_n,h\in Y$.
\end{lemma}
\begin{proof}
    For $\Psi_2$ we take $i=p=\id_{L(\omega_n)}$ and obtain the claimed
    equation from $\Psi_2(K_h)v_\mu=K_hv_\mu$.

    For $\Psi_1$ we define $j:V(\gamma)\to L(\omega_{n-1})$ and
    $p: L(\omega_{n-1})\to V(\gamma)$ as follows:
    \begin{align*}
        j: v_\mu &\mapsto \begin{cases}
            v_{\mu-\omega_1} & \mu-\omega_1\in W\omega_n\\
            (-q)^{2-n}cv_{\mu+\omega_1} & \mu+\omega_1\in W\omega_n
        \end{cases}\\
        p: v_\mu&\mapsto \begin{cases}
            (-q)^{n-2}c^{-1}v_{\mu-\omega_1} & \mu-\omega_1\in W\omega_{n-1}\\
            v_{\mu+\omega_1} & \mu+\omega_1\in W\omega_{n-1}
        \end{cases}.
    \end{align*}
    Note that these maps are well-defined since in the standard notation the
    Weyl orbits of $\omega_n,\omega_{n-1}$ can be described as $n$-tuples
    of $\set{\pm\frac{1}{2}}$ with an even/odd number of $-\frac{1}{2}$'s.
    Consequently, by flipping the first entry (i.e. adding or
    subtracting $\omega_1$) we can always make one weight from the other.
    Note that the maps $j,p$ are $U_{\bullet}$-intertwiners as $\omega_1$'s
    inner product with $h_2,\dots,h_n$ is zero, and as the weights are
    unchanged apart from adding $\pm\omega_1$.

    It therefore suffices to show that $p,j$ preserve the action of $B_1$. For this we note that $B_1v_\mu$ lies in the span of $v_{\mu-\alpha_1}$ and $v_{\mu + \omega_2}$. In particular, since both cannot
    be in the same Weyl orbit as $v_\mu$, it is only one term. Note that if
    $\langle h_1,\mu\rangle=-1$, we just have
    $B_1v_\mu = F_1v_\mu = v_{\mu-\alpha_1}$. On the other hand, if
    $\langle h_1,\mu\rangle=0$ and $\langle h_2,\mu\rangle=-1$, we know
    that $v_\mu$ can be reached from $v_{\omega_{n-1}-\omega_2}$ or $v_{\omega_n-\omega_2}$ by applying a polynomial in $F_3,\dots,F_n$. Each such polynomial commutes with $B_1$, so that it suffices to
    compute $B_1v_{\omega_n-\omega_2}$ and $B_1v_{\omega_{n-1}-\omega_2}$,
    which are multiples of $v_{\omega_n},v_{\omega_{n-1}}$.
    We have
    \begin{align*}
        B_1v_{\mu-\omega_2}
        &= c\ad(E_2\cdots E_nE_{n-2}\cdots E_2)(E_1)K_1^{-1}v_{\mu-\omega_2}\\
        &= c\ad(E_2\cdots E_nE_{n-2}\cdots E_2)(E_1)v_{\mu-\omega_2}.
    \end{align*}
    for $\mu=\omega_{n-1}$ or $\omega_n$. Note that
    $\im(E_i)\cap kv_{\omega_{n-1}}\ne0$ only for $i=n-1$, similarly for
    $\omega_n$ (then $i=n$). Consequently, in $c\ad(E_2\cdots E_nE_{n-2}\cdots E_2)(E_1)v_{\mu-\omega_2}$ all terms vanish that don't
    have $E_{n-1}$ (resp. $E_n$) as their left-most term. Thus,
    \begin{align*}
        &B_1v_{\omega_{n-1}-\omega_2}=\\
        & (-1)^{n-2} c K_2\cdots K_{n-2}K_n\ad(E_{n-1}\cdots E_2)(E_1)
        K_n^{-1}E_nK_{n-2}^{-1}E_{n-2}\cdots K_2^{-1}E_2\eta_{\omega_{n-1}-\omega_2}\\
        &= (-1)^{n-2}q^{2-n} c \ad(E_{n-1}\cdots E_2)(E_1)
        v_{\omega_n-\omega_1}\\
        &= (-q)^{2-n} c v_{\omega_{n-1}}.
    \end{align*}
    We conclude
    \[
        B_1v_\mu = \begin{cases}
            v_{\mu-\alpha_1} & \mu-\alpha_1\in W\mu\\
            (-q)^{2-n}c v_{\mu+\omega_2} & \mu+\omega_2\in W\mu\\
            0 & \text{otherwise}
        \end{cases}.
    \]
    Then we have
    \[
        j(B_1v_\mu) =
        B_1j(v_\mu) = 
        \begin{cases}
            (-q)^{2-n}cv_{\mu+\omega_2-\omega_1} & \mu -\alpha_1\in W\mu\\
            (-q)^{2-n}cv_{\mu+\omega_2-\omega_1} & \mu + \omega_2\in W\mu\\
            0 & \text{otherwise}
        \end{cases}.
    \]
    We obtain similar relations for $p$. Furthermore, $p\circ j=\id_{V(\gamma)}$. Note that the definition of $j$ can be
    summarised as $j(L(\omega_{n-1})_\mu)\subset L(\omega_n)_{\Theta(\mu)}$. We conclude
    \[
        \Psi_1(K_h)v_\mu = q^{\langle h,\Theta(\mu)\rangle}v_\mu.\qedhere
    \]
\end{proof}

\begin{corollary}\label{cor-spin2d-bottom-weight}
    There is a nonzero constant $C\in k$ such that
    \[
        M = C\mqty(q^{n-1}+q^{1-n} & m_{\omega_1}\\m_{\omega_1} &
        q^{n-1}+q^{1-n}).
    \]
\end{corollary}
\begin{proof}
    We have
    \begin{align*}
        \Psi_1(K_{h-\rho})\Psi_1(K_{-h-\rho})v_\mu &=
        q^{\langle h-\rho-h-\rho,\Theta(\mu)\rangle} v_\mu
        = q^{\langle -2\rho,\Theta(\mu)\rangle} v_\mu\\
        \Psi_1(K_{h-\rho})\Psi_2(K_{-h-\rho})v_\mu &=
        q^{\langle h-\rho,\Theta(\mu)\rangle+\langle-h-\rho,\mu\rangle}
        v_\mu= q^{-\langle\rho,\mu+\Theta(\mu)\rangle-\langle h,\mu-\Theta(\mu)\rangle}v_\mu\\
        \Psi_2(K_{h-\rho})\Psi_1(K_{-h-\rho})v_\mu &=
        q^{\langle h-\rho,\mu\rangle + \langle -h-\rho,\Theta(\mu)\rangle} v_\mu= q^{-\langle\rho,\mu+\Theta(\mu)\rangle+\langle h,\mu-\Theta(\mu)\rangle}v_\mu\\
        \Psi_2(K_{h-\rho})\Psi_2(K_{-h-\rho})v_\mu &=
        q^{\langle h-\rho-h-\rho,\mu\rangle}v_\mu =
        q^{\langle -2\rho,\mu\rangle} v_\mu.
    \end{align*}
    Note that $D_n$ is simply laced, so that $2\rho = \sum_{h\in\mathcal{R}^{\vee+}} h$. In particular, $\langle 2\rho,\omega_1\rangle=2(n-1)$,
    so that
    \[
        \langle 2\rho,\mu\rangle = \pm(n-1) + \langle\rho,\mu+\Theta(\mu)\rangle
        = \pm 2(n-1) + \langle2\rho,\Theta(\mu)\rangle
    \]
    based on whether $\mu-\Theta(\mu)=\pm\omega_1$. We conclude that
    \[
        M_{1,1} = \ch_\gamma(-2\Theta(\rho)) = M_{2,2}=\ch_\gamma(-2\rho)
        = \qty(q^{n-1}+q^{1-n})\sum_{\substack{\mu\in W\omega_n\\\mu-\Theta(\mu)=\omega_1}}q^{-\langle\rho,\mu+\Theta(\mu)\rangle},
    \]
    and
    \[
        M_{1,2} = M_{2,1}=  \qty(e^{-\omega_1}+e^{\omega_1})
        \sum_{\substack{\mu\in W\omega_n\\\mu-\Theta(\mu)=\omega_1}}
        q^{-\langle\rho,\mu+\Theta(\mu)\rangle}.
    \]
\end{proof}

Note that the weight $\triangledown$ considered here is the symmetric Macdonald weight
for the root system $S=\set{\pm 2\omega_1 + 2r\where r\in\Z}$ with
$\dim(\mathfrak{g}_\alpha)=2n-2$, i.e.
\[
    \triangledown = \prod_{\epsilon=\pm1}\frac{(e^{2\epsilon\omega_1};q^2)_\infty}{(q^{2n-2}e^{2\epsilon\omega_1};q^2)_\infty},
\]
which is obtained from \cite[\S6.3.1]{Mac03} by doubling powers of $q$ and setting $k=n-1$ (and by replacing $x$ by $e^{\omega_1}$). Consequently,
we have
\[
    \Delta^0 = q^{1-n}\frac{e^{\omega_1}-e^{-\omega_1}}{q^{1-n}e^{\omega_1}-q^{n-1}e^{-\omega_1}}.
\]
For the basis $v_1=1, v_2=q^{n-1}e^{\omega_1}$ we then obtain the following
weight matrix
\begin{equation}\label{eq-spin2d-macdonald-weight}
    m_{i,j} = \frac{1}{2}\sum_{w\in W_\Sigma}
    w \frac{v_1\overline{v_2}}{\Delta^0}\qquad
    m = \frac{q^{1-n}}{2}
    \mqty(q^{n-1}+q^{1-n} & m_{\omega_1}\\m_{\omega_1} & q^{n-1}+q^{1-n})
    \mqty(1 & 0\\0 & q^{2n-2}).
\end{equation}

\begin{lemma}\label{lem-DII-triangular}
    Define $C: k[2L]\to k[2L]$ as a $k[2L]^{W_\Sigma}$-linear
    map mapping
    \begin{align*}
        v_1\mapsto 2q^{n-1}\sum_{\substack{\mu\in W\omega_n\\\mu-\Theta(\mu)=\omega_1}}
        q^{-\langle\rho,\mu+\Theta(\mu)\rangle} v_1\\
        v_2\mapsto 2q^{1-n}\sum_{\substack{\mu\in W\omega_n\\\mu-\Theta(\mu)=\omega_1}}
        q^{-\langle\rho,\mu+\Theta(\mu)\rangle}v_2.
    \end{align*}
    Then $C$ is triangular with respect to the Macdonald
    ordering on $(e^\mu)_{\mu\in 2L}$.
\end{lemma}
\begin{proof}
    See Appendix~\ref{sec-DII-proofs}.
\end{proof}

\begin{lemma}\label{lem-DII-monotonic}
    Define $t: 2L\to X^+(\gamma)$ by
    \[
        -m\omega_1 \mapsto m\omega_1 + \omega_{n-1},\qquad
        (m+1)\omega_1\mapsto m\omega_1 + \omega_n.
    \]
    Then $t$ is a bijection and $t^{-1}$ is monotonic
    with respect to the Macdonald ordering on $2L$ and
    the dominance ordering on $X^+(\gamma)$.
\end{lemma}
\begin{proof}
    See Appendix~\ref{sec-DII-proofs}.
\end{proof}

\begin{theorem}\label{thm-DII-example}
    Let $\Gamma: k[2L]\to k[X]\otimes\operatorname{span}\set{e_{\omega_{n-1}},e_{\omega_n}}$
    be given by $\Gamma(v_1):= e_{\omega_{n-1}}$ and $\Gamma(v_2):= e_{\omega_n}$.
    
    Then, for every $m\in\Z$ there is a
    nonzero constant $C_m$ such that
    \[
        \Gamma(E_m) = C_m\qty(-\rho\triangleright\Res(\Phi^{t(m)}_\gamma)),
    \]
    where $E_m$ is the $m$-th non-symmetric Macdonald polynomial (see \cite[\S6.2]{Mac03} for the
    root system $\set{\pm2\omega_1}$ of type $A_1$ with parameter $k=n-1$ and base $q^2$.
\end{theorem}
\begin{proof}
    Note that $C$ is triangular by Lemma~\ref{lem-DII-triangular},
    $\Gamma^{-1}$ is monotonic with respect to $t^{-1}$ by
    construction,
    and
    $t^{-1}$ is monotonic by Lemma~\ref{lem-DII-monotonic}.
    
    Furthermore, we see from \eqref{eq-spin2d-macdonald-weight} and
    Corollary~\ref{cor-spin2d-bottom-weight} that 
    \[
        \forall b,b'\in\mathfrak{B}(\gamma):\quad
        \frac{1}{\#W_\Sigma}\sum_{w\in W_\Sigma}
        w\frac{\Gamma^{-1}(e_b)\overline{C(\Gamma^{-1}(e_{b'}))}}{\Delta^0}
        = M_{b,b'}.
    \]
    With Corollary~\ref{cor-Intermediate-macdonald-polys} and noting that
    the Intermediate Macdonald polynomials for $J=\emptyset$ are just the
    non-symmetric Macdonald polynomials, we conclude the claim.
\end{proof}

\subsubsection{Case $n=2$}
The $\mathsf{A}_2$ group case is considered in generality in \cite{Ald17}.
This particular case can be found in loc.cit, \S4.1.1.

We consider the quantum group
$\uq=U_q(\mathfrak{sl}(2))^{\otimes2}$ with
$\uqbs$ generated by
\[
    B_1 = F_1 - cE_2K_1^{-1},\qquad
    B_2 = F_2 - cE_1K_2^{-1},\qquad
    K_{h_1-h_2}^{\pm1}.
\]
Classically, this corresponds to the Lie algebras $\mathfrak{g}=\mathfrak{sl}(2)\oplus\mathfrak{sl}(2)$ with $\mathfrak{k}$ the diagonal
algebra. The branching rules correspond to the Clebsch--Gordan rules
for $\mathfrak{sl}(2)$: the simple $\mathfrak{g}$-module $L(m\omega_1+n\omega_2)$ contains the simple modules $L(\ell\omega)$ of
$\mathfrak{k}$ for
\[
    \abs{m-n}\le \ell\le m+n,
\]
where $m+n+\ell$ is an even integer. Consequently, both $L(\omega_1)$ and
$L(\omega_2)$ are irreducible when restricted to $\mathfrak{k}$ and give
us the (same) spin module described by $\ell=1$. This $\mathfrak{k}$-module
is then contained in every simple $\mathfrak{g}$-module $L(m\omega_1+n\omega_2)$ satisfying $\abs{m-n}=1$. Consequently, we have
$2L = \Z(\omega_1+\omega_2)$ and $\mathfrak{B}(\gamma)=\set{\omega_1,\omega_2}$. Moreover, we have $\epsilon_1=\epsilon_2=1$, so that
\[
    2\rho = h_1+h_2.
\]

We take $V(\gamma)=L(\omega_1)$. In the basis $v_{\omega_1},v_{-\omega_1}$,
we have
\[
    \Phi^{\omega_1}_\gamma(K_h) =
    \mqty(q^{\langle h,\omega_1\rangle} & 0\\0 & q^{-\langle h,\omega_1\rangle}),\qquad
    \Phi^{\omega_2}_\gamma(K_h) =
    \mqty(q^{-\langle h,\omega_2\rangle\rangle} & 0\\0 & q^{\langle h,\omega_2\rangle}).
\]
We conclude that
\[
    M = \mqty(q+q^{-1} & m_{\omega_1}\\m_{\omega_1} & q+q^{-1})
\]
where we ordered the bottom elements as $\omega_1,\omega_2$.
This matches the findings of
\cite{Ald17}.
The rest of the arguments
from Section~\ref{sec-DII-msf} also work here, so that we also obtain non-symmetric
Macdonald polynomials for $k=n-1=1$.

\subsection{Multi-Variable 3-Vector: $\mathsf{AI}_2$, $\mathsf{A}_2$-group case, $\mathsf{AII}_5$}\label{sec:multivar}
We now consider three examples of rank 2: the three examples from
\cite{steinbergVariation} which are the three examples from \cite{shimenoA2}.
As is pointed out in loc.cit., Remark~2.3, there is one more Satake diagram whose restricted root system is of type $\mathsf{A}_2$: $\mathsf{EIV}$, but there doesn't seem to be
a suitable $\uqb$-type that produces a similar result as we obtain here.

Moreover, the matrix spherical functions in the $\mathsf{A}_2$-group case has been studied in \cite{Ald17} from the viewpoint of matrix valued orthogonal polynomials, as introduced in Definition \ref{def:matrixtovect}. There the assosiated matrix valued orthogonal polynomials are identified with matrix-valued analogs of Askey-Wilson polynomials.

Rank 2 (and in particular $\Sigma$ of type $\mathsf{A}_2$)
is the simplest case for which
there exist non-trivial examples of Intermediate Macdonald polynomials:
the ones from \cite[Section~6.2]{Sch23}. 
These Intermediate Macdonald polynomials turn out to describe
exactly the MSF we consider in this section.
However, as per
usual with the examples from \cite{Sch23}, the author chose exactly the
wrong conventions when initially working them out.

\subsubsection{Notation}
Let $(I,I_\bullet,\tau)$ correspond to the following Satake diagrams:
$\mathsf{AI}_2$, the group case for $\mathsf{A}_2$, and $\mathsf{AII}_5$. Their respective Satake diagrams are
\[
\dynkin[scale=1.8, labels={1,2}]A{oo},\qquad
\begin{tikzpicture}[baseline]
	\node[circle, draw, scale=0.5, label={[label distance=-1.5mm]below:\scriptsize1}] at (-0.5,0.15) {};
	\node[circle, draw, scale=0.5, label={[label distance=-1.5mm]below:\scriptsize2}] at (0.5,0.15) {};
	\draw[bend left, <->] (-0.5, 0.35) to (0.5, 0.35);
\end{tikzpicture},\qquad
\dynkin[scale=1.8, labels={1,2,3,4,5}]{A}{*o*o*}.
\]
Let $X$ be the weight lattice and $Y$ the coroot lattice.
Let $(\alpha_i)_{i\in I}$ be the simple roots, $(h_i)_{i\in I}$ the
corresponding coroots, and $(\omega_i)_{i\in I}$ the fundamental weights.
We pick $\epsilon_i=1$, so that $K_i=K_{h_i}$ ($i\in I$).

The coideal subalgebra $\uqb$ is generated by the following elements:
\begin{enumerate}
    \item For $\mathsf{AI}_2$:
    \[
        B_i := F_i - c_i E_i K_i^{-1}\qquad (i\in I).
    \]
    \item For the $\mathsf{A}_2$ group case:
    \[
        B_i := F_i - c_i E_{\tau(i)} K_i^{-1}\qquad (i\in I)
    \]
    \item For $\mathsf{AII}_5$: by $E_i,F_i,K_i,K_i^{-1}$ ($i=1,3,5$) and
    \[
        B_2 := F_2 - c_2 \ad(E_1E_3)(E_2)K_2^{-1},\qquad
        B_4 := F_4 - c_4 \ad(E_3E_5)(E_4)K_4^{-1}.
    \]
\end{enumerate}

Moreover, the involution $\Theta$ gives us the following root systems $\Sigma$ of type $\mathsf{A}_2$ and weight lattices $L$:
\begin{enumerate}
    \item For $\mathsf{AI}_2$, we have $\Theta=-1$, so that $\tilde{\alpha}=\alpha$
    for all $\alpha\in\mathcal{R}$. Consequently, we have
    $\Sigma=\mathcal{R}$ and $L=X$, generated by $\tilde{\omega}_1,\tilde{\omega}_2$.
    \item For the $\mathsf{A}_2$ group case, we have $\Theta=-\tau$, so that
    $\tilde{\alpha}_i = \frac{\alpha_i+\alpha_{i+2}}{2}$ for $i=1,2$. $\Sigma$ has
    $\tilde{\alpha}_1,\tilde{\alpha}_2$ as simple roots, and $L$ is generated by
    $\tilde{\omega}_1 = \frac{\omega_1+\omega_3}{2}$ and
    $\tilde{\omega}_2 = \frac{\omega_2+\omega_4}{2}$.
    \item In the common notation of $\alpha_i=e_i-e_{i+1}$, $\Theta$ maps
    exchanges $e_1,e_3,e_5$ with $-e_2,-e_4,-e_6$, respectively. Consequently,
    $\Sigma$ has
    \[
        \tilde{\alpha}_2=\frac{\alpha_1+2\alpha_2+\alpha_3}{2}=\frac{e_1+e_2-e_3-e_4}{2},\qquad
        \tilde{\alpha}_4=\frac{\alpha_3+2\alpha_4+\alpha_4}{2} = \frac{e_3+e_4-e_5-e_6}{2}
    \]
    as simple roots, and its weight lattice $L$ is spanned by
    $\tilde{\omega}_2 = \frac{\omega_2}{2}$ and $\tilde{\omega}_4=\frac{\omega_4}{2}$.
\end{enumerate}

\subsubsection{Classical Branching Rules}
We choose the quantum equivalents of the standard representations of $\mathfrak{k}$
in every case.
From \cite[\S6]{steinbergVariation} we know that in every case, $\mathfrak{B}(\gamma)$
has three elements, and for $\lambda$ one of two elements we get that the
$\mathfrak{g}$-representation of highest weight $\lambda$ stays irreducible when
restricted to $\mathfrak{k}$.
By Theorem~\ref{thm-specialisation}, this stays true on the quantum group level.

In particular, we find that
\begin{enumerate}
    \item for $\mathsf{AI}_2$, $\mathfrak{B}(\gamma)=\set{\omega_1,\omega_1+\omega_2,\omega_2}$, and we pick $V(\gamma) = L(\omega_1)|_{\uqb}$;
    \item for the $\mathsf{A}_2$ group case, $\mathfrak{B}(\gamma)=\set{\omega_1,\omega_2+\omega_3,\omega_4}$, and we pick $V(\gamma)=L(\omega_1)|_{\uqb}$;
    \item for $\mathsf{AII}_5$, $\mathfrak{B}(\gamma)=\set{\omega_1,\omega_3,\omega_5}$, and we pick $V(\gamma)=L(\omega_1)|_{\uqb}$.
\end{enumerate}

\subsubsection{Describing $V(\gamma)$}
In the previous section we picked concrete vector spaces for $V(\gamma)$.
Note that these are all representations whose highest weight is minuscule.
Consequently, we can make use of the description provided in \cite[\S5A]{Jan96},
which we will do (again using the letter $v$ instead of $x$ to denote basis elements).

\begin{proposition}\label{prop-shimeno-actions-on-Vgamma}
    \begin{enumerate}
        \item For $\mathsf{AI}_2$, we order the basis as $v_{\omega_1},v_{-\omega_1+\omega_2},v_{-\omega_2}$.
        Then, the generators $B_1,B_2$ act as the following matrices:
        \[
            B_1 = \mqty(0 & -c_1q & 0\\1 & 0 & 0\\0 & 0 & 0),\qquad
            B_2 = \mqty(0 & 0 & 0\\0 & 0 & -c_2q\\0 & 1 & 0).
        \]
        \item For the $\mathsf{A}_2$ group case, we order the basis the same way.
        Then,
        \[
            B_1v_{\omega_1} = v_{-\omega_2+\omega_1},
            B_2v_{-\omega_1+\omega_2} = v_{-\omega_2},
            B_3v_{-\omega_1+\omega_2} = -c_1v_{\omega_1},
            B_4v_{-\omega_2}=-c_2v_{-\omega_1+\omega_2}.
        \]
        Any other combination of generator and basis element vanishes.
        \item For $\mathsf{AII}_5$, we order the basis as
        $v_{\omega_1},v_{-\omega_i+\omega_{i+1}},v_{-\omega_5}$ ($i=1,\dots,4$).
        The vectors $v_{\omega_1},v_{-\omega_2+\omega_3},
        v_{-\omega_4+\omega_5}$ are $\uq_\bullet$-highest weight vectors of
        weights $\omega_1,\omega_3,\omega_5$, respectively, and $B_2,B_4$ act
        as follows:
        \begin{alignat*}{2}
            B_2v_{-\omega_1+\omega_2} &= v_{-\omega_2+\omega_3},\qquad
            &B_2v_{-\omega_3+\omega_4} &= -c_2q^{-1}v_{\omega_1},\\
            B_4v_{-\omega_3+\omega_4} &= v_{-\omega_4+\omega_5},
            &B_4v_{-\omega_5} &= -c_4q^{-1} v_{-\omega_2+\omega_3}
        \end{alignat*}
        with all other combinations being 0.
    \end{enumerate}
\end{proposition}

\subsubsection{MSF for Bottom Elements}
We enumerate our bottom as follows $\mathfrak{B}(\gamma)=\set{b_1,b_2,b_3}$,
with $b_1,b_2,b_3$ the elements given earlier (in the same order).
We already note that this is chosen in such a way that
the permutation $\sigma$ from \ref{sec:barinv} exchanges $b_1$ and $b_3$,
and leaves $b_2$.

Using the notation of Lemma \ref{lem-emsf-basis}, we write $\Psi_i:=\Phi^{b_i}_\gamma$ for $i=1,2,3$.

\begin{proposition}\label{prop-A2-msf}
    $\Res(\Psi_i)$ is always a diagonal matrix.
    Writing the diagonal entries in columns, we obtain the following:
    \begin{enumerate}
        \item for $\mathsf{AI}_2$:
        \[
            \mqty(e^{\omega_1}\\e^{-\omega_1+\omega_2}\\e^{-\omega_2}),
            \qquad
            \frac{1}{q+q^{-1}}\mqty(qe^{\alpha_2}+q^{-1}e^{-\alpha_2}\\
            qe^{\alpha_1+\alpha_2} + q^{-1}e^{-\alpha_1-\alpha_2}\\
            qe^{\alpha_1} + qe^{-\alpha_1}),
            \qquad
            \mqty(e^{-\omega_1}\\e^{\omega_1-\omega_2}\\e^{\omega_2});
        \]
        \item for the $\mathsf{A}_2$ group case:
        \[
            \mqty(e^{\omega_1}\\
            e^{-\omega_1+\omega_2}\\
            e^{-\omega_2}),\qquad
            \frac{1}{q+q^{-1}}
            \mqty(qe^{\omega_2-\omega_3+\omega_4} + q^{-1}e^{\omega_1-\omega_2-\omega_4}\\
            qe^{\omega_2+\omega_3}
            + q^{-1}e^{-\omega_1-\omega_4}\\
            qe^{\omega_1-\omega_2+\omega_3} + q^{-1}e^{-\omega_1-\omega_3+\omega_4}),\qquad
            \mqty(e^{-\omega_3}\\e^{\omega_3-\omega_4}\\
            e^{\omega_4});
        \]
        \item for $\mathsf{AII}_5$:
        \[
            \mqty(e^{\omega_1}\\
            e^{-\omega_1+\omega_2}\\
            e^{-\omega_2+\omega_3}\\
            e^{-\omega_3+\omega_4}\\
            e^{-\omega_4+\omega_5}\\
            e^{-\omega_5}),\qquad
            \frac{1}{q^2+q^{-2}}
            \mqty(q^2e^{\omega_1-\omega_2+\omega_4}
            + q^{-2}e^{\omega_1-\omega_4}\\
            q^2e^{-\omega_1+\omega_4} + q^{-2}e^{-\omega_1+\omega_2-\omega_4}\\
            q^2e^{\omega_3}+q^{-2}e^{-\omega_2+\omega_3-\omega_4}\\
            q^2e^{\omega_2-\omega_3+\omega_4} + q^{-2}e^{-\omega_3}\\
            q^2e^{\omega_2-\omega_4+\omega_5} + q^{-2}e^{-\omega_2+\omega_5}\\
            q^2e^{\omega_2-\omega_5} + q^{-2}e^{-\omega_2+\omega_4-\omega_5}),\qquad
            \mqty(e^{\omega_1-\omega_2}\\
            e^{-\omega_1}\\
            e^{\omega_3-\omega_4}\\
            e^{\omega_2-\omega_3}\\
            e^{\omega_5}\\
            e^{\omega_4-\omega_5}).
        \]
    \end{enumerate}
\end{proposition}
\begin{proof}
    For $b_1$, take $p=j=\id$.
    For $b_3$, we can take $j,p$ to both map
    $v_\mu\mapsto v_{-\tau(\mu)}$ (note that
    $-\tau(\mu) = -\tau_0\sigma(\mu)=w_0\sigma(\mu)$, so that
    the weights are indeed valid weights of the corresponding other module).
    That they indeed preserve the action of $\uqb$ can be checked using
    the explicit expressions from Proposition~\ref{prop-shimeno-actions-on-Vgamma}.

    For $b_2$, we note that for the $\mathsf{A}_2$ group case and
    $\mathsf{AII}_5$, the weight $b_2$ is minuscule, so we can use
    $(v_\mu)_{\mu\in Wb_2}$ as basis.
    For $\mathsf{AI}_2$, it is only pseudominuscule.
    In that case, we use the notation from
    \cite[\S5A.2]{Jan96} (i.e. $x_\alpha$ for $\alpha\in\mathcal{R}$ and
    $y_1,y_2$, corresponding to Jantzen's $h_{\alpha_1},h_{\alpha_2}$).

    For $\mathsf{AI}_2$, we pick the following maps
    \begin{align*}
        j\mqty(v_{\omega_1}\\
        v_{-\omega_1+\omega_2}\\
        v_{-\omega_2}) &:= \mqty(
        c_2q^2x_{\alpha_2} + x_{-\alpha_2}\\
         -c_1c_2q^3 x_{\alpha_1+\alpha_2} 
        + x_{-\alpha_1-\alpha_2}\\
        -c_2q(c_1q^2x_{\alpha_1} + x_{-\alpha_1}))\\
        p\mqty(x_{\alpha_1}\\
        x_{\alpha_1+\alpha_2}\\
        x_{\alpha_2}\\
        x_{-\alpha_1}\\
        x_{-\alpha_1-\alpha_2}\\
        x_{-\alpha_2}\\
        y_1\\
        y_2)
        &:=
        -\frac{1}{c_1c_2q^2(q+q^{-1})}
        \mqty(v_{-\omega_2}\\
        v_{-\omega_1+\omega_2}\\
        -c_1qv_{\omega_1}\\
        c_1v_{-\omega_2}\\
        -c_1c_2qv_{-\omega_1+\omega_2}\\
        -c_1c_2qv_{\omega_1}\\
        0\\0).
    \end{align*}
    For the $\mathsf{A}_2$ group case, we pick
    \begin{align*}
        j\mqty(v_{\omega_1}\\v_{-\omega_1+\omega_2}\\v_{-\omega_2})
        &:= \mqty(v_{\omega_1-\omega_2-\omega_4} + c_2qv_{\omega_2-\omega_3+\omega_4}\\
        v_{-\omega_1-\omega_4} - c_1c_2qv_{\omega_2+\omega_3}\\
        -c_2v_{-\omega_1-\omega_3+\omega_4}-c_1c_2qv_{\omega_1-\omega_2+\omega_3}),\\
        p\mqty(v_{\omega_2+\omega_2}\\
        v_{\omega_2-\omega_3+\omega_4}\\
        v_{\omega_1-\omega_2+\omega_3}\\
        v_{\omega_1-\omega_2-\omega_4}\\
        v_{-\omega_1-\omega_3+\omega_4}\\
        v_{-\omega_1-\omega_4})
        &:= \frac{1}{c_1c_2(q+q^{-1})}
        \mqty(-v_{-\omega_1+\omega_2}\\
        c_1v_{\omega_1}\\
        -v_{-\omega_2}\\
        c_1c_2qv_{\omega_1}\\
        -c_1qv_{-\omega_2}\\
        c_1c_2qv_{-\omega_1+\omega_2}).
    \end{align*}
    And for $\mathsf{AII}_5$, we pick
    \begin{align*}
        j\mqty(v_{\omega_1}\\v_{-\omega_2+\omega_3}\\v_{-\omega_4+\omega_5}) &=\mqty(v_{\omega_1-\omega_4} - c_4qv_{\omega_1-\omega_2+\omega_4}\\
        v_{-\omega_2+\omega_3-\omega_4} + c_2c_4v_{\omega_3}\\
        -c_4q^{-1}v_{-\omega_2+\omega_5} + c_2c_4v_{\omega_2-\omega_4+\omega_5})\\
        p\mqty(v_{\omega_1-\omega_4}\\
        v_{\omega_1-\omega_2+\omega_4}\\
        v_{-\omega_2+\omega_3-\omega_4}\\
        v_{\omega_3}\\
        v_{-\omega_2+\omega_5}\\
        v_{\omega_1-\omega_4}) &=
        \frac{1}{c_2c_4(q^2+q^{-2})}\mqty(c_2c_4q^{-2}v_{\omega_1}\\
        -qc_2v_{\omega_1}\\
        c_2c_4q^{-2}v_{-\omega_2+\omega_3}\\
        q^2v_{-\omega_2+\omega_3}\\
        -c_2q^{-1}v_{-\omega_4+\omega_5}\\
        q^2v_{-\omega_4+\omega_5}),
    \end{align*}
    extend this $\uq_\bullet$-linearly, and let the maps be 0 on all other basis
    vectors.
    
    These maps all intertwine $\uqb$ and satisfy $p\circ j = \id_{V(\gamma)}$.
    This can be seen especially easily for $\mathsf{AII}_5$, where
    $V(\gamma)$ splits as a direct sum of three two-dimensional
    $\uq_\bullet$-modules of highest weights $\omega_1,\omega_3,\omega_5$,
    respectively.
    In order to map $V(\gamma)$ $\uqb$-linearly to any other $\uq$-module (and
    vice-versa), we must therefore identify the $\uq_\bullet$-highest weight
    vectors of these weights.
    For $L(\omega_5)$, the desired weight spaces are spanned by $v_{\omega_1-\omega_2},v_{\omega_3-\omega_4},v_{\omega_5}$, respectively,
    so any intertwiner must map $v_{\omega_1}$ to the span of $v_{\omega_1-\omega_2}$
    or vice-versa (and similarly for the others).
    By considering the actions of $B_2,B_4$ on $V(\gamma)$, the constants can be
    fixed.
    
    For $L(\omega_3)$, the desired weight spaces are 2-dimensional, but the fact
    that $B_4v_{\omega_1}=0$ fixes the candidate.
    From these maps $p,j$ we see that $\Res(\Psi_i)$ ($i=1,2,3$) have indeed the
    claimed shapes.
\end{proof}

\begin{corollary}
    We obtain the following matrix weights:
    \begin{enumerate}
        \item For $\mathsf{AI}_2$, we obtain
        \[
            M = \mqty([3]_q & m_{2\omega_2} & m_{2\omega_1}\\
        m_{2\omega_1} & \frac{m_{2\alpha_1+2\alpha_2} + 2}{(q+q^{-1})^2} + 1 
        & m_{2\omega_2}\\
        m_{2\omega_2} & m_{2\omega_1} & [3]_q).
        \]
        \item For the $\mathsf{A}_2$ group case, we obtain
        \[
            M = \mqty([3]_q & m_{2\tilde{\omega}_2} & m_{2\tilde{\omega}_1}\\
        m_{2\tilde{\omega}_1} & \frac{m_{2\tilde{\alpha}_1+2\tilde{\alpha}_2} + 2}{(q+q^{-1})^2} + 1 
        & m_{2\tilde{\omega}_2}\\
        m_{2\tilde{\omega}_2} & m_{2\tilde{\omega}_1} & [3]_q).
        \]
        \item For $\mathsf{AII}_5$, we obtain
        \[
            M = (q+q^{-1})\mqty(
                [3]_{q^2} & m_{2\tilde{\omega}_4} & m_{2\tilde{\omega}_2}\\
                m_{2\tilde{\omega}_2} & \frac{m_{2\tilde{\alpha}_2+2\tilde{\alpha}_4}+2}{(q^2+q^{-2})^2} + 1 & m_{2\tilde{\omega}_4}\\
                m_{2\tilde{\omega}_2} & m_{2\tilde{\omega}_4} & [3]_{q^2}).            
        \]
    \end{enumerate}
\end{corollary}
\begin{proof}
    This can be computed from Proposition~\ref{prop-A2-msf},
    using that $M_{b_i,b_j}$ can be computed from
    $(-\rho)\triangleright\Res(\Psi_i),(-\rho)\triangleright\Res(\Psi_j)$.
    We use that $\rho = h_1+h_2$ for $\mathsf{AI}_2$, $\rho=h_1+h_2+h_3+h_4$ for
    the $\mathsf{A}_2$ group case, and
    \[
        \rho=\frac{5}{2}(h_1+h_5) + 4(h_2+h_4) + \frac{9}{2}h_3
    \]
    for $\mathsf{AII}_5$.

    We can furthermore exploit Proposition~\ref{prop-weight-matrix-props}(ii,iii), so that we only
    need to compute the entries $M_{b_1,b_1},M_{b_1,b_2},M_{b_1,b_3},M_{b_2,b_2}$.
\end{proof}

We now write $\varpi_1,\varpi_2$ for $2\tilde{\omega}_1,2\tilde{\omega}_2$ ($\mathsf{AI}_2$ and $\mathsf{A}_2$ group case) or $2\tilde{\omega}_2,2\tilde{\omega}_4$, respectively for the generators of $2L$.

Observe that $M$ is proportional to
\begin{equation}\label{eq-shimeno-matrix-weight1}
    M = \mqty([3]_\tau & m_{\varpi_2} & m_{\varpi_1}\\
        m_{\varpi_1} & \frac{m_{\varpi_1+\varpi_2} + 2}{(\tau+\tau^{-1})^2} + 1 
        & m_{\varpi_2}\\
        m_{\varpi_2} & m_{\varpi_1} & [3]_\tau)
\end{equation}
where $\tau=q$ in the first two cases and $q^2$ in the latter.

We now begin using notation from \cite{Sch23}, in particular
we set $J:=\set{2}$ and take $W_J$ to be the parabolic subgroup $W_J\le W_\Sigma$ generated by $s_2$,
$L_{+,J}$ for the $J$-dominant elements (\cite[Definition~3.12]{Sch23}) of $L$,
and $m_{J,\lambda}$ for the $W_J$-symmetric \enquote{monomial} containing $e^\lambda$ (proof of \cite[Lemma~4.6]{Sch23}).
We consider the $k[2L]^{W_\Sigma}$-basis of $k[2L]^{W_J}$ ($J=\set{2}$)
given by
\begin{equation}\label{eq-A2-basis-elements}
    e_1 = 1,\qquad
    e_{s_1} = \frac{e^{\varpi_2}+e^{\varpi_1-\varpi_2}}{1+\tau^{-2}},\qquad
    e_{s_2s_1} = \tau^2e^{\varpi_1},
\end{equation}
which produces the following matrix weight (using the $\overline{\cdot}$
involution instead of $*$):
\begin{equation}\label{eq-shimeno-matrix-weight2}
    \frac{\tau^4+\tau^2}{6}
    \mqty([3]_\tau & m_{\varpi_1} & m_{\varpi_2}\\
 m_{\varpi_2} & \frac{m_{\varpi_1+\varpi_2}+2}{(\tau+\tau^{-1})^2} + 1 & m_{\varpi_1}\\
    m_{\varpi_1} & m_{\varpi_2} & [3]_\tau)\mqty(1 & 0 & 0\\0 & \tau^2 & 0\\0 & 0 & \tau^4).
\end{equation}
This suggests that we need to map $e_1, e_{s_1},e_{s_2s_1}$ to
$\Psi_3,\Psi_2,\Psi_1$ (or $e_{b_3},e_{b_2},e_{b_1}$). 

\begin{lemma}\label{lem-A2-monotonic}
    The map $t: 2L_{+,J}\to X^+(\gamma)$ given by
    \[
        \varpi_1 + \lambda\mapsto b_1 + \lambda,\qquad
        s_1(\varpi_2 + \lambda)\mapsto b_2 + \lambda,\qquad
        s_1s_2\lambda\mapsto b_3+\lambda
    \]
    for $\lambda\in L_+$ is a well-defined bijection.
    Moreover, $t^{-1}$ is monotonic (considering the
    dominance ordering on $X^+(\gamma)$ and the Macdonald ordering on $2L_{+,J}$, cf. \cite[\S2.7]{Mac03}).
\end{lemma}
\begin{proof}
    See Appendix~\ref{sec-A2-proofs}.
\end{proof}

\begin{lemma}\label{lem-A2-triangular}
    Let $C$ be a $k[2L]^{W_\Sigma}$-linear endomorphism of $k[2L]^{W_J}$ given by
    \[
        Ce_v = a_v e_v\qquad
        v\in W^J.
    \]
    Then $C$ is triangular with respect to $(m_{J,\mu})_{\mu\in L_{+,J}}$.
\end{lemma}
\begin{proof}
    See Appendix~\ref{sec-A2-proofs}.
\end{proof}

\begin{lemma}\label{lem-A2-Gamma-t}
    Define the $k[2L]^{W_\Sigma}$-linear map 
    $\Gamma: k[2L]^{W_J}\to k[X]\otimes\operatorname{span}\set{e_{b_i}\where i=1,2,3}$ by
    mapping
    \[
        e_1\mapsto e_{b_3},\qquad
        e_{s_1}\mapsto e_{b_2},\qquad
        e_{s_1s_2}\mapsto e_{b_1}
    \]
    (recall \eqref{eq-A2-basis-elements} for $e_1,e_{s_1},e_{s_1s_2}$).
    Then $\Gamma^{-1}$ is $t^{-1}$-triangular.
\end{lemma}
\begin{proof}
    See Appendix~\ref{sec-A2-proofs}.
\end{proof}

\begin{theorem}\label{thm-AI-example}
    For every $\mu\in 2L_{+,J}$ there is a nonzero constant $C_\mu$ such that
    \[
        \Gamma(P_{J,\mu}) = C_\mu (-\rho)\triangleright \Res(\Phi^{t(\mu)}_\gamma).
    \]
\end{theorem}
\begin{proof}
    We define the $k[2L]^{W_\Sigma}$-linear automorphism $C\in\End(k[2L]^{W_J})$ given by the matrix
    \[
        \frac{6}{\tau^4+\tau^2}\mqty(1 & 0 & 0\\0 & \tau^{-2} & 0\\0 & 0 & \tau^{-4})
    \]
    for $\tau\in k^\times$. By Lemma~\ref{lem-A2-triangular},
    $C$ is triangular. 
    Moreover, by Lemma~\ref{lem-A2-monotonic}, $t^{-1}$ is monotonic.
    We now show that $t^{-1}$ is monotonic.
    Moreover, by Lemma~\ref{lem-A2-Gamma-t}, $\Gamma^{-1}$ is
    $t^{-1}$-triangular.

    Comparing Equations~\ref{eq-shimeno-matrix-weight1} and Equation~\ref{eq-shimeno-matrix-weight2}, we see that the matrix weights match up.
    Note that by Proposition~\ref{prop-zsf}, the $-\rho$-shifted
    zonal spherical functions for $(\uq,\uqb)$ are the Macdonald polynomials
    for the root system $S(2\Sigma)$ (which equals $S((2\Sigma)^\vee)^\vee$ as $\mathsf{A}_2$ is simply-laced) with parameter $k=\frac{1}{2}\dim(\mathfrak{g}_\alpha)$,
    with all $q$-powers multiplied by $\epsilon$.
    
    We can now conclude the claim
    from Corollary~\ref{cor-Intermediate-macdonald-polys}.
\end{proof}

\subsection{Conclusion to the Examples}\label{sec-conc-examples}
A conclusion that can be drawn from the examples is that in all cases considered,
the $-\rho$-shifted restrictions of MSF can be mapped to
$k[2L]^{W_J}$ for a parabolic subgroup $W_J\le W_\Sigma$ in such a way that
the elementary MSF are mapped to Intermediate Macdonald polynomials.

We shall now make more precise how the parabolic subgroup $W_J$ can be obtained in
the examples and make two general conjectures about when MSF are Intermediate Macdonald
polynomials.

By \cite[Lemma~8.8]{Pe23} and Theorem~\ref{thm-specialisation}, 
the elements of $\mathfrak{B}(\gamma)$ correspond to the irreducible
representations of $M$ ($M$-types) contained in the $K$-representation $V(\gamma)^1$, where $K$ is a Lie group with Lie algebra $\mathfrak{k}$
such that $V(\gamma)^1$ is integrable and $M=Z_K(\mathfrak{a})$ for
any choice of maximal commutative $\mathfrak{a}\subset\mathfrak{p}=\mathfrak{g}^{-\theta}$.

Since $M$ is a normal subgroup of $N_K(\mathfrak{a})$, the restricted
Weyl group $W_\Sigma=N_K(\mathfrak{a})/M$ acts on the
$M$-types of $V(\gamma)^1$, and hence also on $\mathfrak{B}(\gamma)$.

\begin{conjecture}\label{conj-weyl-action}
    Let $(\uq,\uqb,\gamma)$ be a commutative triple.
    The aforementioned action of $W_\Sigma$ on $\mathfrak{B}(\gamma)$
    can be described as follows: the map
    \[
        p: \mathfrak{B}(\gamma)\to \mathfrak{P}(X),\qquad
        b\mapsto W_\bullet b + 2L
    \]
    is an injection and has a $W_\Theta$-invariant domain ($W_\Theta=Z_W(\Theta)$).
    Pulling back the $W_\Theta$-action along this injection gives
    a $W_\Theta$ action on $\mathfrak{B}(\gamma)$ that factors through
    $W_\bullet$, and hence yields an action of $W_\Sigma$
    \cite[(2.2)]{Araki}.
\end{conjecture}

\begin{conjecture}\label{conj-imp}
    Assume now that the action of $W_\Sigma$ on $\mathfrak{B}(\gamma)$
    is transitive and one
    of the stabilisers $W_J$ is a parabolic subgroup of $W_\Sigma$.
    Then the $-\rho$-shifted restrictions of MSF for $(\uq,\uqb,\gamma)$ are Intermediate Macdonald polynomials invariant under
    $W_J$.
\end{conjecture}

We now show that these two conjectures hold for all integrable small $\uqb$-types
and for the four examples $\mathsf{DII}_n,\mathsf{AI}_2$, the $\mathsf{A}_2$ group case,
and $\mathsf{AII}_5$.

\begin{lemma}\label{lem-weyl-action-examples}
    Conjecture~\ref{conj-weyl-action} holds for the six examples from this chapter.
\end{lemma}
\begin{proof}
    Recall that in the examples we had
    \[
    \mathcal{B}(\gamma) = \begin{cases}
        \set{s\omega_n} & \mathsf{BII}\\
        \set{s\omega_1} & \mathsf{CII}_n\\
        \set{\omega_{n-1},\omega_n} & \mathsf{DII}\\
        \set{\omega_1,\omega_2,\omega_1+\omega_2} & \mathsf{AI}_3\\
        \set{\omega_1,\omega_2+\omega_3,\omega_4} & \mathsf{A}_2\text{-group}\\
        \set{\omega_1,\omega_3,\omega_5} & \mathsf{AII}_5
        \end{cases}.
    \]
    For $\mathsf{BII}$, we have
    \[
        W_\bullet s\omega_n + 2L = \set{\qty(\frac{s}{2}+a,\frac{b_2s}{2},\dots,\frac{b_ns}{2})\where a\in\Z,b_2,\dots,b_n\in\set{\pm1}}.
    \]
    For $\mathsf{CII}$, we have
    \[
        W_\bullet s\omega_1 + 2L = \set{\qty(a,b,0,\dots,0)\where \abs{a-b}=s}.
    \]
    For $\mathsf{DII}$, we have
    \begin{align*}
        W_\bullet\omega_{n-1}+2L
        &= \set{\qty(\frac{1}{2}+a,\frac{b_2}{2},\dots,
        \frac{b_n}{2})\where a\in\Z, b_2,\dots,b_n\in\set{\pm1},b_2\cdots b_n=-1}\\
        W_\bullet\omega_n + 2L &= \set{\qty(\frac{1}{2}+a,\frac{b_2}{2},\dots,\frac{b_n}{2})\where a\in\Z,b_2,\dots,b_n\in\set{\pm1},b_2\cdots b_n=1}.
    \end{align*}
    For $\mathsf{AI}_3$, we have
    \begin{align*}
        W_\bullet\omega_1+2L &= \set{a\omega_1+b\omega_2\where
        (a-1),b\in2\Z}\\
        W_\bullet\omega_2+2L &= \set{a\omega_1+b\omega_2\where
        a,(b-1)\in2\Z}\\
        W_\bullet(\omega_1+\omega_2)+2L &=
        \set{a\omega_1+b\omega_2\where a,b\in 1+\Z}.
    \end{align*}
    For the $\mathsf{A}_2$ group case we have
    \begin{align*}
        W_\bullet\omega_1+2L &= \set{(a+1)\omega_1+b\omega_2+a\omega_3+b\omega_4\where
        a,b\in\Z}\\
        W_\bullet(\omega_2+\omega_3)+2L &=\set{a\omega_1+(b+1)\omega_2+(a+1)\omega_3 + b\omega_4\where a,b\in\Z}\\
        W_\bullet\omega_4 + 2L &= \set{a\omega_1+b\omega_2+a\omega_3+(b+1)\omega_4\where a,b\in\Z}.
    \end{align*}
    For $\mathsf{AII}_5$, we have
    \begin{align*}
        W_\bullet\omega_1 + 2L &=
        \set{\pm\omega_1 + a\omega_2+b\omega_4\where a,b\in\Z}\\
        W_\bullet\omega_3 + 2L &=
        \set{\pm\omega_3+a\omega_2+b\omega_4\where a,b\in\Z}\\
        W_\bullet\omega_5 + 2L &=
        \set{\pm\omega_5+a\omega_2+b\omega_4\where a,b\in\Z}.
    \end{align*}
    Note that all these sets are disjoint, so $p$ is injective, and
    every set $p(b)$ is permuted by every element of $W_\bullet$. Next,
    we describe the groups $W_\Theta$:
    \[
        W_\Sigma = \begin{cases}
        \set{W_\bullet, s_1\cdots s_n\cdots s_1W_\bullet} & \mathsf{BII}\\
        \set{W_\bullet,w_0W_\bullet} & \mathsf{CII}\\
        \set{W_\bullet, s_1\cdots s_{n-2}s_ns_{n-2}\cdots s_1 W_\bullet} & \mathsf{DII}\\
        W & \mathsf{AI}_3\\
        \langle s_1s_3, s_2s_4\rangle & \mathsf{A}_2\text{-group}\\
        \langle s_2s_1s_3s_2W_\bullet,s_4s_3s_5s_4W_\bullet\rangle &
        \mathsf{AII}_5
        \end{cases}.
    \]
    Write $r_1:=s_1, s_1s_3,s_2s_1s_3s_2$ for the first generator
    of the last three cases, and $r_2:=s_2,s_2s_4,s_4s_3s_5s_4$ for
    the second.
    These generators act as follows on the bottom elements:
    \begin{description}
        \item[For $\mathsf{BII}$]
        \[
        s_1\cdots s_n\cdots s_1 s\omega_n =
            s\omega_n-s\omega_1 \in  p(s\omega_n)
        \]
        \item[For $\mathsf{CII}$]
        \[
            w_0s\omega_1 = -\omega_1 = s_1s\omega_1 - s\omega_2\in p(s\omega_1).
        \]
        \item[For $\mathsf{DII}$]
        \begin{align*}
            s_1\cdots s_{n-2}s_ns_{n-2}\cdots s_1\omega_{n-1}
            &= \omega_n - \omega_1\in p(\omega_n)\\
            s_1\cdots s_{n-2}s_ns_{n-2}\cdots s_1\omega_n
            &= \omega_{n-1} - \omega_1 \in p(\omega_{n-1}).
        \end{align*}
        \item[For $\mathsf{AI}_3$, the $\mathsf{A}_2$ group case, and for $\mathsf{AII}_5$]
        \begin{align*}
            r_1b_1 &\in p(b_2)\\
            r_1b_2 &\in p(b_1)\\
            r_1b_3 &\in p(b_3)\\
            r_2b_1 &\in p(b_1)\\
            r_2b_2 &\in p(b_3)\\
            r_2b_3 &\in p(b_2)
        \end{align*}
    \end{description}
    Together with the fact that $W_\Theta L=L$ we obtain that
    $W_\Theta$ restricts to a map on $\im(p)$. 
    This therefore descends to the following actions on $\mathfrak{B}(\gamma)$: 
    the trivial action for $\mathsf{BII}$ and the standard permutation actions of
    $S_2$ on $\set{\omega_{n-1},\omega_n}$ ($\mathsf{DII}$) and
    $S_3$ on $\set{b_1,b_2,b_3}$ ($\mathsf{AI}_3$, $\mathsf{A}_2$-group, and $\mathsf{AII}_5$).

    We now show that these actions coincide with the group theoretical actions
    derived earlier.
    \begin{description}
        \item[For $\mathsf{BII}_n$ and $\mathsf{CII}_{n,1}$] Nothing further needs to be shown.
        \item[For $\mathsf{DII}_n$] We take $K=\operatorname{Spin}(2n-1)$. Since the
        adjoint action on $\mathfrak{p}$ factors via $SO(2n-1)$, our group
        $M$ is the preimage of $SO(2n-2)$ under $K\to SO(2n-1)$, i.e.
        $M=\operatorname{Spin}(2n-2)$. In particular, $M$ is a simple
        Lie group, so that its representation theory is determined by the
        highest weights. As was described earlier, $V(\gamma)^1$ splits
        as $L_M(\omega_{n-1}|_\mathfrak{t})\oplus L_M(\omega_n|_{\mathfrak{t}})$, where $\mathfrak{t}\le\mathfrak{m}\le\mathfrak{g}$ is the Cartan subalgebra that was chosen to be contained in the Cartan subalgebra
        $\mathfrak{h}$ of $\mathfrak{g}$. In particular,
        $s_1\cdots s_{n-2}s_ns_{n-2}\cdots s_1 \omega_n|_{\mathfrak{t}}=
        \omega_{n-1}|_{\mathfrak{t}}$ and vice-versa, so that
        $s_1\cdots s_{n-2}s_ns_{n-2}\cdots s_1$, the non-trivial representative of $W_{\Sigma}$ indeed permutes the $M$-types as
        claimed.
        \item[For $\mathsf{AI}_2$]
        We take $K=SO(3)$ and have
        \[
            M= \set{\diag(a,b,c)\where a,b,c\in\set{\pm1},abc=1}.
        \]
        the $K$-type of $\gamma$ corresponds to the standard representation of $K$ on $\C^3$. Note that $M$ can be seen as a subset of the Cartan
        subgroup of $SL(3)$. This gives us a $W$-action on $M$-representations and a direct way of associating weights of $\mathfrak{g}$ with one-dimensional representations of $M$.
        
        In particular, our bottom weights induce the
        following one-dimensional representations of $M$:
        \[
            e^{\omega_1}_M: \diag(a,b,c)\mapsto a,\qquad
            e^{\omega_2}_M: \diag(a,b,c)\mapsto ab,\qquad
            e^{\omega_1+\omega_2}_M: \diag(a,b,c)\mapsto b.
        \]
        In particular, since $M$ consists of elements that square to 1,
        we can ignore any minus signs.
        This means that
        \[
            s_1e^{\omega_1}_M = e^{\omega_2-\omega_1}_M = e^{\omega_1+\omega_2}_M
        \]
        and similarly for all other combinations of elements of $W$ and
        $\mathfrak{B}(\gamma)$. This is exactly the same action as
        was obtained purely from the weights.
        \item[For the $\mathsf{A}_2$ group case]
        We take $K=SU(3)$ (here properly: the
        antidiagonal subgroup of $G=SU(3)\times SU(3)$).
        Then $M\le K$ is the Cartan subgroup, i.e. the group of diagonal
        matrices.
        The bottom weights induce the following one-dimensional representations of $M$:
        \[
            e^{\omega_1}_M:\diag(a,b,c)\mapsto a,\qquad
            e^{\omega_2+\omega_3}_M: \diag(a,b,c)\mapsto b,\qquad
            e^{\omega_4}_M:\diag(a,b,c)\mapsto c.
        \]
        The action of $W_\Theta$ on $M$ is just the action of permutation
        matrices,
        which is the same as is describe earlier.
        \item[For $\mathsf{AII}_5$]
        We take $K=Sp(3)$ and $M=SU(2)\times SU(2)\times SU(2)$.
        If $V$ is the 2-dimensional representation of $SU(2)$, then
        $\omega_1,\omega_3,\omega_5$ correspond to the following
        $M$-types of $V(\gamma)$:
        \[
            V\otimes\C\otimes\C,\quad
            \C\otimes V\otimes \C,\quad
            \C\otimes\C\otimes V.
        \]
        Since the Weyl group of $G$ acts by conjugation with permutation
        matrices, $s_2s_1s_3s_2=(13)(24)$ exchanges the first two $SU(2)$
        blocks and fixes the third and $s_4s_3s_5s_4=(35)(46)$ exchanges
        the last two blocks and fixes the first.
        We conclude that $W_\Sigma$ acts on the representations and hence
        also on the bottom elements exactly like $S_3$ acts on
        $\set{b_1,b_2,b_3}$.
        This is also the action we described earlier.
    \end{description}
\end{proof}

\begin{lemma}\label{lem-integrable-small-b-types}
    Let $\gamma$ be an integrable small $\uqb$-type.
    Then $V(\gamma)$ is one-dimensional (i.e. trivial, or $(\uq,\uqb)$ is of Hermitian type
    and $\gamma$ is a character), or
    $(\uq,\uqb,\gamma)$ is any of the following:
    \begin{enumerate}
        \item $(\uq,\uqb)$ is of type $\mathsf{BII}_n$ ($n>1$) and 
        $\gamma$ corresponds to the representation of highest weight
        $s\varpi_{n-1}$ (in the notation of \cite{Pe23}) for
        $s\in\N_0$.
        \item $(\uq,\uqb)$ is of type $\mathsf{BII}_n$ ($n>1$)and
        $\gamma$ corresponds to the representation of highest weight $s\varpi_n$
        for $s\in\N_0$.
        \item $(\uq,\uqb)$ is of type $\mathsf{CII}_{n,1}$ $(n>2)$ and
        $\gamma$ corresponds to the representation of highest weight $s\varpi_1$
        for $s\in\N_0$.
    \end{enumerate}
\end{lemma}

\begin{proof}
See Appendix \ref{ap:small}
\end{proof}

\begin{lemma}
    Assume $\gamma$ is an integrable small $\uqb$-type.
    Then Conjecture~\ref{conj-weyl-action} holds for
    $(\uq,\uqb,\gamma)$.
\end{lemma}
\begin{proof}
    By Lemma~\ref{lem-integrable-small-b-types}, $\gamma$ is either one-dimensional,
    or $(\uq,\uqb,\gamma)$ is one of the three cases considered as the first two examples
    in Lemma~\ref{lem-weyl-action-examples}.

    So it remains to check the case where $\gamma$ is one-dimensional.
    By \cite[\S2.4]{Mee25}, nontrivial characters exist for $(\uq,\uqb)$ of
    type $\mathsf{AIII}_{n,p},\mathsf{BI}_{n,2},\mathsf{CI},\mathsf{DI}_{n,2},\mathsf{DIII},\mathsf{EIII},\mathsf{EVII}$.
    In addition, we have to consider $\mathsf{AI}_1$ and $\mathsf{AIV}_n$
    (which are considered to be $\mathsf{AIII}_{1,1}$ and $\mathsf{AIII}_{n,1}$ in loc.cit.).

    In each case, the characters that exist form a 1-dimensional lattice
    generated by $\chi$.
    In case $\Sigma$ is reduced and $\gamma=\ell\chi$, we know from
    \cite[Lemma~2.7]{Mee25} that $\mathfrak{B}(\gamma)=\set{b}$ for
    \[
        b = \frac{\abs{\ell}}{2} \sum_{\Sigma^+_{\mathrm{l}}}\alpha
    \]
    where $\Sigma^+_{\mathrm{l}}$ consists of the positive longest
    roots of $\Sigma$.

    We now show that $W_\bullet b + 2L$ is $W_\Theta$-invariant.
    Note that since $b$ lies in the $\Q$-span of $\Sigma$, $W_\Theta$'s
    action on it factors through the quotient group $W_\Sigma$, so
    that's what we look at.
    
    Recall that $W_\Sigma$ is the Weyl group of $\Sigma$, 
    so it suffices to show $wb-b\in 2L$ for the generators $w=r_i$ associated with simple
    roots $\tilde{\alpha}_i$.
    In case $\tilde{\alpha}_i$ not a long root (recall
    that we assumed $\Sigma$ to be reduced), we have
    $r_i\Sigma_{\mathrm{l}}^+=\Sigma_{\mathrm{l}}^+$ and hence
    $r_ib=b$.
    Otherwise, we have $wb-b = \alpha_i$.
    Since $\Sigma$ is of type $\mathsf{B},\mathsf{C}$, or
    $\mathsf{A}_1$, we then have $\alpha_i\in 2P$, hence
    $wb-b\in 2L$.
    The claim follows.

    The cases not covered by the assumption of reducedness are
    $\mathsf{AIII}_{n,p}$ ($2p<n+1$), $\mathsf{AIV}_n$, $\mathsf{DIII}_{2p+1}$,
    and $\mathsf{EIII}$.
    In these cases, we find in \cite[Table~2]{Pe23} that the
    bottom elements are given by $\abs{l}\omega_i$ or
    $\abs{l}\omega_{\tau(i)}$ where $I_\bullet\cup\set{i,\tau(i)}$ is a subdiagram of type $\mathsf{AIV}$, depending
    on the sign of $l$ (and the choice of fundamental character $\chi$).
    Concretely:
    \begin{align*}
        \mathsf{AIII}_{n,p}&:\qquad \abs{\ell}\omega_p,\abs{\ell}\omega_{n+1-p}\\
        \mathsf{AIV}_{n}&:\qquad \abs{l}\omega_1,\abs{\ell}\omega_n\\
        \mathsf{DIII}_{n=2p+1} &:\qquad
        \abs{l}\omega_{2p},\abs{l}\omega_{2p+1}\\
        \mathsf{EIII}&:\qquad \abs{l}\omega_1,\abs{l}\omega_6
    \end{align*}
    For $\mathsf{AIII},\mathsf{AIV}$, note that $W_\Theta$ is generated
    by $W_\bullet,s_1s_n,\dots,s_{p-1}s_{n+2-p},s_p\cdots s_{n+1-p}\cdots s_p$.
    We have $s_is_{n+1-i}b=b$ for $i=1,\dots,p-1$ as well
    as $wb=b$ for $w\in W_\bullet$.
    Furthermore,
    \begin{align*}
        s_p\cdots s_{n+1-p}\cdots s_p\omega_p &=
        \omega_{p-1}-\omega_{n+1-p}+\omega_{n+2-p}
        = \omega_p - 2\tilde{\alpha}_p\\
        s_p\cdots s_{n+1-p}\cdots s_p\omega_{n+1-p} &=
        \omega_{p-1}-\omega_p + \omega_{n+2-p}
        = \omega_{n+1-p}-2\tilde{\alpha}_p.
    \end{align*}
    In any case, $b$ and $wb$ differ by elements contained in $2L$.

    For $\mathsf{DIII}_{2p+1}$, the group $W_\Theta$ is generated by
    $W_\bullet, s_{2i}s_{2i-1}s_{2i+1}s_{2i}, s_{2p}s_{2p-1}s_{2p+1}s_{2p-1}s_{2p}$
    ($i=1,\dots,p-1$).
    $W_\bullet$ fixes $\omega_{2p},\omega_{2p+1}$ as do
    $s_{2i}s_{2i-1}s_{2i+1}s_{2i}$ for $i\le p-1$.
    Lastly,
    \begin{align*}
        s_{2p}s_{2p-1}s_{2p+1}s_{2p-1}s_{2p}\omega_{2p} &=
        \omega_{2p-2}-\omega_{2p+1}
        = \omega_{2p} - 2 (-\omega_{2p-2}+\omega_{2p-1})^\sim
        \in \omega_{2p}+2L\\
        s_{2p}s_{2p-1}s_{2p+1}s_{2p-1}s_{2p}\omega_{2p+1} &=
        \omega_{2p-2}-\omega_{2p}
        = \omega_{2p+1}-2(-\omega_{2p-2}+\omega_{2p-1})^\sim
        \in \omega_{2p+1}+2L.
    \end{align*}

    For $\mathsf{EIII}$, the elements $r_1 = s_1s_3s_4s_5s_6s_1s_3s_4s_5s_1s_3s_4s_1s_3s_1$
    (i.e. the longest element for the $\mathsf{A}_5$ subdiagram generated by
    $\alpha_1,\alpha_3,\alpha_4,\alpha_5,\alpha_6$) and
    $r_2=s_2s_4s_3s_5s_4s_2$ both commute with $w_\bullet$ and with $\tau$.
    Consequently, they commute with $\Theta$.
    Furthermore they implement the simple reflections of $\Sigma$, so
    that $W_\Theta$ is generated by $W_\bullet$ and these elements.
    Note that $W_\bullet$ fixes $\omega_1,\omega_6$, similarly for $r_2$.
    It remains to see how $r_1$ acts.
    Note that we can rewrite $r_1=s_1s_3s_4s_5s_6\langle s_i\mid i<6\rangle$,
    so that
    \[
        r_1\omega_6 = s_1s_3s_4s_5s_6\omega_6 = \omega_2-\omega_1.
    \]
    Applying $\tau$ we obtain $r_1\omega_1=\omega_2-\omega_6$.
    In both cases, we have
    \[
        r_1b-b = \omega_2-\omega_1-\omega_6=-2\tilde{\alpha}_1\in 2L.
    \]
    We thus obtain the conclusion that $W_\bullet b +2L$ is $W_\Theta$-invariant for
    all cases with nontrivial characters.

    In case the character is trivial, we have $b=0$, so the statement follows by
    the fact that $W_\Theta$ acts by linear maps.
\end{proof}

\begin{lemma}
    Conjecture~\ref{conj-imp} holds for all six examples considered here
    as well as for the case where $\gamma$ is an integrable small $\uqb$-type (and
    where the ZSF case has been discussed in Section~\ref{sec:zonal}).
\end{lemma}
\begin{proof}
    Note that for our six examples and the case where $\gamma$ is
    one-dimensional (even if it is a small $\uqb$-type), the premise of
    $W_\Sigma$ acting transitively on $\mathfrak{B}(\gamma)$ holds.

    If $\gamma$ is 1-dimensional, the claim follows from \cite[Theorem~6.1]{Mee25} and the fact that the stabiliser of the
    (single) bottom element is all of $W_\Sigma$.

    For the remaining integrable small $\uqb$-types, the stabiliser of
    the single bottom element is $W_\Sigma$, and by Theorem~\ref{thm-example-BII-CII}
    the $-\rho$-shifted restrictions of matrix-spherical functions
    are $W_\Sigma$-Intermediate Macdonald polynomials (=symmetric Macdonald polynomials).

    For $\mathsf{DII}$, the stabiliser of $\omega_n$ is trivial,
    and we obtained $\{1\}$-intemediate Macdonald polynomials
    (=non-symmetric Macdonald polynomials) as $-\rho$-shifted restrictions of MSF in Theorem~\ref{thm-DII-example}.

    For $\mathsf{AI}_3$, the $\mathsf{A}_2$ group case, and
    $\mathsf{AII}_5$, the stabiliser of $\omega_1$ is the $W_J$
    generated by $s_2$, and we obtained $W_J$-Intermediate Macdonald
    polynomials as $-\rho$-shifted restrictions of MSF in Theorem~\ref{thm-AI-example}.
\end{proof}

\begin{lemma}
    Let $(\uq,\uqb,\gamma)$ be a commutative triple of rank 1 such that
    the Weyl action from Conjecture~\ref{conj-weyl-action} is transitive,
    then one of the following cases holds:
    \begin{enumerate}
        \item $(\uq,\uqb)$ is of type $\mathsf{A}_1$-group and
        $\gamma$ corresponds to the highest weight $\varpi$
        (i.e. the standard representation)
        \item $(\uq,\uqb)$ is of type $\mathsf{AII}_3$ and $\gamma$
        corresponds to the highest weight $\varpi_1$
        (i.e. the standard representation).
        \item $(\uq,\uqb)$ is of type $\mathsf{DII}_n$ and
        $\gamma$ corresponds to the highest weight $\varpi_{n-1}$
        (i.e. the spin representation).
        \item $(\uq,\uqb,\gamma)$ is a small $\uqb$-type.
    \end{enumerate}
\end{lemma}
\begin{proof}
    For the eight possible rank-1 (families of) Satake diagrams
    $\mathsf{A}_1$-group, $\mathsf{AI_1}$, $\mathsf{AII}_3$, $\mathsf{AIV}_n$, $\mathsf{BII}_n$, $\mathsf{CII}_{n,1}$, $\mathsf{DII}_n$, $\mathsf{FII}$, we note
    that modulo $W_\bullet$, the group $W_\Theta$ is generated by
    \[
        s_1s_2,s_1,s_2s_1s_3s_2, s_1\cdots s_n\cdots s_1,s_1\cdots s_n\cdots s_1,-1, (-1,-1,1,\dots,1),-1,
    \]
    respectively.
    For the Satake diagrams $\mathsf{AI}_1,\mathsf{AIV}_n,\mathsf{BII}_n,\mathsf{CII}_{n,1},\mathsf{FII}$, we then see that modulo $W_\bullet$ and
    $2L$, all weights are stabilised by $W_\Theta$,
    so in particular only the small $\uqb$-types $\mathsf{AI}_1,\mathsf{AIV}_n$
    (characters) and $\mathsf{BII}_n,\mathsf{CII}_{n,1}$ have a transitive
    $W_\Theta$-action.

    For the three remaining Satake diagrams, the nontrivial element of
    $W_\Theta/W_\bullet$ acts non-trivially on some weights.
    In particular, for the $\mathsf{A}_1$ group case, the weights $\omega_1,\omega_2$ are exchanged, for $\mathsf{AII}_3$, the weights $\omega_1,\omega_3$ are exchanged, and for $\mathsf{DII}_n$, the weights
    $\omega_{n-1},\omega_n$ are exchanged.
    These orbits are exactly the bottoms of the wells corresponding to the
    first three claimed cases.
\end{proof}

\begin{proposition}
    Conjectures~\ref{conj-weyl-action} and \ref{conj-imp} hold in rank 1.
\end{proposition}

Note that the examples here have focused exclusively on the case where
$\uqb=\uqb'$ and $\gamma=\gamma'$.
As a straightforward generalisation of the integrable small $\uqb$-types
we consider earlier, i.e. cases for which $\#\mathfrak{B}^+(\gamma)=1$,
the authors of \cite{OS05} put forward an example
of two different $\uqb$-representations $\gamma,\gamma'$
satisfying $\#\mathfrak{B}^+(\gamma,\gamma')=1$.
Just like for our small $\uqb$-type examples, the resulting 
matrix-spherical functions are symmetric Macdonald polynomials for an
affine root system of type $(\mathsf{C}_p^\vee,\mathsf{C}_p)$.
From anecdotal evidence, examples of such asymmetric integrable
small $\uqb$-types are abundant.

\section{Outlook}
We conclude by discussing further directions of research.

Firstly, as we require a pairing of Hopf algebras, the 
current formalism only works for finite-type 
quantum-symmetric pairs.
Nevertheless, the formalism from \cite{Kol14} describes
quantum-symmetric pairs also if the root system is of infinite type.
The work \cite{Ko15} suggests that in order to study MSF
in such a setting, completions of various algebras have to be considered.

As a further possible generalisation, we point out the work of
Andrea~Appel, Vidas~Regelskis, and Bart~Vlaar
(\cite{RV20},\cite{RV22},\cite[\S6]{Ap25})
that generalises the notion of a (quantum) symmetric pair.
It should be noted, however, that their theory describes pairs of Lie algebras (and quantisations thereof) $(\mathfrak{g},\mathfrak{k})$ for which $\mathfrak{k}$ is not reductive.
Therefore, in particular the representation theory and
the branching rules of such pairs have not been much studied.

But even in the world of finite-type quantum-symmetric pairs
there is still ample work to be done.
Most notable, as a foundation of all techniques presented in this paper,
the zonal spherical functions for quantum-symmetric pairs of
Satake-type $\mathsf{BI}_{n,2},\mathsf{CI}_n,\mathsf{DI}_{n,2},\mathsf{DIII}_{2p},\mathsf{EVII}$ with non-zero non-standard parameter,
and of $\mathsf{DIII}_{2p+1},\mathsf{EIII},\mathsf{FII}$ are
still not known in full generality.

Next, while we have exhausted all the \enquote{homogeneous}
integral small $\uqb$-types (i.e. $\gamma$ such that
$\#\mathfrak{B}^+(\gamma,\gamma)=1$),
as is demonstrated in \cite{OS05}, there are also a lot of
\enquote{heterogeneous} integral small $\uqb$-types (i.e.
$\gamma,\gamma'$ with $\#\mathfrak{B}^+(\gamma,\gamma')=1$).
This means that there is a wealth even of symmetric Macdonald
polynomials still to find.

And lastly, the Conjectures~\ref{conj-weyl-action} and
\ref{conj-imp} remain to be proven.
This includes or would ideally be complemented by a classification
of quantum commutative triples $(\uq,\uqb,\gamma)$
(or $(\uq,\uqb,\gamma)$ and $(\uq,\uqb',\gamma')$)
with a transitive action of $W_\Sigma$ on $\mathfrak{B}^+(\gamma)$ (or $\mathfrak{B}^+(\gamma,\gamma')$).

\appendix
\section{L-Operators}\label{ap:Loperators}
The goal of this appendix is the computation of the $L$-operators
$L_{n+1-l,l+1}^-$ and $L_{n-l,l}^-$ for $\uq=\uq_q(\mathfrak{gl}(n))$, as introduced in \eqref{eq:Loperator}.

\begin{lemma}\label{lem-l-operators-eij}
    Writing $L^-_{ji}=e_{ij}K_{-\epsilon_i}$ (with in particular
    $e_{i,i+1}=-(q-q^{-1})E_i$) for an element $e_{ij}$ of weight
    $\epsilon_i-\epsilon_j$, we have
    $e_{i,j+1}=-q\ad(E_j)(e_{ij})$.
\end{lemma}
\begin{proof}
    From \cite[(111)]{ks} we can deduce that
    \begin{align*}
        L_{j+1,i}^- &= \frac{1}{q-q^{-1}} \comm{L_{j+1,j}^-}{L_{j,i}^-}
        L_{jj}^+= -\comm{E_jK_{-\epsilon_j}}{e_{ij}K_{-\epsilon_i}}K_{\epsilon_j}= - q\ad(E_j)(e_{ij})K_{-\epsilon_i},
    \end{align*}
    which proves the claim.
\end{proof}

\begin{lemma}\label{lem-l-operators-fij}
    Writing $L^-_{ji}=f_{ij}K_{-\epsilon_j}$ for an element $f_{ij}$ of
    weight $\epsilon_i-\epsilon_j$, with in particular
    $f_{i,i+1}=-(q-q^{-1})E_iK_i^{-1}$,
    we have $f_{i-1,j}=-q\ad_r(E_{i-1})(f_{ij})$.
\end{lemma}

\begin{proof}
    Analogous to Lemma~\ref{lem-l-operators-eij}.
\end{proof}

\begin{proposition}\label{prop-l-operators}
    We have
    \begin{align*}
        L_{n-l,l}^- &= (-1)^n q^{n-2l-1} (q-q^{-1}) \ad(E_{n-l-1}\cdots E_{l+1})(E_l) K_{-\epsilon_l}\\
        L_{n+1-l,l+1}^- &=(-1)^nq^{n-2l-1}(q-q^{-1})\ad_r(E_{n-l-1}\cdots E_{l+1})(E_{n-l}K_{n-l}^{-1})
       K_{-\epsilon_{n+1-l}}.
    \end{align*}
    In particular we have
    \begin{align*}
        L_{n-l,l}^-&\in (-1)^n q^{n-2l-1} (q-q^{-1}) E_{n-l-1}\cdots E_l 
        K_{-\epsilon_l} + \uq\uq_{\bullet,+}\\
        L_{n-l,l}^-&\in- (q-q^{-1}) E_l\cdots E_{n-l-1}K_{-\epsilon_l}
        + \uq_{\bullet,+}\uq\\
        L_{n+1-l,l+1}^-&\in- (q-q^{-1})E_{l+1}\cdots E_{n-l}
       K_{-\epsilon_{l+1}} + \uq\uq_{\bullet,+}\\
        L_{n+1-l,l+1}^-&\in(-1)^nq^{n-2l-1}(q-q^{-1})E_{n-l}\cdots E_{l+1}
       K_{-\epsilon_{l+1}}
       + \uq_{\bullet,+}\uq.
    \end{align*}
\end{proposition}
\begin{proof}
    Iterating the statement from Lemma~\ref{lem-l-operators-eij}, we obtain
    \begin{align*}
        e_{l,n-l} &= (-q)^{n-2l-1} \ad(E_{n-l-1}\cdots E_{l+1})(e_{l,l+1})\\
        &= (-1)^n q^{n-2l-1} (q-q^{-1}) \ad(E_{n-l-1}\cdots E_{l+1})(E_l)\\
        L_{n-l,l}^- &= (-1)^n q^{n-2l-1} (q-q^{-1}) \ad(E_{n-l-1}\cdots E_{l+1})(E_l) K_{-\epsilon_l}.
    \end{align*}
    Similarly, we obtain
    \begin{align*}
       f_{l+1,n+1-l} &= (-q)^{n-2l-1} \ad_r(E_{n-l-1}\cdots E_{l+1})(e_{n-l, n+1-l})\\
       &= (-1)^n q^{n-2l-1} (q-q^{-1}) \ad_r(E_{n-l-1}\cdots E_{l+1})(E_{n-l}K_{n-l}^{-1})\\
       L_{n+1-l, l+1}^- &= (-1)^nq^{n-2l-1}(q-q^{-1})\ad_r(E_{n-l-1}\cdots E_{l+1})(E_{n-l}K_{n-l}^{-1})
       K_{-\epsilon_{n+1-l}}.
    \end{align*}
    Expanding the adjoint actions, we obtain
    \begin{align*}
        L_{n-l,l}^- \in &(-1)^n q^{n-2l-1} (q-q^{-1}) E_{n-l-1}\cdots E_l 
        K_{-\epsilon_l}
        + \uq\uq_{\bullet,+}\\
        \cap& - (q-q^{-1}) E_l\cdots E_{n-l-1}K_{-\epsilon_l}
        + \uq_{\bullet,+}\uq\\
        L_{n+1-l, l+1}^- \in& - (q-q^{-1})E_{l+1}\cdots E_{n-l}
       K_{-\epsilon_{l+1}} + \uq\uq_{\bullet,+}\\
       \cap & (-1)^nq^{n-2l-1}(q-q^{-1})E_{n-l}\cdots E_{l+1}
       K_{-\epsilon_{l+1}}
       + \uq_{\bullet,+}\uq\qedhere
    \end{align*}
\end{proof}

\section{Proofs for Section~\ref{sec-examples}}\label{ap:A}
\subsection{Lemma~\ref{lem-sl2-triple}}\label{ap:lem-sl2-triple}
We begin with a few computation rules for $\mathsf{BII}_n$.
\begin{proposition}\label{prop-higher-spin-computations}
    Let $v\in L(s\omega_n)_{s\omega_n-r\omega_1}$.
    \begin{enumerate}
        \item For $1<i\le n$ we have
        \[
            E_n\cdots E_i F_i\cdots F_nv = [s]_qv.
        \]
        \item For $1\le i < n$ we have
        \[
            F_1\cdots F_iE_i\cdots E_1v=[r]_{q^2}v.
        \]
        \item For $1<i\le n$ we have
        \[
            E_n\cdots E_iF_1\cdots F_nE_{i-1}\cdots E_1v
            = [s-2]_q[r]_{q^2}v + F_1\cdots F_nE_n\cdots E_1v.
        \]
        \item We have
            \[
            \comm{E_n\cdots E_1}{F_1\cdots F_n}v
            = [s-2r]_qv.
            \]
    \end{enumerate}
\end{proposition}
\begin{proof}
    \begin{enumerate}
        \item We proceed by reverse induction on $i$.
        For $n=i$, we obtain
        \[
            E_nF_nv = [K_{h_n}]_qv + F_nE_nv
            = [s]_qv
        \]
        since $s\omega_n - r\omega_1+\alpha_n$ is not contained in the
        support of $L(s\omega_n)$.
        
        Assume the statement holds true for $i+1$ (for $i>1$), then we have
        \begin{align*}
            E_n\cdots E_{i+1}E_iF_iF_{i+1}\cdots F_nv &=
            E_n\cdots E_{i+1}[K_i]_{q^2}F_{i+1}\cdots F_nv
            \\
            &\quad + E_n\cdots E_{i+1}F_i\cdots F_nE_iv.
        \end{align*}
        
        Note that $s\omega_n-r\omega_1+\alpha_i$ does not lie in the
        support of $L(s\omega_n)$, so that the last term is 0. Moreover,
        \[
            \langle h_i, s\omega_n-r\omega_1-\alpha_n-\cdots-\alpha_{i+1}\rangle = 1,
        \]
        consequently, we have
        \[
            = E_n\cdots E_{i+1}F_{i+1}\cdots F_nv =
            [s]_qv,
        \]
        as was claimed.
        \item We proceed by induction on $i$.
        For $i=1$, we have
        \[
            F_1E_1v = -[K_1]_{q^2}v + E_1F_1v
            = [r]_{q^2}v.
        \]
        Assume the claim holds for $i-1$ for $i<r$, then
        \begin{align*}
         F_1\cdots F_iE_i\cdots E_1v &=
            -F_1\cdots F_{i-1}[K_i]_{q^2}E_{i-1}\cdots E_1v
           \\
           &\quad + F_1\cdots F_{i-1}E_{i-1}\cdots E_1E_iv.   
        \end{align*}

        Since $\alpha_i+s\omega_n-r\omega_1$ does not lie in the support
        of $L(s\omega_n)$, the last term is 0. Furthermore, we have
        \[
            \langle h_i, s\omega_n-r\omega_1+\alpha_1+\cdots+\alpha_{i-1}\rangle = -1,
        \]
        so that we have
        \[
            F_1\cdots F_iE_i\cdots E_1v = F_1\cdots F_{i-1}E_{i-1}\cdots E_1v = [r]_{q^2}v,
        \]
        as claimed. 
        \item We proceed by reverse induction on $i$.
        For $i=n$, we have
        \begin{align*}
            &E_nF_1\cdots F_nE_{n-1}\cdots E_1v\\
            &= F_1\cdots F_{n-1}[K_n]_qE_{n-1}\cdots E_1v
            + F_1\cdots F_nE_n\cdots E_1v\\
            &= [s-2]_q F_1\cdots F_{n-1}E_{n-1}\cdots E_1v
            + F_1\cdots F_nE_n\cdots E_1v\\
            &= [s-2]_q[r]_{q^2} v + F_1\cdots F_nE_n\cdots E_1v
        \end{align*}
        by (ii).

        Assume the claim is true for $i+1$ ($i>1$), then we have
        \begin{align*}
            &E_n\cdots E_iF_1\cdots F_nE_{i-1}\cdots E_1v =\\&
            E_n\cdots E_{i+1}F_1\cdots
            F_{i-1}[K_i]_{q^2} F_{i+1}\cdots F_nE_{i-1}\cdots E_1v\\
            &+ E_n\cdots E_{i+1}F_1\cdots F_n E_i\cdots E_1v.
        \end{align*}
        We have
        \[
            \langle h_i, s\omega_n-r\omega_1 + \alpha_1+\cdots+\alpha_{i-1}-\alpha_n-\cdots-\alpha_{i+1}\rangle=0,
        \]
        so the first term vanishes. We then apply the induction hypothesis and obtain
        \[
            E_n\cdots E_iF_1\cdots F_nE_{i-1}\cdots E_1v =
            [s-2]_q[r]_{q^2}v + F_1\cdots F_nE_n\cdots E_1v.
        \]
        \item
            We have
            \begin{align*}
                E_n\cdots E_1F_1\cdots F_nv &=
                E_n\cdots E_2[K_1]_{q^2}F_2\cdots F_nv
                + E_n\cdots E_2F_1\cdots F_nE_1v\\
                &= -[r-1]_{q^2} E_n\cdots E_2F_2\cdots F_nv
                + E_n\cdots E_2F_1\cdots F_nE_1v.
            \end{align*}
            By (i) and (iii) this equals
            \[
                -[r-1]_{q^2}[s]_qv + [s-2]_q[r]_{q^2}v + F_1\cdots F_nE_n\cdots E_1v.
            \]
            We therefore conclude that
            \[
                \comm{E_n\cdots E_1}{F_1\cdots F_n}v
                = \qty([r]_{q^2}[s-2]_q - [r-1]_{q^2}[s]_q)v
                = [s-2r]_qv.\qedhere
            \]
    \end{enumerate}
\end{proof}

\begin{proof}[Proof of \ref{lem-sl2-triple} for $\mathsf{BII}_n$]
    Proposition~\ref{prop-higher-spin-computations}(iv) can be rephrased
    as $\comm{E_J}{F_{\overline{J}}}$ acting on $V'_r$ as $[s-2r]_q$.
    Note that $K_J$ acts on $V'_r$ as
    \[
        q^{\langle 2h_1+\cdots + 2h_{n-1}+h_n,s\omega_n-r\omega_1\rangle}
        = q^{s-2r},
    \]
    so that $\comm{E_J}{F_{\overline{J}}}=[K_J]_q$ on $V'$.

    Moreover, $E_J,F_{\overline{J}}$ have weight $\pm(\alpha_1+\cdots+\alpha_n)=\pm\omega_1$, so $\ad(K_J)(E_J)=q^2E_J$ and $\ad(K_J)(F_{\overline{J}}) = q^{-2}F_{\overline{J}}$.
\end{proof}

We continue with similar computation rules for the case of $\mathsf{CII}_{n,1}$.
\begin{proposition}\label{prop-CII-computations}
    Write $K=(i_1,\dots,i_{2n-2})=(1,\dots,n,\dots,2)$.
    Let $v\in V'_r=L(s\omega_1)_{s\omega_1-r\omega_2}$.
    \begin{enumerate}
        \item For $2\le k<2n-2$ we have
        \[
            E_{i_1}\cdots E_{i_k}F_{i_k}\cdots F_{i_1}v=
            [s-r]_q.
        \]
        In particular,
        \[
            E_1\cdots E_n\cdots E_jF_j\cdots F_n\cdots F_1v
            =[s-r]_qv
        \]
        for $j>2$.
        \item For $2<k\le 2n-2$ we have
        \[
            F_{i_{2n-2}}\cdots F_{i_k}E_{i_k}\cdots E_{i_{2n-2}}v
            = [r]_qv.
        \]
        In particular,
        \[
            F_2\cdots F_n\cdots F_jE_j\cdots E_n\cdots E_2v=[r]_qv
        \]
        for $j>2$.
        \item For $2\le k<2n-2$ we have
        \[
            E_{i_1}\cdots E_{i_k}F_{\overline{J}}E_{i_{k+1}}
            \cdots E_{i_{2n-2}} v = [r]_q[s-1-r]_qv + F_{\overline{J}}E_Jv.
        \]
        \item We have $\comm{E_J}{F_{\overline{J}}}v=[s-2r]_qv$.
    \end{enumerate}
\end{proposition}
\begin{proof}
    \begin{enumerate}
        \item We proceed by induction in $k$. Write
        $A$ for the term on the left-hand side.
        For $k=2$, we have
        \begin{align*}
		A &= E_1E_2F_2F_1 v = E_1 [K_2]_q F_1v  + F_2E_1F_1E_2v\\
		&= -[r-1]_q E_1F_1v + F_2 [K_1]_q E_2v  + F_2F_1E_1E_2v\\
		&= -[r-1]_q [s]_qv + [s-1]_q[r]_q v + F_2F_1E_1E_2v.
	    \end{align*}
	    Note that $E_1E_2v=0$ as $s\omega_1-r\omega_2+\alpha_1+\alpha_2$ is not contained in the support of $L(s\omega_1)$.
	    Moreover, $[s-1]_q[r]_q-[s]_q[r-1]_q = [s-r]_q$.
        
	    Assume that $k>2$ and that the claim holds for $k-1$.
        \begin{itemize}
            \item If $k\le n$, we have
            \begin{align*}
            A
            &= E_1\cdots E_kF_k\cdots F_1v
            = E_1\cdots E_{k-1}[K_k]_{q_k}F_{k-1}\cdots F_1v\\
            &= E_1\cdots E_{k-1}F_{k-1}\cdots F_1v
            = E_{i_1}\cdots E_{i_{k-1}}F_{i_{k-1}}\cdots F_{i_1}v\\
            &= [s-r]_q v
            \end{align*}
            as $E_kv=0$ and by the induction hypothesis.

	          \item If $k>n$, we have
    	    \begin{align*}
    		A
    		&= E_1\cdots E_n\cdots E_{2n-k}F_{2n-k}\cdots F_n\cdots F_1v\\
    		&= E_1\cdots E_n\cdots E_{2n+1-k}\Big(
    		[K_{2n-k}]_q F_{2n+1-k}\cdots F_n\cdots F_1v\\
    		&\qquad+ F_{2n-k}\cdots F_n\cdots F_{2n+1-k}[K_{2n-k}]_q F_{2n-1-k}\cdots F_1v\Big).
    	    \end{align*}
    	    Note that $F_{2n+1-k}$ commutes with $F_1,\dots,F_{2n-1-k}$, so that
    	    in the second expression we can permute $F_{2n+1-k}$ to the right.
    	    Since $F_{2n+1-k}v=0$ (as $2n+1-k\ne1$), the second term vanishes:
    	    \begin{align*}
    		A &= E_1\cdots E_n\cdots E_{2n+1-k}F_{2n+1-k}\cdots F_n\cdots F_1v\\
    		&= E_{i_1}\cdots E_{i_{k-1}}F_{i_{k-1}}\cdots F_{i_1}v
    		= [s-r]_qv
    	    \end{align*}
    	    by the induction hypothesis.
        \end{itemize}
        \item We proceed by reverse induction in $k$. Again, use
        $A$ to denote the term on the left-hand side of the claim.
        For $k=2n-2$,
        we obtain
        \[
            A=F_2E_2v = -[K_2]_qv = [r]_qv.
        \]
        Now assume that $k<2n-2$ and that the claim holds for $k+1$.
        \begin{itemize}
            \item If $k\ge n$, we obtain
    	    \begin{align*}
    		A
    		&= F_2\cdots F_{2n-k}E_{2n-k}\cdots E_2v\\
    		&= -F_2\cdots F_{2n-1-k}[K_{2n-k}]_{q_{2n-k}} E_{2n-1-k}\cdots E_2v\\
    		&= F_2\cdots F_{2n-1-k}E_{2n-1-k}\cdots E_2v\\
    		&= F_{i_{2n-2}}\cdots F_{i_{k+1}}E_{i_{k+1}}\cdots E_{i_{2n-2}}v\\
    		&= [r]_qv
    	    \end{align*}
    	    since $F_{2n-k}v=0$ and by the induction hypothesis.

    	    \item If $k<n$, we obtain
    	    \begin{align*}
    		A
    		&= F_2\cdots F_n\cdots F_kE_k\cdots E_n\cdots E_2v\\
    		&= -F_2\cdots F_n\cdots F_{k+1}
    		\Big([K_k]_q E_{k+1}\cdots E_n\cdots E_2v\\
    		&\qquad + E_k\cdots E_n\cdots E_{k+1}[K_k]_q E_{k-1}\cdots E_2v).
    	    \end{align*}
    	    Note that $E_{k+1}$ commutes with $E_{k-1},\dots,E_2$, so that
    	    we may permute it to the right in the second expression.
    	    Since furthermore $E_{k+1}v=0$ ($k>1$), we conclude that the second
    	    expression vanishes.
    	    Consequently, the above equals
    	    \begin{align*}
    		A &= F_2\cdots F_n\cdots F_{k+1}E_{k+1}\cdots E_n\cdots E_2v\\
    		&= F_{i_{2n-2}}\cdots F_{i_{k+1}}E_{i_{k+1}}\cdots E_{i_{2n-2}}v
    		= [r]_qv
    	    \end{align*}
    	    by the induction hypothesis.
        \end{itemize}
        \item We proceed by induction on $k$.
        Write $A$ for the term on the left-hand side of the claim.
        For $k=2$, we obtain
        \begin{align*}
		A &= E_1 E_2 F_{\overline{J}}E_3\cdots E_n\cdots E_2v\\
		&= E_1 \Big([K_2]_q F_3\cdots F_n\cdots F_1\\
		&\qquad+ F_2\cdots F_n\cdots F_3[K_2]_qF_1\Big )E_3\cdots E_n\cdots E_2v\\
		&\quad + E_1 F_{\overline{J}}E_2\cdots E_n\cdots E_2v\\
		&= -[r-1]_q E_1\Big(F_3\cdots F_n\cdots F_2\\
		&\qquad + F_2\cdots F_n\cdots F_3\Big)F_1E_3\cdots E_n\cdots E_2v\\
		&\quad + F_2\cdots F_n\cdots F_2[K_1]_q E_2\cdots E_n\cdots E_2v\\
		&\quad + F_{\overline{J}}E_Jv\\
		&= -[r-1]_q E_1\Big(F_3\cdots F_n\cdots F_2\\
		&\qquad + F_2\cdots F_n\cdots F_3)F_1E_3\cdots E_n\cdots E_2v\\
		&\quad+ [s-2]_q F_2\cdots F_n\cdots F_2 E_2\cdots E_n\cdots E_2v\\
		&\quad+ F_{\overline{J}}E_Jv.
	    \end{align*}
	    Note that $F_2E_3\cdots E_n\cdots E_2v=E_1E_3\cdots E_n\cdots E_2v=0$,
	    so that
	    \begin{align*}
		A &= -[r-1]_q\Big(F_3\cdots F_n\cdots F_2\\
		&\qquad + F_2\cdots F_n\cdots F_3)[K_1]_qE_3\cdots E_n\cdots E_2v\\
		&\quad- [s-2]_qF_2\cdots F_n\cdots F_3 [K_2]_q E_3\cdots E_n\cdots E_2v\\
		&\quad+ F_{\overline{J}}E_Jv\\
		&= \qty([r]_q[s-2]_q-[r-1]_q[s-1]_q)F_2\cdots F_n\cdots F_3E_3\cdots E_n\cdots E_2v
		+ F_{\overline{J}}E_Jv\\
		&= [r]_q[s-1-r]_qv
		+ F_{\overline{J}}E_Jv.
	    \end{align*}
	    Next, assume that $k>2$ and that the claim holds for $k-1$.
        \begin{itemize}
            \item If $k<n$, we have
    	    \begin{align*}
    		A&= E_{i_1}\cdots E_{i_k}F_{\overline{J}} E_{i_{k+1}}\cdots E_{i_{2n-2}}v\\
    		&= E_1\cdots E_JF_{\overline{J}} E_{k+1}\cdots E_n\cdots E_2v\\
    		&= E_1\cdots E_{k-1}\Big(F_2\cdots F_{k-1}[K_k]_qF_{k+1}\cdots F_n\cdots F_1\\
    		&\qquad + F_2\cdots F_n\cdots F_{k+1}[K_k]_qF_{k-1}\cdots F_1\Big)E_{k+1}\cdots E_n\cdots E_2v\\
    		&\quad + E_1\cdots E_{k-1}F_{\overline{J}}E_J\cdots E_n\cdots E_2v\\
    		&= E_1\cdots E_{k-1}F_{\overline{J}}E_k\cdots E_n\cdots E_2v,
    	    \end{align*}
    	    which equals $[r]_q[s-1-r]_qv + F_{\overline{J}}E_Jv$ by the induction
    	    hypothesis.
            
            \item If $k=n$, we have
    	    \begin{align*}
    		A &= E_1\cdots E_nF_{\overline{J}}E_{n-1}\cdots E_2v \\
    		&= E_1\cdots E_{n-1}F_2\cdots F_{n-1}[K_n]_{q^2}F_{n-1}\cdots F_1
    		E_{n-1}\cdots E_2v\\
    		&\quad + E_1\cdots E_{n-1}F_{\overline{J}} E_n\cdots E_2v\\
    		&= E_1\cdots E_{n-1}F_{\overline{J}}E_n\cdots E_2v,
    	    \end{align*}
    	    which equals $[r]_q[s-1-r]_qv + F_{\overline{J}}E_Jv$ by the induction
    	    hypothesis.

	        \item If $k>n$, we have
    	    \begin{align*}
    		A &=E_1\cdots E_n\cdots E_{2n-k} F_{\overline{J}} E_{2n-1-k}\cdots E_2v\\
    		&= E_1\cdots E_n\cdots E_{2n+1-k}\Big(
    		F_2\cdots F_{2n-1-k}[K_{2n-k}]_q F_{2n+1-k}\cdots F_n\cdots F_1\\
    		&\qquad + F_2\cdots F_n\cdots F_{2n+1-k}[K_{2n-k}]_q F_{2n-1-k}\cdots F_1\Big)
    		E_{2n-1-k}\cdots E_2v\\
    		&\quad + E_1\cdots E_n\cdots E_{2n+1-k} F_{\overline{J}}
    		E_{2n-k}\cdots E_2v\\
    		&= E_1\cdots E_n\cdots E_{2n+1-k}F_{\overline{J}}
    		E_{2n-k}\cdots E_2v,
    	    \end{align*}
    	    which equals $[r]_q[s-1-r]_qv + F_{\overline{J}}E_Jv$ by the induction
    	    hypothesis.
        \end{itemize}
        \item We have
        \begin{align*}
            E_JF_{\overline{J}}v &= E_1\cdots E_n\cdots E_3\Big(
    		[K_2]_q F_3\cdots F_n\cdots F_1v\\
    		&\qquad+ F_2\cdots F_n\cdots F_3[K_2]_qF_1v\Big)\\
    		&\quad+ E_1\dots E_n\cdots E_3F_{\overline{J}}E_2v\\
    		&= -[r-1]_qE_1\cdots E_n\cdots E_3F_3\cdots F_n\cdots F_1v
    		+ E_1\cdots E_n\cdots E_3F_{\overline{J}}E_2v\\
    		&= -[r-1][s-r]_qv + [r]_q[s-r-1]_qv + F_{\overline{J}}E_Jv\\
    		&= [s-2r]_qv + F_{\overline{J}}E_Jv.
        \end{align*}
        by (i) and (iii).
    \end{enumerate}
\end{proof}

\begin{proof}[Proof of Lemma~\ref{lem-sl2-triple} for $\mathsf{CII}_{n,1}$]
    Similar to the $\mathsf{BII}_n$ case, using Proposition~\ref{prop-CII-computations}.
\end{proof}

\subsection{Lemma~\ref{lem-B-action-V'}}\label{ap:lem-B-action}
We begin with $\mathsf{BII}_n$.
\begin{proof}[Lemma~\ref{lem-B-action-V'} for $\mathsf{BII}_n$]
    We have
    \[
        B_{\overline{J}}
        = F_{\overline{J}} + c\ad(E_2\cdots E_n^2\cdots E_2)(E_1)K_1^{-1}F_2\cdots F_n.
    \]
    This and the weights of the two terms shows that
    $B_{\overline{J}}F_{\overline{J}}^r\eta$ can be expanded
    as claimed, with $b_r=0$ and some yet-undetermined constant
    $c_r$.

    Let now $v\in V'_r$,
    then $F_2\cdots F_nv$ has weight $s\omega_n-(r-1)\omega_1-\omega_2=\qty(\frac{s}{2}-r,\frac{s}{2}-1,\frac{s}{2},\dots)$.
    This shows that applying $K_1^{-1}$ to it yields $q^{2r-2}$.
    Furthermore, $E_iF_2\cdots F_nv\ne0$ implies $i=1$ or 2.
    Moreover,
    $V'_{r-1}\cap\im(E_i)$ is empty unless $i=n$.
    Therefore, we can neglect any term that doesn't have $E_n$
    on the left:
    This shows that
    \begin{align*}
        c_rF_{\overline{J}}^{r-1}\eta &=
        cq^{2r-2}\ad(E_2\cdots E_n^2\cdots E_2)(E_1)
        F_2\cdots F_nF_{\overline{J}}^r\eta\\
        &= (-1)^n cq^{2r-2}
        K_2\cdots K_{n-1}\ad(E_n^2\cdots E_2)(E_1)
        K_{n-1}E_{n-1}\cdots K_2E_2F_2\cdots F_nF_{\overline{J}}^r\eta\\
        &= (-1)^n cq^{2r-2}
        \ad(K_2)(\cdots\ad(K_{n-1})(\ad(E_n^2\cdots E_2)(E_1))E_{n-1}\cdots)E_2F_2\cdots F_nF_{\overline{J}}^r\eta.
    \end{align*}
    Note that every $\ad(K_i)$ has inside it an expression
    of weight $\omega_1+\alpha_{i+1}+\cdots +\alpha_n$.
    Consequently, the $\ad(K_i)$ adds a factor of $q^{-2}$,
    hence
    \begin{align*}
        c_rF_{\overline{J}}^{r-1}\eta &= 
        (-1)^n cq^{2r+2-2n}
        \ad(E_n^2\cdots E_2)(E_1)E_{n-1}\cdots E_2F_2\cdots F_nF_{\overline{J}}^r\eta\\
        &= (-1)^n cq^{2r+2-2n}
        \ad(E_n^2\cdots E_2)(E_1)F_nF_{\overline{J}}^r\eta
    \end{align*}
    (as in the proof of Proposition~\ref{prop-higher-spin-computations}(i)).
    We continue:
    \begin{align*}
        c_rF_{\overline{J}}^{r-1}\eta &=
        (-1)^n cq^{2r+2-2n}
        \Big(E_n^2\ad(E_{n-1}\cdots E_2)(E_1)\\
        &\qquad- (K_nE_n+E_nK_n)\ad(E_{n-1}\cdots E_2)(E_1)K_n^{-1}E_n\Big)F_nF_{\overline{J}}^r\eta\\
        &= (-1)^n cq^{2r+2-2n}
        \Big(E_n^2\ad(E_{n-1}\cdots E_2)(E_1)\\
        &\qquad- (1+q^{-2})E_n\ad(E_{n-1}\cdots E_2)(E_1)E_n\Big)F_nF_{\overline{J}}^r\eta.
    \end{align*}
    Since $\im(F_n)\cap V'=0$, we can ignore any terms with $F_n$
    on the left.
    Moreover, $E_i V'=0$ unless $i=1$.
    Consequently,
    \begin{align*}
    c_rF_{\overline{J}}^{r-1}\eta &= (-1)^n cq^{2r+2-2n}
        \Big((E_n[K_n]_q+[K_n]_qE_n)\ad(E_{n-1}\cdots E_2)(E_1)\\
        &\qquad- (1+q^{-2})E_n\ad(E_{n-1}\cdots E_2)(E_1)[K_n]_q\Big)F_{\overline{J}}^r\eta\\
        &= (-1)^n cq^{2r+2-2n}
        \Big(([s-2]_q+[s]_q)E_J\\
        &\qquad- (1+q^{-2})[s]_qE_J\Big)F_{\overline{J}}^r\eta\\
        &= (-1)^ncq^{2r+2-2n-s}[2]_qE_JF_{\overline{J}}^r\eta\\
        &= (-1)^ncq^{2r+2-2n-s}[2]_q[r]_q[s+1-r]_q F_{\overline{J}}^{r-1}\eta.
    \end{align*}
    We now conclude that
    \[
        c_r = (-1)^ncq^{2r+2-2n-s}[2]_q[r]_q[s+1-r]_q
        = -\frac{(-1)^nq^{1-2n}}{(q-q^{-1})^2}[2]_q
        \qty(1-q^{2r})\qty(1-q^{2r-2-2s})
    \]
    as claimed.
\end{proof}

Now for $\mathsf{CII}_{n,1}$.
\begin{proof}[Lemma~\ref{lem-B-action-V'} for $\mathsf{CII}_{n,1}$]
    Let $v\in V'_r$, we compute $B_{\overline{J}}v$.
    We begin with $B_2F_1v$.
    We have
    \begin{align*}
	B_2F_1v &= F_2F_1v + c\ad(E_1E_3\cdots E_n\cdots E_3)(E_2)K_2^{-1}F_1v\\
	&= F_2F_1v + cq^{r-1}\ad(E_1E_3\cdots E_n\cdots E_3)(E_2)F_1v.
    \end{align*}
    Next, note that $E_iF_1v=F_1E_iv$ if $i\ne1$, which is zero unless $i=2$.
    Consequently, $E_iF_1v\ne0$ implies that $i=1$ or $i=2$.
    Therefore, all terms of $\ad(\cdots)(E_2)$ that don't have an $E_1$ or $E_2$ at the end vanish.
    This shows that
    \begin{align*}
	B_2F_1v &= F_2F_1v + cq^{r-1} E_1E_3\cdots E_n\cdots E_2F_1v
	- cq^{r-2} E_3\cdots E_n\cdots E_1F_1v\\
	&= F_2F_1v + cq^{r-1} E_1F_1 E_3\cdots E_n\cdots E_2v
	- cq^{r-2}[s]_q E_3\cdots E_n\cdots E_2v.
    \end{align*}
    Applying $F_3\cdots F_n\cdots F_3$, we obtain
    \begin{align*}
	F_3\cdots F_n\cdots F_3B_2F_1v &=
	F_3\cdots F_n\cdots F_1v \\&\quad + cq^{r-1} E_1F_1 F_3\cdots F_n\cdots F_3 E_3\cdots E_n\cdots E_2v\\
	&\quad - cq^{r-2}[s]_q F_3\cdots F_n\cdots F_3E_3\cdots E_n\cdots E_2v.
    \end{align*}
    By the proof of Proposition~\ref{prop-CII-computations}(ii), this equals
    \begin{align*}
	\cdots &= F_3\cdots F_n\cdots F_1v + cq^{r-1} E_1F_1E_2v
	- cq^{r-2}[s]_q E_2v\\
	&= F_3\cdots F_n\cdots F_1v - cq^{r-s-1} E_2v
    \end{align*}
    Applying $F_2$ to this, we obtain
    \[
	F_2\cdots F_n\cdots F_3B_2F_1v = F_{\overline{J}}v - cq^{r-s-1}[r]_qv. 
    \]
    We conclude that
    \[
	B_{\overline{J}}v = F_{\overline{J}}v - cq^{r-s-1}[r]_qv + cA
    \]
    where $A=\ad(E_1E_3\cdots E_n\cdots E_3)(E_2)K_2^{-1}F_3\cdots F_n\cdots F_3B_2F_1v$.
    Note that
    \[
	K_2^{-1}F_3\cdots F_n\cdots F_3B_2F_1v = q^{r-1}F_3\cdots F_n\cdots F_1v
	- cq^{2r-s-3} E_2 v.
    \]
    Next, we show that $\ad(E_1E_3\cdots E_n\cdots E_3)(E_2)$ contains the term $-q^{4-2n}E_1\cdots E_n\cdots E_3$.
    To achieve this, we take the first term of $\ad(E_1)$ and the last of every subsequent $\ad$, yielding
    \begin{align*}
	B &= -E_1K_3\cdots K_n\cdots K_3 E_2K_3^{-1}E_3\cdots K_n^{-1}E_n\cdots K_3^{-1}E_3\\
	&= - E_1 \ad(K_3)(\cdots \ad(K_n)(\cdots \ad(K_3)(E_2)E_3\cdots )E_n\cdots )E_3.
    \end{align*}
    Considering the innermost brackets, the term inside $\ad(K_i)$ for $i<n$ has weight $\alpha_2+\cdots+\alpha_{i-1}$, so its pairing with
    $\epsilon_ih_i$ is $-1$.
    This shows that
    \[
	B = - q^{3-n}E_1\ad(K_3)(\cdots\ad(K_n)(E_2\cdots E_{n-1})E_n\cdots)E_3.
    \]
    Next, $E_2\cdots E_{n-1}$ has weight $\alpha_2+\cdots+\alpha_{n-1}$, whose pairing
    with $\epsilon_nh_n = 2h_n$ is $-2$, so that
    \[
	B = -q^{1-n}E_1\ad(K_3)(\cdots\ad(K_{n-1})(E_2\cdots E_n)E_{n-1}\cdots)E_3.
    \]
    Now, the term inside every $\ad(K_i)$ has weight
    \[
	\alpha_2+\cdots+\alpha_i+2\alpha_{i+1}+\cdots+2\alpha_{n-1}+\alpha_n
    \]
    whose pairing with $\epsilon_ih_i=h_i$ is $-1$, so that
    \[
	B = -q^{4-2n}E_1\cdots E_n\cdots E_3.
    \]
    Having established this, we note that the two terms of $F_3\cdots F_n\cdots F_3B_2F_1v$
    (of weights $s\omega_1-(r+1)\omega_2 + \alpha_2$ and $s\omega_1-r\omega_2+\alpha_2$, respectively)
    give rise to two terms in $A$ of weights $s\omega_1 - r\omega_2$ and $s\omega_1-(r-1)\omega_2$.
    For both weights, the only $E_i$ whose image has a non-zero intersection with that weight
    space is $E_1$.
    Similarly, the image of $E_1E_i$ intersects non-trivially with these two weight spaces iff $i=2$
    (or $i=1$).
    Otherwise, $E_1E_i=E_iE_1$, but $E_i$'s image's intersection with both weight spaces is trivial.
    This shows that of the terms of $\ad(E_1E_3\cdots E_n\cdots E_3)(E_2)$, only $B$ contributes non-trivially
    to $A$.
    In particular,
    \begin{align*}
	A &= q^{r-1}BF_3\cdots F_n\cdots F_1v
	- cq^{2r-s-3} BE_2 v\\
	&= -q^{r+3-2n}E_1\cdots E_n\cdots E_3F_3\cdots F_n\cdots F_1v
	+ cq^{2r-s+1-2n} E_1\cdots E_n\cdots E_2v\\
	&= -q^{r+3-2n} [s-r]_q v + cq^{2r-s+1-2n} E_Jv.
    \end{align*}
    In conclusion,
    \[
	B_{\overline{J}}v = F_{\overline{J}}v - cq^{r-s-1}[r]_qv - cq^{r+3-2n} [s-r]_q v + c^2q^{2r-s+1-2n} E_Jv,
    \]
    which implies
    \[
        B_{\overline{J}}F_{\overline{J}}^r\eta
        = F_{\overline{J}}^{r+1}\eta - c\qty(q^{r-s-1}[r]_q+q^{r+3-2n} [s-r]_q)F_{\overline{J}}^r\eta  + c^2q^{2r-s+1-2n}[r]_q[s+1-r]_q F_{\overline{J}}^{r-1}\eta,
    \]
    which yields the claimed expressions for $b_r,c_r$.
\end{proof}
\subsection{$\mathsf{DII}$ Example}\label{sec-DII-proofs}
We start by showing that a $k[2L]^{W_\Sigma}$-linear map
that is diagonal with respect to our basis, is in fact also 
triangular with respect to the Macdonald ordering.
\begin{proof}[Proof of Lemma~\ref{lem-DII-triangular}]
    We begin by expanding monomials in terms of basis elements:
    \[
        e^{(m+1)\omega_1} = fv_1 + gv_2,\qquad
        e^{-m\omega_1} = hv_1 + jv_2
    \]
    where $f$ contains $m_\mu$ for $\mu < (m+1)\omega_1$ in the
    dominance ordering, and
    \begin{align*}
        g &= q^{1-n}m_{m\omega_1} + \lot,\\
        h &= m_{m\omega_1} + \lot\\
        j &= -q^{1-n}m_{(m-1)\omega_1}+\lot.
    \end{align*}
    Then
    \begin{align*}
        C(e^{(m+1)\omega_1}) &\propto
        q^{n-1}fv_1 + q^{1-n}gv_2 = q^{n-1}e^{(m+1)\omega_1}+\lot\\
        C(e^{-m\omega_1}) &\propto
        q^{n-1}hv_1 + q^{1-n}jv_2 = q^{n-1}e^{-m\omega_1}
        + (q^{n-1}-q^{1-n})e^{m\omega_1}+ \lot
    \end{align*}
    ($\propto$ denotes proportionality).
    Since $-m\omega_1>m\omega_1$, we conclude that $C$ is triangular.
\end{proof}

Now we show that the $t^{-1}$ is monotonic.
\begin{proof}[Proof of Lemma~\ref{lem-DII-monotonic}]
    Let $\mu,\mu'\in X^+(\gamma)$ with $\mu>\mu'$. 
    We write $\mu=m\omega_1+b$ and $\mu'=m'\omega_1+b'$.
    
    In case $b=b'$, we have $m-m'\in 2\N$ and 
    $t^{-1}(\mu)_+ - t^{-1}(\mu')_+ = (m-m')\omega_1$, 
    which is a positive linear combination of roots. 
    
    In case $b=\omega_{n-1},b'=\omega_n$, we have
    \begin{align*}
     &Q^+\ni\mu-\mu'=(m-m')\omega_1 + \omega_{n-1}-\omega_n
        \\
        &= \alpha_1+\cdots+\alpha_{n-1} + \frac{m-m'-1}{2}\qty(
        2\alpha_1+\cdots + 2\alpha_{n-2}+\alpha_{n-1}+\alpha_n).   
    \end{align*}
    This shows that $m-m'\in 1 + 2\N_0$. 
    We then have
    \[
        t^{-1}(\mu) = -m\omega_1,\qquad
        t^{-1}(\mu') = (m'+1)\omega_1.
    \]
    This shows that $t^{-1}(\mu)_+\ge t^{-1}(\mu')_+$. 
    Together with the fact that $t^{-1}(\mu)$ is antidominant 
    (hence the larger element in its Weyl orbit), 
    this implies that $t^{-1}(\mu)>t^{-1}(\mu')$.
    
    In case $b=\omega_n,b'=\omega_{n-1}$, we again conclude that $m-m'\in 1+2\N_0$.
    This time, however, we have
    \[
        t^{-1}(\mu)=(m+1)\omega_1,\qquad
        t^{-1}(\mu')=-m'\omega_1,
    \]
    so in particular $t^{-1}(\mu)_+>t^{-1}(\mu')_+$ and hence $t^{-1}(\mu)>t^{-1}(\mu')$.
\end{proof}

\subsection{$\mathsf{A}_2$ Examples}\label{sec-A2-proofs}
We start by showing that $t^{-1}$ is monotonic.
\begin{proof}[Proof of Lemma~\ref{lem-A2-monotonic}]
Let $\mu\in 2L_{+,J}$, we first show that there exists $\lambda\in 2L_+$
    such that either $\mu=\varpi_1+\lambda$, or $\mu=s_1(\varpi_2+\lambda)$,
    or $\mu=s_1s_2\lambda$. 
    Let $\mu=a\varpi_1+b\varpi_2$.
    That $\mu\in 2L_{+,J}$ means that $b\ge0$.
    In case $a>0$, we have $a\ge1$, so that $\mu\in \varpi_1+2L_+$.
    In case $a\le0$ and $a+b>0$, we have
    \[
        s_1\mu = -a\varpi_1 + (a+b)\varpi_2.
    \]
    Since $(a+b)\ge1$, we have $s_1\mu\in\varpi_2+2L_+$,
    so that $\mu \in s_1(\varpi_2+2L_+)$.

    In case $a,a+b\le0$, we have
    \[
        s_2s_1\mu = b\varpi_1 - (a+b)\varpi_2 \in 2L_+,
    \]
    so that $\mu\in s_1s_2(2L_+)$.

    Since only one of these three cases can occur, we conclude that
    \[
        2L_{+,J} = \qty(\varpi_1 + 2L_+)\sqcup s_1(\varpi_2+2L_+)
        \sqcup s_2s_1(2L_+).
    \]
    This shows that $t$ is well-defined.
    Since moreover it is a bijection on each subset and the images are
    disjoint, it is a bijection.

    Let $\mu'=b'+\lambda'<\mu=b+\lambda$ in $X^+$.
    Write $\mu-\mu'=\alpha\in Q^+$.

    First note that $Q^+\cap 2L=2(\Z\Sigma)^+$ (and in
    particular $Q\cap 2L = 2\Z\Sigma$) in all cases.
    Depending on $b,b'$ we can therefore see in which equivalence class modulo
    $Q\cap 2L$ the element $\alpha$ is contained.
    We obtain the following:
    \begin{enumerate}
        \item For $\mathsf{AI}_2$: If $b=b'$, we have $\alpha\in 2\Z\Sigma$.
        If $\set{b,b'}=\set{b_1,b_2}$, we have $\alpha\in 2\varpi_1-\varpi_2+2\Z\Sigma$.
        If $\set{b,b'}=\set{b_1,b_3}$, we have $\alpha\in \varpi_1+\varpi_2+2\Z\Sigma$.
        If $\set{b,b'}=\set{b_2,b_3}$, we have $\alpha\in -\varpi_1+2\varpi_2+
        2\Z\Sigma$.
        \item For the $\mathsf{A}_2$ group case, we obtain the following table:\\
        \begin{tabular}{r | c | c | c}
             $\downarrow b',b\rightarrow$ & $b_1$ & $b_2$ & $b_3$\\\hline
             $b_1$ & $2\Z\Sigma$ & $\alpha_3+2\Z\Sigma$ & $\alpha_3+\alpha_4+2\Z\Sigma$\\\hline
             $b_2$ & $\alpha_1+2\Z\Sigma$ & $2\Z\Sigma$ & $\alpha_4+2\Z\Sigma$\\\hline
             $b_3$ & $\alpha_1+\alpha_2+2\Z\Sigma$ & $\alpha_2+2\Z\Sigma$ & $2\Z\Sigma$.
        \end{tabular}
        \item For $\mathsf{AII}_5$, we obtain the following:\\
        \begin{tabular}{r | c | c | c}
             $\downarrow b',b\rightarrow$ & $b_1$ & $b_2$ & $b_3$\\\hline
             $b_1$ & $2\Z\Sigma$ & $\alpha_2+\alpha_3+2\Z\Sigma$ & $\alpha_2+\cdots+\alpha_5+2\Z\Sigma$\\\hline
             $b_2$ & $\alpha_1+\alpha_2+2\Z\Sigma$ & $2\Z\Sigma$ & $\alpha_4+\alpha_5+2\Z\Sigma$\\\hline
             $b_3$ & $\alpha_1+\cdots+\alpha_4+2\Z\Sigma$ & $\alpha_3+\alpha_4+2\Z\Sigma$ & $2\Z\Sigma$.
        \end{tabular}
    \end{enumerate}
    In particular, since $\alpha\in Q^+$, we can add $^+$ to all $2(\Z\Sigma)$'s appearing in these formulae.
    Moreover, from $\lambda-\lambda'=b'-b+\alpha$ and the aforementioned expressions for $\alpha$, we can now conclude
    the following information on $\lambda-\lambda'$:\\
    \begin{tabular}{r | c | c | c}
         $\downarrow b',b\rightarrow$ & $b_1$ & $b_2$ & $b_3$\\\hline
         $b_1$ & $2(\Z\Sigma)^+$ & $\varpi_1-\varpi_2+2(\Z\Sigma)^+$ & $\varpi_1+2(\Z\Sigma)^+$\\\hline
         $b_2$ & $\varpi_1+2(\Z\Sigma)^+$ & $2(\Z\Sigma)^+$ & $\varpi_2+2(\Z\Sigma)^+$\\\hline
         $b_3$ & $\varpi_2+2(\Z\Sigma)^+$ & $-\varpi_1+\varpi_2+2(\Z\Sigma)^+$ & $2(\Z\Sigma)^+$.
    \end{tabular}\\
    Moreover, we obtain the following expressions for $t^{-1}(\mu)_+-t^{-1}(\mu')_+$:\\
    
   \begin{tabular}{r | c | c | c}
         $\downarrow b',b\rightarrow$ & $b_1$ & $b_2$ & $b_3$\\\hline
         $b_1$ & $\underbrace{\lambda-\lambda'}_{\in 2(\Z\Sigma)^+}$ & $\underbrace{\lambda-\lambda'-\varpi_1+\varpi_2}_{\in 2(\Z\Sigma)^+}$ & $\underbrace{\lambda-\lambda'-\varpi_1}_{\in 2(\Z\Sigma)^+}$\\\hline
         $b_2$ & $\underbrace{\lambda-\lambda'+\varpi_1-\varpi_2}_{\in 2\varpi_1-\varpi_2+2(\Z\Sigma)^+}$ & $\underbrace{\lambda-\lambda'}_{\in 2(\Z\Sigma)^+}$ & $\underbrace{\lambda-\lambda'-\varpi_2}_{\in 2(\Z\Sigma)^+}$\\\hline
         $b_3$ & $\underbrace{\lambda-\lambda'+\varpi_1}_{\in\varpi_1+\varpi_2+2(\Z\Sigma)^+}$ & $\underbrace{\lambda-\lambda'+\varpi_2}_{\in -\varpi_1+2\varpi_2+(\Z\Sigma)^+}$ & $\underbrace{\lambda-\lambda'}_{\in 2(\Z\Sigma)^+}$.
    \end{tabular}\\

    For the three strictly lower-diagonal cases, we conclude that $t^{-1}(\mu)_+>t^{-1}(\mu')_+$, and that therefore $t^{-1}(\mu)>t^{-1}(\mu')$.
    For the other cases, we have the non-strict inequality,
    but $t^{-1}(\mu)$ lies in a less dominant Weyl chamber than $t^{-1}(\mu')$, whence
    $t^{-1}(\mu)>t^{-1}(\mu')$.
\end{proof}

Then we show that a diagonal map (on the basis elements) is
also triangular with respect to the Macdonald order.
\begin{proof}[Proof of Lemma~\ref{lem-A2-triangular}]
    We define
    \[
        \widetilde{m}_{J,\lambda} := \sum_{w\in W_J} e^{w\lambda},\qquad
        \widetilde{m}_\lambda := \sum_{w\in W} e^{w\lambda}
    \]
    Then $\widetilde{m}_{J,\lambda}=nm_{J,\lambda}$ with $n=\#W_{J,\lambda}$ and
    $\widetilde{m}_\lambda = nm_\lambda$ for $n=\#W_{\lambda}$. 
    Consequently, $C$ is triangular with respect to $(m_{J,\lambda})_{\lambda\in L_{+,J}}$ if and
    only if it is triangular with respect to $(\widetilde{m}_{J,\lambda})_{\lambda\in L_{+,J}}$.
    
    Let $\lambda\in 2L_+$, then
    \begin{align*}
        \widetilde{m}_\lambda e_1 &= \widetilde{m}_\lambda = \widetilde{m}_{J,\lambda} + \widetilde{m}_{J,s_1\lambda} + \widetilde{m}_{J,s_1s_2\lambda}\\
        \widetilde{m}_\lambda e_{s_1} &=\frac{1}{1+\tau^{-2}}\big(
        \widetilde{m}_{J,\lambda + \varpi_2} + \widetilde{m}_{J,\lambda+\varpi_1-\varpi_2}+ \widetilde{m}_{J,s_1\lambda + \varpi_2} \\
        &\quad+ \widetilde{m}_{J,s_1\lambda+\varpi_1-\varpi_2}+ \widetilde{m}_{J,s_1s_2\lambda + \varpi_2} + \widetilde{m}_{J,s_1s_2\lambda + \varpi_1 -\varpi_2}\big)\\
        \widetilde{m}_{\lambda}e_{s_2s_1} &= \tau^2 \qty(\widetilde{m}_{J,\lambda+\varpi_1}
        + \widetilde{m}_{J,s_1\lambda + \varpi_1} + \widetilde{m}_{J,s_1s_2\lambda + \varpi_1}).
    \end{align*}
    If $\lambda$ is generic, then $\lambda-\varpi_1$ and $\lambda-\varpi_2$ are dominant, so we can rewrite the $J$-dominant weights
    occurring in $\widetilde{m}_\lambda e_{s_1}$ as
    \[
        \lambda + \varpi_2, \lambda + \varpi_1-\varpi_2,
        s_1(\lambda + \varpi_2),s_1(\lambda-\varpi_1),s_1s_2(\lambda + \varpi_1-\varpi_2),s_1s_2(\lambda-\varpi_1)
    \]
    and those in $\widetilde{m}_\lambda e_{s_2s_1}$ as $\lambda + \varpi_1,
    s_1(\lambda - \varpi_1+\varpi_2),s_1s_2(\lambda - \varpi_2)$. Consequently,
    we can write
    \begin{align*}
        \widetilde{m}_\lambda e_{s_1} &=\frac{1}{1+\tau^{-2}}\qty(
        \widetilde{m}_{J,\lambda + \varpi_2} + \widetilde{m}_{J,s_1(\lambda+\varpi_2)})
        + \sum_{\mu\in L_{+,J},\mu_+<\lambda+\varpi_2} k\widetilde{m}_{J,\mu}\\
        \widetilde{m}_\lambda e_{s_2s_1} &= \tau^2 \widetilde{m}_{J,\lambda+\varpi_1} + \sum_{\mu\in L_{+,J},\mu_+<\lambda+\varpi_1} k\widetilde{m}_{J,\mu}.
    \end{align*}
    In case $W_\lambda=\langle s_1\rangle$, the $J$-dominant weights contained in
    $\widetilde{m}_\lambda e_{s_1}$ are
    \[
        \lambda + \varpi_2,\lambda+\varpi_1-\varpi_2,s_1s_2(\lambda+\varpi_1-\varpi_2).
    \]
    and those in $\widetilde{m}_\lambda e_{s_2s_1}$ are $\lambda+\varpi_1,s_1s_2(\lambda-\varpi_2)$.
    Consequently, we have
    \begin{align*}
        \widetilde{m}_\lambda e_{s_1} &=\frac{2}{1+\tau^{-2}}
        \widetilde{m}_{J,\lambda + \varpi_2}
        + \sum_{\mu\in L_{+,J},\mu_+<\lambda+\varpi_2} k\widetilde{m}_{J,\mu}\\
        \widetilde{m}_\lambda e_{s_2s_1} &= 2\tau^2 \widetilde{m}_{J,\lambda+\varpi_1} + \sum_{\mu\in L_{+,J},\mu_+<\lambda+\varpi_1} k\widetilde{m}_{J,\mu}.
    \end{align*}
    If $W_\lambda=\langle s_2\rangle$, the $J$-dominant weights contained in
    $\widetilde{m}_\lambda e_{s_1}$ are
    \[
        \lambda + \varpi_2, s_1(\lambda+\varpi_2),s_1(\lambda-\varpi_1)
    \]
    and those in $\widetilde{m}_\lambda e_{s_2s_1}$ are
    $\lambda+\varpi_1,s_1(\lambda-\varpi_1+\varpi_2)$.
    Consequently, we have
    \begin{align*}
       \widetilde{m}_\lambda e_{s_1} &=\frac{2}{1+\tau^{-2}}
        \qty(\widetilde{m}_{J,\lambda + \varpi_2}
        + \widetilde{m}_{J,s_1(\lambda+\varpi_2)})
        + \sum_{\mu\in L_{+,J},\mu_+<\lambda+\varpi_2} k\widetilde{m}_{J,\mu}\\
        \widetilde{m}_\lambda e_{s_2s_1} &= \tau^2 \widetilde{m}_{J,\lambda+\varpi_1} + \sum_{\mu\in L_{+,J},\mu_+<\lambda+\varpi_1} k\widetilde{m}_{J,\mu}.
    \end{align*}
    And lastly,
    \[
        \widetilde{m}_0e_{s_1} = \frac{6}{1+\tau^{-2}}\widetilde{m}_{J,\varpi_2},
        \qquad
        \widetilde{m}_0e_{s_1s_2} = 3\tau^2\widetilde{m}_{J,\varpi_1}.
    \]
    We now proceed by induction. Assume that $C$ is triangular restricted to the
    span of
    \[
    \set{m_{J,\mu}\where \mu\in L_{+,J},\mu_+<\lambda}
    \]
    is triangular. We
    now show that the same is also true on $W\lambda$. If $\lambda$ is regular,
    we have
    \[
        \widetilde{m}_{J,\lambda} = \begin{cases}
            \tau^{-2} \widetilde{m}_{\lambda-\varpi_1}e_{s_2s_1} & \langle h_1,\lambda\rangle>2\\
            \frac{1}{2}\tau^{-2}\widetilde{m}_{\lambda-\varpi_1}e_{s_2s_1} &
            \langle h_1,\lambda\rangle=2
        \end{cases}
        + \sum_{\nu\in L_{+,J},\nu_+<\lambda} k\widetilde{m}_{J,\nu},
    \]
    so that
    \[
        C\widetilde{m}_{J,\lambda} = a_{s_2s_1}\widetilde{m}_{J,\lambda} +
        \sum_{\nu\in L_{+,J},\nu_+<\lambda} k\widetilde{m}_{J,\nu}
        = a_{s_2s_1}\widetilde{m}_{J,\lambda} + \lot
    \]
    (here we use $\lot$ to explicitly denote elements with weights in lower orbits).
    Furthermore,
    \[
        \widetilde{m}_{J,s_1\lambda} = \begin{cases}
            (1 + \tau^{-2})\widetilde{m}_{\lambda-\varpi_2}e_{s_1} & \langle h_2,\lambda\rangle>2\\
            \frac{1+\tau^{-2}}{2} \widetilde{m}_{\lambda-\varpi_2}e_{s_1} 
            & \langle h_2, \lambda\rangle=2
        \end{cases}
        - \widetilde{m}_{J,\lambda}
        + \sum_{\nu\in L_{+,J},\nu_+<\lambda}
        k\widetilde{m}_{J,\nu},
    \]
    so that
    \begin{align*}
        C\widetilde{m}_{J,s_1\lambda} &= a_{s_1}\widetilde{m}_{J,s_1\lambda}
        + (a_{s_1}-a_{s_2s_1})\widetilde{m}_{J,\lambda} + \lot \\
        C\widetilde{m}_{J,s_1s_2\lambda} &=
        C(\widetilde{m}_\lambda e_1 - m_{J,s_1\lambda} - m_{J,\lambda})
        = a_1\widetilde{m}_{J,s_1s_2\lambda} \\
        &\qquad\qquad\qquad\qquad\qquad\qquad\,\,\quad+ 
        (a_1-a_{s_1})(\widetilde{m}_{J,s_1\lambda} + \widetilde{m}_{J,\lambda})
        + \lot
    \end{align*}
    In case $W_\lambda = \langle s_1\rangle$, we
    have
    \[
        \widetilde{m}_{J,\lambda} = \begin{cases}
            \frac{1+\tau^{-2}}{2} \widetilde{m}_{\lambda-\varpi_2}e_{s_1} & \langle h_2,\lambda\rangle>2\\
            \frac{1+\tau^{-2}}{6} \widetilde{m}_{\lambda-\varpi_2}e_{s_1} &
            \langle h_2,\lambda\rangle=2
        \end{cases}
        + \lot
    \]
    so that
    \begin{align*}
        C\widetilde{m}_{J,\lambda} &= a_{s_1} \widetilde{m}_{J,\lambda} + \lot\\
        C\widetilde{m}_{J,s_1s_2\lambda} &= C(\widetilde{m}_\lambda e_1 - 2\widetilde{m}_{J,\lambda}) = a_1 \widetilde{m}_{J,s_1s_2\lambda}
        + 2(a_1-a_{s_1})\widetilde{m}_{J,\lambda} + \lot
    \end{align*}
    In case $W_\lambda = \langle s_2\rangle$, we have
    \[
        \widetilde{m}_{J,\lambda} = \begin{cases}
            \tau^{-2}\widetilde{m}_{\lambda-\varpi_1}e_{s_2s_1} & 
            \langle h_1,\lambda\rangle>2\\
            \frac{1}{3\tau^2}\widetilde{m}_{\lambda-\varpi_1}e_{s_2s_1} & \langle h_1,\lambda\rangle=2
        \end{cases} + \lot
    \]
    so that
    \begin{align*}
        C\widetilde{m}_{J,\lambda} &= a_{s_2s_1} \widetilde{m}_{J,\lambda} + \lot\\
        C\widetilde{m}_{J,s_1\lambda} &= \frac{1}{2}C(\widetilde{m}_\lambda e_1
        - \widetilde{m}_{J,\lambda})
        = a_1\widetilde{m}_{J,s_1\lambda}
        + \frac{a_1-a_{s_2s_1}}{2}\widetilde{m}_{J,\lambda} + \lot
    \end{align*}
    And lastly, if $\lambda=0$ we have $C\widetilde{m}_{J,\lambda} = 2Ce_1=a_1\widetilde{m}_{J,\lambda}$.

    We conclude that $C$ is also triangular on the span of $\set{m_{J,\nu}\where \nu\in L_{+,J},\nu_+\le\lambda}$.
\end{proof}

Now we show that our correspondence of functions is also
monotonic.
\begin{proof}[Proof of Lemma~\ref{lem-A2-Gamma-t}]
    We proceed on a case-by-case basis.
    We have $\Gamma^{-1}(m_\lambda\otimes e_{b_3})=m_\lambda$. 
    If $\lambda$ is generic, this equals $m_{J,\lambda}+m_{J,s_1\lambda}+m_{J,s_1s_2\lambda}$,
    whose leading term is $m_{J,s_1s_2\lambda}$. 
    If $W_\lambda$ is generated by $s_1$ we have
    $m_{\lambda}=m_{J,\lambda}+m_{J,s_1s_2\lambda}$ whose leading term is
    again $m_{J,s_1s_2\lambda}$. 
    If $W_\lambda$ is generated by $s_2$, we
    have $m_\lambda = m_{J,\lambda}+ m_{J,s_1\lambda}$, whose leading term
    is $m_{J,s_1\lambda}=m_{J,s_1s_2\lambda}$. 
    Lastly, if $W_\lambda$ is generated by $s_1,s_2$, we have $m_\lambda = m_{J,\lambda}=m_{J,s_1s_2\lambda}$. 
    In all cases, the leading exponent is
    $s_1s_2\lambda=t^{-1}(b_3+\lambda)$.

    Next, we have
    \[
        \Gamma^{-1}(m_\lambda\otimes e_{b_2})
        = m_\lambda e_2 = \frac{m_\lambda e^{\varpi_2} + m_\lambda e^{\varpi_1-\varpi_2}}{1+\tau^{-2}}.
    \]
    If $\lambda$ is generic, then both $\lambda-\varpi_1,\lambda+\varpi_1-\varpi_2$ are dominant.
    The above expression then
    contains monomials contained in $W(\lambda+\varpi_1),W(\lambda+\varpi_1-\varpi_2),W(\lambda-\varpi_1)$.
    Since $\lambda+\varpi_2$ dominates the others, only the
    terms contained in the $\lambda+\varpi_2$ orbit are relevant for the
    leading term. 
    In particular, these terms are
    \begin{align*}
     &\frac{1}{1+\tau^{-2}}\qty(e^{\lambda+\varpi_2} + e^{s_1\lambda + \varpi_2}
        + e^{s_2\lambda + \varpi_1-\varpi_2} + e^{s_2s_1\lambda + \varpi_1-\varpi_2})\\
        &= \frac{1}{1+\tau^{-2}}\qty(m_{J,\lambda+\varpi_2} +
        m_{J,s_1(\lambda + \varpi_2)}).   
    \end{align*}
    If $W_\lambda$ is generated by $s_1$, we have
    \[
        m_\lambda e_2 =
        \frac{1}{1+\tau^{-2}}\qty(m_{J,\lambda+\varpi_2} + m_{J,\lambda+\varpi_1-\varpi_2}
        + m_{J,s_1s_2(\lambda+\varpi_1-\varpi_2)}).
    \]
    If $W_\lambda$ is
    generated by $s_2$, we have
    \[
        m_\lambda e_2 =
        \frac{1}{1+\tau^{-2}}
        \qty(m_{J,\lambda+\varpi_2} +
        m_{J,s_1(\lambda+\varpi_2)}
        + m_{J,s_2s_1(\lambda-\varpi_1)}).
    \]
    Lastly, for $\lambda=0$ we have $m_\lambda e_2=e_2
    = \frac{m_{J,\varpi_2}}{1+\tau^{-2}}$.
    Note that in every case, the leading exponent is $s_1(\lambda+\varpi_2)=t^{-1}(b_2+\lambda)$.

    Finally, note that
    \[
        \Gamma^{-1}(m_\lambda\otimes e_{b_1})
        = m_\lambda e_3 = \tau^2 m_\lambda e^{\varpi_1}.
    \]
    For $\lambda$ generic, this equals
    \[
        \tau^2\sum_{w\in W_\Sigma} e^{w\lambda + \varpi_1}
        = \tau^2\qty(m_{J,\lambda + \varpi_1} + m_{J,s_1(\lambda - \varpi_1+\varpi_2)}
        + m_{J,s_1s_2(\lambda-\varpi_2)})
    \]
    (note that $\lambda-\varpi_2,\lambda-\varpi_1+\varpi_2$ are dominant as
    well). 
    If $W_\lambda$ is generated by $s_1$ we have
    \[
        m_\lambda e_3 = \tau^2 \qty(e^{\lambda+\varpi_1} + e^{s_2\lambda + \varpi_1}
        + e^{s_1s_2\lambda + \varpi_1})
        = \tau^2\qty(m_{J,\lambda+\varpi_1} + m_{J,s_1s_2(\lambda - \varpi_2)}).
    \]
    If $W_\lambda$ is generated by $s_2$, we have
    \[
        m_\lambda e_3 = \tau^2\qty(e^{\lambda+\varpi_1} + e^{s_1\lambda + \varpi_1} + e^{s_2s_1\lambda + \varpi_1})
        = \tau^2\qty(m_{J,\lambda + \varpi_1} + m_{J,s_1(\lambda - \varpi_1+\varpi_2)}),
    \]
    and if $\lambda=0$, we have $m_\lambda e_3 = \tau^2 m_{J,\varpi_1}$. In
    all cases the leading exponent is $\lambda + \varpi_1=t^{-1}(b_1+\lambda)$. 
    This shows that $\Gamma^{-1}$ is triangular.
\end{proof}

\section{Classifying Integrable Small $\uqb$-Types}\label{ap:small}
The small $K$-types of real Lie groups are classified in \cite{smallK},
and similar to what we do in Section~\ref{sec-singlevarsmalk}, their spherical functions
are identified with appropriate Heckman--Opdam hypergeometric functions in loc.cit.

We shall now translate their findings in the language of Satake diagrams and weights
as in \cite{Pe23}.
The following $K$-types are small:
\begin{description}
    \item[For $\mathsf{AI}_n$] If $n=1$: $l\varpi_1$ for $l\in\Z$. If $n>1$ odd:
    $\varpi_{\frac{n+1}{2}-1},\varpi_{\frac{n+1}{2}}$.
    If $n>1$ even: $\varpi_{\frac{n}{2}}$.
    \item[For $\mathsf{AII}_n$] None.
    \item[For $\mathsf{AIII}_{n,p}$] $l\varpi_p$ for $l\in\Z$.
    \item[For $\mathsf{AIV}_n$] $l\varpi_1$ for $l\in\Z$.
    \item[For $\mathsf{BI}_{n,p}$] (i.e. $\mathfrak{k}\cong\mathfrak{so}(p)\times\mathfrak{so}(2n+1-p)$).
    If $p=2$: $l\varpi_0$ for $l\in\Z$.
    For $p>2$ even: $\varpi_{\frac{p}{2}-1},\varpi_{\frac{p}{2}}$.
    For $p>2$ odd: $\varpi_{\frac{p-1}{2}},\varpi_{\frac{2n-1-p}{2}}',\varpi_{\frac{2n+1-p}{2}}'$.
    \item[For $\mathsf{BII}_n$] $s\varpi_{n-1},s\varpi_n$ for $s\in\N_0$.
    \item[For $\mathsf{CI}_n$] $l\varpi_{n+1}$ for $l\in\Z$.
    \item[For $\mathsf{CII}_{n,p}$] If $p=1$ and $n>2$: $s\varpi_1$ for $s\in\N_0$.
    If $p=1$ and $n=2$: $s\varpi_1,s\varpi_1'$ for $s\in\N_0$.
    \item[For $\mathsf{DI}_{n,p}$] If $p=2$: $l\varpi_0$ for $l\in\Z$.
    If $p<n$ and $p$ even: $\varpi_{\frac{p}{2}-1},\varpi_{\frac{p}{2}}$.
    If $p<n$ and $p$ odd: $\varpi_{\frac{p-1}{2}}$
    If $p=n$ and $n$ even: $\varpi_{\frac{n}{2}-1},\varpi_{\frac{n}{2}},\varpi_{\frac{n}{2}-1}',\varpi_{\frac{n}{2}}'$.
    If $p=n$ and $n$ odd: $\varpi_{\frac{n-1}{2}},\varpi_{\frac{n-1}{2}}'$.
    \item[For $\mathsf{DII}_n$] None.
    \item[For $\mathsf{DIII}_n$] $l\varpi_{n+1}$ for $l\in\Z$.
    \item[For $\mathsf{EI}$] $\varpi_1$.
    \item[For $\mathsf{EII}$] (i.e. $\mathfrak{k}\cong\mathfrak{su}(6)\times\mathfrak{su}(2)$)
    $\varpi_1'$.
    \item[For $\mathsf{EIII}$] (i.e. $\mathfrak{k}\cong\mathfrak{so}(10)\times\C$)
    $l\varpi_0$ for $l\in\Z$.
    \item[For $\mathsf{EIV}$] None.
    \item[For $\mathsf{EV}$] $\varpi_1,\varpi_7$.
    \item[For $\mathsf{EVI}$] (i.e. $\mathfrak{k}\cong\mathfrak{so}(12)\times\mathfrak{su}(2)$)
    $\varpi_1'$.
    \item[For $\mathsf{EVII}$] (i.e. $\mathfrak{k}\cong \mathfrak{e}_6\times\C$)
    $l\varpi_0$.
    \item[For $\mathsf{EVIII}$] $\varpi_1$.
    \item[For $\mathsf{EIX}$] (i.e. $\mathfrak{k}\cong\mathfrak{e}_7\times\mathfrak{su}(2)$)
    $\varpi_1'$.
    \item[For $\mathsf{FI}$] (i.e. $\mathfrak{k}\cong\mathfrak{sp}(6)\times\mathfrak{su}(2)$)
    $\varpi_1'$.
    \item[For $\mathsf{FII}$] None.
    \item[For $\mathsf{G}$] $\varpi_1,\varpi_1'$.
\end{description}
Any occurrence of $\ell\in\Z$ occurs for a character of $\mathfrak{k}$ for
\[
    \mathsf{AI}_1,\mathsf{AIII}_{n,p},\mathsf{AIV}_n,\mathsf{BI}_{n,2},
    \mathsf{CI}_n,\mathsf{DI}_{n,2},\mathsf{DIII}_n,\mathsf{EIII},\mathsf{EVII},
\]
the Hermitian cases.
In these cases, the only small $K$-types are the characters.

There are four other families, namely $\mathsf{BII}_n$ with $s\varpi_{n-1}$ ($s\in\N_0$),
$\mathsf{BII}_n$ with $s\varpi_n$ ($s\in\N_0$),
$\mathsf{CII}_{n,1}$ with $s\varpi_1$ ($s\in\N_0$),
and $\mathsf{CII}_{2,1}$ with $s\varpi_1'$ ($s\in\N_0$).
Note that the Satake diagrams $\mathsf{CII}_{2,1}$ and $\mathsf{BII}_2$ are identical
(exchanging 1 and 2),
and under this isomorphism, $\varpi_1,\varpi_1'$ correspond to $\varpi_1,\varpi_2$, so we can
assume that $n>2$ when dealing with $\mathsf{CII}_{n,1}$.

The remaining cases are the following \enquote{sporadic/exceptional cases}:
spin representations $\varpi_{n-1},\varpi_n$ or $\varpi_n$ when part of $\mathfrak{k}$ is of type $\mathfrak{so}(2n)$ or $\mathfrak{so}(2n+1)$ and 
\begin{align}\label{eq-sporadic-small-k-types}
    &(\mathsf{EI},\varpi_1),(\mathsf{EII},\varpi_1'),(\mathsf{EV},\varpi_1),(\mathsf{EV},\varpi_7),
    (\mathsf{EVI},\varpi_1'), (\mathsf{EVIII},\varpi_1),(\mathsf{EIX},\varpi_1')\\\nonumber
    &(\mathsf{FI},\varpi_1'),(\mathsf{G},\varpi_1),(\mathsf{G},\varpi_1').
\end{align}
We shall show that none of these exceptional cases is integrable.

\begin{proposition}\label{prop-k-types-nonintegrable}
    Let $(\mathfrak{g},\mathfrak{k},\lambda)$ correspond to any of the cases from \eqref{eq-sporadic-small-k-types}.

    Let $\mathfrak{t}\le\mathfrak{k}$ and
    $\mathfrak{h}\le\mathfrak{g}$ be Cartan subalgebras satisfying $\mathfrak{t}\le\mathfrak{h}$.
    I.e. we pick $\mathfrak{h}$ to be maximally compact.
    Then there exists no integral weight $\mu$ such that $\mu|_{\mathfrak{t}}=\lambda$.
\end{proposition}
\begin{proof}
    Consulting \cite[Appendix C]{Kna02}, we obtain a convention to realise the simple roots
    of $\mathfrak{k}$ inside $\mathfrak{h}$.

    We now consider the cases.
    For every case, we define 
    simple roots $\alpha_1,\dots,\alpha_n$ that
    generate the root system $\Sigma(\mathfrak{g}:\mathfrak{h})$
    and simple roots $\beta_1,\dots,\beta_r$ that generate the
    root system $\Sigma(\mathfrak{k}:\mathfrak{t})$.
    Let $\omega_1,\dots,\omega_n$ and $\tilde{\omega}_1,\dots,\tilde{\omega}_r$
    be the corresponding fundamental weights.

    Next, we assume that $\mu=\sum_{i=1}^n a_i\omega_i\in\mathfrak{h}^*$
    restrict to a certain weight and show that this means that $\mu$ cannot be
    integral.
    \begin{description}
        \item[Case $\mathsf{AI}_n$ for $n=2r-1>1$ odd]
        We define $\alpha_i:= e_i-e_{i+1}$ for $i=1,\dots,n$
        and
        \[
            \beta_i:= \begin{cases}
                \frac{1}{2}(\alpha_i+\alpha_{n+1-i}) & i<r\\
                \frac{1}{2}(\alpha_{r-1} + 2\alpha_r + \alpha_{r+1}) & i=r
            \end{cases}\qquad
            \beta_i^\vee = \begin{cases}
                \alpha_i^\vee + \alpha_{n+1-i}^\vee & i<r\\
                \alpha_{r-1}^\vee + 2\alpha_r^\vee + \alpha_{r+1}^\vee & i=r.
            \end{cases}
        \]
        \begin{itemize}
            \item In case $\mu|_{\mathfrak{t}}=\tilde{\omega}_{r-1}$, we have
            $a_{r-1}+a_{r+1}=1$ and $a_{r-1}+a_{r+1}+2a_r=0$, which implies
            $a_r\not\in\Z$.
            \item In case $\mu|_{\mathfrak{t}}=\tilde{\omega}_r$, we have
            $a_{r-1}+a_{r+1}=0$ and $a_{r-1}+a_{r+1}+2a_r=1$, which implies
            $a_r\not\in\Z$.
        \end{itemize}
        \item[Case $\mathsf{AI}_n$ for $n=2r>1$ even]
        With the same $\alpha_1,\dots,\alpha_n$ as before and
        \[
            \beta_i := \frac{1}{2}(\alpha_i+\alpha_{n+1-i})\quad (i\le r),\qquad
            \beta_i^\vee = \begin{cases}
                \alpha_i^\vee + \alpha_{n+1-i}^\vee & i < r\\
                2\alpha_r^\vee + 2\alpha_{r+1}^\vee & i=r
            \end{cases}
        \]
        Assume that $\mu|_{\mathfrak{t}}=\tilde{\omega}_r$, then
        $2a_r+2a_{r+1}=1$, which means that at least one of $a_r,a_{r+1}$ is not an integer.
        \item[Case $(\mathfrak{so}(2n+1),\mathfrak{so}(2p)\times\mathfrak{so}(2q+1))$ for $1<p,q$]
        Let $\alpha_i:= e_i-e_{i+1}$ ($i<n$) and $\alpha_n:=e_n$.
        Let
        \begin{align*}
            \beta_i &:= \begin{cases}
                \alpha_i & i\ne p\\
                \alpha_{p-1} + 2\alpha_p + \cdots + 2\alpha_n & i=p
            \end{cases}\\
            \beta_i^\vee &= \begin{cases}
                \alpha_i^\vee & i\ne p\\
                \alpha_{p-1}^\vee + 2\alpha_p^\vee + \cdots + 2\alpha_{n-1}^\vee + \alpha_n^\vee & i=p.
            \end{cases}
        \end{align*}
        \begin{itemize}
            \item If $\mu|_{\mathfrak{t}}=\tilde{\omega}_{p-1}$, we have $a_i=0$ for $i\ne p-1,p$ as well
            as $a_{p-1}=1$ and $a_{p-1}+2a_p=0$, hence $a_p\not\in\Z$.
            \item If $\mu|_{\mathfrak{t}}=\tilde{\omega}_p$, we have $a_i=0$ for $i\ne p$ and
            $2a_p=1$, which again implies $a_p\not\in\Z$.
            \item If $\mu|_{\mathfrak{t}}=\tilde{\omega}_n$, we have $a_i=0$ for $i\ne p,n$ as well as
            $a_n=1$ and $2a_p+a_n=0$, which implies $a_p\not\in\Z$.
        \end{itemize}
        In no case is it possible to choose $\mu$ integral.
        This statement covers all cases of type $\mathsf{BI}_{n,p}$ for $p>2$.
        \item[Case $\mathsf{DI}_{n,2p}$ with $2<2p$]
        Let $\alpha_i:=e_i-e_{i+1}$ ($i<n$) and $\alpha_n:=e_{n-1}+e_n$.
        Let
        \[
            \beta_i := \begin{cases}
                \alpha_i & i\ne p\\
                \alpha_{p-1}+2\alpha_p + \cdots + 2\alpha_{n-2}+\alpha_{n-1}+\alpha_n & i=p
            \end{cases}
        \]
        Note that all roots are of the same length, so we can add $\vee$'s on all $\alpha$'s and
        $\beta$'s in the above statement and obtain a true statement.
        \begin{itemize}
            \item If $\mu|_{\mathfrak{t}}=\tilde{\omega}_{p-1}$ then $a_i=0$ for $i\ne p-1,p$ as well
            as $a_{p-1}=1$ and $a_{p-1}+2a_p=0$, which implies $a_p\not\in\Z$.
            \item If $\mu|_{\mathfrak{t}}=\tilde{\omega}_p$, then $a_i=0$ for $i\ne p$ and
            $2a_p=1$, which implies $a_p\not\in\Z$.
            \item If $\mu|_{\mathfrak{t}}=\tilde{\omega}_{n-1}$, then $a_i=0$ for $i\ne p,n-1$ as well as
            $a_{n-1}=1$ and $2a_p + a_{n-1}=0$, which implies $a_p\not\in\Z$.
            \item If $\mu|_{\mathfrak{t}}=\tilde{\omega}_n$, then $a_i=0$ for $i\ne p,n$ as well as
            $a_n=1$ and $2a_p + a_n=0$, which implies $a_p\not\in\Z$.
        \end{itemize}
        \item[Case $\mathsf{DI}_{n,2p+1}$ with $1\le p$]
        Let $\alpha_1,\dots,\alpha_n$ as before and let
        \begin{align*}
            \beta_i &:= \begin{cases}
                \alpha_i & i<p \text{ or } p<i\le n-2\\
                e_p = \alpha_p+\cdots+\alpha_{n-2}+\frac{\alpha_{n-1}+\alpha_n}{2} & i=p\\
                e_{n-1} = \frac{\alpha_{n-1}+\alpha_n}{2} & i=n-1.
            \end{cases}\\
            \beta_i^\vee &= \begin{cases}
                \alpha_i^\vee & i<p \text{ or } p<i\le n-2\\
                e_p^\vee = 2\alpha_p^\vee+\cdots+2\alpha_{n-2}^\vee+\alpha_{n-1}^\vee+\alpha_n^\vee & i=p\\
                e_{n-1}^\vee = \alpha_{n-1}^\vee+\alpha_n^\vee & i=n-1.
            \end{cases}
        \end{align*}
        Let $\mu\in\mathfrak{h}^*$ be expanded as $\mu=\sum_{i=1}^n a_i\omega_i$.
        \begin{itemize}
            \item If $\mu|_{\mathfrak{t}}=\tilde{\omega}_p$, we have
            $a_i=0$ for $i\ne p,n-1,n$.
            Moreover,
            \[
                2a_p + a_n+a_{n-1}=1,\qquad
                a_n+a_{n-1}=0,
            \]
            which implies $a_p\not\in\Z$.
            \item If $\mu|_{\mathfrak{t}}=\tilde{\omega}_{n-1}$, we have
            $a_i=0$ for $i\ne p,n-1,n$.
            Moreover,
            \[
                2a_p + a_n + a_{n-1}=0,\qquad
                a_n+a_{n-1}=1,
            \]
            which implies $a_p\not\in\Z$.
        \end{itemize}
        \item[Case $\mathsf{EI}$]
        Let $\alpha_1,\dots,\alpha_6$ as in \cite[Planche~V]{Bour68}.
        Let
        \[
            \beta_i := \begin{cases}
                \alpha_2+\alpha_4 + \frac{\alpha_3+\alpha_5}{2} & i=1\\
                \frac{\alpha_1+\alpha_6}{2} & i=2\\
                \frac{\alpha_3+\alpha_5}{2} & i=3\\
                \alpha_4 & i=4
            \end{cases}\qquad
            \beta_i^\vee = \begin{cases}
                2\alpha_2^\vee + \alpha_3^\vee + 2\alpha_4^\vee + \alpha_5^\vee & i=1\\
                \alpha_1^\vee + \alpha_6^\vee & i=2\\
                \alpha_3^\vee + \alpha_5^\vee & i=3\\
                \alpha_4^\vee & i=4.
            \end{cases}
        \]
        In case $\mu|_{\mathfrak{t}}=\tilde{\omega}_1$, we obtain
        $a_4=a_3+a_5=a_1+a_6=0$ and hence $2a_2=1$, which implies $a_2\not\in\Z$.
        \item[Case $\mathsf{EII}$]
        Let $\alpha_1,\dots,\alpha_6$ be as previously and let
        \begin{align*}
            \beta_i &:= \begin{cases}
                \alpha_i & i\ne2\\
                \alpha_1+\alpha_6+2(\alpha_2+\alpha_3+\alpha_5)+3\alpha_4=\omega_2 & i=2
            \end{cases}\\
            \beta_i^\vee &=
            \begin{cases}
                \alpha_i^\vee & i\ne2\\
                \alpha_1^\vee+\alpha_6^\vee+2(\alpha_2^\vee+\alpha_3^\vee+\alpha_5^\vee)+3\alpha_4^\vee & i=2.
            \end{cases}
        \end{align*}
        In case $\mu|_{\mathfrak{t}}=\tilde{\omega}_2$, we have $a_i=0$ for $i\ne2$ and
        $2a_2=1$, which implies $a_2\ne\in\Z$.
        \item[Case $\mathsf{EV}$]
        Let $\alpha_1,\dots,\alpha_7$ be as in \cite[Planche~VI]{Bour68} and let
        \[
            \beta_i := \begin{cases}
                \alpha_i & i\ne2\\
                \alpha_1+\alpha_6 + 2(\alpha_2+\alpha_3+\alpha_5) + 3\alpha_4 & i=2.
            \end{cases}
        \]
        Analogously to the last case, we obtain the the expressions for $\beta_1^\vee,\dots,\beta_7^\vee$.
        \begin{itemize}
            \item If $\mu|_{\mathfrak{t}}=\tilde{\omega}_1$,
            we have $a_i=0$ for $i\ne 1,2$ and $a_1+2a_2=0$,
            which shows that $a_2\not\in\Z$.
            \item If $\mu|_{\mathfrak{t}}=\tilde{\omega}_2$,
            we have $a_i=0$ for $i\ne 2$ and $2a_2=1$,
            which shows that $a_2\not\in\Z$.
        \end{itemize}
        Assume that $\mu|_{\mathfrak{t}}=\tilde{\omega}_1$ we have
        $a_i=0$ for $i\ne 1,2$ and $a_1=1$ and $a_1+2a_2=0$, which shows that
        $a_2\not\in\Z$.
        \item[Case $\mathsf{EVI}$]
        Let $\alpha_1,\dots,\alpha_7$ be as before and let
        \[
            \beta_i := \begin{cases}
                \alpha_i & i\ne1\\
                \omega_1
            \end{cases}\qquad
            \beta_i^\vee = \begin{cases}
                \alpha_i^\vee & i\ne1\\
                \alpha_7^\vee + 2(\alpha_1^\vee + \alpha_2^\vee + \alpha_7^\vee)
                + 3(\alpha_3^\vee + \alpha_5^\vee) + 4\alpha_4^\vee & i=1.
            \end{cases}
        \]
        Assume that $\mu|_{\mathfrak{t}}=\tilde{\omega}_1$, then $2a_1=1$, which 
        implies $a_1\not\in\Z$.
        \item[Case $\mathsf{EVIII}$]
        Let $\alpha_1,\dots,\alpha_8$ be as in \cite[Planche~VII]{Bour68} and let
        $\beta_i$ and $\beta_i^\vee$ be defined as in the last item (except $i$ can be 8).
        As before, if $\mu|_{\mathfrak{t}}=\tilde{\omega}_1$, we have $2a_1=1$, which implies
        that $a_1\not\in\Z$.
        \item[Case $\mathsf{EIX}$]
        Let $\alpha_1,\dots,\alpha_8$ be as before and let
        \[
            \beta_i := \begin{cases}
                \alpha_i & i\ne8\\
                \omega_8 & i=8,
            \end{cases}
        \]
        with
        \[
            \beta_8^\vee = 2(\alpha_1^\vee + \alpha_8^\vee) + 3(\alpha_2^\vee + \alpha_7^\vee)
            + 4(\alpha_3^\vee + \alpha_6^\vee) + 5\alpha_5^\vee + 6\alpha_4^\vee.
        \]
        If $\mu|_{\mathfrak{t}}=\tilde{\omega}_8$, we have
        $2a_8=1$, which implies $a_8\not\in\Z$.
        \item[Case $\mathsf{FI}$]
        Let $\alpha_1,\dots,\alpha_4$ be as in \cite[Plance~VIII]{Bour68} and let
        \[
            \beta_i := \begin{cases}
                \alpha_i & i\ne1\\
                \omega_1 & i=1
            \end{cases}\qquad
            \beta_i^\vee = \begin{cases}
                \alpha_i^\vee & i\ne1\\
                2\alpha_1^\vee + 3\alpha_2^\vee + 2\alpha_3^\vee + \alpha_4^\vee & i=1.
            \end{cases}
        \]
        If $\mu|_{\mathfrak{t}}=\tilde{\omega}_1$, then $2a_1=1$, which implies
        $a_1\not\in\Z$.
        \item[Case $\mathsf{G}$]
        Let $\alpha_1,\alpha_2$ be as in \cite[Planche~IX]{Bour68} and let
        \[
            \beta_i = \begin{cases}
                \alpha_1 & i=1\\
                \omega_2 & i=2
            \end{cases}\qquad
            \beta_i^\vee = \begin{cases}
                \alpha_1^\vee & i=1\\
                3\alpha_1^\vee + \frac{2}{3}\alpha_2^\vee & i=2.
            \end{cases}
        \]
        \begin{itemize}
            \item If $\mu|_{\mathfrak{t}}=\tilde{\omega}_1$, then
            $a_1=1$ and $3a_1 + \frac{2}{3}a_2=0$, i.e.
            $a_2=\frac{9}{2}\not\in\Z$.
            \item If $\mu|_{\mathfrak{t}}=\tilde{\omega}_2$, then
            $a_1=0$ and $\frac{2}{3}a_2=1$, i.e. $a_2=\frac{3}{2}\not\in\Z$.
        \end{itemize}
    \end{description}
\end{proof}

\begin{proof}[Proof of Lemma~\ref{lem-integrable-small-b-types}]
    Let $(\uq,\uqb,\gamma)$ be a commutative triple with $\gamma$ specialisable, then
    $\cl(\gamma)$ is a small $\cl(\uqb)$-type in the sense that for the
    classical branching rules of $(\cl(\uq),\cl(\uqb))$ we have
    $\#\mathfrak{B}(\cl(\gamma))=1$.
    But then $\cl(\gamma)$ is one of the small $K$-types listed earlier.
    Since $\gamma$ is integrable, so is $\cl(\gamma)$, so that by
    Proposition~\ref{prop-k-types-nonintegrable}, the only possibilities we are left with
    are $\gamma$ one dimensional or $(\uq,\uqb,\gamma)$ one of the three families from the claim.
\end{proof}


\newpage

\printnoidxglossary[type=symbols, style=long, title={List of Symbols}, toctitle={List of Symbols}]
\newpage
\printbibliography
\Addresses
\end{document}